\documentclass[11pt,reqno]{amsart}
\usepackage[foot]{amsaddr}
\usepackage[margin=1in]{geometry}
\usepackage{amsmath}
\usepackage{amssymb}
\usepackage{amsthm}
\usepackage{enumitem}
\usepackage{mathrsfs}

\usepackage[dvipsnames]{xcolor} 
\usepackage{graphicx}
\usepackage{float}
\definecolor{TabBlue}{HTML}{1F77B4}
\definecolor{TabPurple}{HTML}{9467BD}

\usepackage[
colorlinks=true,
citecolor=BurntOrange, 
linkcolor=TabBlue,     
urlcolor=TabBlue
]{hyperref}

\usepackage[backend=biber,style=alphabetic,sorting=nyt,maxalphanames=4,minalphanames=3,maxbibnames=99]{biblatex}
\usepackage{tikz}
\usetikzlibrary{shapes.geometric, positioning, decorations.pathreplacing, decorations.text, calc, bending, arrows.meta}

\tikzset{
    tensor_node/.style={circle, draw=black, very thick, fill=yellow!20, minimum size=1.2cm, font=\bfseries},
    leaf_node/.style={circle, draw=blue!80, thick, fill=blue!10, minimum size=1cm},
    eigen_node/.style={circle, draw=red!80, thick, fill=red!10, minimum size=1cm},
    matrix_node/.style={circle, draw=green!60!black, thick, fill=green!10, minimum size=1cm, font=\small},
    conn/.style={thick, draw=gray!80},
    double_conn/.style={thick, draw=green!40!black},
    arrow/.style={-{Latex[scale=1.5]}, line width=2pt, draw=blue!50!black}
}

\def\R{\mathbb{R}}
\def\t{\mathbf{t}}
\def\a{\mathbf{a}}
\def\g{\mathbf{g}}
\def\e{\mathbf{e}}
\def\s{\mathbf{s}}
\def\u{\mathbf{u}}
\def\v{\mathbf{v}}
\def\x{\mathbf{x}}
\def\y{\mathbf{y}}
\def\z{\mathbf{z}}
\def\w{\mathbf{w}}

\def\P{\mathbb{P}}
\def\C{\mathbb{C}}
\def\E{\mathbb{E}}
\def\N{\mathbb{N}}
\def\R{\mathbb{R}}
\def\1{\mathbf{1}}
\def\cA{\mathcal{A}}
\def\cU{\mathcal{U}}
\def\cX{\mathcal{X}}
\def\cY{\mathcal{Y}}
\def\cZ{\mathcal{Z}}
\def\cB{\mathcal{B}}
\def\cE{\mathcal{E}}
\def\cF{\mathcal{F}}
\def\cG{\mathcal{G}}
\def\cQ{\mathcal{Q}}
\def\cP{\mathcal{P}}
\def\cC{\mathcal{C}}
\def\cK{\mathcal{K}}
\def\cM{\mathcal{M}}
\def\cV{\mathcal{V}}
\def\cL{\mathcal{L}}
\def\cI{\mathcal{I}}
\def\cR{\mathcal{R}}
\def\cT{\mathcal{T}}
\def\cN{\mathcal{N}}
\def\bD{\mathbf{D}}

\def\Im{\operatorname{Im}}
\def\Re{\operatorname{Re}}
\def\op{\mathrm{op}}
\def\matop{{\mathrm{mat}\text{-}\mathrm{op}}}
\def\eps{\varepsilon}
\def\Oprec{O_{\prec}}
\def\deg{\text{deg}}
\def\val{\text{val}}
\def\vec{\operatorname{vec}}
\def\mat{\operatorname{mat}}
\def\Tr{\operatorname{Tr}}
\def\Var{\operatorname{Var}}
\def\diag{\operatorname{diag}}
\def\dist{\operatorname{dist}}
\def\supp{\operatorname{supp}}
\def\rank{\operatorname{rank}}

\def\<{\langle} 
\def\>{\rangle}

\newtheorem{theorem}{Theorem}[section]
\newtheorem{lemma}[theorem]{Lemma}
\newtheorem{coro}[theorem]{Corollary}

\newtheorem{proposition}[theorem]{Proposition}

\newtheorem{assumption}{Assumption}

\theoremstyle{definition}
\newtheorem{defi}[theorem]{Definition}
\newtheorem{example}[theorem]{Example}
\newtheorem{remark}[theorem]{Remark}

\numberwithin{equation}{section}
\numberwithin{figure}{section}

\makeatletter
\renewcommand{\paragraph}{%
  \@startsection{paragraph}{4}%
  {\z@}%
  {1em \@plus 0.5ex \@minus 0.2ex}%
  {-1em}%
  {\normalfont\normalsize\bfseries}%
}
\makeatother

\newcommand{\Prto}{\xrightarrow{\P}}

\title{Anisotropic local law for non-separable sample covariance matrices}
\author{Zhou Fan$^{\ast}$}
\email{zhou.fan@yale.edu}
\author{Renyuan Ma$^{\ast}$}
\address{$^{\ast}$Department of Statistics and Data Science, Yale University}
\email{jack.ma.rm2545@yale.edu}
\author{Elliot Paquette$^{\dagger}$}
\address{$^{\dagger}$Department of Mathematics and Statistics, McGill University}
\email{elliot.paquette@mcgill.ca}
\author{Zhichao Wang$^{\ddagger}$}
\address{$^{\ddagger}$Department of Statistics and ICSI, University of California Berkeley}
\email{zhichao.wang@berkeley.edu}

\begin{document}

\begin{abstract}
We establish local laws for sample covariance matrices $K = N^{-1}\sum_{i=1}^N \g_i\g_i^*$ where the random vectors $\g_1, \ldots, \g_N \in \R^n$ are independent with common covariance $\Sigma$. Previous work has largely focused on the separable model $\g = \Sigma^{1/2}\w$ with $\w$ having independent entries, but this structure is rarely present in statistical applications involving dependent or nonlinearly transformed data. Under a concentration assumption for quadratic forms $\g^*A\g$, we prove an optimal averaged local law showing that the Stieltjes transform of $K$ converges to its deterministic limit uniformly down to the optimal scale $\eta \geq N^{-1+\eps}$. Under an additional structural assumption on the cumulant tensors of $\g$ --- which interpolates between the highly structured case of independent entries and generic dependence --- we establish the full anisotropic local law, providing entrywise control of the resolvent $(K-zI)^{-1}$ in arbitrary directions. We discuss several classes of non-separable examples satisfying our assumptions, including conditionally mean-zero distributions, the random features model $\g = \sigma(X\w)$ arising in machine learning, and Gaussian measures with nonlinear tilting. The proofs introduce a tensor network framework for analyzing fluctuation averaging in the presence of higher-order cumulant structure.
\end{abstract}

\maketitle

\begingroup
\hypersetup{linkcolor=TabPurple}
\tableofcontents
\endgroup
\hypersetup{linkcolor=TabBlue} 

\section{Introduction}

\subsection{Sample covariance matrices and the deformed Marchenko-Pastur law}

Random matrices arise naturally across mathematics, statistics, and physics. In statistics, Wishart \cite{wishart1928generalised} introduced random covariance matrices in 1928 for multivariate analysis, while in physics, Wigner \cite{wigner1955characteristic} employed random matrices in 1955 to model energy levels in heavy nuclei.

A fundamental random matrix in statistics involves the observation of $N$ independent random vectors $\g_1,\ldots,\g_N \in\R^n$ equal in distribution to a vector $\g \in \R^n$ with mean zero and covariance matrix $\Sigma$. Understanding the spectral properties of the associated sample covariance matrix
\[
K = \frac{1}{N}\sum_{i=1}^N \g_i\g_i^* = \frac{1}{N}GG^*,
\qquad \text{where } G = [\g_1,\ldots,\g_N] \in \R^{n \times N},
\]
is central to covariance estimation, principal component analysis, and numerous other statistical procedures \cite{bai2010spectral}. The associated \emph{Gram matrix}
\[
\widetilde{K} = \frac{1}{N}G^*G \in \R^{N \times N}
\]
is central to kernel methods and has the same nonzero eigenvalues as $K$. Many problems in optimization theory involving random data, as well as related questions in machine learning, can also be reduced to understanding the singular value and singular vector structure of the data matrix $G$ \cite{belkin2019reconciling,bartlett2020benign,mei2022generalization,paquette2025homogenization}. Beyond classical statistics, the spectral theory of such matrices has found applications in wireless communications, where the channel capacity of MIMO systems is determined by the log-determinant $\frac{1}{n}\log\det(I + \text{SNR} \cdot K)$ \cite{tulino2004random}, as well as in kernel methods \cite{el2010spectrum}, random features regression \cite{rahimi2007random}, and the analysis of neural networks \cite{pennington2017nonlinear,jacot2018neural}; see \cite{couillet2022random} for a comprehensive introduction to random matrix theory in machine learning.

In the high-dimensional regime where the dimension $n$ and sample size $N$ grow proportionally, $n/N \to \gamma \in (0,\infty)$, the classical theory initiated by Marchenko and Pastur \cite{marvcenko1967distribution} describes the limiting spectral distribution of $K$. When $\Sigma = I_n$ and $\g$ has i.i.d.\ entries with mean zero and variance one, the empirical spectral distribution of $K$ converges weakly to the \emph{Marchenko-Pastur distribution} with density
\[
\rho_{\mathrm{MP}}(x) = (1-1/\gamma)_+\delta_0(x)+\frac{1}{2\pi\gamma x}\sqrt{(x - \lambda_-)(\lambda_+ - x)}\,\1_{[\lambda_-,\lambda_+]}(x),
\]
where $\lambda_\pm = (1 \pm \sqrt{\gamma})^2$ are the spectral edges.

For general covariance $\Sigma$, a deterministic approximation for the spectral distribution is the \emph{deformed Marchenko-Pastur law} $\mu_0$, characterized through its Stieltjes transform $m_0(z) = \int (x-z)^{-1}\,d\mu_0(x)$ which satisfies the self-consistent equation
\begin{equation}\label{eq:intro:MP}
m_0(z) = \frac{1}{n}\sum_{\alpha=1}^n \frac{1}{\sigma_\alpha(1-\gamma-\gamma zm_0(z))-z},
\end{equation}
where $\sigma_1,\ldots,\sigma_n$ are the eigenvalues of $\Sigma$. This law can also be understood as the multiplicative free convolution of the standard Marchenko-Pastur law with the empirical eigenvalue distribution of $\Sigma$ \cite{voiculescu1987multiplication}.

A sufficient condition for the convergence of the empirical spectral distribution of $K$ to the deformed Marchenko-Pastur law was established by Bai and Zhou \cite{bai2008large}: the weak convergence in probability of quadratic forms
\begin{equation}\label{eq:intro:quadraticform}
\frac{1}{n}(\g^*A\g - \Tr \Sigma A) \Prto 0,
\end{equation}
for all matrices $A$ with bounded operator norm. This condition is anticipated by the classical literature, as it is essentially all that is needed for the arguments of \cite{silverstein1995strong}; see also \cite{yaskov2016necessary} for a necessary and sufficient formulation in the isotropic case.
This condition is also readily verified in simple cases, most notably in the \emph{separable} (or \emph{linear}) model where $\g = \Sigma^{1/2}\w$ for a vector $\w$ with i.i.d.\ standard entries, the setting of the foundational work of Silverstein and Bai \cite{silverstein1995empirical}. 

\subsection{Local laws: from global to optimal scale}

While the convergence of the empirical spectral distribution provides a global description of the spectrum, many applications in random matrix theory --- including universality of local eigenvalue statistics, eigenvalue rigidity, and eigenvector delocalization --- require understanding the spectrum at \emph{local} scales, down to the scale of individual eigenvalue spacings.

An \emph{averaged local law} asserts that the Stieltjes transform $m(z) = n^{-1}\Tr(K - zI)^{-1}$ remains close to its deterministic approximation $m_0(z)$ uniformly over spectral parameters $z = E + i\eta$ with $\eta$ as small as $N^{-1+\eps}$ for any $\eps > 0$. Specifically, in the setting where $\g$ has identity covariance and satisfies concentration of quadratic forms, the celebrated work of Pillai and Yin \cite{pillai2014universality} established
\begin{equation}\label{eq:intro:averaged}
|m(z) - m_0(z)| \prec (N\eta)^{-1}
\end{equation}
uniformly over $z$ in suitable regular spectral domains extending down to the optimal scale $\eta \geq N^{-1+\eps}$. (The $\prec$ notation means that the inequality holds except on an event of probability at most a power of $N$ and up to errors that grow slower than any power of $N$; see Section \ref{sec:fluctuation averaging} for details.)

For many applications, including eigenvector delocalization, the distribution of eigenvectors, and the analysis of finite-rank deformations, the averaged local law is insufficient. This leads to the \emph{isotropic local law} \cite{bloemendal2014isotropic} and more generally the \emph{anisotropic local law}. For the \emph{separable model} $\g = \Sigma^{1/2}\w$, Knowles and Yin \cite{knowles2017anisotropic}, building on \cite{bloemendal2014isotropic} in the case $\Sigma =I_n$, established that
\begin{equation}\label{eq:intro:anisotropic}
|\u^* R(z) \v - \u^* \Pi(z) \v| \prec \Psi(z),
\qquad \text{for all unit vectors } \u, \v \in \C^n
\end{equation}
where $\Pi(z) = (-zI - z\widetilde{m}_0(z)\Sigma)^{-1}$, $\Psi(z) = \sqrt{\Im \widetilde{m}_0(z)/(N\eta)} + (N\eta)^{-1}$ is the optimal error parameter, and $\widetilde{m}_0(z)$ is the Stieltjes transform of the limiting spectral distribution of the Gram matrix $\widetilde{K} = N^{-1}G^*G$.

The anisotropic local law has profound implications: it yields eigenvector delocalization in any fixed basis (not just the standard basis) and eigenvalue rigidity at the optimal scale.  Combined with a comparison argument, it leads to Tracy-Widom fluctuations for edge eigenvalues and bulk universality of local eigenvalue statistics \cite{lee2016bulk,lee2016tracy}. These results have made the anisotropic local law a cornerstone of modern random matrix theory.

\subsection{Beyond the separable model: the main question}

The separable model $\g = \Sigma^{1/2}\w$ with independent entries in $\w$ is restrictive in practice. Many statistical and machine learning applications naturally involve random vectors with more complex dependence structures. This brings us to the central question of the present work:
\vspace{0.5em}
\begin{center}
\emph{Do optimal local laws hold for distributions of $\g$ beyond the separable setting?}
\end{center}
\vspace{0.5em}

There are many possible answers to this question, depending on the class of distributions under consideration. One important class is the \emph{random features model}
\begin{equation}\label{eq:intro:RF}
\g = \sigma(X\w),
\end{equation}
where $X \in \R^{n \times d}$ is a deterministic feature matrix, $\w \in \R^d$ is a random vector with independent entries, and $\sigma: \R \to \R$ is a (possibly entrywise) nonlinear activation function. This model arises naturally in the study of random features regression, which is a statistical approximation of kernel methods \cite{hu2022universality,louart2018random,mei2022generalization,hastie2022surprises} and linear approximants for neural networks \cite{pennington2017nonlinear,fan2020spectra,benigni2021eigenvalue}. We discuss applications and related work on random features models further in Section~\ref{sec:randomfeatures} and further details in Section \ref{subsec:examples}.

The weak convergence condition \eqref{eq:intro:quadraticform} is not sufficient for the local law to hold at optimal scales, as a quantitative strengthening in terms of the rate of concentration in \eqref{eq:intro:quadraticform}
can be seen to be necessary for an averaged local law to hold (see Section \ref{subsec:examples}). A first contribution of our work is to show that such a strengthening is also sufficient for the averaged local law to hold. In brief, we work in a standard high-dimensional regime where $n/N$ remains bounded and $\|\Sigma\|_\op \leq C$, and we assume the strong concentration
\[
\g^*A\g - \Tr \Sigma A \prec \|A\|_F
\]
for all deterministic matrices $A$ (see Assumptions~\ref{assum:basic} and \ref{assum:concentration} in Section~\ref{sec:cumulant} for the precise formulation). This concentration is known to hold in the separable case under moment bounds on $\w$, and also for vectors satisfying a log-Sobolev inequality \cite{adamczak2015log} or log-concave distributions \cite{bao2025extreme}. We show that the averaged local law holds under this assumption alone (Theorem~\ref{thm:entrywise law}).

To establish the full anisotropic local law, which provides control of the resolvent in arbitrary directions, our analyses require additional structure on the higher-order cumulants of $\g$. We introduce Assumption~\ref{assum:cumulant} in Section~\ref{sec:cumulant}, a condition on the cumulant tensors of $\g$ that interpolates between the highly structured case of independent entries (where cumulant tensors are diagonal) and generic tensors (where no such structure exists). This assumption is significantly more general than requiring a separable model, yet remains verifiable for important examples including the random features model \eqref{eq:intro:RF}.

\subsection{Main contributions}

The main contributions of this paper are:

\begin{enumerate}[label=(\roman*)]
\item \textbf{Averaged local law under concentration of quadratic forms} (Theorem~\ref{thm:entrywise law}): Under Assumptions~\ref{assum:basic} and \ref{assum:concentration}, we establish the optimal averaged local law \eqref{eq:intro:averaged} together with an entrywise local law for the Gram matrix resolvent $\widetilde{R}(z)$. This extends the results of \cite{pillai2014universality} to the case of non-identity covariance.

\item \textbf{Anisotropic local law under cumulant structure} (Theorem~\ref{thm:anisotropic law}): Under the additional Assumption~\ref{assum:cumulant} on cumulant tensors, we establish the optimal anisotropic local law \eqref{eq:intro:anisotropic} for the linearized resolvent \eqref{eq:linearizedresolvent}. This 
extends results of \cite{knowles2017anisotropic} beyond the separable model, providing the first anisotropic local law down to optimal spectral scales $\eta \sim N^{-1+\eps}$ for a general class of non-separable sample covariance matrices.

\item \textbf{Verification for non-separable examples} (Section~\ref{subsec:examples}): We verify Assumptions~\ref{assum:concentration} and \ref{assum:cumulant} for several classes of non-separable distributions, including conditionally mean-zero distributions (Proposition~\ref{prop:sign change invariant}), the random features model $\g = \sigma(X\w)$ (Proposition~\ref{prop:RF}), and Gaussian measures with random features tilting (Proposition~\ref{prop:spiked cumulant}).

\item \textbf{Tensor network methodology}: Our proofs employ a novel tensor network framework for analyzing fluctuation averaging in the presence of higher-order cumulant structure. This framework, which systematically reduces complex moment expansions to graph-theoretic bounds, may be of independent interest for other problems in random matrix theory.  We give a brief introduction in Section \ref{subsec:proofideas}.
\end{enumerate}

As corollaries, we obtain eigenvalue rigidity at the optimal scale and delocalization of eigenvectors of both $K$ and $\widetilde{K}$ in any fixed basis (see Section \ref{subsec:anisotropic local law} for more details).

\subsection{Related work}\label{sec:relatedwork}\label{sec:randomfeatures}
\paragraph{Random features and Gaussian equivalence.}
The random features model \eqref{eq:intro:RF} was introduced by Rahimi and Recht \cite{rahimi2007random} as a computationally efficient approximation to kernel methods, and is equivalent to a two-layer neural network with frozen first-layer weights and trainable readout. A remarkable phenomenon, observed empirically and established rigorously in several settings \cite{hastie2022surprises,mei2022generalization,montanari2019generalization,goldt2020modelling,gerace2020generalisation,hu2022universality}, is that the random features model is asymptotically equivalent to a surrogate linear Gaussian model: in the proportional limit, the training and generalization errors of the nonlinear model match those of a Gaussian equivalent. The anisotropic local law established in this paper (Theorem~\ref{thm:anisotropic law}) can be viewed as a \emph{strong form of Gaussian equivalence} at the level of the resolvent: it asserts that $(K - zI)^{-1}$ is well-approximated at optimal scales by a deterministic matrix $\Pi(z)$ depending only on the first two moments of $\g$. This extends naturally to spectral filters $\varphi(K)$ for functions $\varphi$ smooth on any scale larger than the mean eigenvalue spacing. The model also extends to deep architectures with frozen intermediate layers \cite{schroder2023deterministic}; our cumulant assumption is designed to accommodate such compositional structure, though verifying it for deep models with many layers remains an open problem.

\paragraph{Covariance estimation and concentration.}
The estimation of covariance matrices in high dimensions has a vast literature; see \cite{bai2010spectral,vershynin2010introduction} for comprehensive treatments. A central question is: what conditions on the distribution of $\g$ ensure that the sample covariance $K = N^{-1}\sum_{i=1}^N \g_i\g_i^*$ concentrates around $\Sigma$ in operator norm?

For log-concave distributions, optimal $\|K - \Sigma\|_{\op} = O(\sqrt{n/N})$ rates were established by Adamczak, Litvak, Pajor, and Tomczak-Jaegermann \cite{adamczak2010quantitative}. The work of Srivastava and Vershynin \cite{srivastava2013covariance} showed that if all $k$-dimensional marginals of the whitened vector $\Sigma^{-1/2}\g$ have $(2+\eps)$-moment tails decaying like $t^{-(1+\eta)}$ for $t > Ck$, then $N = O(n)$ samples suffice for operator norm concentration. This was refined by Tikhomirov \cite{tikhomirov2018sample}, who obtained the optimal Bai-Yin rate under $L_p$-$L_2$ norm equivalence for $p > 4$, which gives moment growth estimates of \emph{linear forms}, i.e.\ uniform control of moments of $\<\u, \g\>$ in place of our \emph{quadratic forms} $\g^*A\g$.
Most recently, Abdalla and Zhivotovskiy \cite{abdalla2024covariance} proved a non-asymptotic, dimension-free Bai-Yin theorem: if $\|\g\|_2 \leq N^{1/2+\eps}$ and one-dimensional marginals satisfy $L_p$-$L_2$ norm equivalence for some $p > 4$, then the sample covariance achieves the optimal rate depending on the effective rank rather than the ambient dimension.

Hence, linear form estimates tend to be sufficient for the \emph{covariance estimation problem}.  Importantly, boundedness of linear forms  does not imply concentration of quadratic forms, even in the weak sense of \eqref{eq:intro:quadraticform}. A simple counterexample is the \emph{mixture model}: let $\g = \xi \cdot \z$ where $\z \sim N(0, I_n)$ is Gaussian and $\xi$ is an independent bounded scalar variable with unit second moment. Then $\E\g = 0$ and $\E\g\g^* = I_n$, and linear forms $\<\u, \g\>$ are sub-Gaussian. However, the variance of $\xi^2$ may induce fluctuations in the quadratic form $\g^*A\g = \xi^2 \z^*A\z$ that are large enough to cause the standard Marchenko-Pastur law to \emph{not} be a deterministic equivalent for the spectral distribution (see Example \ref{ex:mixture model} for further discussion).

\paragraph{Local laws for Wigner matrices and generalizations.}
Th\textbf{}e development of local laws began with Wigner matrices (symmetric matrices with i.i.d.\ entries), where the semicircle law holds globally. Pioneering work of Erd\H{o}s, Schlein, and Yau \cite{erdHos2009local,erdHos2009semicircle} established local semicircle laws down to scales $\eta \sim N^{-1+\eps}$. This was subsequently extended to generalized Wigner matrices and models of mild sparsity \cite{erdHos2012bulk,erdos2013local,erdHos2013averaging,erdos2013spectral}. A powerful generalization to less homogeneous matrices is the \emph{matrix Dyson equation} (MDE) formalism developed by Ajanki, Erd\H{o}s, and Kr\"uger \cite{ajanki2017universality,ajanki2019stability}. For a random matrix $H$ with mean $A = \E H$ and covariance operator $\mathcal{S}[R] = \E(H-A)R(H-A)$, the resolvent $G = (H - zI)^{-1}$ is approximated by a deterministic matrix $M$ satisfying
\[
I + (zI - A + \mathcal{S}[M])M = 0.
\]
This framework accommodates matrices with correlated entries (provided correlations decay sufficiently fast), non-identical variances, and non-zero means.

\paragraph{The characteristic flow method.}
A complementary approach to local laws proceeds dynamically: one interpolates the given matrix toward a Gaussian ensemble by running a Dyson Brownian motion, and studies how the resolvent evolves along this flow. The key observation, going back to Pastur \cite{pastur1972spectrum}, is that the Stieltjes transform of the empirical spectral measure satisfies a transport equation whose characteristics are curves in the upper half-plane contracting away from the spectral support as time increases; this allows one to propagate estimates that hold at larger scales to the optimal shorter scales $\eta \sim N^{-1+\eps}$. This method was pioneered for local laws by \cite{huang2019rigidity,adhikari2020dyson} in the context of $\beta$-ensembles.  \cite{bourgade2022extremegaps,benigni2020eigenvectors} used related stochastic advection equations to study   extreme gap statistics and eigenvector statistics respectively. The method has been also extended beyond the self-adjoint setting: \cite{bourgade2022liouville} employed it in the context of non-Hermitian dynamics, and \cite{cipolloni2022mesoscopic} introduced the \emph{zig-zag method}, which applies the characteristic flow iteratively to prove multi-resolvent local laws by alternating short flow steps with Greens function comparison estimates; this method has also been applied to generalized Wigner models \cite{erdHos2024eigenstate}, which are closely related to the linearized resolvents of sample covariance matrices. Although the present paper does not use the characteristic flow method, identifying the correct assumptions on non-separable random vectors that would allow one to run this dynamical argument remains interesting for future work, and we expect it would be similar in spirit to what we need here.  See \cite{vonSoosten2019random} for a short and elegant dynamical proof of a local semicircle law via the characteristic flow method.

\paragraph{Local laws for sample covariance and Gram matrices.}
For sample covariance matrices with identity covariance, the averaged local law program was initiated by Pillai and Yin \cite{pillai2014universality} (see also \cite{bao2015universality}). The work of Knowles and Yin \cite{knowles2017anisotropic} established the anisotropic local law for separable covariance matrices (c.f. Proposition \ref{prop:linear}) via a strategy closely related to the method here, though they use a sophisticated interpolation argument to show the full anisotropic local law. Alt, Erd\H{o}s, and Kr\"uger extended these results to general Gram matrices and correlated sample covariance models using the Dyson equation approach \cite{alt2017local,alt2019dyson,alt2021correlated}. An important statistical application of these local laws is to edge universality: Lee and Schnelli \cite{lee2016tracy} used local laws as input to establish Tracy-Widom fluctuations for the largest eigenvalue of sample covariance matrices with general population covariance, via a Green function comparison argument. This was extended to Tracy-Widom fluctuations at each regular edge of the spectrum by Fan and Johnstone \cite{fan2022tracy}.

\paragraph{Polynomial scaling regimes and spike eigenvalues.}
The proportional regime $n \asymp N$ considered in this paper is not the only scaling of interest. In the \emph{polynomial regime} $N \asymp d^\ell$ for integer $\ell \geq 1$, limiting spectral distributions of random inner-product kernel matrices have been analyzed by \cite{lu2025equivalence,xiao2022precise}; universality for general distributions was established in \cite{dubova2023universality,pandit2025universality}. Polynomial scalings of random features regression have also been studied \cite{hu2024asymptotics,defilippis2024dimension}, revealing subtleties such as the failure of Gaussian equivalence in the quadratic scaling regime \cite{wen2025conditional}.
Beyond the bulk spectrum, ``spike'' eigenvalues separated from the bulk are often the primary features of interest. The celebrated Baik-Ben Arous-P\'ech\'e (BBP) phase transition \cite{baik2005phase} describes when spike eigenvalues emerge from the bulk in spiked covariance models. For the separable model $\g = \Sigma^{1/2}\x$, these phenomena are well-understood \cite{bai2012sample,benaych2012singular}. The work of Wang, Wu, and Fan \cite{wang2024nonlinear} extends this to nonlinear models, characterizing how spiked structure in input data propagates through neural network layers and establishing a Gaussian equivalence for spike eigenvalue behavior for non-separable sample covariance matrix.

\subsection{Notation}\label{subsec:notation}
We use $C, c$ for positive constants that may change from line to line but depend only on fixed model parameters. The notation $a \asymp b$ means $c \leq a/b \leq C$ for constants $c, C > 0$. We write $\|\cdot\|_\op$ and $\|\cdot\|_F$ for the operator (spectral) norm and Frobenius norm of matrices, and $\|\cdot\|_2$ for the Euclidean norm of vectors. For a general tensor $T\in(\R^n)^{\otimes k}$, we write $\|T\|_\infty$ for the entrywise $l_\infty$ (maximum) norm, and $\|T\|_F$ for its Frobenius tensor norm, i.e.\ $\|\cdot\|_2$-norm of its vectorization. We write $X \prec Y$ to denote stochastic domination: for any $\eps, D > 0$, there exists $C \equiv C(\eps,D)>0$ such that $\P[|X| > N^\eps Y] \leq CN^{-D}$ for all $N \geq 1$. Properties of this notation are further reviewed in Section \ref{sec:fluctuation averaging}.

\section{Main results}

\subsection{Assumptions}\label{sec:cumulant}

We work in the following standard high-dimensional setting.
\begin{assumption}[Basic assumptions]\label{assum:basic}
There exist constants $C,c>0$ such that $c \leq n/N \leq C$,
$\|\Sigma\|_\op \leq C$, and $\Sigma$ is positive semidefinite with at most
$(1-c)n$ of its eigenvalues belonging to $[0,c]$.
\end{assumption}

The central assumption we make on quadratic forms is the following.
\begin{assumption}[Concentration of quadratic forms]\label{assum:concentration}
$\g_1,\ldots,\g_N$ are independent random vectors satisfying $\E\g_i=0$,
$\E \g_i\g_i^*=\Sigma$, and
\[\E\|\g_i\|_2^k \leq n^{C_k} \text{ for each } k \geq 1 \text{ and constants }
C_k>0.\]
Furthermore, for any $\eps,D>0$, there exists a constant $C \equiv C(\eps,D)>0$
such that for each $i=1,\ldots,N$ and any $A \in \R^{n \times n}$,
\[\P[|\g_i^* A\g_i-\Tr\Sigma A| \geq n^\eps \|A\|_F] \leq Cn^{-D}.\]
\end{assumption}
\noindent Assumption~\ref{assum:concentration} can equivalently be formulated in terms of $L^p$ bounds on the re-centered quadratic form $\g_i^*A\g_i - \Tr \Sigma A$, by Markov's inequality. It is known to hold in the separable case $\g_i=\Sigma^{1/2}\w_i$ under moment bounds on the entries of $\w_i$, and also for vectors satisfying a convex concentration inequality with at most logarithmic dependence on the dimension \cite{adamczak2015log}, including isotropic log-concave random vectors \cite{bao2025extreme}.

Under Assumptions~\ref{assum:basic} and \ref{assum:concentration}, we will
establish an optimal local law for the Stieltjes transforms of $K$ and
$\widetilde K$ and an entrywise local law for the resolvent of
$\widetilde K$ (c.f.\ Theorem \ref{thm:entrywise law})
over spectral domains of $\C^+$ satisfying suitable regularity conditions.

We will then establish an optimal anisotropic local law for the
resolvents of $K$ and $\widetilde K$ (c.f.\ Theorem
\ref{thm:anisotropic law}) under an additional assumption for the higher-order
cumulant tensors of $\g_1,\ldots,\g_N$. To formulate this assumption, we
introduce the following definition.

\begin{defi}[$\cU$-norm]
For any subset $\cU\subseteq\R^n$ containing all standard basis vectors
$\e_1,\ldots,\e_n$, and for any $k \geq 1$, we define a norm on tensors
$T \in (\R^n)^{\otimes k}$ by
\[\|T\|_\cU=\sup_{\x_1,\ldots,\x_k\in\cU}|\<\x_1\otimes\cdots\otimes\x_k,T\>|.\]
\end{defi}

Note that for $\cU$ containing the standard basis vectors $\e_1,\ldots,\e_n$
and having polynomial cardinality $|\cU| \leq n^C$, the norm $\|T\|_{\cU}$ may
be understood as a certain strengthening of the entrywise $\ell_\infty$ norm. Our additional assumption for the cumulant tensors of $\g_1,\ldots,\g_N$ is
the following.

\begin{assumption}\label{assum:cumulant}
Let $\kappa_k(\g_i) \in (\R^n)^{\otimes k}$ denote the $k$-th order cumulant
tensor of $\g_i$, whose entries are given by
\[
\langle\kappa_k(\g_i),\e_{\alpha_1}\otimes\cdots\otimes\e_{\alpha_k}\rangle=\kappa_k(\e_{\alpha_1}^*\g_i,\dots,\e_{\alpha_k}^*\g_i)
\]
where $\kappa_k(\cdot)$ on the right side is the $k^\text{th}$ mixed cumulant
of entries of $\g_i$.

Then for each $k\geq 3$, there exists a constant $C_k>0$ and a subset of
deterministic vectors $\cU_k\subset \R^n$ satisfying
    \[
        \{\e_1,\dots,\e_n\}\subseteq\cU_k,\quad |\cU_k|\leq n^{C_k},\quad
\|\x\|_2\leq C_k\;\text{ for all } \x\in\cU_k,
    \]
such that the following holds: For any $\eps>0$, there exists a constant
$C \equiv C(\eps,k)>0$ such that for each $i=1,\ldots,N$
and any $m \in \{1,\ldots,k-1\}$,
$\s_1,\dots,\s_m\in\R^n$, and $T\in(\R^n)^{\otimes k-m}$,
\begin{equation}\label{eq:cumulantassumption}
|\<\kappa_k(\g_i),\s_1\otimes\cdots \otimes \s_m\otimes T\>|\leq
Cn^\eps(\sqrt{n})^{k-m-1}\|T\|_{\cU_k}\prod_{t=1}^m\|\s_t\|_2.
\end{equation}
\end{assumption}


\begin{remark}\label{remark:genericcumulant}
To provide some intuition for Assumption \ref{assum:cumulant}, we note that a simple cumulant growth condition that would imply the needed concentration of quadratic forms in Assumption \ref{assum:concentration} is
\begin{equation}\label{eq:cumulantconditionanalogue}
|\<\kappa_k(\g),S \otimes T\>|
\prec \|S\|_F\|T\|_F
\end{equation}
for all fixed $k \geq 3$, $m \in \{1,\ldots,k-1\}$, $S \in (\R^n)^{\otimes m}$,
and $T \in (\R^n)^{\otimes k-m}$.
(That \eqref{eq:cumulantconditionanalogue}
implies Assumption \ref{assum:concentration} follows from the high moment estimate of Lemma
\ref{lem:concentration of quadratic form} and Markov's inequality.)
Specializing this growth condition
\eqref{eq:cumulantconditionanalogue} to $S=\s_1 \otimes \ldots \otimes \s_m$ and applying $\|T\|_F \leq \sqrt{n}^{k-m}\|T\|_\infty$ would give
\begin{equation}\label{eq:cumulantconditionanalogueimplication}
|\<\kappa_k(\g),\s_1 \otimes \ldots \otimes \s_m \otimes T\>|
\prec (\sqrt{n})^{k-m}\|T\|_\infty\prod_{t=1}^m \|\s_t\|_2,
\end{equation}
which is a version of \eqref{eq:cumulantassumption} with a weaker bound $\sqrt{n}^{k-m}$ in place of $\sqrt{n}^{k-m-1}$. Although the condition \eqref{eq:cumulantassumption} is not directly comparable to
\eqref{eq:cumulantconditionanalogue},
one may interpret \eqref{eq:cumulantassumption} as a certain strengthening of \eqref{eq:cumulantconditionanalogue} by an extra factor of $n^{-1/2}$ when restricted to rank-1 tensors $S$ and measured in a $\ell_\infty$-type norm $\|T\|_{\cU_k}$ rather than $\|T\|_F$.

One may also check that for a ``generic'' tensor $\kappa_k \in (\R^n)^{\otimes k}$ satisfying \eqref{eq:cumulantconditionanalogue}, Assumption \ref{assum:cumulant} with the stronger bound of $\sqrt{n}^{k-m-1}$ cannot hold. Indeed, given any set $\cU_k \subset \R^n$ satisfying the conditions of Assumption \ref{assum:cumulant}, there exists a unit vector $\s \in \R^n$ satisfying
\[\sup_{\u_2,\ldots,\u_k \in \cU_k} |\<\kappa_k,\s \otimes \u_2 \otimes \ldots \otimes \u_k\>|
\prec n^{-1/2}
\sup_{\u:\|\u\|_2=1}
\sup_{\u_2,\ldots,\u_k \in \cU_k}
|\<\kappa_k, \u \otimes \u_2 \otimes \ldots \otimes \u_k\>|
\prec n^{-1/2}\|\kappa_k\|_\text{inj}\]
where $\|\kappa_k\|_\text{inj}$ is the injective norm, and also
\[\|\s\cdot \kappa_k\|_F^2 \asymp n^{-1}\|\kappa_k\|_F^2\]
where $\s\cdot \kappa_k \in (\R^n)^{\otimes k-1}$ is the contraction of $\s$ with $\kappa_k$ in the first coordinate.
(It is easily checked that both conditions hold with high probability over a uniform random choice of $\s$ on the unit sphere.)
Then for $m=1$ and $T=\s \cdot \kappa_k$, we have
\[|\<\kappa_k,\s \otimes T\>|
=\|\s \cdot \kappa_k\|_F^2
\asymp n^{-1}\|\kappa_k\|_F^2,
\qquad \|T\|_{\cU_k}
\prec n^{-1/2}\|\kappa_k\|_\text{inj}.\]
If $\kappa_k$ is generic and satisfies \eqref{eq:cumulantconditionanalogue} --- e.g., $\kappa_k=n^{-(k-1)/2}Z_k$ where $Z_k \in (\R^n)^{\otimes k}$ is a symmetric Gaussian tensor with i.i.d.\ $N(0,1)$ entries up to symmetry --- then we would expect
$\|\kappa_k\|_F^2 \asymp n^{-(k-1)}\|Z_k\|_F^2 \asymp n$ and $\|\kappa_k\|_\text{inj}
\asymp n^{-(k-1)/2}\|Z_k\|_\text{inj}
\asymp n^{-k/2+1}$ (with high probability over $Z_k$). Thus \eqref{eq:cumulantassumption} can only hold with the weaker factor $\sqrt{n}^{k-1}=\sqrt{n}^{k-m}$ instead of
$\sqrt{n}^{k-m-1}$, so Assumption \ref{assum:cumulant} requires a certain non-genericity of each cumulant tensor $\kappa_k$.
\end{remark}

In the simplest setting of a separable model where $\g_i=\Sigma^{1/2}\w_i$ and $\w_i$ has independent
entries with bounded moments, it is easily checked (c.f.\ Proposition \ref{prop:linear}) that Assumption \ref{assum:cumulant} holds in a stronger form with $\sqrt{n}^{k-m-1}$ in \eqref{eq:cumulantassumption} replaced by $\sqrt{n}^{\1\{m=1\}}$.
Importantly, Assumptions \ref{assum:concentration} and \ref{assum:cumulant} are significantly more general than requiring this separable form, and we discuss several illustrative examples in Section
\ref{subsec:examples}.
We show also in Section
\ref{subsec:examples} that neither Assumption \ref{assum:concentration} or
\ref{assum:cumulant} strictly implies the other: for a Gaussian mixture model with random variance, there is a regime where Assumption \ref{assum:cumulant} holds while Assumption \ref{assum:concentration} fails; conversely, for the spherical $4$-spin model at very high temperature (e.g.\ $\beta=n^{-1/6}$), Assumption \ref{assum:concentration} holds while Assumption \ref{assum:cumulant} fails, precisely through the large gap between injective and Frobenius tensor norms discussed above.

\subsection{Deformed Marchenko-Pastur law and regular spectral domains}\label{subsec:mp}



Under general conditions encompassing those of Assumptions \ref{assum:basic} and
\ref{assum:concentration}, the empirical eigenvalue distributions of
$K,\widetilde K$ are well-approximated in large dimensions by
deterministic laws $\mu_0,\tilde \mu_0$
\cite{marvcenko1967distribution,silverstein1995strong,silverstein1995empirical}. We
review here relevant background regarding $\mu_0,\tilde \mu_0$ and their
regularity properties.

Denote the ordered eigenvalues of $\Sigma$ by
\[\sigma_1 \geq \ldots \geq \sigma_n \geq 0.\]
For each $z \in \C^+$, there exists a unique solution $\widetilde m_0(z) \in
\C^+$ \cite{marvcenko1967distribution} to the Marchenko-Pastur equation
\begin{equation}\label{eq:MP}
z=-\frac{1}{\widetilde{m}_0(z)}+\frac{1}{N}\sum_{\alpha=1}^n
\frac{\sigma_\alpha}{1+\sigma_\alpha\widetilde{m}_0(z)}.
\end{equation}
Define also $m_0(z) \in \C^+$ by
\begin{equation}\label{eq:m0m0tilde}
\widetilde{m}_0(z)=\gamma m_0(z)+(1-\gamma)(-1/z), \qquad \gamma=n/N.
\end{equation}
Then $\mu_0$ and $\tilde \mu_0$ are the (unique) probability distributions
whose Stieltjes transforms are given by $m_0,\widetilde m_0:\C^+ \to \C^+$,
\begin{equation}\label{eq:stieltjesdef}
m_0(z)=\int \frac{1}{x-z}\,d \mu_0(x),
\qquad \widetilde m_0(z)=\int \frac{1}{x-z}\,d \widetilde \mu_0(x).
\end{equation}
This coincides with \eqref{eq:intro:MP}, and moreover $\mu_0$ is the deformed Marchenko-Pastur law.
Note that the relation \eqref{eq:m0m0tilde} implies that
$\widetilde \mu_0=\gamma \mu_0+(1-\gamma)\delta_0$ (denoting a mixture of
$\mu_0$ and an atom at 0) if $\gamma \leq 1$, or
$\mu_0=(1/\gamma)\widetilde \mu_0+(1-1/\gamma)\delta_0$ if $\gamma \geq 1$. In
particular,
\[\supp(\mu_0) \cap (0,\infty)=\supp(\widetilde \mu_0) \cap (0,\infty).\]


It is shown in \cite{silverstein1995analysis} that $\widetilde \mu_0$
admits a continuous density at each point $x>0$, given by
\[\rho_0(x)=\lim_{z\in\C^+\rightarrow x}\frac{1}{\pi}\Im \widetilde
m_0(z).\]
Furthermore, $\supp(\widetilde \mu_0) \cap (0,\infty)$ is a finite union of
intervals characterized in the following way: Let $\bar n=\rank(\Sigma)$ and
consider the meromorphic function
with poles $P=\{0,{-}\sigma_1^{-1},\ldots,-\sigma_{\bar n}^{-1}\}$,
\begin{equation}\label{eq:z0m}
z_0(m)={-}\frac{1}{m}+\frac{1}{N}\sum_{\alpha=1}^n
\frac{\sigma_\alpha}{1+\sigma_\alpha m}, \qquad m \in \C \setminus P,
\end{equation}
which locally inverts $\widetilde m_0(z)$ defined by \eqref{eq:MP}.
Then under Assumption \ref{assum:basic} (c.f.\ \cite[Lemmas 2.4--2.6]{knowles2017anisotropic}),
$z_0$ has an even number $2p$ of critical points counting multiplicity on the
extended real line $\bar \R=\R \cup \{\infty\}$,
with exactly one critical point $m_1 \in
(-\sigma_1^{-1},0)$, either 0 or 2 critical points in each interval
$(-\sigma_{k+1}^{-1},-\sigma_k^{-1})$ where $\sigma_k \neq \sigma_{k+1}$,
and exactly one critical point $m_{2p} \in (-\infty,-\sigma_{\bar n}^{-1}) \cup
(0,\infty]$ which is $m_{2p}=\infty$ if $\bar n=N$. Ordering these critical points as
$0>m_1>m_2 \geq m_3>m_4 \geq m_5> \ldots >m_{2p-2} \geq m_{2p-1}>-\sigma_{\bar n}^{-1}$
and $m_{2p} \in (-\infty,-\sigma_{\bar n}^{-1}) \cup (0,\infty]$, we have
\begin{equation}\label{eq:edges}
C> x_1>x_2 \geq x_3>x_4 \geq x_5>\ldots>x_{2p-2}\geq x_{2p-1}>x_{2p} \geq 0,
\quad \text{ where } x_j=z_0(m_j),
\end{equation}
for a constant $C>0$ and
\[\supp(\widetilde \mu_0) \cap (0,\infty)=\Big(\bigcup_{k=1}^p
[x_{2k},x_{2k-1}]\Big) \cap (0,\infty)\subseteq[0,C].\]
If $m_j$ is a critical point of multiplicity 1, then $x_j$ is an edge of
$\widetilde \mu_0$ (i.e.\ a boundary point of $\supp(\widetilde \mu_0)$),
which is a right edge for $j$ odd and a left edge for $j$ even. If
$m_{2k}=m_{2k+1}$ is a critical point of multiplicity 2, then $x_{2k}=x_{2k+1}$
is a cusp of $\widetilde \mu_0$ (i.e.\ interior point $x \in \supp(\widetilde
\mu_0)$ where the density $\rho_0(x)=0$). We refer to $[x_{2k},x_{2k-1}]$ for
$k=1,\ldots,p$ as the bulk components of $\widetilde \mu_0$.

Corresponding to these bulk components and edges/cusps, and denoting $z=E+i\eta
\in \C^+$ where $E=\Re z$ and $\eta=\Im z$, 
we define for $\delta>0$ and $\tau \in (0,1)$ the domains
\begin{equation}\label{eq:spectraldomains}
\begin{aligned}
\bD_k^b(\delta,\tau)&=\{z \in \C^+:E \in [x_{2k}+\delta,x_{2k-1}-\delta],\,
N^{-1+\tau} \leq \eta \leq 1\}\\
\bD_j^e(\delta,\tau)&=\{z \in \C^+:E \in [x_j-\delta,x_j+\delta],\,
N^{-1+\tau} \leq \eta \leq 1\}\\
\bD^o(\delta,\tau)&=\{z \in \C^+:\dist(z,\supp(\widetilde \mu_0)) \geq \delta,\,|z| \geq \delta,\, |E| \leq \delta^{-1},\, N^{-1+\tau} \leq \eta \leq 1\}.
\end{aligned}
\end{equation}
Here, $\dist(z,U)=\inf_{x \in U} |x-z|$.
Our local law results will hold over those domains which satisfy appropriate
regularity properties as discussed in \cite{knowles2017anisotropic}, and which we
summarize in the following definition.

\begin{defi}[Regular spectral domain]\label{def:regular}
We call
$\bD \subset \{z \in \C^+:\eta \in (0,1]\}$
a \emph{spectral domain} if $z=E+i \eta \in \bD$ implies that $z'=E+i \eta' \in \bD$ for each $\eta' \in [\eta,1]$.
Corresponding to $z \in \bD$, denote
$\kappa=\min_{j=1}^{2p} |x_j-E|$, and define
\[L(z)=\{z\}\cup\{w\in\bD:\Re{w}=E,\,\Im{w}\in[\eta,1]\cap
\{N^{-5},2N^{-5},3N^{-5},\ldots,1\}\}.\]
Then the domain $\bD$ is \emph{regular} if there exist constants $C,c>0$ such
that the following hold:
\begin{enumerate}[label=(\alph*)]
\item (Basic bounds for $\widetilde m_0(z)$) For all $z \in\bD$,
\begin{equation}\label{eq:regularbounds}
\begin{gathered}
c \leq |z| \leq C,\qquad c\leq|\widetilde{m}_0(z)|\leq C,\qquad
\min_{\alpha\in[n]}|1+\sigma_\alpha\widetilde{m}_0(z)|\geq c,\\
cg(z) \leq \Im \widetilde m_0(z) \leq Cg(z)
\quad \text{ where } \quad g(z)=\begin{cases}
\sqrt{\kappa+\eta} & \text{if } E \in \supp(\widetilde \mu_0) \\
\frac{\eta}{\sqrt{\kappa+\eta}} & \text{if } E \notin \supp(\widetilde \mu_0)
\end{cases}
\end{gathered}
\end{equation}
\item (Stability of the Marchenko-Pastur equation)
Let $u:\C^+\rightarrow\C^+$ be the Stieltjes transform of any probability measure, and let $\Delta : \bD \rightarrow (0,\infty)$ be any deterministic function satisfying:
\begin{itemize}
    \item (Boundedness) $N^{-2}\leq \Delta(z)\leq (\log N)^{-1}$
for all $z\in \bD$,
    \item (Lipschitz continuity) $|\Delta(z)-\Delta(w)|\leq N^{2}|z-w|$
for all $z,w\in\bD$, and
    \item (Monotonicity) For each fixed $E$, the function
$\eta\mapsto\Delta(E+i\eta)$ is non-increasing over $\eta \in [N^{-1+\tau},1]$.
\end{itemize}
For each $z \in \bD$, if $|z_0(u(w))-w|\leq\Delta(w)$ for every $w\in L(z)$, then
\begin{equation}\label{eq:stability}
|u(z)-\widetilde{m}_0(z)|\leq\frac{C\Delta(z)}{\sqrt{\kappa+\eta}+\sqrt{\Delta(z)}}.
\end{equation}
\end{enumerate}
\end{defi}

The following result was shown in \cite{knowles2017anisotropic} for the regularity of the
preceding domains $\bD_k^b(\delta,\tau)$,
$\bD_j^e(\delta,\tau)$, and $\bD^o(\delta,\tau)$.

\begin{lemma}[\cite{knowles2017anisotropic}]\label{lem:properties of m}
Suppose Assumption \ref{assum:basic} holds. Fix any constant $\tau \in (0,1)$.
\begin{enumerate}[label=(\alph*)]
\item (Regular bulk component) Suppose, for a bulk component
$[x_{2k},x_{2k-1}]$, there exist constants $\delta,c_0>0$ such that $\rho_0(x)>c_0$
for all $x \in [x_{2k}+\delta,x_{2k-1}-\delta]$. Then $\bD_k^b(\delta,\tau)$ is
regular.
\item (Regular edge) Suppose, for an edge $x_j=z_0(m_j)$, there exists a
constant $\delta>0$ such that
\[x_j>\delta, \qquad |x_j-x_k|>\delta \text{ for each } k \neq j,
\qquad \min_{\alpha} |m_j+\sigma_\alpha^{-1}|>\delta.\]
Then there exists a constant $\delta'>0$ such that
$\bD_j^e(\delta',\tau)$ is regular.
\item (Outside the spectrum) Fix any constant $\delta>0$.
Then $\bD^o(\delta,\tau)$ is regular.
\end{enumerate}
(For each statement, Definition \ref{def:regular} holds with constants $C,c>0$
depending on $\tau,\delta,\delta',c_0$.)
\end{lemma}
\begin{proof}
Parts (a), (b), and regularity of $\bD^o(\delta,\tau) \cap \{z \in \C^+:\dist(E,\supp(\widetilde \mu_0) \cup \{0\}) \geq \delta/2\}$ in (c) follow from \cite[Lemmas A.4, A.5, A.6,
A.8]{knowles2017anisotropic}. The remaining points $z \in \bD^o(\delta,\tau)$ with $\dist(E,\supp(\widetilde \mu_0) \cup \{0\})<\delta/2$ must satisfy $\eta \geq \delta/2$: For these points, \eqref{eq:regularbounds} follows from the bounds $\Im \widetilde m_0 \geq c\eta$ and $|\widetilde m_0| \leq \eta^{-1}$ implied by \eqref{eq:stieltjesdef}, and \eqref{eq:stability} follows directly from the condition $|z_0(u(z))-z| \leq \Delta(z)$, uniqueness of the solution in $\C^+$ to \eqref{eq:MP} which implies
$u(z)=\widetilde{m}_0(z+\Delta)$ for some $|\Delta| \leq \Delta(z) \leq (\log N)^{-1}$, and the Lipschitz continuity $|\widetilde m_0'(z)| \leq \eta^{-2}$.
\end{proof}

We remark that the regularity conditions of Definition~\ref{def:regular} are satisfied in a wide range of examples; see \cite[Section 3]{knowles2017anisotropic} for detailed verifications under various conditions on the population covariance $\Sigma$.

\vspace{1em}

\subsection{Local law for the Stieltjes transform}\label{subsec:averaged local law}

This and the following section state the main results of our paper. Proofs of Theorems \ref{thm:entrywise law}, \ref{thm:anisotropic law}, \ref{thm:outside spec} and their corollaries will be given in Section \ref{sec:mainproofs}.

For spectral parameters $z=E+i\eta \in\C^+$, define the resolvents and Stieltjes
transforms of $K$ and $\widetilde{K}$ by
\[
    R(z) = (K - zI_n)^{-1}, \quad m(z) = n^{-1}\Tr R(z),
\]
\[
    \widetilde{R}(z) = (\widetilde{K}-zI_N)^{-1},\quad \widetilde{m}(z)=N^{-1}\Tr\widetilde{R}(z).
\]
We will denote the entries of these resolvents by
$R_{\alpha\beta}(z)=\e_\alpha^*R(z)\e_\beta$ and
$\widetilde{R}_{ij}(z)=\e_i^*\widetilde{R}(z)\e_j$.
For $z=E+i\eta \in \C^+$, define also the error control parameter
\[\Psi(z)=\sqrt{\frac{\Im{\widetilde{m}_0(z)}}{N\eta}}+\frac{1}{N\eta}\]
where we recall the bounds for $\Im \widetilde m_0(z)$
from \eqref{eq:regularbounds}.

The following theorem establishes a local law for the Stieltjes transforms
$m(z)$ and $\widetilde m(z)$, together with an entrywise local law for the
resolvent $\widetilde R(z)$ of the Gram matrix $\widetilde K$,
over a regular spectral domain.

\begin{theorem}(Averaged and entrywise local laws)\label{thm:entrywise law}
Suppose Assumptions \ref{assum:basic} and \ref{assum:concentration} hold,
and $\bD$ is a regular spectral domain. Then for any $\eps,D>0$, there
exists a constant $C \equiv C(\eps,D)>0$ such that
with probability at least $1-CN^{-D}$, for all $z=E+i\eta \in \bD$ we have
\begin{equation}
|m(z)-m_0(z)| \leq N^\eps \frac{\Psi(z)^2}{\sqrt{\kappa+\eta}+\Psi(z)}, \qquad
|\widetilde{m}(z)-\widetilde{m}_0(z)| \leq N^\eps
\frac{\Psi(z)^2}{\sqrt{\kappa+\eta}+\Psi(z)},
\label{eq:average law in bulk edge}
\end{equation}
\begin{equation}
\max_{1\leq i,j\leq N} |\widetilde{R}_{ij}(z)-\widetilde{m}_0(z)\1\{i=j\}|
\leq N^{\eps}\Psi(z).\label{eq:entry law in bulk edge}
\end{equation}
\end{theorem}

\begin{remark}
The bound \eqref{eq:regularbounds} implies $\Im \widetilde m_0(z) \leq
C\sqrt{\kappa+\eta}$ in all cases, so the right side of
\eqref{eq:average law in bulk edge} may be further bounded as
\[\frac{\Psi(z)^2}{\sqrt{\kappa+\eta}+\Psi(z)}
\leq \frac{2C\sqrt{\kappa+\eta}(N\eta)^{-1}+2(N\eta)^{-2}}
{\sqrt{\kappa+\eta}+(N\eta)^{-1}} \leq \frac{2(C+1)}{N\eta}.\]
The result stated in \eqref{eq:average law in bulk edge} is equivalent to a
bound of $N^\eps(N\eta)^{-1}$ for $E \in \supp(\tilde \mu_0)$ inside the
spectral support, where $\Im \widetilde m_0(z) \asymp
\sqrt{\kappa+\eta}$. The stronger form of \eqref{eq:average law in bulk edge} 
for $E \notin \supp(\tilde \mu_0)$ has served useful to show optimal concentration of the eigenvalues of $\widetilde K$ around
$\supp(\tilde \mu_0)$ (see e.g.\ \cite[Eq.\ (8.8)]{pillai2014universality}).
\end{remark}

We state here a few known implications 
of Theorem \ref{thm:entrywise law} on eigenvalue rigidity and eigenvector delocalization \cite{erdHos2009semicircle,erdHos2012rigidity}. Let
\[\lambda_1 \geq \ldots \geq \lambda_{\min(n,N)}\]
be the leading eigenvalues of $\widetilde K$ (or equivalently of $K$).
For any $i \in \{1,\ldots,\min(n,N)\}$, define the classical eigenvalue location
$\theta(i)>0$ through
\[N\int_{\theta(i)}^{\infty}\rho_0(x)\;dx=i-\frac{1}{2}\]
where we recall the density $\rho_0(x)$ of $\widetilde \mu_0$ on $(0,\infty)$.
Corresponding to each edge/cusp $x_j$ defined in \eqref{eq:edges}, denote
\begin{equation}\label{eq:Nj}
N_j=N\int_{x_j}^{\infty}\rho_0(x)\;dx.
\end{equation}
(It is shown in \cite[Lemma A.1]{knowles2017anisotropic} that $N_j$ is always an integer.)

\begin{coro}\label{coro:coro of entrywise law}
Suppose Assumptions \ref{assum:basic} and \ref{assum:concentration} hold.
Then for any $\delta,\eps,D>0$, there exist constants $\tau \equiv \tau(\eps)
\in (0,1)$ and $C \equiv C(\delta,\eps,D)>0$ such that
with probability at least $1-CN^{-D}$, the following hold:
\begin{enumerate}[label=(\alph*)]
\item (Spectral support) $K$ has no eigenvalues outside
$\{x \in \R:\dist(x,\supp(\mu_0)) \leq \delta\}$, and $\widetilde K$ has no eigenvalues outside
$\{x \in \R:\dist(x,\supp(\widetilde \mu_0)) \leq \delta\}$.
\item ($N^{-2/3}$-concentration at regular edges)
\label{eq:concentration near edge} If $x_{2k-1}$ is any right edge for which
$\bD_{2k-1}^e(\delta,1/3)$ is regular, then $K$ and $\widetilde K$ have no
eigenvalues in $[x_{2k-1}+N^{-2/3+\eps},x_{2k-1}+\delta]$.

The analogous statement holds for a left edge $x_{2k}$
and $[x_{2k}-\delta,x_{2k}-N^{-2/3+\eps}]$.

    \item (Rigidity of eigenvalues)\label{eq:Rigidity of eigenvalues}
    If $x_{2k-1}$ is any right edge for which
$\bD_{2k-1}^e(\delta,\tau)$ is regular, then
        \[|\lambda_i-\theta(i)| \leq (i-N_{2k-1})^{-1/3}N^{-2/3+\eps}
\text{ for all $i$ such that $\theta(i)\in [x_{2k-1}-\delta,x_{2k-1}]$.}\]
If $x_{2k}$ is any left edge for which $\bD_{2k}^e(\delta,\tau)$ is regular,
then
\[|\lambda_i-\theta(i)| \leq (N_{2k}-i)^{-1/3}N^{-2/3+\eps}
\text{ for all $i$ such that $\theta(i) \in [x_{2k},x_{2k}+\delta]$}.\]
If $[x_{2k},x_{2k-1}]$ is any bulk component for which
$\bD_k^b(\delta,\tau)$ is regular, and at least one edge domain
$\bD_{2k}^e(\delta,\tau)$ or $\bD_{2k-1}^e(\delta,\tau)$ is also regular, then 
\[|\lambda_i-\theta(i)| \leq N^{-1+\eps}
\text{ for all $i$ such that $\theta(i) \in [x_{2k}+\delta,x_{2k-1}-\delta]$.}\]
    \item ($l_\infty$-delocalization of eigenvectors of $\widetilde
K$)\label{eq:infty delocalization of right singular vector} If $x_j$ is any edge for which $\bD_j^e(\delta,\eps)$ is regular, then for each eigenvalue
$\lambda \in [x_j-\delta,x_j+\delta]$ of $\widetilde K$ and each associated eigenvector $\widetilde \x$,
\[\|\widetilde \x\|_\infty \leq \frac{N^{\eps}}{\sqrt{N}}.\]
If $[x_{2k},x_{2k-1}]$ is any bulk component for which $\bD_k^b(\delta,\tau)$ is regular, then the same holds for every eigenvector of $\widetilde K$ corresponding to an eigenvalue $\lambda \in [x_{2k}+\delta,x_{2k-1}-\delta]$.
\end{enumerate}
\end{coro}

\vspace{1em}

\subsection{Anisotropic local law for the linearized resolvent}\label{subsec:anisotropic local law}

For all $z\in\C^+$, define the linearized resolvent
\begin{equation}\label{eq:linearizedresolvent}
\Pi(z)=\begin{bmatrix}
        -zI_n & \frac{1}{\sqrt{N}}G \\
        \frac{1}{\sqrt{N}}G^* & -I_N
    \end{bmatrix}^{-1} \in\C^{(n+N)\times(n+N)}.
\end{equation}
Note that by Schur's complement, we have
\[\Pi(z)=\begin{bmatrix} R(z) & * \\ * & z\widetilde R(z) \end{bmatrix}\]
where the two diagonal blocks contain the resolvents of $K$ and $\widetilde K$.
Under the additional condition of Assumption \ref{assum:cumulant},
the following theorem establishes an optimal anisotropic
local law for $\Pi(z)$.

\begin{theorem}[Anisotropic local law]\label{thm:anisotropic law}
Suppose Assumptions \ref{assum:basic}, \ref{assum:concentration}, and
\ref{assum:cumulant} hold, and $\bD$ is a regular spectral domain. 
Then for any $\eps,D>0$, there exists a constant $C \equiv C(\eps,D)>0$ such
that the following holds:

Fix any deterministic unit vectors $\v_1,\v_2\in\C^N$ and $\u_1,\u_2\in\C^n$.
Then with probability at least $1-CN^{-D}$, for all $z=E+i\eta \in \bD$ we have
\[\left|\begin{bmatrix}
    \u_1^* & \v_1^*
    \end{bmatrix}\left(\Pi(z)-\begin{bmatrix}
        (-zI_n-z\widetilde{m}_0(z)\Sigma)^{-1} & 0\\
        0 & z\widetilde{m}_0(z)I_N
    \end{bmatrix}\right)\begin{bmatrix}
        \u_2\\
        \v_2
    \end{bmatrix}\right|\leq N^{\eps}\Psi(z).\]
\end{theorem}

This implies the delocalization of eigenvectors of both $K$ and $\widetilde K$
in any fixed basis, extending the entrywise delocalization for $\widetilde K$ in
Corollary \ref{coro:coro of entrywise law}(d).

\begin{coro}[Delocalization of eigenvectors of $K$ and $\widetilde K$]\label{coro:Singular vector lies in general position}
Suppose Assumptions \ref{assum:basic}, \ref{assum:concentration}, and
\ref{assum:cumulant} hold. Fix any deterministic unit vectors $\u \in \R^n$ and
$\v \in \R^N$. Then for any $\delta,\eps,D>0$, there exists a constant $C \equiv
C(\delta,\eps,D)>0$ such that with probability at least $1-CN^{-D}$, the following
holds:

If $x_j$ is any edge of $\widetilde \mu_0$ for which $\bD_j^e(\delta,\eps)$ is regular, then for each eigenvalue
$\lambda \in [x_j-\delta,x_j+\delta]$ and associated eigenvector $\x$ or $\widetilde \x$ of $K$ or $\widetilde K$,
\[|\u^*\x|\leq\frac{N^{\eps}}{\sqrt{N}},
\qquad
|\v^*\widetilde \x| \leq \frac{N^{\eps}}{\sqrt{N}}.\]
If $[x_{2k},x_{2k-1}]$ is any bulk component for which $\bD_k^b(\delta,\tau)$ is regular, then the same holds for every eigenvector $\x$ or $\widetilde \x$ of $K$ or $\widetilde K$ corresponding to an eigenvalue $\lambda \in [x_{2k}+\delta,x_{2k-1}-\delta]$.
\end{coro}

Although the primary focus of our work is on Theorems \ref{thm:entrywise law}
and \ref{thm:anisotropic law} for local spectral parameters $z \in \C^+$, let
us clarify that in the case of $z=E+i\eta$ where $E$ has
constant separation from the spectral support,
our arguments also establish the analogue of Theorem \ref{thm:anisotropic law}
under Assumptions \ref{assum:basic} and \ref{assum:concentration} alone. We
summarize this result here.

\begin{theorem}[Outside the spectrum]\label{thm:outside spec}
Suppose Assumptions \ref{assum:basic} and \ref{assum:concentration} hold.
Fix any constant $\delta>0$ and consider
\[\bar \bD^o(\delta)=\{z \in \C:
\dist(z,\supp(\widetilde{\mu}_0)) \geq \delta,\,|z| \geq \delta,\,|E| \leq \delta^{-1},\,
\eta \in [-1,1]\}.\]
Then for any $\eps,D>0$, there exists a constant
$C \equiv C(\delta,\eps,D)>0$ such that the following hold:

With probability $1-CN^{-D}$, for all $z \in \bar\bD^o(\delta)$ we have
\[|m(z)-m_0(z)| \leq \frac{N^\eps}{N},
\qquad |\widetilde m(z)-\widetilde m_0(z)| \leq \frac{N^\eps}{N}.\]
Furthermore, fix any deterministic unit vectors
$\v_1,\v_2\in\C^N$ and $\u_1,\u_2\in\C^n$.
With probability $1-CN^{-D}$, for all $z \in \bar\bD^o(\delta)$ we have
\[\left|\begin{bmatrix}
    \u_1^* & \v_1^*
    \end{bmatrix}\left(\Pi(z)-\begin{bmatrix}
        (-zI_n-z\widetilde{m}_0(z)\Sigma)^{-1} & 0\\
        0 & z\widetilde{m}_0(z)I_N
    \end{bmatrix}\right)\begin{bmatrix}
        \u_2\\
        \v_2
    \end{bmatrix}\right|\leq \frac{N^{\eps}}{\sqrt{N}}.\]
\end{theorem}

This approximation for $R(z)$ (the upper-left block of $\Pi(z)$) was
established previously in \cite{wang2024nonlinear}, which used this result to
give a characterization of the outlier eigenvalues and eigenvectors of the
sample covariance matrix $K$ under spiked models for $\Sigma$.

\begin{remark}
If $n<N$ and $\dist(0,\supp(\mu_0)) \geq 2\delta$ for some constant
$\delta>0$, then Corollary \ref{coro:coro of entrywise law}(a) implies that with high probability, $m(z)$, $m_0(z)$, $R(z)$, and $(-zI_n-z\widetilde m_0(z)\Sigma)^{-1}$ are all analytic in a neighborhood of 0. Then the bounds $|m-m_0| \leq N^\eps/N$ and
$|\u_1^* (R-(-zI_n-z\widetilde m_0\Sigma)^{-1})\u_2| \leq N^\eps/\sqrt{N}$
in Theorem \ref{thm:outside spec} extend to $\{z \in \C:|z|\leq\delta\}$ by the maximum modulus principle.
Similarly, the bounds $|\widetilde m-\widetilde m_0| \leq N^\eps/N$
and $|\v_1^*(\widetilde R-\widetilde m_0I_N)\v_2| \leq
N^\eps/\sqrt{N}$ extend to $\{z \in \C:|z|\leq\delta\}$ when
$n>N$ and $\dist(0,\supp(\widetilde \mu_0)) \geq 2\delta$.
\end{remark}

\vspace{1em}
\subsection{Examples}\label{subsec:examples}
In this section, assuming $\g_1,\ldots,\g_N$ are equal in law to $\g \in \R^n$, we provide concrete examples of distributions for $\g$ that satisfy Assumptions \ref{assum:concentration} and \ref{assum:cumulant}. All proofs of results in this section are deferred to Section \ref{sec:analysis-examples}.

\subsubsection{Separable distributions}
In the separable setting of \cite{knowles2017anisotropic} where $\g=X\w$ for a vector
$\w \in \R^d$ having independent entries, the cumulant tensor $\kappa_k(\g)$
is a contraction of a diagonal tensor $\kappa_k(\w) \in (\R^d)^{\otimes k}$
with $X$ along each coordinate axis. It is easily checked that in this case,
Assumption \ref{assum:cumulant} holds in a stronger form as guaranteed in the
following proposition.

\begin{proposition}\label{prop:linear}
Suppose Assumption \ref{assum:basic} holds, where $\Sigma=XX^*$ for a matrix $X
\in \R^{n \times d}$ and $cn \leq d \leq Cn$ for constants $C,c>0$.
Suppose $\g=X\w$ where $\w=(\w[\alpha])_{\alpha=1}^d
\in\R^d$ has independent entries satisfying
$\E \w[\alpha]=0$, $\E \w[\alpha]^2=1$, and $\E |\w[\alpha]|^k<C_k$ for each $k
\geq 3$ and a constant $C_k>0$.

Then $\g$ satisfies the conditions of Assumptions \ref{assum:concentration} and
\ref{assum:cumulant}. Each set $\cU_k$
of Assumption \ref{assum:cumulant} may be taken as
\[\cU_k \equiv \cU=\{\e_1,\ldots,\e_n\}\cup\{X\e_1,\ldots,X\e_d\},\]
and the bound \eqref{eq:cumulantassumption} holds in the stronger form
(for any $k \geq 3$, $1 \leq m \leq k-1$, and a constant $C \equiv C(k)>0$)
\begin{equation}\label{eq:linearcumulant}
|\<\kappa_k(\g),\s_1\otimes\cdots \otimes \s_m\otimes T\>|\leq
C\sqrt{n}^{\1\{m=1\}}\|T\|_{\cU}\prod_{t=1}^m\|\s_t\|_2.
\end{equation}
\end{proposition}

\subsubsection{Conditionally mean-zero distributions}

We next consider a more general class of distributions $\g=X\w$ where the entries of $\w$ need not be independent, but satisfy a conditional mean-zero assumption that $\E \prod_{i=1}^k \w[\alpha_i]=0$ whenever one index among $\alpha_1,\ldots,\alpha_k \in [d]$ is distinct from the rest. We check that in such settings, Assumption \ref{assum:cumulant} for $\g$ is implied by Assumption \ref{assum:concentration} for $\w$.

\begin{proposition}\label{prop:sign change invariant}
Suppose Assumption \ref{assum:basic} holds, where $\Sigma=XX^*$ for a matrix
$X \in \R^{n \times d}$ and $cn \leq d \leq Cn$ for constants $C,c>0$.
Suppose $\g=X\w$ where Assumption
\ref{assum:concentration} holds for $\w \in \R^d$, $\E \w\w^*=I$, and $\E \prod_{i=1}^k \w[\alpha_i]=0$
for any indices $\alpha_1,\ldots,\alpha_k \in [d]$ such that $\alpha_1 \notin \{\alpha_2,\ldots,\alpha_k\}$.

Then $\g$ satisfies
Assumptions \ref{assum:concentration} and \ref{assum:cumulant}, where each set $\cU_k$ may again be taken as
\[\cU_k \equiv \cU=\{\e_1,\ldots,\e_n\}\cup\{X\e_1,\ldots,X\e_d\}.\]
\end{proposition}

\begin{example}
Suppose that $\w$ is a uniformly random vector on the sphere $\{\w \in \R^d:\|\w\|_2^2=d\}$. It is clear that all conditions of Proposition \ref{prop:sign change invariant} hold for $\w$.
\end{example}

\begin{example}
Suppose $\w$ is an isotropic ``unconditionally'' log-concave random vector, i.e.\ the log-density $\log p(\w) \in [-\infty,\infty)$ is a concave function on $\R^d$, $\E\w\w^*=I$,
and $\w$ is equal in law to $\w \odot \s$
(the entrywise product) for any sign vector $\s \in \{\pm 1\}^d$. The random Gram matrix
\[\widetilde K
=\frac{1}{N}G^*G,
\qquad G=[\w_1,\ldots,\w_N]\]
where $\w_1,\ldots,\w_N$ are i.i.d.\ copies of $\w$ was studied recently in
\cite{bao2025extreme}, who built upon an entrywise local law for $\widetilde K$ in \cite{pillai2014universality} to directly show Tracy-Widom fluctuations of its leading eigenvalue.

For such distributions, Assumption \ref{assum:concentration} for $\w$ follows from a log-concave isoperimetric inequality, as was shown in \cite[Lemma 3.3]{bao2025extreme}. Then it is readily checked that all conditions of Proposition \ref{prop:sign change invariant} hold for $\w$, so our results imply that an isotropic local law holds for $G=[\w_1,\ldots,\w_N]$, and also that an anisotropic local law holds for
$G=[\g_1,\ldots,\g_N]$ where $\g_i=X\w_i$.
\end{example}

\begin{example}\label{ex:mixture model}
Suppose $\w$ follows a mixture model where there exists a random variable $\lambda$ such that $\w=(\w[\alpha])_{\alpha=1}^d$ has i.i.d.\ entries conditional on $\lambda$, and for some constants $C_k>0$,
\begin{equation}\label{eq:mixtureconditionbasic}
\begin{gathered}
\E[\w[\alpha] \mid \lambda]=0,
\;\;\E[|\w[\alpha]|^k \mid \lambda] \leq C_k \quad \text{ for all } k \geq 2 \text{ and } \lambda \in \Lambda,\\
\E[\E[\w[\alpha]^2 \mid \lambda]]=1.
\end{gathered}
\end{equation}
Thus $\w$ has a mean-zero, isotropic, and exchangeable law, and the conditional law $(\w \mid \lambda)$ satisfies Assumption \ref{assum:concentration} uniformly over $\lambda \in \Lambda$. If, in addition,
\begin{equation}\label{eq:mixturecondition}
|\E[\w[\alpha]^2 \mid \lambda]-1| \prec 1/\sqrt{n},
\end{equation}
then all conditions of Proposition \ref{prop:sign change invariant} hold for $\w$. (See Section~\ref{sec:analysis-examples} for a verification of this and the remaining claims in this example.)

In general, the condition \eqref{eq:mixturecondition} is required for the local law estimates of Theorems \ref{thm:entrywise law} and \ref{thm:anisotropic law} to hold for such a model: Consider for example
\begin{equation}\label{eq:mixtureexample}
\P[\lambda=1]=\P[\lambda=-1]=1/2, 
\quad
(\w \mid \lambda) \sim N(0,(1+c_n\lambda)I) \text{ for some } c_n>0.
\end{equation}
Then
$|\E[\w[\alpha]^2 \mid \lambda]-1|=c_n$, so \eqref{eq:mixturecondition} requires $c_n \prec n^{-1/2}$. Note that if $c_n \equiv c$ is a fixed constant, then even the convergence in probability of \eqref{eq:intro:quadraticform} fails; in this setting, the standard Marchenko-Pastur law $\mu_0$ (corresponding to $\Sigma=I$) is not the correct deterministic equivalent for the spectral distribution of $K=N^{-1}\sum_{i=1}^N \w_i\w_i^*$, and one must instead view $K$ as a Gram matrix and take the ``companion" law of the deformed Marchenko-Pastur law that results from conditioning on the corresponding latent variables $\lambda_1,\ldots,\lambda_N$:
\[
    K=N^{-1}\widetilde{W}^*\diag(1+c\lambda_1,...,1+c\lambda_N)\widetilde{W},\quad \widetilde{W}_{ij}\mathop\sim^{\text{i.i.d.}} N(0,1).
\]
If $n^{-1/2} \ll c_n \ll 1$, then $\mu_0$ provides a global approximation for the spectral distribution of $K$, but Assumption \ref{assum:concentration} does not hold, and an averaged local law also does not hold in the quantitative form of Theorem \ref{thm:entrywise law}.

For this model \eqref{eq:mixtureexample}, one may also directly check that Assumption \ref{assum:cumulant} holds under the weaker condition $c_n \prec n^{-1/4}$. Thus for $n^{-1/2} \ll c_n \ll n^{-1/4}$, Assumption \ref{assum:cumulant} holds while Assumption \ref{assum:concentration} does not,
illustrating that Assumption \ref{assum:cumulant} is not strictly stronger than Assumption \ref{assum:concentration}.
\end{example}

\subsubsection{Random features models}
Consider the model $\g=\sigma(X\w)$ where $\w \in \R^d$
has independent entries and $\sigma(\cdot)$ is an entrywise nonlinear map.
This model has been of considerable interest in the statistical learning
literature on random features regression models
\cite{louart2018random,hastie2022surprises,mei2022generalization} and linear
approximants for neural networks \cite{pennington2017nonlinear,fan2020spectra,benigni2021eigenvalue}. The following proposition verifies
that Assumptions \ref{assum:concentration} and \ref{assum:cumulant} both hold
in a linear scaling regime $n \asymp d$ under suitable conditions for
$\sigma(\cdot)$.  We note that we allow the activation function to be different on each entry.

\begin{proposition}\label{prop:RF}
Let $X \in \R^{n \times d}$ be a deterministic matrix,
where $\|X\|_\op \leq C$ and $d \leq Cn$ for a constant $C>0$.
Let $\w \in \R^d$ be a random vector with independent entries, and let
\[\g=\sigma(X\w) \in \R^n\]
where $\sigma:\R^n \to \R^n$ is given by $\sigma(\x)=(\sigma_i(\x[i]))_{i=1}^n$
for scalar functions $\sigma_1,\ldots,\sigma_n:\R \to \R$ applied entrywise.
Suppose $\E\sigma(X\w)=0$ and $\E\sigma(X\w)\sigma(X\w)^*=\Sigma$, where
$\Sigma \in \R^{n \times n}$ satisfies Assumption~\ref{assum:basic}.

\begin{enumerate}[label=(\alph*)]
\item Suppose the entries of $\w=(\w[\alpha])_{\alpha=1}^d$ 
satisfy $\E \w[\alpha]=0$,
$\E \w[\alpha]^2=1$, and $\E|\w[\alpha]|^k \leq C_k$ for each $k \geq 3$ and a
constant $C_k>0$. If, for some constants $C,D>0$, each function
\[\sigma_i(x)=\sum_{l=0}^D a_{il}x^l\]
is a polynomial with degree at most $D$ and coefficients satisfying
$\max_{i=1}^n\max_{l=0}^D |a_{il}|<C$, then $\g$ satisfies
Assumptions \ref{assum:concentration} and \ref{assum:cumulant}.
\item Suppose that $\w \sim N(0,I_d)$ is a standard Gaussian vector.
If, for some constants $C>0$ and $\beta>1/2$, each function $\sigma_i$ admits a
(absolutely convergent) series representation
\[\sigma_i(x)=\sum_{l=0}^\infty a_{il}x^l\]
where $\max_{i=1}^n |a_{il}|<C(l!)^{-\beta}$, then $\g$ also
satisfies Assumptions \ref{assum:concentration} and \ref{assum:cumulant}.
\end{enumerate}
\end{proposition}

Condition (b) is equivalent to the condition that $\sigma_i(x)$ extends to an
entire function on $\C$ satisfying the growth $|\sigma_i(z)|\leq
C\exp(c|z|^{1/\beta})$ for some constants $C,c>0$ and $\beta>1/2$.
In particular, for $\beta=1$, this is the class of entire functions of
exponential type (i.e.\ having Fourier transform with compact support
\cite[Theorem 19.3]{rudin1987real}),
which has also been considered recently in a slightly different context of random inner-product kernel
matrices \cite{kogan2024extremal,lu2025equivalence}.

\subsubsection{Random features tilt}
A related class of examples arises when the random features structure appears not in the random vector itself but in its density. Consider a random vector $\w \in \R^n$ with density proportional to $\exp(-U(\w))$, where
\begin{equation}\label{eq:tiltpotential}
U(\w) = \frac{1}{2}\|\w\|_2^2 + \lambda\sum_{i=1}^d \sigma_i(\x_i^*\w)
\end{equation}
for deterministic directions $\x_1,\ldots,\x_d \in \R^n$, smooth nonlinearities $\sigma_i$, and a coupling constant $\lambda$. The following proposition establishes that for $|\lambda|$ sufficiently small, $d$ at most polynomial in $n$, and under a boundedness condition for the derivatives of $\sigma_i(\cdot)$, the centered vector $\g=\w-\E\w$ satisfies both Assumptions~\ref{assum:concentration} and \ref{assum:cumulant}.

\begin{proposition}\label{prop:spiked cumulant}
Let $X\in\R^{d\times n}$ be a deterministic matrix with rows $\{\x_i\}_{i=1}^d$
such that $\|X\|_\op\leq C$ and $d\leq n^C$ for a constant $C>0$. Suppose
$\w \in \R^n$ is a random vector having density proportional to $\exp(-U(\w))$,
where $U(\w)$ is given by \eqref{eq:tiltpotential}, and $\sigma_1,\ldots,\sigma_d:\R\rightarrow\R$ are smooth functions
satisfying
\[\sup_{\a\in\R^d}\|(\sigma_i^{(k)}(\a[i]))_{i=1}^d\|_2 \leq C^k \text{ for all } k
\geq 1\]
(where $\sigma_i^{(k)}$ is the $k^\text{th}$-order derivative). Let
\[\g=\w-\E \w, \qquad \Sigma=\E \g\g^*.\]
Then there exists a constant $\lambda_*>0$ such that whenever
$|\lambda|<\lambda_*$, Assumptions
\ref{assum:basic}, \ref{assum:concentration}, and \ref{assum:cumulant} hold for
$\Sigma$ and $\g$. Each set $\cU_k$ of Assumption \ref{assum:cumulant} may be
taken as
\[\cU_k\equiv\cU=\{\e_1,\ldots,\e_n\}\cup\{\x_1,\ldots,\x_d\},\]
and the bound \eqref{eq:cumulantassumption} holds in the stronger form
(for any $k \geq 3$, $1 \leq m \leq k-1$, and a constant $C \equiv C(k)>0$)
\[\<\kappa_k(\g),\s_1\otimes\cdots\otimes\s_m\otimes T\>\leq C
\left(\prod_{i=1}^m\|\s_i\|_2\right)\|T\|_\cU.\]
\end{proposition}

This model is of particular interest when $d$ is bounded and $\|\x_i\|_2 = 1$: the higher-order cumulant tensors of $\g$ then contain spikes along the directions $\{\x_i\}_{i=1}^d$.
Moreover this higher cumulant structure persists even in the 
whitened vector $\tilde{\g} = \Sigma^{-1/2}\g$ which has identity covariance and satisfies both Assumption \ref{assum:concentration} and \ref{assum:cumulant}. Our local law implies that the resolvent of the sample covariance of i.i.d.\ copies of $\tilde\g$ is well-approximated by the Marchenko-Pastur deterministic equivalent, showing that even refined local statistics of the eigenvalues and eigenvectors of the sample covariance matrix alone cannot recover the directions $\x_1,\ldots,\x_d$, in the regime of sample size $N$ proportional to the dimension $n$. This style of ``spiked cumulant'' model and the associated hidden-direction recovery task is closely related to sample-complexity and computational lower-bound results for non-Gaussian component analysis and higher-order correlation models \cite{wang2019onlineica,tan2018ngca,diakonikolas2023sq,diakonikolas2024sosngca,szekely2024higherorder}. 

\subsection{Negative examples}\label{sec:negexample}

For illustration, we discuss also some examples which do not satisfy the hypotheses of Assumptions \ref{assum:concentration} and \ref{assum:cumulant}, to illustrate the necessity of the assumptions and/or limitations of our current techniques.

\subsubsection{Chaos vectors}
Let $M,d$ be integers with $1 \leq M \leq d$.
For a vector $\x = (\x[1], \ldots, \x[d]) \in \R^d$, of i.i.d.\ standard normal random variables, define a vector
\[
\g[i,j] = \x[i]\x[j] \quad \text{for } 1 \leq i < j \leq \min\{i + M, d\}.
\]
(The case of $M=d$ corresponds to all pairwise products of entries of $X$.)
We set $n$ to be the size of the indexing set $\{(i,j):1 \leq i<j \leq \min(i+M,d)\}$, which is up to constants $n \asymp M d$.  This random vector provides an example which satisfies the weak concentration of quadratic forms \eqref{eq:intro:quadraticform} but does not have the rate in Assumption \ref{assum:concentration}; moreover, it is easily verified not to satisfy the averaged local law with the optimal rate.

More precisely, for any test matrix $A \in \R^{n \times n}$, it is readily checked that
\[
\E|\g^* A \g - \E \g^* A \g|^2 \asymp M \|A\|_F^2,
\]
and hence from hypercontractivity of the Gaussian measure, for any $k \geq 2$,
\[
\E |\g^* A \g - \E \g^* A \g|^k \leq C_k (\E |\g^* A \g - \E \g^* A \g|^2)^{k/2} \lesssim M^{k/2} \|A\|_F^k.
\]
This suffices for the Bai-Zhou theorem \cite{bai2008large} for any $M \leq d$, i.e.\ the Marchenko-Pastur law holds at the global spectral scale for the sample covariance matrix $K$ defined by $N \asymp n$ samples. However, it is also easily checked that
\[
\operatorname{Var}(\operatorname{Tr} K) \asymp M n/N \asymp M.
\]
By a simple fourth moment estimate for $\operatorname{Tr} K - \E \operatorname{Tr} K$, we can argue with probability bounded below independently of $n$ that
\[
    |\operatorname{Tr} K - \E \operatorname{Tr} K| \gtrsim \sqrt{M}.
\]
Hence when $M = d^{\alpha}$ for any $\alpha \in (0,1]$, this contradicts the eigenvalue rigidity that would follow from an optimal local law. Thus the rate of concentration assumed in Assumption \ref{assum:concentration} is needed to ensure that an averaged local law holds at the optimal scale.

\subsubsection{Spin glass models}
As an illustration of the role of Assumption~\ref{assum:cumulant}, consider a spherical spin glass model. Let $J \in (\R^n)^{\otimes p}$ be a symmetric random tensor with i.i.d.\ standard Gaussian entries (up to symmetry), scaled so that $\|J\|_F \asymp n^{p/2}$, and let $\g$ be drawn from the Gibbs measure
\[
d\mu(\g) \propto \exp\bigl(\tfrac{\beta}{n^{(p-1)/2}}\<J,\g^{\otimes p}\>\bigr)
\]
on the sphere $\{\g \in \R^n:\|\g\|_2^2=n\}$, where $\beta>0$ is the inverse temperature or coupling constant. For $\beta < \beta_0$ some absolute constant, the measure $\mu$ is easily checked to satisfy a log-Sobolev inequality with a constant independent of $n$, which implies concentration of quadratic forms; hence Assumption~\ref{assum:concentration} holds and the averaged local law of Theorem~\ref{thm:entrywise law} applies to the sample covariance matrix of i.i.d.\ draws from $\mu$.

However, for the spherical $4$-spin model, it can be quantified rigorously (c.f.\ \cite{fanma2026spherical4spin}) that
\[
\|\kappa_4\|_{\text{inj}} \leq C\Big(\beta^3 + \frac{\beta}{\sqrt{n}} + \frac{1}{n}\Big),
\qquad
\|\kappa_4\|_F \geq c\beta \sqrt{n}.
\]
At $\beta=n^{-1/6}$, this gives $\|\kappa_4\|_{\text{inj}}/\|\kappa_4\|_F^2 \leq C n^{-7/6}$, exhibiting a small ratio of injective norm to Frobenius norm.
Then by the reasoning discussed in Remark \ref{remark:genericcumulant},
Assumption \ref{assum:cumulant} does not hold, so our current results do not imply an optimal anisotropic local law in the form of Theorem \ref{thm:anisotropic law} for the sample covariance matrix of i.i.d.\ draws from such a model.
In this case of ``disordered'' cumulant tensors, we note that cancellations of fluctuations in the cumulant tensors themselves must be considered to obtain optimal estimates of various cumulant contractions, and this is not captured by the proof method we employ based on Assumption \ref{assum:cumulant}.

\vspace{1em}
\subsection{Proof ideas}\label{subsec:proofideas}

We conclude this section by explaining briefly the main novelties of the
proofs. For any subset $S \subset
[N]$, let $R^{(S)}=(N^{-1}\sum_{i \notin S} \g_i\g_i^*-zI)^{-1}$ denote the
resolvent of the sample covariance matrix leaving out the samples in $S$. Then
a classical argument of \cite{silverstein1995strong}, applying the definition of
$z_0(\widetilde m)$ and the Sherman-Morrison identity
$R=R^{(i)}-N^{-1}R^{(i)}\g_i\g_i^*R^{(i)}/(1+N^{-1}\g_i^*R^{(i)}\g_i)$, shows the
equality
\[z_0(\widetilde m)-z={-}\frac{1}{\widetilde m} \cdot \frac{1}{N}
\sum_{i=1}^n \frac{N^{-1}\Tr(\g_i\g_i^*-\Sigma)R^{(i)}(I+\widetilde
m\Sigma)^{-1}}{1+N^{-1}\g_i^*R^{(i)}\g_i}.\]
We introduce a control parameter $\Phi=\Phi(z) \in [N^{-1/2},1]$ depending on $z$ and $N$, whose explicit form is not important here (c.f.\ Lemma \ref{lem:fluctuation_averaging}).
Assuming the two a priori estimates
\begin{equation}\label{eq:apriori}
N^{-1}\|R^{(S)}\|_F \prec \Phi \text{ over all } S \subset [N]
\text{ of bounded cardinality},
\quad |\widetilde m-\widetilde m_0| \ll 1,
\end{equation}
it is readily checked from
Assumption \ref{assum:concentration} and regularity of the spectral domain that
\[\left|\frac{N^{-1}\Tr(\g_i\g_i^*-\Sigma)R^{(i)}(I+\widetilde
m\Sigma)^{-1}}{1+N^{-1}\g_i^*R^{(i)}\g_i}\right| \prec \Phi,\]
and hence also $|z_0(\widetilde m)-z| \prec \Phi$. The main technical step to
establish an optimal averaged local law is to improve this estimate to
\begin{equation}\label{eq:FAintro}
|z_0(\widetilde m)-z| \prec \Phi^2.
\end{equation}
The averaged local law then follows by applying the stability condition of \eqref{eq:stability} in conjunction with
a stochastic continuity argument as developed and used in
e.g.\ \cite{erdHos2009local,erdos2013local,erdos2013spectral,pillai2014universality,knowles2017anisotropic}.

To illustrate the fluctuation averaging mechanism that leads to
\eqref{eq:FAintro}, and also to explain our main novelties for extending this
to an anisotropic local law, consider the slightly simpler quantity
\begin{align*}
\Delta(A)=\sum_{i=1}^n N^{-1}\Tr(\g_i\g_i^*-\Sigma)R^{(i)}A
\end{align*}
for any deterministic matrix $A \in \C^{n \times n}$ with $\|A\|_\op \leq 1$.
The bound \eqref{eq:FAintro} is analogous to the simpler bound $|\Delta(A)| \prec
N\Phi^2$ (ignoring the factors in the denominator of $z_0(\tilde{m})-z$, which do not significantly alter the proof strategy).
Applying the Sherman-Morrison identity
to resolve the dependence of $R^{(i)}$ on $\g_j$ and vice versa,
\begin{align*}
&\E|\Delta(A)|^2=\sum_{i,j=1}^n
\E\Big[N^{-1}\Tr (\g_i\g_i^*-\Sigma)R^{(i)}A
\cdot \overline{N^{-1}\Tr (\g_j\g_j^*-\Sigma)R^{(j)}A}\Big]\\
&=\Oprec(N\Phi^2)+\frac{1}{N}\sum_{i \neq j}
\E\Bigg[
\Bigg(\underbrace{\Tr(\g_i\g_i^*-\Sigma)\frac{R^{(ij)}A}{\sqrt{N}}}_{:=\cY_i^{(ij)}}
-\underbrace{\Tr(\g_i\g_i^*-\Sigma)\frac{R^{(ij)}}{N}\g_j\g_j^*\frac{R^{(ij)}A}{\sqrt{N}}}_{:=Z_{ijj}^{(ij)}}
\cdot \underbrace{\frac{1}{1+N^{-1}\g_j^*R^{(ij)}\g_j}}_{:=\cQ_j^{(ij)}}\Bigg)\\
&\hspace{0.5in}
\times \overline{\Bigg(\underbrace{\Tr
(\g_j\g_j^*-\Sigma)\frac{R^{(ij)}A}{\sqrt{N}}}_{:=\cY_j^{(ij)}}
-\underbrace{\Tr (\g_j\g_j^*-\Sigma)\frac{R^{(ij)}}{N}\g_i\g_i^*\frac{R^{(ij)}A}{\sqrt{N}}}_{:=\cZ_{jii}^{(ij)}}
\cdot \underbrace{\frac{1}{1+N^{-1}\g_i^*R^{(ij)}\g_i}}_{:=\cQ_i^{(ij)}}\Bigg)}\Bigg]
\end{align*}
Observe that $\E[\cY_i^{(ij)}\overline{\cY_j^{(ij)}}]=0$,
$\E[\cY_i^{(ij)}\overline{\cZ_{jii}^{(ij)}\cQ_i^{(ij)}}]=0$, and
$\E[\cZ_{ijj}^{(ij)}\cQ_j^{(ij)}\overline{\cY_j^{(ij)}}]=0$.
Assumption~\ref{assum:concentration} and the a priori estimates
\eqref{eq:apriori} imply that $|\cZ_{jii}^{(ij)}|,|\cZ_{ijj}^{(ij)}| \prec
\sqrt{N}\Phi^2$ and $|\cQ_i^{(ij)}|,|\cQ_j^{(ij)}| \prec 1$. Hence
$\E|\Delta(A)|^2 \prec N^2\Phi^4$, and extending this to a high
moment estimate shows $|\Delta(A)| \prec N\Phi^2$ as desired. This fluctuation averaging idea was used in \cite{wang2024nonlinear} in the same context of a
non-separable sample covariance matrix model as our current work, to analyze outlier eigenvalues/eigenvectors in the setting of a spiked covariance $\Sigma$.

Showing an optimal anisotropic local law for the resolvent $R(z)$ will instead
require a stronger estimate that is analogous to a bound of the form
$|\Delta(\u\u^*)| \prec \Phi$ for
a rank-one matrix input $A=\u\u^*$, assuming the a priori conditions
in place of \eqref{eq:apriori},
\begin{equation}\label{eq:apriori2}
N^{-1/2}\|R^{(S)}\x\|_2 \prec \Phi \text{ over all deterministic unit vectors }
\x \in \C^n,
\quad |\widetilde m-\widetilde m_0| \ll 1.
\end{equation}
Applying the same expansion for $\E|\Delta(\u\u^*)|^2$, we may readily check as
above that
\begin{equation}\label{eq:Deltauu}
\E|\Delta(\u\u^*)|^2=\Oprec(\Phi^2)+\frac{1}{N}\sum_{i \neq j}
\E[\cZ_{ijj}^{(ij)}\cQ_j^{(ij)}\overline{\cZ_{jii}^{(ij}\cQ_i^{(ij)}}]
\end{equation}
where now $|\cZ_{ijj}^{(ij)}|,|\cZ_{jii}^{(ij)}| \prec \Phi^2$
and $|\cQ_i^{(ij)}|,|\cQ_j^{(ij)}| \prec 1$ under
\eqref{eq:apriori2}. However, the primary challenge is that this only
gives a bound of $\E|\Delta(\u\u^*)|^2 \prec N\Phi^4$, leading to
\begin{equation}\label{eq:FAbad}
|\Delta(\u\u^*)| \prec \sqrt{N}\Phi^2,
\end{equation}
which does not imply the desired bound
$|\Delta(\u\u^*)| \prec \Phi$ when $z \in \C^+$ is a local spectral parameter.
Overcoming this challenge was a central technical contribution of
\cite{knowles2017anisotropic},
who proposed instead a method of bootstrapping on the spectral scale in
multiplicative increments of $N^{-\delta}$ in the imaginary part, and
using the additional a priori estimate
\begin{equation}\label{eq:apriori3}
|\x^* R^{(S)}\y| \prec N^{2\delta} \text{ over all deterministic unit
vectors } \x,\y \in \C^n
\end{equation}
obtained from the preceding bootstrap iteration to establish the bound
\begin{equation}\label{eq:bootstrapintro}
|\x^* R^{(S)}\y-\x^*(-zI-z\widetilde m_0\Sigma)^{-1}\y| \prec
N^{C\delta}\Phi.
\end{equation}
This then implies also $|\x^* R^{(S)}\y| \prec 1$, which allows one to
continue the bootstrap. The argument in \cite{knowles2017anisotropic} for showing
\eqref{eq:bootstrapintro} from \eqref{eq:apriori3} was specific to a separable
model $\g=\Sigma^{1/2}\w$ where $\w$ has independent entries, and relied
on a technically intricate interpolation between $\w$ and a Gaussian vector
$\z \sim N(0,I)$.

In this work, we employ the same bootstrapping structure of
\cite{knowles2017anisotropic}, but
pursue a more general proof by directly establishing the desired fluctuation averaging bound
\begin{equation}\label{eq:FAimproved}
|\Delta(\u\u^*)| \prec N^{2\delta} \Phi
\end{equation}
under the a priori assumptions
\eqref{eq:apriori2} and \eqref{eq:apriori3}, without
interpolation. We note that the above estimates
$|\cZ_{ijj}^{(ij)}|,|\cZ_{jii}^{(ij)}| \prec \Phi^2$ 
and $|\cQ_i^{(ij)}|,|\cQ_j^{(ij)}| \prec 1$ leading to \eqref{eq:FAbad}
are sharp, and that this improved bound \eqref{eq:FAimproved} will instead arise from verifying that under the cumulant condition of
Assumption~\ref{assum:cumulant}, the expectation of
$\cZ_{ijj}^{(ij)}\cQ_j^{(ij)}\overline{\cZ_{jii}^{(ij)}\cQ_i^{(ij)}}$
is much smaller than its typical size.
We will show this through a power series expansion of
$\cQ_i^{(ij)}$ together with a moment-cumulant expansion of
$\E\cZ_{ijj}^{(ij)}\cQ_j^{(ij)}\overline{\cZ_{jii}^{(ij)}\cQ_i^{(ij)}}$. 

To illustrate this calculation, consider the simplest term
$\E\cZ_{ijj}^{(ij)}\overline{\cZ_{jii}^{(ij)}}$.
Recalling the forms of $\cZ_{ijj}^{(ij)}$ and $\cZ_{jii}^{(ij)}$ with
$A=\u\u^*$, we observe that $\E\cZ_{ijj}^{(ij)}\overline{\cZ_{jii}^{(ij)}}$
depends on the 4th-order moment tensors $\E \g_i^{\otimes 4}$ and
$\E \g_j^{\otimes 4}$, and one term in the resulting moment-cumulant expansion takes the form
\begin{align}\label{eq:momentcumulantintro}
\val:=&\sum_{\alpha_1,\ldots,\alpha_4,\beta_1,\ldots,\beta_4=1}^n
\kappa_4(\g_i)[\alpha_1,\ldots,\alpha_4]
\kappa_4(\g_j)[\beta_1,\ldots,\beta_4] \notag\\
&\hspace{1in}\times
\overline{\u}[\alpha_1]\frac{\overline{R^{(ij)}\u}}{\sqrt{N}}[\alpha_2]
\u[\beta_1]\frac{R^{(ij)}\u}{\sqrt{N}}[\beta_2]\frac{R^{(ij)}}{N}[\alpha_3,\beta_3]
\frac{\overline{R^{(ij)}}}{N}[\alpha_4,\beta_4].
\end{align}
We observe that $\val$ takes the form $\val=\langle \kappa_4(\g_i),\,\overline{\u} \otimes \frac{\overline{R^{(ij)}\u}}{\sqrt{N}} \otimes T_i\rangle$, where $T_i \in (\C^n)^{\otimes 2}$ is the order-2 tensor encoding the remaining contraction involving $\kappa_4(\g_j)$.
Applying Assumption~\ref{assum:cumulant} with $k=4$ and $m=2$ gives
\begin{equation}\label{eq:step1intro}
|\val| \prec (\sqrt{N})^{4-2-1}\,\|\u\|_2 \left\|\frac{R^{(ij)}\u}{\sqrt{N}}\right\|_2\,\|T_i\|_{\cU_4}.
\end{equation}
To bound the $\cU_4$-norm, we expand $\|T_i\|_{\cU_4}=\sup_{\x,\y \in \cU_4}|\langle \x \otimes \y,T_i\rangle|$.
For each fixed $\x,\y \in \cU_4$,
\[\langle \x \otimes \y,T_i\rangle
=\Big\langle \kappa_4(\g_j),\,\u \otimes \tfrac{R^{(ij)}\u}{\sqrt{N}}
\otimes \underbrace{\tfrac{R^{(ij)*}\x}{N}\otimes\tfrac{\overline{R^{(ij)}}^*\y}{N}}_{=:\,T_j(\x,\y)}\Big\rangle.\]
A second application of Assumption~\ref{assum:cumulant} with $k=4$ and $m=2$ then gives
\begin{align*}
|\langle \x \otimes \y,T_i\rangle|
&\prec (\sqrt{N})^{4-2-1}\,\|\u\|_2
\left\|\frac{R^{(ij)}\u}{\sqrt{N}}\right\|_2\,\|T_j(\x,\y)\|_{\cU_4}\\
&\leq (\sqrt{N})^{4-2-1}\,\|\u\|_2
\left\|\frac{R^{(ij)}\u}{\sqrt{N}}\right\|_2
\left(\sup_{\v,\w \in \cU_4}\left|\frac{\v^* R^{(ij)}\w}{N}\right|\right)^2.
\end{align*}
Since $|\cU_4| \leq n^{C_4}$ has polynomial cardinality, combining with \eqref{eq:step1intro} and applying the a priori conditions \eqref{eq:apriori2}
and \eqref{eq:apriori3} shows that $|\val|=\Oprec(N^{-1+4\delta}\Phi^2)$.

For general terms in the high moment expansion, to manage the algebraic complexity, we will introduce a tensor network formalism and systematically map the moment expansions to a
family of tensor networks. Note that the example above has the following diagrammatic representation:
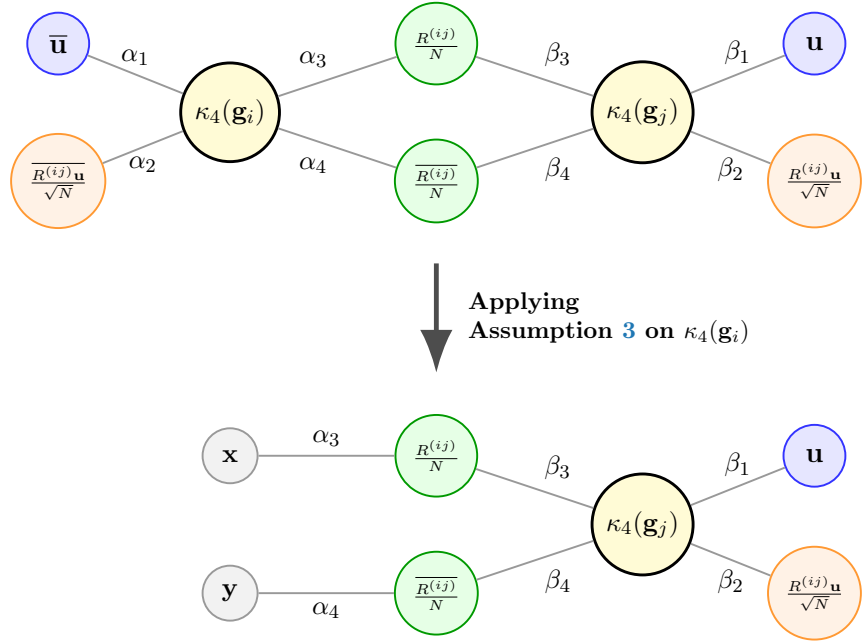
\begin{figure}[H]
    \centering
    \resizebox{0.7\textwidth}{!}{
    \begin{tikzpicture}[
    node distance=2.5cm,
    leaf_node/.style={circle, draw=blue!80, thick, fill=blue!10, minimum size=0.9cm, font=\large},
    tensor_node/.style={circle, draw=black, very thick, fill=yellow!20, minimum size=1.2cm, font=\bfseries},
    matrix_node/.style={circle, draw=green!60!black, thick, fill=green!10, minimum size=1.0cm, font=\footnotesize},
    vec_node/.style={circle, draw=orange!80, thick, fill=orange!10, minimum size=1.0cm, font=\footnotesize},
    input_vec/.style={circle, draw=gray!80, thick, fill=gray!10, minimum size=0.8cm, font=\itshape},
    conn/.style={thick, draw=gray!80},
    arrow_style/.style={-{Latex[scale=1.5]}, line width=2pt, draw=black!70}
]


    \node[tensor_node] (A_ki) at (-3,0) {$\kappa_4(\g_i)$};

    \node[tensor_node] (A_kj) at (3,0) {$\kappa_4(\g_j)$};

    \node[matrix_node] (A_m1) at (0, 1.0) {$\frac{R^{(ij)}}{N}$};
    \node[matrix_node] (A_m2) at (0, -1.0) {$\frac{\overline{R^{(ij)}}}{N}$};

    \node[leaf_node] (A_ubar) at (-5.5, 1.0) {$\overline{\u}$};
    \node[vec_node] (A_Rubar) at (-5.5, -1.0) {$\frac{\overline{R^{(ij)}\u}}{\sqrt{N}}$};

    \node[leaf_node] (A_u) at (5.5, 1.0) {$\u$};
    \node[vec_node] (A_Ru) at (5.5, -1.0) {$\frac{R^{(ij)}\u}{\sqrt{N}}$};

    \draw[conn] (A_ubar) -- (A_ki) node[midway, above] {$\alpha_1$};
    \draw[conn] (A_Rubar) -- (A_ki) node[midway, below] {$\alpha_2$};
    \draw[conn] (A_ki) -- (A_m1) node[midway, above left] {$\alpha_3$};
    \draw[conn] (A_m1) -- (A_kj) node[midway, above right] {$\beta_3$};
    \draw[conn] (A_ki) -- (A_m2) node[midway, below left] {$\alpha_4$};
    \draw[conn] (A_m2) -- (A_kj) node[midway, below right] {$\beta_4$};
    \draw[conn] (A_kj) -- (A_u) node[midway, above] {$\beta_1$};
    \draw[conn] (A_kj) -- (A_Ru) node[midway, below] {$\beta_2$};

    \draw[arrow_style] (0, -2.2) -- (0, -3.8) 
        node[midway, right, xshift=0.3cm, align=left, font=\small\bfseries] {Applying\\Assumption \ref{assum:cumulant} on $\kappa_4(\g_i)$};

    \begin{scope}[yshift=-6cm]

        \node[tensor_node] (B_kj) at (3,0) {$\kappa_4(\g_j)$};

        \node[matrix_node] (B_m1) at (0, 1.0) {$\frac{R^{(ij)}}{N}$};
        \node[matrix_node] (B_m2) at (0, -1.0) {$\frac{\overline{R^{(ij)}}}{N}$};

        \node[input_vec] (B_x) at (-3, 1.0) {$\x$};
        \node[input_vec] (B_y) at (-3, -1.0) {$\y$};

        \node[leaf_node] (B_u) at (5.5, 1.0) {$\u$};
        \node[vec_node] (B_Ru) at (5.5, -1.0) {$\frac{R^{(ij)}\u}{\sqrt{N}}$};

        \draw[conn] (B_x) -- (B_m1) node[midway, above] {$\alpha_3$};
        \draw[conn] (B_y) -- (B_m2) node[midway, below] {$\alpha_4$};
        
        \draw[conn] (B_m1) -- (B_kj) node[midway, above right] {$\beta_3$};
        \draw[conn] (B_m2) -- (B_kj) node[midway, below right] {$\beta_4$};
        
        \draw[conn] (B_kj) -- (B_u) node[midway, above] {$\beta_1$};
        \draw[conn] (B_kj) -- (B_Ru) node[midway, below] {$\beta_2$};
    \end{scope}

\end{tikzpicture}
    }
\end{figure}
\noindent As illustrated in this example, we employ a recursive ``peeling'' strategy: we successively
remove high-degree vertices (representing higher-order cumulants) from the graph,
using Assumption~\ref{assum:cumulant} to control the cost of each removal,
until the network is reduced to simple chains and cycles that can be bounded
directly (see Lemmas \ref{lem:remove tensor weak}--\ref{lem:pairing case}).

Finally, we emphasize the dual role of our tensor network framework. Beyond driving the proofs of our main results, it provides the essential machinery for verifying Assumption~\ref{assum:cumulant} in complex nonlinear settings. While verification is relatively straightforward for linear or sign-invariant models (Propositions \ref{prop:linear} and \ref{prop:sign change invariant}), the random features model (Proposition \ref{prop:RF}) presents significant combinatorial challenges. In Section \ref{subsec:RF}, we address these by coupling a monomial expansion of $\g=\sigma(X\w)$ with a tensor network representation of the cumulant form $\langle\kappa(\g),\s_1 \otimes \ldots \otimes \s_m \otimes T\rangle$. This allows us to systematically bound the resulting diagrams and rigorously establish the required cumulant decay.

\section{Fluctuation averaging lemmas}\label{sec:fluctuation averaging}

The primary technical input for our main results consists of a sequence of
fluctuation averaging lemmas
that resolve the weak dependence of the
resolvent $R(z)$ on each individual sample $\g_j$ for $j \neq i$.
In this section, we first state and prove these
fluctuation averaging lemmas.

\begin{defi}[Minors]
For any $S \subseteq [N]$, let $G^{(S)} \in \R^{n \times N}$ be the matrix with
columns
\[G^{(S)}\e_i = \begin{cases} \g_i & \text{if $i\notin S$}\\
        0 & \text{otherwise}. \end{cases}\]
Let
\[
    K^{(S)} = \frac{1}{N}G^{(S)} G^{(S)*}=\frac{1}{N}\sum_{i\in [N]\setminus
S}\g_i\g_i^*,\qquad \widetilde{K}^{(S)}=\frac{1}{N}G^{(S)*} G^{(S)},
\]
and denote their resolvents and Stieltjes transforms
\[R^{(S)}(z)=(K^{(S)}-zI_n)^{-1}, \quad m^{(S)}(z) = n^{-1}\Tr R^{(S)}(z),\]
\[\widetilde{R}^{(S)}(z)=(\widetilde{K}^{(S)}-zI_N)^{-1},\quad
\widetilde{m}^{(S)}(z)=N^{-1}\Tr\widetilde{R}^{(S)}(z).\]
The expectation with respect to the columns $\{\g_i\}_{i \in S}$ of $G$ is
denoted as
\[\E_S[\cdot] := \E[\;\cdot \mid \{\g_i\}_{i \in [N] \setminus S}].\]
\end{defi}
We will often omit the spectral argument $z$ for brevity when the meaning is
clear.

\begin{defi}[Stochastic domination]
For $(N,n)$-dependent random variables $X \in \R$ and $Y \geq 0$, we write
\[X \prec Y \quad \text{ or } \quad X=\Oprec(Y)\]
if, for any constants $\eps,D>0$, there exists $C \equiv C(\eps,D)>0$ such that
$\P[|X| \geq N^\eps Y] \leq CN^{-D}$. (Equivalently, there exists $N_0 \equiv
N_0(\eps,D)$ such that $\P[|X| \geq N^\eps Y] \leq N^{-D}$ for all $N \geq
N_0$.) For $X \equiv X(u)$ and $Y \equiv Y(u)$ depending on a parameter $u \in
\cU^{(N)}$, we say that
\[X \prec Y \quad \text{ or } \quad X=\Oprec(Y) \qquad \text{ uniformly over }
u \in \cU^{(N)}\]
if the preceding holds for a uniform constant $C \equiv C(\eps,D)>0$ and all $u
\in \cU^{(N)}$.

This constant $C(\eps,D)$ may depend on other constant quantities in the
context of the statement, including $\tau$ and the regularity constants
$C,c>0$ in Definition \ref{def:regular} for a regular spectral domain.
\end{defi}

We will often use implicitly the following well-known properties of stochastic domination $\prec$. For the proof of the following lemma, we refer to \cite[Lemma D.2]{fan2022tracy}. 
\begin{lemma}\label{lemma:domination}
\phantom{ }
    \begin{enumerate}[label=(\alph*)]
        \item If $X(u,v)\prec \zeta(u,v)$ uniformly over $u\in U$ and $v\in V$,
and $|V|\leq n^C$ for some constant $C>0$, then uniformly over $u \in U$,
        \begin{align*}
            \sum_{v\in V} X(u,v) \prec \sum_{v\in V}\zeta(u,v)
        \end{align*}
        \item If $X_1 \prec \zeta_1$ and $X_2 \prec \zeta_2$
uniformly over $u \in U$, then also $X_1X_2\prec \zeta_1\zeta_2$
uniformly over $u \in U$.
        \item 
Suppose $X \prec \zeta$ uniformly over $u \in U$, where $\zeta$ is
deterministic, $\zeta>n^{-C}$,
and $\E[|X|^k] \leq n^{C_{k}}$ for all $k \in [1,\infty)$ and some constants
$C,C_{k}>0$. Then $\E[X \mid \mathscr{G}] \prec \zeta$ uniformly over
$u \in U$ and over all sub-sigma-fields $\mathscr{G} \subseteq \mathscr{F}$
of the underlying probability space $(\Omega,\mathscr{F},\P)$.
    \end{enumerate}
\end{lemma}

For any $z=E+i\eta\in \C^+$ and $S \subset [N]$,
let us define the control parameter
\[\Gamma^{(S)}(z)=\max_{i \notin S}|\widetilde{R}_{ii}^{(S)}(z)-\widetilde{m}_0(z)|.\]
The following lemma will be used to show the averaged local law of
Theorem \ref{thm:entrywise law}.

\begin{lemma}\label{lem:fluctuation_averaging}
Suppose Assumptions \ref{assum:basic} and \ref{assum:concentration} hold,
and $\bD \subset \C^+$ is a regular spectral domain. Suppose
there exists a constant $\tau'>0$ and a deterministic function
$\Phi:\bD \rightarrow [N^{-1/2},N^{-\tau'}]$ such that for any fixed
$L\geq 1$, uniformly over all $S\subseteq[N]$ with $|S|\leq L$,
all $z \in \bD$, we have
\[\Gamma^{(S)}\prec N^{-\tau'}, \qquad \frac{\|R^{(S)}\|_F}{N} \prec \Phi, \]
Then uniformly over all $z\in\bD$ and deterministic matrices $A\in\C^{n\times
n}$ with $\|A\|_\op\leq 1$,
\[\left|\sum_{i=1}^N \frac{N^{-1}\Tr(\g_i\g_i^*-\Sigma)R^{(i)}A}{1+N^{-1}\g_i^*R^{(i)}\g_i}\right|
\prec N\Phi^2.\]
\end{lemma}

To establish anisotropic local laws, our fluctuation averaging lemmas will
assume, as a stronger input, the condition $N^{-1/2}\|R^{(S)}\x\|_2 \prec \Phi$
uniformly over unit vectors $\x \in \C^n$. We note that this implies
the preceding condition $N^{-1}\|R^{(S)}\|_F \prec \Phi$, since
$N^{-1}\|R^{(S)}\|_F \leq \max_{\alpha=1}^n N^{-1/2}\|R^{(S)}\e_\alpha\|_2$.
The following lemma will be used to show Theorem \ref{thm:outside spec} outside
the spectrum.

\begin{lemma}\label{lem:fluctuation_averaging_lowrank}
Suppose Assumptions \ref{assum:basic} and \ref{assum:concentration} hold,
and $\bD \subset \C^+$ is a regular spectral domain. Suppose
there exists a constant $\tau'>0$ and a deterministic function
$\Phi:\bD \rightarrow [N^{-1/2},N^{-\tau'}]$ such that for any fixed
$L\geq 1$, uniformly over all $S\subseteq[N]$ with $|S|\leq L$,
all $z \in \bD$, and all deterministic unit vectors $\x \in \C^n$, we have
\[\Gamma^{(S)}\prec N^{-\tau'}, \qquad \frac{\|R^{(S)}\x\|_2}{\sqrt{N}}
\prec \Phi.\]
Then 
\begin{enumerate}[label=(\alph*)]
\item Uniformly over all $z \in \bD$ and deterministic unit
vectors $\u \in \C^n$,
\[\left|\sum_{i=1}^N \frac{N^{-1}\u^*(\g_i\g_i^*-\Sigma)R^{(i)}\u}{1+N^{-1}\g_i^*R^{(i)}\g_i}\right|
\prec \sqrt{N}\Phi^2.\]
\item Uniformly over all $z\in\bD$ and deterministic unit vectors $\v\in\C^N$,
\[\left|\sum_{i\neq
j}\frac{\bar{\v}[i]\v[j]N^{-1}\g_i^*R^{(ij)}\g_j}{(1+N^{-1}\g_i^*R^{(i)}\g_i)(1+N^{-1}\g_j^*R^{(ij)}\g_j)}\right|\prec
N\Phi^3.\]
\item Uniformly over all $z\in\bD$ and deterministic unit vectors $\v\in\C^N$ and $\u\in\C^n$,
\[\left|\sum_{i=1}^N\frac{\bar{\v}(i)N^{-1/2}\g_i^*
R^{(i)}\u}{1+N^{-1}\g_i^*R^{(i)}\g_i}\right|\prec \sqrt{N}\Phi^2.\]
\end{enumerate}
\end{lemma}

The following strengthening of Lemma \ref{lem:fluctuation_averaging} will show,
under the additional condition of Assumption \ref{assum:cumulant}, the full
anisotropic local law of Theorem \ref{thm:anisotropic law}.
For a resolvent with spectral decomposition
$R=\sum_\alpha (\lambda_\alpha-z)^{-1}\v_\alpha\v_\alpha^*$,
we write $\|R\|_1=\sum_\alpha |\lambda_\alpha-z|^{-1}$.

\begin{lemma}\label{lem:fluctuation_averaging_lowrank2}
Suppose, in addition to the conditions of Lemma
\ref{lem:fluctuation_averaging_lowrank}, that Assumption \ref{assum:cumulant}
also holds, and there exists a constant $\delta \geq 0$ such that for any fixed
$L\geq 1$, uniformly over all $S\subseteq[N]$ with $|S|\leq L$,
all $z \in \bD$, and all deterministic unit vectors $\x,\y \in \C^n$, we have
\[N^{-1}\|R^{(S)}\|_1 \prec 1, \qquad \qquad |\x^*R^{(S)}\y|\prec N^\delta.\]
Then
\begin{enumerate}[label=(\alph*)]
\item Uniformly over all $z\in\bD$ and 
deterministic unit vectors $\u\in\C^n$,
\begin{equation}\label{eq:L2 fluctuation averaging}
\left|\sum_{i=1}^N\frac{N^{-1}\u^*(\g_i\g_i^*-\Sigma)R^{(i)}\u}{1+N^{-1}\g_i^* R^{(i)}\g_i}\right|\prec N^\delta\Phi.
\end{equation}
\item Uniformly over all $z\in\bD$ and deterministic unit vectors $\v\in\C^N$,
\[\left|\sum_{i\neq j}\frac{\bar{\v}(i)\v(j)N^{-1}\g_i^*R^{(ij)}\g_j}{(1+N^{-1}\g_i^*R^{(i)}\g_i)(1+N^{-1}\g_j^*R^{(ij)}\g_j)}\right|
\prec N^{2\delta}\Phi.\]
\item Uniformly over all $z\in\bD$ and deterministic unit vectors $\u \in
\C^n$ and $\v\in\C^N$,
\[\left|\sum_{i=1}^N\frac{\bar{\v}(i)N^{-1/2}\g_i^*R^{(i)}\u}{1+N^{-1}\g_i^*R^{(i)}\g_i}\right|
\prec N^\delta\Phi.\]
\end{enumerate}
\end{lemma}

In the remainder of this section, we prove
Lemmas \ref{lem:fluctuation_averaging},
\ref{lem:fluctuation_averaging_lowrank}, and
\ref{lem:fluctuation_averaging_lowrank2}.

\subsection{Preliminaries}

\begin{lemma}[Resolvent identities]\label{lem:resolventidentities}
For any $S\subseteq[N]$:
\begin{enumerate}[label=(\alph*)]
    \item For any distinct $i,j\notin S$,
    \begin{equation}\label{eq:schur complement}
        \widetilde{R}^{(S)}_{ii}=-\frac{1}{z}\frac{1}{1+N^{-1}\g_i^*
R^{(iS)}\g_i},\qquad
\widetilde{R}^{(S)}_{ij}=z\widetilde{R}^{(S)}_{ii}\widetilde{R}^{(iS)}_{jj}
\cdot N^{-1}\g_i^* R^{(ijS)}\g_j.
    \end{equation}
    \item For any $i,j,k\notin S$ (including $i=j$) with $k\notin\{i,j\}$,
    \begin{equation}\label{eq:resolvent identity}
    \widetilde{R}^{(S)}_{ij}=\widetilde{R}^{(kS)}_{ij}+\frac{\widetilde{R}^{(S)}_{ik}\widetilde{R}^{(S)}_{kj}}{\widetilde{R}^{(S)}_{kk}},\qquad \frac{1}{\widetilde{R}_{ii}^{(S)}}=\frac{1}{\widetilde{R}^{(kS)}_{ii}}-\frac{\widetilde{R}^{(S)}_{ik}\widetilde{R}^{(S)}_{ki}}{\widetilde{R}^{(kS)}_{ii}\widetilde{R}^{(S)}_{kk}\widetilde{R}_{ii}^{(S)}}.
    \end{equation}
    \item (Sherman-Morrison) For any $i\notin S$,
\begin{equation}\label{eq:Sherman-Morrison}
R^{(S)}=R ^{(iS)}-\frac{N^{-1}R ^{(iS)}\g_i\g_i^* R^{(iS)}}
{1+N^{-1}\g_i^*R^{(iS)}\g_i}
=R^{(iS)}+z\widetilde{R}^{(S)}_{ii} \cdot N^{-1}R ^{(iS)}\g_i\g_i^* R ^{(iS)}.
\end{equation}
    \item (Ward's identity)
    \begin{equation}\label{eq:Ward's identity}
    \|R^{(S)}\|_F^2:=\Tr R^{(S)}R^{(S)*}=\frac{\Im\Tr R^{(S)}}{\Im z}.
    \end{equation}
\end{enumerate}
\end{lemma}
\begin{proof}
Parts (a--b) follow from applying Schur's complement to the
linearized resolvent $\Pi(z)$. Part (c) is the standard Sherman-Morrison formula
for a rank-one update of the matrix inverse $R^{(S)}$.
Part (d) follows from the identity
    $R^{(S)}-R^{(S)*}=R^{(S)}(z-\bar{z})R^{(S)*}=(2\Im{z})R^{(S)}R^{(S)*}$.
\end{proof}

\begin{lemma}[Concentration]\label{lem:concentration}
Under Assumptions \ref{assum:basic} and \ref{assum:concentration}, we have
\begin{enumerate}[label=(\alph*)]
    \item (Linear and quadratic forms)
Uniformly over all deterministic $A\in\C^{n\times n}$, $\u\in\C^n$,
and $i,j\in[N]$ with $i \neq j$,
    \[
        \g_i^*A\g_j \prec \|A\|_F,
        \quad \|A\g_i\|_2 \prec \|A\|_F,
        \quad \u^*\g_i \prec \|\u\|_2.
    \]
    \item (Operator norm)
For every $S\subseteq[N]$, $\|K^{(S)}\|_\op \leq \|K\|_\op \prec 1$.
\end{enumerate}
\end{lemma}
\begin{proof}
For (a), we have $|\u^*\g_i|^2=\g_i^*\u\u^*\g_i=\u^*\Sigma\u +
\Oprec(\|\u\|_2^2)=\Oprec(\|\u\|_2^2)$, where the second equality follows from
Assumption \ref{assum:concentration} applied to the matrix $\u\u^*$. Thus
$\u^*\g_i \prec \|\u\|_2$. Next, since $\g_i$ and $\g_j$ are independent, this
implies
\[
    \g_i^* A\g_j \prec \|A\g_j\|_2.
\]
Moreover,
\[
    \|A\g_j\|_2^2=\g_j^*A^*A\g_j=\Tr\Sigma A^*A + \Oprec(\|A^*A\|_F)=\Oprec(\|A\|_F^2),
\]
so $\|A\g_j\|_2 \prec \|A\|_F$ and thus $\g_i^*A\g_j \prec \|A\|_F$. This proves (a).

For (b), note that $0 \preceq K^{(S)} \preceq K$ in the positive-definite
ordering, so $\|K^{(S)}\|_\op \leq \|K\|_\op$. To show $\|K\|_\op\prec 1$, define $A_i=N^{-1}(\g_i\g_i^*-\Sigma)$. Then $K=\sum_{i=1}^NA_i+\Sigma$.
By Assumption \ref{assum:basic}, $\|\Sigma\|_\op\leq C$, so it
suffices to show $\|\sum_{i=1}^N A_i\|_\op \prec 1$. Since the $A_i$ are
independent, mean-zero, symmetric random matrices, the matrix Rosenthal inequality implies that for any $k\geq 1$, there exists a constant $C_k>0$ such that
\[
    \E\,\frac{1}{N}\Tr\left(\sum_{i=1}^N A_i\right)^{2k}\leq
C_k\underbrace{\left\|\sum_{i=1}^N\E
A_i^2\right\|_\op^k}_{=I}+C_k\underbrace{\sum_{i=1}^N\E\|A_i\|_\op^{2k}}_{=II}.
\]
By Assumption \ref{assum:basic}, for each $i\in[N]$,
    $\|A_i\|_\op\leq N^{-1}\|\g_i\|_2^2+N^{-1}\|\Sigma\|_\op \leq
N^{-1}\|\g_i\|_2^2+C/N$.
Moreover, $\|\g_i\|_2^2=\g_i^*I_n\g_i=\Tr\Sigma + \Oprec(\|I_n\|_F)=\Oprec(n)$,
so $N^{-1}\|\g_i\|_2^2\prec 1$. Thus $\|A_i\|_\op\prec 1$. It follows that
for any fixed $k \geq 1$, $\|A_i\|_\op^{2k}\prec 1$ and hence $II\prec N$.

For the first term $I$, consider any unit vector $\v\in\C^n$. Then,
since $0 \preceq \E A_i^2=N^{-2}[\E(\g_i\g_i^*)^2-\Sigma^2]
\preceq N^{-2}\E(\g_i\g_i^*)^2$, we have
$\sum_{i=1}^N\v^*\E A_i^2\v \leq \sum_{i=1}^N N^{-2}\E\v^*(\g_i\g_i^*)^2\v$.
Since also $\v^*\g_i \prec 1$ and $\|\g_i\|_2^2 \prec n$ as shown above,
this shows $\sum_{i=1}^N\v^*\E A_i^2\v  \prec 1$. Thus
for any fixed $k \geq 1$, $\|\sum_{i=1}^n \E A_i^2\|_\op^k \prec 1$ and hence
$I \prec 1$. Putting this together yields
\[
    \E\,\frac{1}{N}\Tr \left(\sum_{i=1}^N A_i\right)^{2k}\prec N.
\]
The claim now follows from Markov's inequality: For any $\eps,D>0$, there exist
constants $k=k(\eps,D)>0$ and $C \equiv C(\eps,D)>0$ large enough such that
\[\P\left(\left\|\sum_{i=1}^NA_i\right\|_\op\geq N^\eps\right)\leq N^{-2\eps
k}\E\left\|\sum_{i=1}^NA_i\right\|_\op^{2k}
\leq N^{-2\eps k}\E\Tr \left(\sum_{i=1}^NA_i\right)^{2k}
\leq CN^{-2\eps k}N^{2+\eps}\leq CN^{-D}.\]
\end{proof}

\subsection{Sherman-Morrison recursions}

For any $A \in \C^{n \times n}$, $\u \in \C^n$, $S \subset [N]$, and $i, j, k \in [N] \setminus S$, define
\begin{align*}
    \cY_i^{(S)}[A] &= \Tr (\g_i \g_i^* - \Sigma) \frac{R^{(S)}A}{\sqrt{N}}, \\
    \cZ_{ijk}^{(S)}[A] &= \Tr (\g_i \g_i^* - \Sigma) \frac{R^{(S)}}{N}\g_j
\g_k^* \frac{R^{(S)}A}{\sqrt{N}}, \quad i \notin \{j, k\}, \\
    \cX_i^{(S)}[\u] &= \g_i^* \frac{R^{(S)}\u}{\sqrt{N}}, \\
    \cB_{jk}^{(S)} &= \g_j^* \frac{R^{(S)}}{N} \g_k, \quad j \neq k, \\
    \cP_i^{(S)} &= \Tr (\g_i \g_i^* - \Sigma) \frac{R^{(S)}}{N}, \\
    \cQ_i^{(S)} &= \Big(1 + \g_i^* \frac{R^{(S)}}{N} \g_i\Big)^{-1}, \\
    \cC^{(S)} &= \Big(1 + \Tr \Sigma \frac{R^{(S)}}{N}\Big)^{-1}.
\end{align*}
Note that each $R^{(S)}$ in the expressions is scaled by $N^{-1}$, and each
$R^{(S)}A$ and $R^{(S)}\u$ is scaled by $N^{-1/2}$.
We remark that the vectors $\g_i$ are real and $R^{(S)}$ is symmetric (in the
real sense, without complex conjugation) so $\cB_{jk}^{(S)}=\cB_{kj}^{(S)}$.

We first establish basic bounds for these quantities under the conditions of
Lemmas \ref{lem:fluctuation_averaging} and
\ref{lem:fluctuation_averaging_lowrank}.

\begin{lemma}\label{lem:prec bounds on y_i}
Suppose the conditions of Lemma \ref{lem:fluctuation_averaging} hold.
Then for any fixed $L \geq 1$,
uniformly over all $S \subseteq [N]$ with $|S| \leq L$, all $i, j, k \in [N]
\setminus S$ with $i \notin \{j,k\}$, all $A \in \C^{n \times n}$ with
$\|A\|_\op \leq 1$, all unit vectors $\u \in \C^n$, and all $z=E+i\eta \in \bD$,
we have
\[|\cY_i^{(S)}[A]| \prec N^{1/2}\Phi,
\quad |\cZ_{ijk}^{(S)}[A]| \prec N^{1/2}\Phi^2,\]
\[|\cB_{ij}^{(S)}| \prec \Phi, \quad |\cP_i^{(S)}| \prec \Phi,
\quad |\cQ_i^{(S)}| \prec 1, \quad |\cC^{(S)}| \prec 1.\]
If, in addition, the condition $N^{-1/2}\|R^{(S)}\u\|_2 \prec \Phi$ of Lemma
\ref{lem:fluctuation_averaging_lowrank} holds, then
\[|\cY_i^{(S)}[\u\u^*]| \prec \Phi,
\quad |\cZ_{ijk}^{(S)}[\u\u^*]| \prec \Phi^2,
\quad |\cX_i^{(S)}[\u]| \prec \Phi.\]
\end{lemma}
\begin{proof}
For the first two bounds for $\cY_i^{(S)}$ and $\cZ_{ijk}^{(S)}$,
Assumption \ref{assum:concentration} and Lemma \ref{lem:concentration}(a) give
\begin{align*}
|\cY_i^{(S)}[A]| &\prec N^{-1/2} \|R^{(S)} A\|_F,\\
|\cZ_{ijk}^{(S)}[A]| &\prec \|N^{-1}R^{(S)}\g_j\g_k^* N^{-1/2}R^{(S)}A\|_F
\prec N^{-1}\|R^{(S)}\|_F \cdot N^{-1/2}\|R^{(S)}A\|_F.
\end{align*}
Since $\|R^{(S)} A\|_F \leq \|R^{(S)}\|_F\|A\|_\op \leq \|R^{(S)}\|_F$ and
$N^{-1}\|R^{(S)}\|_F \prec \Phi$, the first two bounds follow.
For a unit vector $\u$ and $A=\u\u^*$,
we have $\|R^{(S)}\u\u^*\|_F=\|R^{(S)}\u\|_2$, so
under the additional assumption $N^{-1/2}\|R^{(S)}\u\|_2 \prec \Phi$,
this shows $|\cY_i^{(S)}[\u\u^*]| \prec \Phi$ and $|\cZ_{ijk}^{(S)}[\u\u^*]|
\prec \Phi$ as well as
\[|\cX_i^{(S)}[\u]| \prec N^{-1/2} \|R^{(S)} \u\|_2 \prec \Phi.\]

The bounds $|\cB_{ij}^{(S)}| \prec \Phi$ and $|\cP_i^{(S)}| \prec \Phi$ follow
directly from Assumption \ref{assum:concentration} and the condition
$N^{-1}\|R^{(S)}\|_F \prec \Phi$.
For $\cQ_i^{(S)}$ and $\cC^{(S)}$, Lemma \ref{lem:resolventidentities}(a)
and the condition $\Gamma^{(S)} \prec N^{-\tau'}$ imply
$\cQ_i^{(S)}=-z \widetilde{R}_{ii}^{(S \setminus \{i\})}
=-z\widetilde m_0+\Oprec(\Gamma^{(S \setminus \{i\})})
=-z\widetilde m_0+\Oprec(N^{-\tau'})$.
The conditions $|z| \asymp 1$ and $|\widetilde m_0(z)| \asymp 1$
in Definition \ref{def:regular} for regularity of $\bD$ then imply that
for some constants $C,c>0$ and any $D>0$,
\[\P[c \leq |\cQ_i^{(S)}| \leq C] \geq 1-C(D)n^{-D}.\]
This shows $|\cQ_i^{(S)}| \prec 1$. Moreover,
\[
    (\cQ_i^{(S)})^{-1} = 1 + \g_i^* (N^{-1} R^{(S)}) \g_i = 1 + \Tr \Sigma (N^{-1} R^{(S)}) + \Tr (\g_i \g_i^* - \Sigma) (N^{-1} R^{(S)}) = (\cC^{(S)})^{-1} + \Oprec(\Psi).
\]
Since $\Psi \leq N^{-\tau'}$, the above bounds for $|\cQ_i^{(S)}|$ then also
imply $|\cC^{(S)}| \prec 1$.
\end{proof}

The following lemma describes a system of recursions for these quantities,
derived from the Sherman-Morrison identity \eqref{eq:Sherman-Morrison}.

\begin{lemma}\label{lem:sm recursions}
For any $A \in \C^{n \times n}$, $\u \in \C^n$, $S \subseteq [N]$, 
$i, j, k, l \in [N] \setminus S$ with $l \notin \{i,j,k\}$ and $i \notin
\{j,k\}$, and $z \in \C^+$, we have
\begin{align}
    \cY_i^{(S)}[A] &= \cY_i^{(Sl)}[A] - \cZ_{ill}^{(Sl)}[A] \cQ_l^{(Sl)}, \label{eq:y_i recursion} \\
    \cZ_{ijk}^{(S)}[A] &= \cZ_{ijk}^{(Sl)}[A] - \cZ_{ijl}^{(Sl)}[A] \cB_{kl}^{(Sl)} \cQ_l^{(Sl)} - \cZ_{ilk}^{(Sl)}[A] \cB_{lj}^{(Sl)} \cQ_l^{(Sl)} + \cZ_{ill}^{(Sl)}[A] \cB_{kl}^{(Sl)} \cB_{lj}^{(Sl)} (\cQ_l^{(Sl)})^2, \label{eq:z_ijk recursion} \\
    \cZ_{ikk}^{(S)}[A] &= \cZ_{ikk}^{(Sl)}[A] - \cZ_{ikl}^{(Sl)}[A] \cB_{kl}^{(Sl)} \cQ_l^{(Sl)} - \cZ_{ilk}^{(Sl)}[A] \cB_{lk}^{(Sl)} \cQ_l^{(Sl)} + \cZ_{ill}^{(Sl)}[A] \cB_{kl}^{(Sl)} \cB_{lk}^{(Sl)} (\cQ_l^{(Sl)})^2, \label{eq:z_ikk recursion} \\
    \cX_i^{(S)}[\u] &= \cX_i^{(Sl)}[\u] - \cX_l^{(Sl)}[\u] \cB_{il}^{(Sl)}, \label{eq:x_i recursion} \\
    \cB_{ij}^{(S)} &= \cB_{ij}^{(Sl)} - \cB_{il}^{(Sl)} \cB_{lj}^{(Sl)} \cQ_l^{(Sl)}. \label{eq:b_ij recursion}
\end{align}
Moreover, fix any constants $L \geq 1$ and $D>0$. 
If the conditions of Lemma \ref{lem:fluctuation_averaging} hold, then
uniformly over all $z \in \bD$, $S
\subseteq [N]$ with $|S| \leq L$, and $i,l \in [N] \setminus S$ with $i \neq l$,
\begin{align}
    \cQ_i^{(S)} &= \sum_{m=0}^D (\cQ_i^{(Sl)})^{m+1} (\cB_{il}^{(Sl)}
\cB_{li}^{(Sl)} \cQ_l^{(Sl)})^m + \Oprec(N^{-2\tau'(D+1)}), \label{eq:q_i recursion} \\
    \cQ_i^{(S)} &= \sum_{m=0}^D (\cC^{(S)})^{m+1} (-\cP_i^{(S)})^m +
\Oprec(N^{-\tau'(D+1)}). \label{eq:q_i to c_i}
\end{align}
\end{lemma}
\begin{proof}
The identities (\ref{eq:y_i recursion}-\ref{eq:b_ij recursion}) follow directly
from applying the Sherman-Morrison formula \eqref{eq:Sherman-Morrison} to the
left sides. For \eqref{eq:q_i recursion}, by Lemma \ref{lem:prec bounds on y_i}
we have
\[|\cQ_i^{(S)}|,|\cC^{(S)}| \prec 1, \quad |\cB_{ij}^{(S)}|,|\cP_i^{(S)}| \prec
\Phi \leq N^{-\tau'}.\]
Applying Sherman-Morrison to $(\cQ_i^{(S)})^{-1}$ gives
\[
    (\cQ_i^{(S)})^{-1} = 1 + \g_i^* (N^{-1} R^{(S)}) \g_i = (\cQ_i^{(Sl)})^{-1}
- \cB_{il}^{(Sl)} \cB_{li}^{(Sl)} \cQ_l^{(Sl)}.
\]
Using the identity for scalar values $a,b \neq 0$
\begin{equation}\label{eq:AB taylor expan}
    a^{-1} = \sum_{m=0}^D b^{-(m+1)} (b-a)^m + a^{-1} b^{-(D+1)} (b-a)^{D+1},
\end{equation}
we obtain
\begin{align*}
    \cQ_i^{(S)} &= \sum_{m=0}^D (\cQ_i^{(Sl)})^{m+1} (\cB_{il}^{(Sl)}
\cB_{li}^{(Sl)} \cQ_l^{(Sl)})^m + (\cQ_i^{(S)})(\cQ_i^{(Sl)})^{D+1}
(\cB_{il}^{(Sl)} \cB_{li}^{(Sl)} \cQ_l^{(Sl)})^{D+1} \\
    &= \sum_{m=0}^D (\cQ_i^{(Sl)})^{m+1} (\cB_{il}^{(Sl)} \cB_{li}^{(Sl)}
\cQ_l^{(Sl)})^m + \Oprec(N^{-2\tau'(D+1)}),
\end{align*}
since $|\cB_{il}^{(Sl)} \cB_{li}^{(Sl)} \cQ_l^{(Sl)}| \prec N^{-2\tau'}$ and
$|\cQ_i^{(S)}|, |\cQ_i^{(Sl)}| \prec 1$. For \eqref{eq:q_i to c_i}, apply \eqref{eq:AB taylor expan} to
\[
    (\cQ_i^{(S)})^{-1} = 1 + \g_i^* (N^{-1} R^{(S)}) \g_i = (\cC^{(S)})^{-1} +
\cP_i^{(S)},
\]
yielding similarly
\[
    \cQ_i^{(S)} = \sum_{m=0}^D (\cC^{(S)})^{m+1} (-\cP_i^{(S)})^m +
\Oprec(N^{-\tau'(D+1)}).
\]
\end{proof}

\subsection{Proof of Lemma \ref{lem:fluctuation_averaging}}

We fix a matrix $A \in \C^{n \times n}$ with $\|A\|_\op \leq 1$, and abbreviate
\[\cY_i^{(S)}=\cY_i^{(S)}[A], \qquad \cZ_{ijk}^{(S)}=\cZ_{ijk}^{(S)}[A].\]
All subsequent $\Oprec(\cdot)$ bounds are implicitly uniform over all such
matrices $A$.

The following lemma characterizes the combinatorics of an expansion of
$\cY_i^{(i)}\cQ_i^{(i)}$ to resolve the dependence on variables
$\{\g_j\}_{j \in S}$ in a set $S \subset [N]$, using the recursions of
Lemma \ref{lem:sm recursions}.

\begin{lemma}\label{lem:y_i q_i expansion}
Suppose the conditions of Lemma \ref{lem:fluctuation_averaging} hold.
Fix any $L \geq 1$ and $D>0$. Then uniformly over $z \in \bD$,
$S\subset [N]$ with $|S|\leq L$, and $i\in S$, the following holds:

Denote $S^{(i)}=S \setminus \{i\}$. There exists a collection of monomials
$\cM_{i,S}$ such that the quantity
$\cY_i^{(i)}\cQ_i^{(i)}$ can be expanded as
\begin{align}
    \cY_i^{(i)}\cQ_i^{(i)}&=\sum_{q\in\cM_{i,S}}
q\left(\cY_i^{(S)},\{\cZ_{ijk}^{(S)}\}_{j,k\in S^{(i)}},
\{\cB_{ij}^{(S)}\}_{j \in S^{(i)}},
\{\cB_{jk}^{(S)}\}_{j\neq k\in S^{(i)}},\cP_i^{(S)},
\{\cP_j^{(S)}\}_{j \in S^{(i)}},\cC^{(S)}\right)\notag\\
&\hspace{1in}+\Oprec(N^{-\tau'(D+1)+1/2}).\label{eq:y_i expansion}
\end{align}
Each monomial $q\in\cM_{i,S}$ is a product of $\pm 1$ and one or more of its
inputs, allowing repetition.
We have $q=\Oprec(N^{1/2}\Phi)$ uniformly over $q\in\cM_{i,S}$, and the
number of monomials $|\cM_{i,S}|$ is at most a constant depending only on
$L,D$.

Furthermore, for each $q \in \cM_{i,S}$, letting $m_Y,m_Z,m_B^*,m_B$
denote the numbers of factors of the forms $\cY_i^{(S)}$,
$\{\cZ_{ijk}^{(S)}\}_{j,k \in S^{(i)}}$,
$\{\cB_{ij}^{(S)}\}_{j \in S^{(i)}}$, and
$\{\cB_{jk}^{(S)}\}_{j\neq k\in S^{(i)}}$ in $q$, we have:
\begin{enumerate}[label=(\alph*)]
    \item $m_Y+m_Z=1$.
\item $m_B^*$ is even.
    \item The number of distinct indices of $S^{(i)}$ that appear as lower
indices across all factors of $q$ is at most $m_Z+m_B+\frac{1}{2}m_B^*$.
\end{enumerate}
\end{lemma}
\begin{proof}
We arbitrarily order the indices of $S^{(i)}$ as
$l_1,\ldots,l_{|S|-1}$. Beginning with the term $\cY_i^{(i)}\cQ_i^{(i)}$,
iteratively for $j=1,\ldots,|S|-1$, we replace all factors with superscript
$(il_1\ldots l_{j-1})$ by a sum of terms with superscript $(il_1\ldots l_j)$,
using the recursions (\ref{eq:y_i recursion}-\ref{eq:z_ikk recursion}) for $\cY$
and $\cZ$, (\ref{eq:b_ij recursion}) for $\cB$, and (\ref{eq:q_i recursion}) for
$\cQ$. After we have replaced all superscripts to be $(S)=(il_1\ldots l_{|S|-1})$,
we then apply the recursion (\ref{eq:q_i to c_i}) for $\cQ$ to replace each
factor $\cQ_l^{(S)}$ by factors $\cC^{(S)}$ and $\cP_l^{(S)}$. It is then direct to check that this gives a representation of the form (\ref{eq:y_i expansion}), where:
\begin{itemize}
    \item Each application of (\ref{eq:y_i recursion}-\ref{eq:z_ikk recursion})
replaces a factor $\cY_i^{(\ldots)}$ or $\cZ_{ijk}^{(\ldots)}$ by
terms having exactly one such factor. Thus, each monomial $q\in\cM_{i,S}$ has
exactly one factor $\cY_i^{(S)}$ or $\cZ_{ijk}^{(S)}$, i.e.\ $m_Y+m_Z=1$.
    \item The number of total applications of (\ref{eq:y_i
recursion}-\ref{eq:z_ikk recursion}), (\ref{eq:b_ij recursion}),
(\ref{eq:q_i recursion}), and (\ref{eq:q_i to c_i}) is bounded by a constant
depending on $L,D$, so $|\cM_{i,S}|$ and the number of factors of each $q
\in \cM_{i,S}$ are also bounded by constants depending on $L,D$.
By the bounds of Lemma \ref{lem:prec bounds on y_i}, the factor
$\cY_i^{(S)}$ or $\cZ_{ijk}^{(S)}$ of $q$ is $\Oprec(N^{1/2}\Phi)$, and all
other factors are $\Oprec(1)$. Thus
each $q\in\cM_{i,S}$ satisfies $\Oprec(N^{1/2}\Phi)$, and the remainder in
(\ref{eq:y_i expansion}) is at most $\Oprec(N^{-\tau'(D+1)+1/2})$. 
\item Each application of (\ref{eq:q_i recursion}) replaces $\cQ_i^{(\ldots)}$
or $\cQ_j^{(\ldots)}$
by terms having an even number of factors $\{\cB_{ij}^{(\ldots)}\}_{j \in
S^{(i)}}$, and each application of
(\ref{eq:y_i recursion}-\ref{eq:z_ikk recursion}) does not change the number
of factors of the form $\{\cB_{ij}^{(\ldots)}\}_{j \in S^{(i)}}$.
Thus $m_B^*$ is even.
\item If an index of $S^{(i)}$ appears as a lower index on the left side of (\ref{eq:y_i recursion}-\ref{eq:z_ikk recursion}),
(\ref{eq:b_ij recursion}), or (\ref{eq:q_i recursion}), then it also appears as
a lower index on some factor of each term of the right side.
Furthermore, each term on the
right side of (\ref{eq:y_i recursion}-\ref{eq:z_ikk recursion}) and
(\ref{eq:b_ij recursion}) that contains any factor
with the new lower index $l$ has at least one more factor of the form
$\{\cZ_{ijk}^{(\ldots)}\}_{j,k \in S^{(i)}}$ or $\{\cB_{jk}^{(\ldots)}\}_{j \neq
k \in S^{(i)}}$ than the left side, and similarly each term on the right side of
(\ref{eq:q_i recursion}) that contains the new lower index $l$ has at least two 
more factors of the form $\{\cB_{ij}^{(\ldots)}\}_{j \in S^{(i)}}$
than the left side. I.e., whenever a new
lower index of $S^{(i)}$ is introduced, either $m_Z+m_B$ increases by at least
1, or $m_B^*$ increases by at least 2. Thus the number of distinct lower indices
of $S^{(i)}$ across all factors of $q$ is at most $m_Z+m_B+\frac{1}{2}m_B^*$.
\end{itemize}
Combining these observations yields the lemma.
\end{proof}

\begin{proof}[Proof of Lemma \ref{lem:fluctuation_averaging}]
Fix constants $L\geq 1$ and $D>0$. A high moment expansion gives
\begin{align*}
\E\left|\frac{1}{\sqrt{N}}\sum_{i=1}^N\cY_i^{(i)}\cQ_i^{(i)}\right|^{2L}
&=N^{-L}\sum_{i_1,\ldots,i_{2L}=1}^N\underbrace{\E\prod_{l=1}^L\cY_{i_l}^{(i_l)}\cQ_{i_l}^{(i_l)}
\prod_{l=L+1}^{2L}\overline{\cY_{i_l}^{(i_l)}\cQ_{i_l}^{(i_l)}}}_{\E m(i_1,\ldots,i_{2L})}
\end{align*}
Let $S=\{i_1,\ldots,i_{2L}\}$ denote the set of distinct indices in
$i_1,\ldots,i_{2L}$. Applying Lemma \ref{lem:y_i q_i expansion}
to expand each $\cY_{i_l}^{(i_l)}\cQ_{i_l}^{(i_l)}$ over $S$,
\begin{align*}
\E m(i_1,\ldots,i_{2L})&=\sum_{q_1\in\cM_{i_1,S}} \cdots
\sum_{q_{2L}\in\cM_{i_{2L,S}}}
\E\underbrace{\prod_{l=1}^L q_l \prod_{l=L+1}^{2L} \bar
q_l}_{:=q}+\Oprec(N^{-\tau'(D+1)+L}).
   \end{align*}
Hence, noting that the number of index tuples $(i_1,\ldots,i_{2L})$ with
$|\{i_1,\ldots,i_{2L}\}|=s$ is at most $C(L)N^s$ for some
constant $C(L)>0$,
\begin{equation}\label{eq:YQbound}
\E\left|\frac{1}{\sqrt{N}}\sum_{i=1}^N\cY_i^{(i)}\cQ_i^{(i)}\right|^{2L}\prec
\max_{\substack{i_1,\ldots,i_{2L} \in [N]\\
q_1\in\cM_{i_1,S},\ldots,q_{2L}\in\cM_{i_{2L},S}}}
\left\{N^{|S|-L}\E|\E_S q|\right\}+N^{-\tau'(D+1)+L}
\end{equation}
where $\E_S$ denotes the partial expectation over $\{\g_i:i \in S\}$.

We now bound $\E_S q=\E_S \prod_{l=1}^L q_l \prod_{l=L+1}^{2L} \bar q_l$.
Let $K\subseteq S$ correspond to the indices appearing exactly once in the list
$(i_1,\ldots,i_{2L})$. For each $l \in [2L]$, let
$m_Z(l),m_B(l),m_B^*(l)$ denote the counts $m_Z,m_B,m_B^*$ of Lemma
\ref{lem:y_i q_i expansion} for $q_l$, and let $m_P^*(l)$ denote also the number
of factors $\cP_{i_l}^{(S)}$ in $q_l$. Set
\begin{equation}\label{eq:qcounts}
m_Z=\sum_{l=1}^{2L} m_Z(l), \quad m_B=\sum_{l=1}^{2L} m_B(l),
\quad m_B^*=\sum_{l=1}^{2L} m_B^*(l), \quad m_P^*=\sum_{l=1}^{2L} m_P^*(l).
\end{equation}
We consider three cases for each index $i_l\in K$:
\begin{itemize}
\item $i_l$ does not appear as a lower index on any factor of $q_l$ other than
$\cY_{i_l}^{(S)}$ or $\cZ_{i_ljk}^{(S)}$, or on any factor of
$\{q_{l'}:l' \neq l\}$. In this case, Lemma \ref{lem:y_i q_i expansion}(a)
implies that the only factor of $q$
which depends on $\g_{i_l}$ is the (exactly one) factor $\cY_{i_l}^{(S)}$ or
$\cZ_{i_ljk}^{(S)}$ of $q_l$. Since
$\E_{i_l}[\cY_{i_l}^{(S)}]=\E_{i_l}[\cZ_{i_ljk}^{(S)}]=0$, it follows
that $\E_S q=0$.
\item $i_l$ appears as a lower index in some $\{q_{l'}:l' \neq l\}$. Since $i_l$
is distinct from $i_{l'}$ for each $l' \neq l$ and hence belongs to
$S^{(i_l')}$, Lemma \ref{lem:y_i q_i expansion}(c)
ensures that the total number of such indices $i_l \in K$
is at most $m_Z+m_B+\frac{1}{2}m_B^*$.
\item $i_l$ does not appear as a lower index on any $\{q_{l'}:l' \neq l\}$,
but appears as a lower index on at least one factor $\cP_{i_l}^{(S)}$
or $\{\cB_{i_lj}^{(S)}\}_{j \in S^{(i_l)}}$ of $q_l$. Then either $m_P^*(l) \geq
1$, or $m_P^*(l)=0$ in which case Lemma \ref{lem:y_i q_i expansion}(b)
ensures that $m_B^*(l) \geq 2$. So the number of such indices $i_l \in K$
is at most $m_P^*+\frac{1}{2}m_B^*$.
\end{itemize}
Combining these cases, we see that either $\E_S q=0$, or
\[m_Z+m_B+\frac{1}{2}m_B^*+m_P^*+\frac{1}{2}m_B^*
=m_Z+m_B+m_B^*+m_P^* \geq |K|.\]
In the latter case, using Lemma \ref{lem:prec bounds on y_i} to bound
$|\cY_i^{(S)}| \prec N^{1/2}\Phi$,
$|\cZ_{ijk}^{(S)}| \prec N^{1/2}\Phi^2$,
$|\cB_{ij}^{(S)}| \prec \Phi$ (for the factors counted by both $m_B$ and
$m_B^*$), $|\cP_i^{(S)}| \prec \Phi$,
and each other factor of $q$ by $\Oprec(1)$, since $q$ has exactly $2L$ factors
$|\cY_i^{(S)}|$ and $|\cZ_{ijk}^{(S)}|$ by Lemma \ref{lem:y_i q_i expansion}(a),
we get
\begin{equation}\label{eq:FAqbound}
|\E_S q| \prec \Phi^{|K|}(N^{1/2}\Phi)^{2L}.
\end{equation}
Hence
\begin{align*}
\E\left|\frac{1}{\sqrt{N}}\sum_{i=1}^N\cY_i^{(i)}\cQ_i^{(i)}\right|^{2L}\prec
N^{|S|}\Phi^{|K|+2L}+N^{-\tau'(D+1)+L}.
\end{align*}
Since indices in $K$ appear exactly once, the remaining $2L-|K|$ indices must
appear at least twice, so $(2L-|K|)/2+|K|=2L+|K|/2 \geq |S|$
(the number of distinct indices). Thus $N^{|S|}\Phi^{|K|+2L}
\leq (N\Phi^2)^{|K|/2+L} \leq (N\Phi^2)^{2L}$, where the last
inequality uses $|K| \leq 2L$. Since $\Phi \geq N^{-1/2}$, this is larger than
the second term $N^{-\tau'(D+1)+L}$ for a sufficiently large choice of
constant $D \equiv D(\tau',L)>0$. Thus
\[\E\left|\frac{1}{\sqrt{N}}\sum_{i=1}^N\cY_i^{(i)}\cQ_i^{(i)}\right|^{2L}
\prec (N\Phi^2)^{2L}.\]
Then Markov's inequality implies
$|N^{-1/2}\sum_{i=1}^N\cY_i^{(i)}\cQ_i^{(i)}| \prec N\Phi^2$, which shows
Lemma \ref{lem:fluctuation_averaging}.
\end{proof}

\subsection{Proof of Lemma \ref{lem:fluctuation_averaging_lowrank}}

Fixing a unit vector $\u \in \C^n$, let us now abbreviate
\[\cY_i^{(S)}=\cY_i^{(S)}[\u\u^*], \qquad
\cZ_{ijk}^{(S)}=\cZ_{ijk}^{(S)}[\u\u^*],
\qquad \cX_i^{(S)}=\cX_i^{(S)}[\u].\]
All subsequent $\Oprec(\cdot)$ bounds are implicitly uniform in $\u$.
We now show Lemma \ref{lem:fluctuation_averaging_lowrank} under the additional
condition that $N^{-1/2}\|R^{(S)}\u\|_2 \prec \Phi$.

\begin{proof}[Proof of Lemma \ref{lem:fluctuation_averaging_lowrank}(a)]
The proof is identical to Lemma \ref{lem:fluctuation_averaging}, where now
by Lemma \ref{lem:prec bounds on y_i} we have the bounds
$|\cY_i^{(S)}| \prec \Phi$ and $|\cZ_{ijk}^{(S)}| \prec \Phi^2$. This gives,
instead of \eqref{eq:FAqbound}, $|\E_S q| \prec \Phi^{|K|}\Phi^{2L}$, and hence
\begin{align*}
\E\left|\frac{1}{\sqrt{N}}\sum_{i=1}^N\cY_i^{(i)}\cQ_i^{(i)}\right|^{2L}\prec
N^{|S|-L}\Phi^{|K|+2L}+N^{-\tau'(D+1)+L}
\leq N^{-L}(N\Phi^2)^{2L}+N^{-\tau'(D+1)+L}
\end{align*}
Choosing $D$ large enough and applying Markov's inequality shows 
Lemma \ref{lem:fluctuation_averaging_lowrank}(a).
\end{proof}

We next show Lemma \ref{lem:fluctuation_averaging_lowrank}(b).
The following lemma is similar to Lemma \ref{lem:y_i q_i expansion}.

\begin{lemma}\label{lem:B_ij expansion}
Suppose the conditions of Lemma \ref{lem:fluctuation_averaging_lowrank} hold.
Fix any $L \geq 1$ and $D>0$. Then uniformly over $z \in \bD$, $S\subset[N]$ with $|S|\leq L$,
and $i,j\in S$ with $i\neq j$, the following holds:

Denote $S^{(ij)}=S\setminus\{i,j\}$. Then there exists a collection of monomials $\cM_{ij,S}$ such that $\cB_{ij}^{(ij)}\cQ_i^{(ij)}\cQ_j^{(ij)}$ can be expanded as
\begin{align}
\cB_{ij}^{(ij)}\cQ_i^{(i)}\cQ_j^{(ij)}&=\sum_{q\in\cM_{ij,S}}
q\left(\{\cB_{ij}^{(S)}\},\{\cB_{ik}^{(S)}\}_{k\in S^{(ij)}},
\{\cB_{jk}^{(S)}\}_{k \in S^{(ij)}},\{\cB_{kl}^{(S)}\}_{k\neq
l\in S^{(ij)}},\cP_i^{(S)},\cP_j^{(S)},\{\cP_k^{(S)}\}_{k \in
S^{(ij)}},\cC^{(S)}\right)\notag\\
&\hspace{1in}+\Oprec(N^{-\tau'(D+1)}).\label{eq:Bijq}
\end{align}
Each monomial $q\in\cM_{ij,S}$ is a product of $\pm 1$ and one or more 
of its inputs, allowing repetition.
We have $q=\Oprec(\Phi)$ uniformly over $q\in\cM_{ij,S}$, and the number
of monomials $|\cM_{ij,S}|$ is at most a constant depending only on $L,D$.

Furthermore, for each $q \in \cM_{ij,S}$, letting
$m_B^{**},m_B^{*1},m_B^{*2},m_B$ denote the numbers of factors of the forms
$\{\cB_{ij}^{(S)}\},\{\cB_{ik}^{(S)}\}_{k\in S^{(ij)}},
\{\cB_{jk}^{(S)}\}_{k \in S^{(ij)}},\{\cB_{kl}^{(S)}\}_{k\neq
l\in S^{(ij)}}$ appearing in $q$, we have:
\begin{enumerate}[label=(\alph*)]
    \item Either $m_B^{**} \geq 1$, or both $m_B^{*1} \geq 1$ and
$m_B^{*2} \geq 1$.
    \item The number of occurrences of $i$ as a lower index in all factors
$\cB_{...}^{(S)}$ of $q$ (i.e.\ $m_B^{**}+m_B^{*1}$) is odd. Similarly,
the number of such occurrences of $j$ (i.e.\ $m_B^{**}+m_B^{*2}$) is odd.
For each $k \in S^{(ij)}$, the number of
occurrences of $k$ as a lower index across all factors
$\cB_{...}^{(S)}$ of $q$ is even.
\item The number of distinct indices of $S^{(ij)}$ that appear as lower indices
across all factors of $q$ is at most $m_B+\frac{1}{2}(m_B^{*1}+m_B^{*2})$.
\end{enumerate}
\end{lemma}
\begin{proof}
We may first apply \eqref{eq:q_i recursion} to expand $\cQ_i^{(i)}$ in $j$,
to get
\begin{equation}\label{eq:BQQinitial}
    \cB_{ij}^{(ij)}\cQ_i^{(i)}\cQ_j^{(ij)}=\sum_{m=0}^D[\cB_{ij}^{(ij)}]^{2m+1}[\cQ_i^{(ij)}]^{m+1}[\cQ_j^{(ij)}]^{m+1}+\Oprec(N^{-2\tau'(D+1)}).
\end{equation}
We may then successively expand $\cB$ and $\cQ$ in the
indices of $S^{(ij)}$ using \eqref{eq:b_ij recursion} and \eqref{eq:q_i
recursion}, and finally expand each $\cQ_{\ldots}^{(S)}$ in $\cC^{(S)}$ and
$\cP_{\ldots}^{(S)}$ using \eqref{eq:q_i to c_i}. This gives a representation of
the form \eqref{eq:Bijq}.
Since $|\cB_{ij}^{(S)}| \prec \Phi$ and $|\cQ_i^{(S)}| \prec 1$, we have
$q=\Oprec(\Phi)$, and the remainder is $\Oprec(N^{-\tau'(D+1)})$.

Each term of \eqref{eq:BQQinitial} has at least one
factor of $\cB_{ij}^{(ij)}$. If \eqref{eq:b_ij recursion} is applied to expand
$\cB_{ij}^{(...)}$, then the resulting terms either have a factor
of $\cB_{ij}^{(...)}$ or have a factor each of $\cB_{ik}^{(...)}$ and
$\cB_{jl}^{(...)}$ for some $k,l \in S^{(ij)}$, and this is preserved
in later steps of the expansion. This shows property (a).

Property (b) holds because $i$ and $j$ each occur an odd number of times as a
lower index of $\cB_{...}^{(ij)}$ in \eqref{eq:BQQinitial}, and each $k \in
S^{(ij)}$ occurs 0 times. The parities of these occurrences do not change in the
applications of \eqref{eq:b_ij recursion} and \eqref{eq:q_i recursion}.

Property (c) holds as in Lemma \ref{lem:y_i q_i expansion}
because each application of \eqref{eq:b_ij recursion} or \eqref{eq:q_i
recursion} which introduces a new lower index introduces at least one additional
factor of $\{\cB_{kl}^{(...)}\}_{k,l \in S^{(ij)}}$
or two additional factors of
$\{\cB_{ik}^{(...)}\}_{k \in S^{(ij)}} \cup
\{\cB_{jk}^{(...)}\}_{k \in S^{(ij)}}$.
\end{proof}

\begin{proof}[Proof of Lemma \ref{lem:fluctuation_averaging_lowrank}(b)]
Fix any $L \geq 1$ and $D>0$. For any unit vector $\v \in \R^N$, we may expand
\begin{align*}
&\E\left|\sum_{i\neq
j}\bar{\v}(i)\v(j)\cB_{ij}^{(ij)}\cQ_i^{(i)}\cQ_j^{(ij)}\right|^{2L}\\
&=\sum_{i_1\neq j_1}\cdots \sum_{i_{2L}\neq j_{2L}}\prod_{l=1}^L
\bar{\v}(i_l)\v(j_l) \prod_{l=L+1}^{2L} \v(i_l)\bar{\v}(j_l)
\E\left[\prod_{l=1}^L\cB_{i_lj_l}^{(i_lj_l)}\cQ_{i_l}^{(i_l)}\cQ_{j_l}^{(i_lj_l)}\prod_{l=L+1}^{2L}\overline{\cB_{i_lj_l}^{(i_lj_l)}\cQ_{i_l}^{(i_l)}\cQ_{j_l}^{(i_lj_l)}}\right].
\end{align*}
Writing $S=\{i_1,\ldots,i_{2L},j_1,\ldots,j_{2L}\}$ for the set of distinct
indices, $q_l \in \cM_{i_lj_l,S}$ for the monomials in the expansion
of $\cB_{i_lj_l}^{(i_lj_l)}\cQ_{i_l}^{(i_l)}\cQ_{j_l}^{(i_lj_l)}$ in
Lemma \ref{lem:B_ij expansion}, and $q=\prod_{l=1}^L q_l
\prod_{L+1}^{2L} \bar q_l$, we have
\begin{align*}
\E\left|\sum_{i\neq
j}\bar{\v}(i)\v(j)\cB_{ij}^{(ij)}\cQ_i^{(i)}\cQ_j^{(ij)}\right|^{2L}
&\leq \sum_{i_1\neq j_1}\cdots \sum_{i_{2L}\neq j_{2L}}\prod_{l=1}^{2L}
|\v(i_l)\v(j_l)|\\
&\hspace{0.2in}
\left(\sum_{q_1 \in \cM_{i_1j_1,S},\ldots,q_{2L} \in \cM_{i_{2L}j_{2L},S}}
\E|\E_S q|
+\Oprec(N^{-\tau'(D+1)})\right).
\end{align*}
Let $K \subseteq S$ denote the subset of indices that appear exactly once in the
list $(i_1,\ldots,i_{2L},j_1,\ldots,j_{2L})$.
Applying $\sum_i |\v(i)| \leq \sqrt{N}$ and $\sum_i |\v(i)|^p \leq 1$ for each
$p \geq 2$, for any fixed $k \in \{1,\ldots,4L\}$ and a constant $C \equiv
C(L)>0$ we then have
\[\mathop{\sum_{i_1\neq j_1,\ldots,i_{2L}\neq j_{2L}}}_{|K|=k}
\prod_{l=1}^{2L} |\v(i_l)\v(j_l)|
\leq CN^{k/2} \leq CN^{2L}.\]
Thus
\begin{equation}\label{eq:vvBQQval}
\E\left|\sum_{i\neq
j}\bar{\v}(i)\v(j)\cB_{ij}^{(ij)}\cQ_i^{(i)}\cQ_j^{(ij)}\right|^{2L}
\prec \max_{\substack{i_1,\ldots,i_{2L},j_1,\ldots,j_{2L} \in [N]\\
q_1 \in \cM_{i_1j_1,S},\ldots,q_{2L} \in \cM_{i_{2L}j_{2L},S}}}
\Big\{N^{|K|/2}\E|\E_S q|\Big\}+\Oprec(N^{-\tau'(D+1)+2L}).
\end{equation}
Here $S=\{i_1,\ldots,i_{2L},j_1,\ldots,j_{2L}\}$ denotes the set of distinct
indices and $K \subseteq S$ denotes the subset appearing exactly once, both of
which depend implicitly on $(i_1,\ldots,i_{2L},j_1,\ldots,j_{2L})$.

For each $l\in[2L]$, let $m_B^{**}(l),m_B^{*1}(l),m_B^{*2}(l),m_B(l)$ denote the
counts $m_B^{**},m_B^{*1},m_B^{*2},m_B$ of Lemma \ref{lem:B_ij expansion} for
$q_l$, and let $m_P^{*1}(l),m_P^{*2}(l)$ denote also the total number of factors
$\cP_{i_l}^{(S)},\cP_{j_l}^{(S)}$ in $q_l$ respectively. Set
\[m_B^{**}=\sum_{l=1}^{2L} m_B^{**}(l),
\quad m_B^*=\sum_{l=1}^{2L} m_B^{*1}(l)+m_B^{*2}(l),
\quad m_B=\sum_{l=1}^{2L} m_B(l),
\quad m_P^*=\sum_{l=1}^{2L} m_P^{*1}(l)+m_P^{*2}(l),\]
and denote also
\[E(l)=m_B^{**}(l)+\frac{1}{2}(m_B^{*1}(l)+m_B^{*2}(l))
+m_P^{*1}(l)+m_P^{*2}(l)-1.\]
Note that Lemma \ref{lem:B_ij expansion}(a) implies $E(l) \geq 0$ for each
$l=1,\ldots,2L$. Consider three cases for an index $i_l \in K$ (or $j_l \in K$),
analogous to the proof of Lemma \ref{lem:fluctuation_averaging}:
\begin{itemize}
\item $i_l$ appears as a lower index on exactly one factor of
$q_l$ (which must be $\cB_{i_lk}^{(S)}$ for some $k \neq i_l$),
and it does not appear as a lower index on any factor of
$\{q_{l'}:l' \neq l\}$. Then $\E_S q=0$.
\item $i_l$ appears as a lower index in some $\{q_{l'}:l' \neq l\}$.
Since $i_l$ is distinct from $\{i_{l'},j_{l'}\}$ for each $l' \neq l$,
Lemma \ref{lem:B_ij expansion}(c) ensures that the total number of such
indices $i_l,j_l \in K$ is at most $m_B+\frac{1}{2}m_B^*$.
\item $i_l$ does not appear as a lower index on any $\{q_{l'}:l' \neq l\}$,
but appears as a lower index on at least two factors of $q_l$. Then Lemma
\ref{lem:B_ij expansion}(a--b) ensures that either $m_P^{*1}(l) \geq 1$
and $m_B^{**}(l)+\frac{1}{2}(m_B^{*1}(l)+m_B^{*2}(l)) \geq 1$, or
$m_P^{*1}(l)=0$ and $m_B^{**}(l)+\frac{1}{2}(m_B^{*1}(l)+m_B^{*2}(l)) \geq 2$
(since $m_B^{**}(l)+m_B^{*1}(l) \geq 2$ is odd, and also
$m_B^{*2}(l) \geq 1$ if $m_B^{**}(l)=0$). This implies $E(l) \geq 1$.

Furthermore, if also $j_l \in K$ and $j_l$ appears
as a lower index on at least two factors of $q_l$, then similarly, either
$m_P^{*1}(l)+m_P^{*2}(l) \geq 2$ and
$m_B^{**}(l)+\frac{1}{2}(m_B^{*1}(l)+m_B^{*2}(l)) \geq 1$, or
$m_P^{*1}(l)+m_P^{*2}(l)=1$ and
$m_B^{**}(l)+\frac{1}{2}(m_B^{*1}(l)+m_B^{*2}(l)) \geq 2$, or
$m_P^{*1}(l)+m_P^{*2}(l)=0$ and
$m_B^{**}(l)+\frac{1}{2}(m_B^{*1}(l)+m_B^{*2}(l)) \geq 3$. This
implies $E(l) \geq 2$.

Thus, the total number of such indices $i_l,j_l \in K$
is at most $\sum_{l=1}^{2L} E(l)$.
\end{itemize}
Combining these cases, we see that either $\E_S q=0$, or
\[m_B+\frac{1}{2}m_B^*+\sum_{l=1}^{2L} E(l)
=m_B^{**}+m_B^*+m_B+m_P^*-2L\geq |K|.\]
In the latter case, using
Lemma \ref{lem:prec bounds on y_i} to bound $|\cB_{ij}^{(S)}|,|\cP_i^{(S)}|\prec
\Phi$ and each other factor of $q$ by $\Oprec(1)$, we get $|\E_S q| \prec
\Phi^{2L+|K|}$, and hence
\[
    \E\left|\sum_{i\neq
j}\bar{\v}(i)\v(j)\cB_{ij}^{(ij)}\cQ_i^{(i)}\cQ_j^{(ij)}\right|^{2L}\prec
N^{|K|/2}\Phi^{2L+|K|}+N^{-\tau'(D+1)+2L}
\leq (\sqrt{N}\Phi)^{4L}\Phi^{2L}+N^{-\tau'(D+1)+2L}\]
where the last inequality applies $\sqrt{N}\Phi \geq 1$ and $|K| \leq 4L$.
For sufficiently large $D$, the second term is at most the first, so 
Lemma \ref{lem:fluctuation_averaging_lowrank}(b) again follows from Markov's
inequality.
\end{proof}

Finally, we show Lemma \ref{lem:fluctuation_averaging_lowrank}(c).
The following lemma is similar to Lemmas \ref{lem:y_i q_i expansion}
and \ref{lem:B_ij expansion}.

\begin{lemma}\label{lem:x_i q_i expansion}
Suppose the conditions of Lemma \ref{lem:fluctuation_averaging} hold.
Fix any $L \geq 1$ and $D>0$. Then uniformly over $z \in \bD$, $S\subset[N]$ with $|S|\leq L$,
and $i\in S$, the following holds:

Denote $S^{(i)}=S\setminus{i}$. There exists a collection of monomials
$\cM_{i,S}$ such that the quantity $\cX_i^{(i)}\cQ_i^{(i)}$ can be expanded as
\begin{align*}
    \label{eq:y_i expansion}
    \cX_i^{(i)}\cQ_i^{(i)}&=\sum_{q\in\cM_{i,S}} q\left(\cX_i^{(S)},
\{\cX_j^{(S)}\}_{j \in S^{(i)}},
\{\cB_{ij}^{(S)}\}_{j \in S^{(i)}},\{\cB_{jk}^{(S)}\}_{j\neq k\in S^{(i)}},
\cP_i^{(S)},\{\cP_j^{(S)}\}_{j \in S^{(i)}},\cC^{(S)}\right)\\
    &+ \Oprec(N^{-\tau'(D+1)})
\end{align*}
Each monomial $q\in\cM_{i,S}$ is a product of $\pm 1$ and one or more of its
inputs, allowing repetition. We
have $q=\Oprec(\Phi)$ uniformly over $q\in\cM_{i,S}$, and the
number of monomials $|\cM_{i,S}|$ is at most a constant depending only on $L,D$.

Furthermore, for each $q \in \cM_{i,S}$, letting $m_X^*,m_X,m_B^*,m_B$ denote
the numbers of factors of the forms $\cX_i^{(S)}$,
$\{\cX_j^{(S)}\}_{j \in S^{(i)}}$,
$\{\cB_{ij}^{(S)}\}_{j \in S^{(i)}}$, and
$\{\cB_{jk}^{(S)}\}_{j \neq k \in S^{(i)}}$ in $q$, we have:
\begin{enumerate}[label=(\alph*)]
\item $m_X^*+m_X=1$.
\item $m_X^*+m_B^*$ is odd, and for each $j \in S^{(i)}$, the number of
occurrences of $j$ as a lower index of the factors
$\{\cX_j^{(S)}\}_{j \in S^{(i)}}$ and
$\{\cB_{jk}^{(S)}\}_{j \neq k \in S^{(i)}}$ is even.
\item The number of distinct indices of $S^{(i)}$ that appear as a lower
index across all factors of $q$ is at most $\frac{1}{2}m_X+m_B+\frac{1}{2}m_B^*$.
\end{enumerate}
\end{lemma}
\begin{proof}
The proof is similar to the proofs of Lemmas \ref{lem:y_i q_i expansion}
and \ref{lem:B_ij expansion}, where we instead apply the recursion (\ref{eq:x_i
recursion}) and the bound $|\cX_j^{(S)}| \prec \Phi$.
We omit the details for brevity.
\end{proof}

\begin{proof}[Proof of Lemma \ref{lem:fluctuation_averaging_lowrank}(c)]
Fixing any $L \geq 1$ and $D>0$, a high moment expansion gives
\begin{align*}
\E\left|\sum_{i=1}^N \bar{\v}(i)\cX_i^{(i)}\cQ_i^{(i)}\right|^{2L}
&=\sum_{i_1,\ldots,i_{2L}=1}^N
\prod_{l=1}^L \bar \v(i_l) \prod_{l=L+1}^{2L} \v(i_l)
\E\prod_{l=1}^L\cX_{i_l}^{(i_l)}\cQ_{i_l}^{(i_l)}
\prod_{l=L+1}^{2L}\overline{\cX_{i_l}^{(i_l)}\cQ_{i_l}^{(i_l)}}\\
&\leq \sum_{i_1,\ldots,i_{2L}=1}^N
\prod_{l=1}^{2L} |\v(i_l)|
\left(\sum_{q_1 \in \cM_{i_1,S},\ldots,q_{2L} \in \cM_{i_{2L},S}}
\E|\E_S q|+\Oprec(N^{-\tau'(D+1)})\right),
\end{align*}
where $S=\{i_1,\ldots,i_{2L}\}$ is the set of distinct indices,
$q_l \in \cM_{i_l,S}$ are the monomials in the expansion of
$\cX_{i_l}^{(i_l)}\cQ_{i_l}^{(i_l)}$ in Lemma \ref{lem:x_i q_i expansion},
and $q=\prod_{l=1}^L q_l \prod_{l=L+1}^{2L} \bar q_l$.
Let $K \subseteq S$ be those indices appearing exactly once in
$(i_1,\ldots,i_{2L})$. Then, applying $\sum_i |\v(i)| \leq \sqrt{N}$
and $\sum_i |\v(i)|^p \leq 1$ for each $p \geq 2$, this implies as in
Lemma \ref{lem:fluctuation_averaging}(b) that
\begin{equation}\label{eq:vXQval}
\E\left|\sum_{i=1}^N \bar{\v}(i)\cX_i^{(i)}\cQ_i^{(i)}\right|^{2L}
\prec \mathop{\max_{i_1,\ldots,i_{2L} \in [N]}}_{q_1 \in \cM_{i_1,S},
\ldots,q_{2L} \in \cM_{i_{2L},S}}
\Big\{N^{|K|/2}\E|\E_S q|\Big\}+\Oprec(N^{-\tau'(D+1)+L}).
\end{equation}

For each $l\in[2L]$, let $m_X^*(l),m_X(l),m_B^*(l),m_B(l)$ denote the counts of
Lemma \ref{lem:x_i q_i expansion} for $q_l$, and let $m_P^*(l)$ denote the total
number of factors $\cP_{i_l}^{(S)}$ in $q_l$. Set
\[m_X^*=\sum_{l=1}^{2L} m_X(l),
\quad m_X=\sum_{l=1}^{2L} m_X(l),
\quad m_B^*=\sum_{l=1}^{2L} m_B^*(l),
\quad m_B=\sum_{l=1}^{2L} m_B(l),
\quad m_P^*=\sum_{l=1}^{2L} m_P^*(l),\]
and set also
\[E(l)=m_X^*(l)+\frac{1}{2}(m_X(l)+m_B^*(l))+m_P^*(l)-1.\]
Note that Lemma \ref{lem:x_i q_i expansion}(a) implies $E(l) \geq 0$ for each
$l=1,\ldots,2L$. Consider three cases for an index $i_l \in K$:
\begin{itemize}
\item $i_l$ appears as a lower index on exactly one factor of
$q_l$ (which must be $\cX_{i_l}^{(S)}$ or $\cB_{i_lk}^{(S)}$ for some
$k \neq i_l$), and it does not appear as a lower index on any factor of
$\{q_{l'}:l' \neq l\}$. Then $\E_S q=0$.
\item $i_l$ appears as a lower index in some $\{q_{l'}:l' \neq l\}$.
Lemma \ref{lem:x_i q_i expansion}(c) ensures that the total number of such
indices $i_l \in K$ is at most $\frac{1}{2}m_X+m_B+\frac{1}{2}m_B^*$.
\item $i_l$ does not appear as a lower index on any $\{q_{l'}:l' \neq l\}$,
but appears as a lower index on at least two factors of $q_l$. Then Lemma
\ref{lem:x_i q_i expansion}(a--b) ensures that either $m_P^*(l) \geq 1$
and $m_X^*(l)+\frac{1}{2}(m_X(l)+m_B^*(l)) \geq 1$, or $m_P^*(l)=0$ and
$m_X^*(l)+\frac{1}{2}(m_X(l)+m_B^*(l)) \geq 2$, so $E(l) \geq 1$.
Thus the total number of such indices $i_l \in K$ is at most $\sum_{l=1}^{2L}
E(l)$.
\end{itemize}
Combining these cases, either $\E_S q=0$, or
\[\frac{1}{2}m_X+m_B+\frac{1}{2}m_B^*+\sum_{l=1}^{2L} E(l)
=m_X^*+m_X+m_B^*+m_B+m_P^*-2L \geq |K|.\]
Bounding $|\cX_i^{(S)}|,|\cB_{ij}^{(S)}|,|\cP_i^{(S)}| \prec \Phi$
gives $|\E_S q| \prec \Phi^{2L+|K|}$. Then
\[\E\left|\sum_{i=1}^N \bar{\v}(i)\cX_i^{(i)}\cQ_i^{(i)}\right|^{2L}
\prec N^{|K|/2}\Phi^{2L+|K|}+N^{-\tau'(D+1)+L}
\leq (\sqrt{N}\Phi)^{2L}\Phi^{2L}+N^{-\tau'(D+1)+L}\]
where the last inequality uses $\sqrt{N}\Phi \geq 1$ and $|K| \leq 2L$.
For large enough $D>0$, the second term is at most the first, so 
Lemma \ref{lem:fluctuation_averaging_lowrank}(c)
follows again by Markov's inequality.
\end{proof}

\subsection{Proof of Lemma \ref{lem:fluctuation_averaging_lowrank2}}

\subsubsection{Weak cumulant tensor bound}

We will use implicitly the observation that if Assumption 
\ref{assum:cumulant} holds, then the condition \eqref{eq:cumulantassumption} is
valid also for complex inputs $\s_1,\ldots,\s_m \in \C^n$ and $T \in
(\C^n)^{\otimes k-m}$, as it may be applied separately to the real and complex
parts. In what follows, $\<\cdot,\cdot\>$ represents the non-conjugate, scalar product on $(\C^n)^{\otimes k}$.

The condition in Assumption \ref{assum:cumulant} pertains to settings where
$m,k-m \in \{1,\ldots,k-1\}$. The following lemma clarifies that for
all values of $m$ including $\{0,k\}$, Assumption \ref{assum:cumulant}
has the following weaker implication.

\begin{lemma}[Weak cumulant tensor bound]\label{lem:weak cumulant bound}
Suppose Assumption \ref{assum:cumulant} holds. Fix any $k \geq 3$,
let $\|\cdot\|_{\cU_k}$ be as in Assumption
\ref{assum:cumulant}, and for a scalar quantity $T \in (\C^n)^0$ define
$\|T\|_{\cU_k}=|T|$. Then for any $\eps>0$, there exists a constant $C \equiv
C(\eps,k)>0$ such that for all $i\in[N]$, 
$0\leq m\leq k$, $\s_1,\dots,\s_m\in\C^n$, and $T\in(\C^n)^{\otimes k-m}$,
\[
    |\<\kappa_k(\g_i),\s_1\otimes\cdots\s_m\otimes T\>|\leq
Cn^\eps(\sqrt{n})^{k-m}\|T\|_{\cU_k}\prod_{t=1}^m\|\s_t\|_2.
\]
\end{lemma}
\begin{proof}
For $1 \leq m \leq k-1$, the result is implied by Assumption
\ref{assum:cumulant}. For $m=k$, the result also follows by identifying a single
vector $\s_t$ as $T$ in Assumption \ref{assum:cumulant} and applying
$\|\s_t\|_{\cU_k}\leq C_k\|\s_t\|_2$.

For $m=0$, let $T\in(\C^n)^{\otimes k}$ be any $k$-th order tensor. We can write
\[
    T=\sum_{\alpha_1,\dots,\alpha_k}T[\alpha_1,\dots,\alpha_k]\e_{\alpha_1}\otimes\cdots\otimes\e_{\alpha_k}=\sum_{\alpha_1}\e_{\alpha_1}\otimes\underbrace{\sum_{\alpha_2,\dots,\alpha_k}T[\alpha_1,\dots,\alpha_k]\e_{\alpha_2}\otimes\cdots\otimes\e_{\alpha_k}}_{\widetilde{T}(\alpha_1)\in(\C^n)^{\otimes k-1}}.
\]
Fix $\eps>0$. By Assumption \ref{assum:cumulant}, for a constant $C \equiv
C(\eps,k)>0$, we have
\begin{align*}
    |\<\kappa_k(\g_i),T\>|&\leq\sum_{\alpha_1}\left|\<\kappa_k(\g_i),\e_{\alpha_1}\otimes\widetilde{T}(\alpha_1)\>\right|\\
    &\leq n \cdot
Cn^\eps(\sqrt{n})^{k-1-1}\max_{\alpha_1}\left(\|\e_{\alpha_1}\|_2\cdot\|\widetilde{T}(\alpha_1)\|_{\cU_k}\right)=
Cn^\eps(\sqrt{n})^{k}\max_{\alpha_1}\|\widetilde{T}(\alpha_1)\|_{\cU_k}.
\end{align*}
The proof is completed by the definition of the $\cU_k$-norm,
\begin{align*}
    \max_{\alpha_1}\|\widetilde{T}(\alpha_1)\|_{\cU_k}&=\max_{\alpha_1}\max_{\u_1,\dots,\u_{k-1}\in\cU_k}\left|\<\widetilde{T}(\alpha_1),\u_1\otimes\cdots\u_{k-1}\>\right|\\
    &=\max_{\alpha_1}\max_{\u_1,\dots,\u_{k-1}\in\cU_k}\left|\<T,\e_{\alpha_1}\otimes\u_1\otimes\cdots\u_{k-1}\>\right|\leq\|T\|_{\cU_k}.
\end{align*}
\end{proof}

\subsubsection{Moment-cumulant expansion}

For any random vector $\g\in\R^n$ with finite moments of all orders, any integer $k\in\N$, and any partition $\pi$ of $[k]$, we define the tensor $\kappa_\pi(\g)\in(\R^n)^{\otimes k}$ as
\[
\kappa_\pi(\g)=\bigotimes_{B\in\pi} \kappa_{|B|}(\g),
\]
where $\kappa_{|B|}(\g)$ is the $|B|$-th order cumulant tensor of $\g$, and
each factor $\kappa_{|B|}(\g)$ of the tensor product corresponds to the indices
of $B \in \pi$. We have the following moment-cumulant decomposition:

\begin{lemma}\label{lem:cumulant pairing}
Suppose $\g\in\R^n$ is a random vector with $\E \g=0$, $\E\g\g^*=\Sigma$,
and finite moments of all orders. Then for any integers $k,m \geq 0$, we have
\[
    \E[(\g\g^*-\Sigma)^{\otimes k}\otimes\g^{\otimes m}]=\sum_{\pi\in \dot{P}_{k,m}}\kappa_\pi(\g),
\]
where $\dot{P}_{k,m}$ is the set of all partitions of $[2k+m]$ that does not
have any singleton block or any of the cardinality-2 blocks
$\{1,2\},\{3,4\},\ldots,\{2k-1,2k\}$.
\end{lemma}
\begin{proof}
We prove by induction on $k$: For $k=0$, let $P_m$ denote the set of all
partitions of $[m]$. By the moment-cumulant relations, for any vectors
$\v_1,\ldots,\v_m \in \R^n$,
\[\E\<\g^{\otimes m},\v_1\otimes \ldots \otimes \v_m\>
=\sum_{\pi \in P_m} \prod_{B \in \pi} \kappa_{|B|}(\g^* \v_i:i \in B),\]
hence $\E[\g^{\otimes m}]=\sum_{\pi \in P_m} \kappa_\pi(\g)$. If $\pi$ has a
singleton block, then $\kappa_\pi(\g)=0$ since $\kappa_1(\g)=\E[\g]=0$, and the
result follows for $m=0$.

Now suppose the claim holds for $k-1$ for all $m \geq 0$. Let $Q:(\R^n)^{2k+m}
\to (\R^n)^{2k+m}$ be the permutation for which
$Q(T)(i_1,\ldots,i_{2k+m})=T(i_1,\ldots,i_{2k-2},i_{2k+m-1},i_{2k+m},i_{2k-1},\ldots,i_{2k+m-2})$,
i.e.\ rotating the last two coordinates to positions $2k-1$ and $2k$.
Then
\begin{align*}
    \E[(\g\g^*-\Sigma)^{\otimes k}\otimes\g^{\otimes
m}]&=\E[(\g\g^*-\Sigma)^{\otimes k-1}\otimes\g^{\otimes
m+2}]-Q\Big(\E[(\g\g^*-\Sigma)^{\otimes k-1}\otimes \g^{\otimes m}]\otimes
\Sigma\Big)\\
    &=\sum_{\pi\in \dot{P}_{k-1,m+2}}\kappa_\pi(\g)-\sum_{\pi\in
\dot{P}_{k-1,m}} Q\Big(\kappa_\pi(\g)\otimes\Sigma\Big).
\end{align*}
where the second equality holds by the induction hypothesis.
The first sum is over partitions $\pi$ of $[2k+m]$ that do not contain
singletons or the blocks $\{1,2\},\ldots,\{2k-3,2k-2\}$. The second sum may be
understood as the sum of $\kappa_{\pi'}(\g)$ for partitions $\pi'$ of $[2k+m]$
that do not contain singletons or the blocks
$\{1,2\},\ldots,\{2k-3,2k-2\}$, but contains the block $\{2k-1,2k-2\}$. Hence
their difference is exactly $\sum_{\pi \in \dot P_{k,m}} \kappa_\pi(\g)$,
completing the induction.
\end{proof}

\subsubsection{Tensor network representation}

We again fix a deterministic unit vector $\u \in \C^n$ and abbreviate
\[\cY_i^{(S)}=\cY_i^{(S)}[\u\u^*], \qquad
\cZ_{ijk}^{(S)}=\cZ_{ijk}^{(S)}[\u\u^*], \qquad
\cX_i^{(S)}=\cX_i^{(S)}[\u].\]
All subsequent $\Oprec(\cdot)$ bounds are implicitly uniform over all such
vectors $\u$.

To prove Lemma \ref{lem:fluctuation_averaging_lowrank2},
we shall now develop a tensor network language to express the high moment
expansions of the quantities to be bounded.

For any tensor $T\in(\C^n)^{\otimes k}$, we write $\deg(T)=k$ to mean the degree
or order of the tensor. In particular, this means for vectors $\x\in\C^n$
that $\deg(\x)=1$ and for matrices $A\in\C^{n\times n}$ that $\deg(A)=2$.

\begin{defi}\label{def:tensor network}
For any integer $L\geq 1$ and $S\subset[N]$ with $|S|\leq L$,
we say that $(\cG, f_\cG)$ is a
\textit{$(S,L)$-valid tensor network} if there exist constants $M,C \geq 1$ 
depending only on $L$ such that the following holds: 

Let $\{\v_\alpha\}_{\alpha=1}^n$
be the (random, real) eigenvectors of $R^{(S)}$, let $\u \in \C^n$ be a
fixed deterministic unit vector, and let $\cU=\bigcup_{k=1}^M \cU_k$
where $\{\cU_k\}_{k \geq 1}$ are the sets of Assumption \ref{assum:cumulant}.
Then $\cG=(V_\cG,E_\cG)$ is an undirected multi-graph with no self loops
and $|V_\cG|<C$, and
$f_\cG:V_\cG\rightarrow\cT$ is a labeling function on the vertices $V_\cG$
taking values in
a set of tensors $\cT=\cL\cup\cR\cup\cE\cup\cI_m\cup\cI_t$, where
\[\cL:=\{\u,\Bar{\u}\},\quad \cR:=\cR_o\cup\cR_{n},\quad
\cR_o:=\left\{\frac{R^{(S)}}{\sqrt{N}}\u,\frac{\Bar{R}^{(S)}}{\sqrt{N}}\Bar{\u}\right\},\quad
\cR_{n}:=\bigcup_{\x\in \cU}\left\{\frac{R^{(S)}}{\sqrt{N}}\x,\frac{\Bar{R}^{(S)}}{\sqrt{N}}\x\right\},
\]
\[
\cE:=\{\v_1,\ldots,\v_n\},\quad
\cI_m:=\left\{\frac{R^{(S)}}{N},\frac{\Bar{R}^{(S)}}{N}\right\}
\quad
    \cI_t:=\bigcup_{k=2}^M\{\kappa_k(\g_1),\ldots,\kappa_k(\g_N)\}\text{
(where $\kappa_2(\g_i)=\Sigma$)}.
\]
We say that $v\in V_\cG$ is a \textit{left leaf}, a \textit{right leaf}, a \textit{right original leaf}, a \textit{right new leaf}, an \textit{eigen-leaf}, a \textit{matrix vertex}, or a \textit{tensor vertex} 
if $f_\cG(v)$ belongs to $\cL,\cR,\cR_o,\cR_{n},\cE,\cI_m$, or $\cI_t$
respectively. We denote the corresponding sets of vertices
\[
    V_\cG^{l},V_\cG^{r},V_\cG^{ro},V_\cG^{rn},V_\cG^{e},V_\cG^{m},V_\cG^t
\subseteq V_\cG.
\]
Moreover, we require that the following hold for $(\cG,f_\cG)$:
\begin{enumerate}[label=(\alph*)]
    \item For all $v\in V_\cG$, we have $\deg(v)=\deg(f_\cG(v))$, where
$\deg(v)$ is the vertex degree of $v \in V_\cG$ and $\deg(f_\cG(v))$ is the
degree/order of its label $f_\cG(v)$. In particular, $\cG$ has no vertices of
degree 0, each vertex of $V_\cG^l\cup V_\cG^r \cup V_\cG^e$ has degree 1,
and each vertex of $V_\cG^m$ has degree 2.
    \item $\cG$ is a bipartite multi-graph between 
$V_\cG^l \cup V_\cG^r \cup V_\cG^e \cup V_\cG^m$ and $V_\cG^t$, i.e., each
edge of $E_\cG$ connects some $u \in V_\cG^l \cup V_\cG^r \cup V_\cG^e \cup
V_\cG^m$ with some $v \in V_\cG^t$.
    \item If $u \in V_\cG^m$ and $v \in V_\cG^t$ have 2 edges between them,
then $\deg(v) \geq 3$. We refer to such a vertex $u \in V_\cG^m$ as a
\textit{type 1 matrix vertex}, and we refer to each other matrix vertex
$u \in V_\cG^m$ (connecting to two distinct vertices $v,v' \in V_\cG^t$)
as a \textit{type 2 matrix vertex}. We denote
the sets of such vertices by $V_\cG^{m1},V_\cG^{m2} \subseteq V_\cG^m$.
\end{enumerate}
\end{defi}

For each tensor vertex $v\in V_\cG^t$, we will write
$l(v),r(v),ro(v),rn(v),e(v),m(v),m1(v),m2(v)$ for the numbers of left leaves,
right leaves, right original leaves, right new leaves, eigen-leaves, matrix
vertices, type 1 matrix vertices, and type 2 matrix vertices adjacent to $v$.
(Each adjacent type 1 matrix neighbor contributes a count of 1 to $m1(v)$,
even though it has 2 edges connecting to $v$.)

The proof of Lemma \ref{lem:fluctuation_averaging_lowrank2}(a) 
will consider networks having only labels
$\cL,\cR_o,\cR_n,\cI_m,\cI_t$ (thus no eigen-leaves), while the proof
of Lemma \ref{lem:fluctuation_averaging_lowrank2}(b--c) will pertain to
networks having only labels $\cE,\cR_o,\cR_n,\cI_m,\cI_t$ (thus no left leaves).
For convenience, we state here several general results that pertain to networks
which may have vertices of all types.

\begin{defi}
For a $(S,L)$-valid tensor network $(\cG,f_\cG)$, we define its contracted value to be
\[
    \mathrm{val}(\cG,f_\cG):=\sum_{\alpha\in[n]^{E_\cG}}\prod_{v\in V_\cG}[f_\cG(v)]_{\alpha(\partial v)}
\]
where $\alpha(\partial v)$ denotes the multi-set of indices $\alpha_e$
associated with the edges $e$ incident to $v$. Note that this is well-defined
since all tensors $f_\cG(v) \in \cT$ are
symmetric (under permutations of coordinates, without complex conjugation).
\end{defi}

\begin{remark}\label{remk:splitting of network}
$\mathrm{val}(\cG,f_\cG)$ is multiplicative across disjoint connected components
of the graph. I.e., if $(\cG,f_\cG)$ is a $(S,L)$-valid tensor network such
that $\cG$ splits into two connected components $\cG_1,\cG_2$, then both
$(\cG_1,f_{\cG_1}), (\cG_2,f_{\cG_2})$ must be $(S,L)$-valid tensor network,
with $f_{\cG_i}$ defined as the restriction of $f_\cG$ to $\cG_i$, and
$\mathrm{val}(\cG,f_\cG)
=\mathrm{val}(\cG_1,f_{\cG_1}) \times \mathrm{val}(\cG_2,f_{\cG_2})$.
\end{remark}

We will proceed to bound the value of a $(S,L)$-valid tensor network via
operations that successively remove tensor vertices $v \in V_\cG^t$ with
$\deg(v) \geq 3$ from $\cG$.
This procedure is formalized by the following definition.

\begin{defi}\label{def:removev}
For a $(S,L)$-valid tensor network $(\cG,f_\cG)$ and $v\in V_\cG^t$ with
$\deg(v)\geq 3$, we say that $(\cG',\cF_{\cG'})$ is \textit{the family of
networks generated by removing $v$} if $G'$ is constructed by removing $v$, its
incident edges, and all of its adjacent leaves and adjacent
type 1 matrix vertices, and
$\cF_{\cG'}$ is the set of all labelings $f_{\cG'}:V_{\cG'}\rightarrow\cT$
satisfying the following:
\begin{itemize}
    \item If $u\in V_\cG$ is not adjacent to $v$, then $f_\cG(u)=f_{\cG'}(u).$
    \item If $u \in V_\cG^{m2}$ is adjacent to $v$, i.e.\ $u$ is a type 2 matrix
neighbor of $v$ (and hence remains in $\cG'$), then $f_{\cG'}(u) \in \cR_n$.
\end{itemize}
In words, each $(\cG',f_{\cG'})$ is constructed by
removing $v$ and replacing each type 2 matrix neighbor of $v$ by a
right new leaf, and $\cF_{\cG'}$ comprises all possible labelings of these right
new leaves.
\end{defi}

We note that this procedure of removing $v$ from $G$ does not change the degree
of any vertices which remain in $G$.
The following lemma gives a basic bound for $\mathrm{val}(\cG,f_\cG)$
via such a removal.

\begin{lemma}\label{lem:remove tensor weak}
Suppose the conditions of Lemma \ref{lem:fluctuation_averaging_lowrank2} hold.
Let $(\cG,f_\cG)$ be a $(S,L)$-valid tensor network, let
$v\in V_\cG^t$ be a tensor vertex satisfying $\deg(v) \geq 3$, and let
$(\cG',\cF_{\cG'})$ be the family of networks generated by removing $v$. Then
\[
    |\mathrm{val}(\cG,f_\cG)| \prec
\Phi^{r(v)}\times\max_{f_{\cG'}\in\cF_{\cG'}}|\mathrm{val}(\cG',f_{\cG'})|
\]
uniformly over $z\in\bD$ and all $S\subseteq[N]$ with $|S|\leq L$.
\end{lemma}

\begin{proof}
By Remark \ref{remk:splitting of network}, we may consider without loss of
generality the case where $\cG$ is a single connected component. A
representative form of $(\cG,f_\cG)$ is given by the following picture, where
$v \in V_\cG^t$ has label $\kappa_k(\g_i)$, $\x$ denotes a deterministic
vector in $\cU$, $\v_\alpha$ denotes an eigenvector of $R^{(S)}$,
and $T$ denotes the contraction of all tensors
connecting to $v$ via its neighboring type 2 matrix vertices:
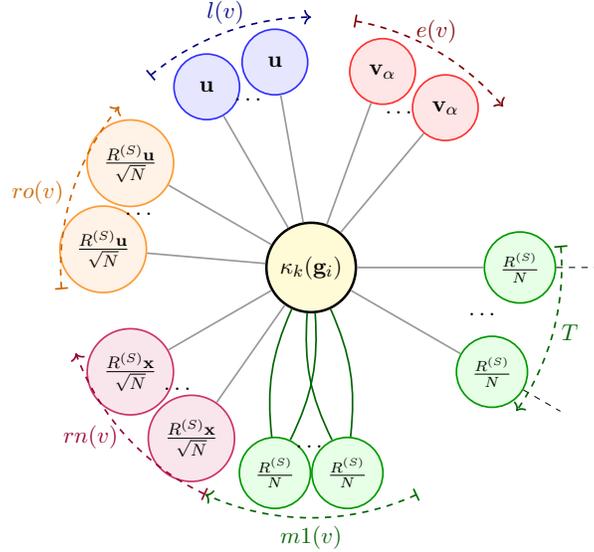
\begin{figure}[H]
    \centering
    \resizebox{0.5\textwidth}{!}{
        \begin{tikzpicture}[
    node distance=3.5cm,
    center_node/.style={circle, draw=black, very thick, fill=yellow!20, minimum size=1.5cm, font=\bfseries},
    leaf_node/.style={circle, draw=blue!80, thick, fill=blue!10, minimum size=1.1cm},
    eigen_node/.style={circle, draw=red!80, thick, fill=red!10, minimum size=1.1cm},
    matrix_node/.style={circle, draw=green!60!black, thick, fill=green!10, minimum size=1.1cm, font=\small},
    conn/.style={thick, draw=gray!80},
    double_conn/.style={thick, draw=green!40!black}
]

    \node[center_node] (center) at (0,0) {$\kappa_k(\g_i)$};

    
    \foreach \angle/\i in {100/1, 120/2} {
        \node[leaf_node] (l\i) at (\angle:3.5cm) {$\u$};
        \draw[conn] (center) -- (l\i);
    }
    \node at (110:3cm) {$\dots$}; 
    \draw[|->, thick, dashed, blue!50!black] (130:4.2cm) arc (130:90:4.2cm) node[midway, above] {$l(v)$};

    \foreach \angle/\i in {50/1, 70/2} {
        \node[eigen_node] (e\i) at (\angle:3.5cm) {$\v_\alpha$};
        \draw[conn] (center) -- (e\i);
    }
    \node at (60:3cm) {$\dots$};
    \draw[|->, thick, dashed, red!50!black] (80:4.2cm) arc (80:40:4.2cm) node[midway, above] {$e(v)$};

    \node[matrix_node] (t1) at (0:3.5cm) {$\frac{R^{(S)}}{N}$};
    \node[matrix_node] (t2) at (-30:3.5cm) {$\frac{R^{(S)}}{N}$};
    \draw[conn] (center) -- (t1);
    \draw[conn] (center) -- (t2);
    \draw[dashed, shorten >=5pt] (t1) -- +(0:1.5cm); 
    \draw[dashed, shorten >=5pt] (t2) -- +(-30:1.5cm);
    \node at (-15:3cm) {$\dots$};
    \draw[|->, thick, dashed, green!40!black] (5:4.2cm) arc (5:-35:4.2cm) node[midway, right] {$T$};

    \node[matrix_node] (m1) at (-80:3.5cm) {$\frac{R^{(S)}}{N}$};
    \node[matrix_node] (m2) at (-100:3.5cm) {$\frac{R^{(S)}}{N}$};
    
    \draw[double_conn] (center) to[bend left=15] (m1);
    \draw[double_conn] (center) to[bend right=15] (m1);
    \draw[double_conn] (center) to[bend left=15] (m2);
    \draw[double_conn] (center) to[bend right=15] (m2);
    
    \node at (-90:3cm) {$\dots$};
    \draw[|->, thick, dashed, green!40!black] (-65:4.2cm) arc (-65:-115:4.2cm) node[midway, below] {$m1(v)$};

    \node[leaf_node, fill=orange!10, draw=orange!80] (ro1) at (150:3.5cm) {$\frac{R^{(S)}\u}{\sqrt{N}}$};
    \node[leaf_node, fill=orange!10, draw=orange!80] (ro2) at (175:3.5cm) {$\frac{R^{(S)}\u}{\sqrt{N}}$};
    \draw[conn] (center) -- (ro1);
    \draw[conn] (center) -- (ro2);
    \node at (162.5:3cm) {$\dots$};
    \draw[|->, thick, dashed, orange!80!black] (185:4.2cm) arc (185:140:4.2cm) node[midway, left] {$ro(v)$};
    
    \node[leaf_node, fill=purple!10, draw=purple!80] (rn1) at (210:3.5cm) {$\frac{R^{(S)}\x}{\sqrt{N}}$};
    \node[leaf_node, fill=purple!10, draw=purple!80] (rn2) at (235:3.5cm) {$\frac{R^{(S)}\x}{\sqrt{N}}$};
    \draw[conn] (center) -- (rn1);
    \draw[conn] (center) -- (rn2);
    \node at (222.5:3cm) {$\dots$};
    \draw[|->, thick, dashed, purple!80!black] (245:4.2cm) arc (245:200:4.2cm) node[midway, left] {$rn(v)$};

\end{tikzpicture}
    }
\end{figure}
\noindent
(Other forms of $(\cG,f_\cG)$ may replace certain copies of $R^{(S)}$ or $\u$
by their complex conjugates, have differing vectors $\x \in \cU$ for the
$rn(v)$ neighbors that are right new leaves, and/or have differing
eigenvectors $\v_\alpha$ for the $e(v)$ neighbors that are eigen-leaves.
For all such networks, $\mathrm{val}(\cG,f_\cG)$ may be bounded
similarly. Throughout this and the subsequent proofs, we will focus on a single
representative example of $\mathrm{val}(\cG,f_\cG)$ for notational convenience.)

Writing the spectral decomposition
$R^{(S)}=\sum_\alpha (\lambda_\alpha-z)^{-1}\v_\alpha^{\otimes 2}$,
$|\mathrm{val}(\cG,f_\cG)|$ for the above network may be bounded using
Lemma \ref{lem:weak cumulant bound} as
\begin{align*}
    &\left|\left\<\kappa_k(\g_i),(N^{-1}R^{(S)})^{\otimes
m1(v)}\otimes\v_\alpha^{\otimes e(v)}\otimes\u^{\otimes
l(v)}\otimes(N^{-1/2}R^{(S)}\u)^{\otimes
ro(v)}\otimes(N^{-1/2}R^{(S)}\x)^{\otimes rn(v)}\otimes T\right\>\right|\\
    &\leq N^{-m1(v)}\times\sum_{\alpha_1,\ldots,\alpha_{m1(v)}}\prod_{l=1}^{m1(v)}\frac{1}{|\lambda_{\alpha_l}-z|}\times\\
    &\quad\left|\left\<\kappa_k(\g_i),\bigotimes_{l=1}^{m1(v)}\v_{\alpha_l}^{\otimes
2}\otimes\v_\alpha^{\otimes e(v)}\otimes\u^{\otimes
l(v)}\otimes(N^{-1/2}R^{(S)}\u)^{\otimes
ro(v)}\otimes(N^{-1/2}R^{(S)}\x)^{\otimes rn(v)}\otimes T\right\>\right|\\
&\leq (N^{-1}\|R^{(S)}\|_1)^{m1(v)} \times\\
&\quad \max_{\alpha_1,\ldots,\alpha_{m1(v)}}
\left|\left\<\kappa_k(\g_i),\bigotimes_{l=1}^{m1(v)}\v_{\alpha_l}^{\otimes
2}\otimes\v_\alpha^{\otimes e(v)}\otimes\u^{\otimes
l(v)}\otimes(N^{-1/2}R^{(S)}\u)^{\otimes
ro(v)}\otimes(N^{-1/2}R^{(S)}\x)^{\otimes rn(v)}\otimes T\right\>\right|\\
    &\overset{(a)}{\prec} (N^{-1}\|R^{(S)}\|_1)^{m1(v)}\times
(\sqrt{N})^{\deg(T)} \times \|N^{-1/2}R^{(S)}\u\|_2^{ro(v)}
\times \|N^{-1/2}R^{(S)}\x\|_2^{rn(v)} \times\|T\|_{\cU_k}\\
    &\overset{(b)}{\prec} \Phi^{r(v)} \times
\|(\sqrt{N})^{\deg(T)}T\|_{\cU_k}
\overset{(c)}{\prec}
\Phi^{r(v)}\times\max_{f_{\cG'}\in\cF_{\cG'}}|\mathrm{val}(\cG',f_{\cG'})|.
\end{align*}
Here (a) applies Lemma \ref{lem:weak cumulant bound}, (b) applies the
bounds $N^{-1}\|R^{(S)}\|_1 \prec 1$,
 $N^{-1/2}\|R^{(S)}\u\|_2 \prec \Phi$, and $N^{-1/2}\|R^{(S)}\x\|_2 \prec
\Phi \|\x\|_2 \prec \Phi$ uniformly over all vectors $\x \in \cU$,
and (c) uses that $\|(\sqrt{N})^{\deg(T)}T\|_{\cU_k}$ is the maximum of
$|\mathrm{val}(\cG',f_{\cG'})|$ over a subset of the labelings $\cF_{\cG'}$ in
the network generated by removing $v$.
\end{proof}

For certain types of vertices $v \in V_\cG^t$, we will require a refinement of
the above bound using Assumption \ref{assum:cumulant} directly, instead of its
weaker implication of Lemma \ref{lem:weak cumulant bound}.

\begin{lemma}\label{lem:remove tensor strong}
Suppose the conditions of Lemma \ref{lem:fluctuation_averaging_lowrank2} hold.
Let $(\cG,f_\cG)$ be a $(S,L)$-valid tensor network, let $v\in V_\cG^t$ be
a tensor vertex satisfying $\deg(v)\geq 3$, and let
$(\cG',\cF_{\cG'})$ be the the family of networks generated by removing $v$.
Then:
\begin{enumerate}[label=(\alph*)]
    \item If $m1(v)\geq 1$, or if $l(v)+e(v)=1$
and $r(v)+m1(v)=0$, then
    \[
        |\mathrm{val}(\cG,f_\cG)|\prec\frac{N^\delta}{\sqrt{N}}\times\Phi^{r(v)}\times\max_{f_{\cG'}\in\cF_{\cG'}}|\mathrm{val}(\cG',f_{\cG'})|
    \]
uniformly over all $S \subset [N]$ with $|S|\leq L$ and over $z\in\bD$.
    \item If $r(v) \geq 1$, then
    \[
        |\mathrm{val}(\cG,f_\cG)|\prec\frac{N^\delta}{\sqrt{N}}\times\Phi^{r(v)-1}\times\max_{f_{\cG'}\in\cF_{\cG'}}|\mathrm{val}(\cG',f_{\cG'})|.
    \]
uniformly over all $S \subset [N]$ with $|S|\leq L$ and over $z\in\bD$.
\end{enumerate}
\end{lemma}

\begin{proof}
We may again consider the case where $\cG$ is a single connected component.

For (a), suppose first $m1(v)\geq 1$ and the only vertices adjacent to $v$ are
type 1 matrix vertices. This implies that in fact $m1(v) \geq 2$,
because $\deg(v) \geq 3$. Then $v$ and these adjacent vertices form a connected
component, and thus must be all of $\cG$ since $\cG$ is connected. 
In this case, $\mathrm{val}(\cG,f_\cG)$ takes the representative
form $|\<\kappa_k(\g_i),(N^{-1}R^{(S)})^{\otimes m1(v)}\>|$, for which we have
\begin{align*}
    &\left|\left\<\kappa_k(\g_i),(N^{-1}R^{(S)})^{\otimes m1(v)}\right\>\right|\\
    &\leq
N^{-(m1(v)-1)}\sum_{\alpha_1,\ldots,\alpha_{m1(v)-1}}\left(\prod_{t=1}^{m1(v)-1}\frac{1}{|\lambda_{\alpha_t}-z|}\right)\left|\left\<\kappa_k(\g_i),\bigotimes_{t=1}^{m1(v)-1}\v_{\alpha_t}^{\otimes
2}\otimes N^{-1}R^{(S)}\right\>\right|\\
    &\leq(N^{-1}\|R^{(S)}\|_1)^{m1(v)-1}\max_{\alpha_1,\ldots,\alpha_{m1(v)-1}}\left|\left\<\kappa_k(\g_i),\bigotimes_{t=1}^{m1(v)-1}\v_{\alpha_t}^{\otimes
2}\otimes N^{-1}R^{(S)}\right\>\right|\\
    &\prec (N^{-1}\|R^{(S)}\|_1)^{m1(v)-1}(\sqrt{N})^{2-1}\|N^{-1}
R^{(S)}\|_{\cU_k}\prec N^{-1/2}N^\delta.
\end{align*}
by Assumption \ref{assum:cumulant} and the bounds
$\|N^{-1}R^{(S)}\|_1 \prec 1$ and
$\|R^{(S)}\|_{\cU_k} \prec N^\delta$. This
implies the lemma since $r(v)=0$ and $\cG'$ is empty.

If $m1(v)\geq 1$ and the only vertices adjacent to $v$ are matrix vertices, of
which at least one is type 2, then a representative form of
$\mathrm{val}(\cG,f_\cG)$ is bounded using Assumption \ref{assum:cumulant} 
similarly as
\begin{align*}
    &\left|\left\<\kappa_k(\g_i),(N^{-1}R^{(S)})^{\otimes m1(v)}\otimes T\right\>\right|\\
    &\leq
N^{-m1(v)}\sum_{\alpha_1,\ldots,\alpha_{m1(v)}}\prod_{t=1}^{m1(v)}\frac{1}{|\lambda_{\alpha_t}-z|}\left|\left\<\kappa_k(\g_i),\bigotimes_{t=1}^{m1(v)}\v_{\alpha_t}^{\otimes
2}\otimes T\right\>\right|\\
    &\prec (N^{-1}\|R^{(S)}\|_1)^{m1(v)}(\sqrt{N})^{\deg(T)-1}\|T\|_{\cU_k}\prec
N^{-1/2}\times\max_{f_{\cG'}\in\cF_{\cG'}}|\mathrm{val}(\cG',f_{\cG'})|.
\end{align*}
This implies the lemma since  $r(v)=0$ and $\delta \geq 0$.

If $m1(v)\geq 1$ and $v$ is adjacent to at least one leaf vertex,
i.e.\ $l(v)+e(v)+r(v)\geq 1$, then a representative form of
$\mathrm{val}(\cG,f_\cG)$ is bounded similarly using
Assumption \ref{assum:cumulant} as
\begin{align*}
    &\left|\left\<\kappa_k(\g_i),\u^{\otimes l(v)}\otimes\v_\beta^{\otimes
e(v)}\otimes(N^{-1/2}R^{(S)}\u)^{ro(v)}\otimes(N^{-1/2}R^{(S)}\x)^{\otimes
rn(v)}\otimes(N^{-1}R^{(S)})^{\otimes m1(v)}\otimes T\right\>\right|\\
    &\leq
N^{-(m1(v)-1)}\sum_{\alpha_1,\ldots,\alpha_{m1(v)-1}}\left(\prod_{t=1}^{m1(v)-1}\frac{1}{|\lambda_{\alpha_t}-z|}\right)\\
    &\left|\left\<\kappa_k(\g_i),\u^{\otimes l(v)}\otimes\v_\beta^{\otimes
e(v)}\otimes(N^{-1/2}R^{(S)}\u)^{ro(v)}\otimes(N^{-1/2}R^{(S)}\x)^{\otimes
rn(v)}\otimes\bigotimes_{t=1}^{m1(v)-1}\v_{\alpha_t}^{\otimes
2}\otimes(N^{-1}R^{(S)})\otimes T\right\>\right|\\
    &\prec(N^{-1}\|R^{(S)}\|_1)^{m1(v)-1}(\sqrt{N})^{\deg(T)+2-1}
\|N^{-1/2}R^{(S)}\u\|_2^{ro(v)}\|N^{-1/2}R^{(S)}\x\|_2^{rn(v)}
\|N^{-1} R^{(S)}\|_{\cU_k}\|T\|_{\cU_k}\\
    &\prec
N^{-1/2}N^\delta \times \Phi^{r(v)} \times
\|(\sqrt{N})^{\deg(T)}T\|_{\cU_k}\\
&\prec N^{-1/2}N^\delta \times \Phi^{r(v)} \times
\max_{f_{\cG'}\in\cF_{\cG'}}|\mathrm{val}(\cG',f_{\cG'})|.
\end{align*}
This covers all cases where $m1(v) \geq 1$.

Next, suppose $l(v)+e(v)=1$ and $r(v)+m1(v)=0$. Then, since $\deg(v)\geq
3$, $v$ is adjacent to at least one type 2 matrix vertex, so a
representative form of $\mathrm{val}(\cG,f_\cG)$ is bounded as
\[
    \left|\left\<\kappa_k(\g_i),\u^{\otimes l(v)}\otimes\v_\beta^{\otimes
e(v)}\otimes T\right\>\right|\prec (\sqrt{N})^{\deg(T)-1}\|T\|_{\cU_k}\prec
N^{-1/2}\times\max_{f_{\cG'}\in\cF_{\cG'}}|\mathrm{val}(\cG',f_{\cG'})|.
\]
This again implies the lemma since $r(v)=0$ and $\delta \geq 0$.
This shows part (a).

For (b), if $m1(v)\geq 1$, then the statement is implied by (a) as $\Phi \leq
1$. Suppose then that $m1(v)=0$. If $v$ is adjacent to at least one 
type 2 matrix vertex,
then a representative form of $\mathrm{val}(\cG,f_\cG)$ is bounded as
\begin{align*}
    &\left|\left\<\kappa_k(\g_i),\u^{\otimes l(v)}\otimes\v_\beta^{\otimes
e(v)}\otimes(N^{-1/2}R^{(S)}\u)^{ro(v)}\otimes(N^{-1/2}R^{(S)}\x)^{\otimes
rn(v)}\otimes T\right\>\right|\\
    &\prec (\sqrt{N})^{\deg(T)-1}\Phi^{r(v)}\|T\|_{\cU_k}\prec
N^{-1/2}\times\Phi^{r(v)}\times\max_{(\cG',f_{\cG'})\in\cF_{\cG'}}|\mathrm{val}(\cG',f_{\cG'})|
\end{align*}
which again implies the lemma.
Finally, if $v$ is not adjacent to any matrix vertex, then all neighbors of $v$
have degree 1, and $v$ and its neighbors constitute all of $\cG$ which is a
star. Isolating one factor of $R^{(S)}\u$ or $R^{(S)}\x$ as $T$
in Assumption \ref{assum:cumulant},
a representative form of $\mathrm{val}(\cG,f_\cG)$ is bounded as
\begin{align*}
    &\left|\left\<\kappa_k(\g_i),\u^{\otimes l(v)}\otimes\v_\beta^{\otimes
e(v)}\otimes(N^{-1/2}R^{(S)}\u)^{ro(v)}\otimes(N^{-1/2}R^{(S)}\x)^{\otimes
rn(v)}\right\>\right|\\
    &\prec \max\{N^{-1/2}\|R^{(S)}\x\|_{\cU_k},N^{-1/2}\|R^{(S)}\u\|_{\cU_k}\}\times\max\{N^{-1/2}\|R^{(S)}\u\|_2,N^{-1/2}\|R^{(S)}\x\|_2\}^{r(v)-1}\\
    &\prec N^{-1/2}N^\delta\times\Phi^{r(v)-1}.
\end{align*}
This implies the lemma since $\cG'$ is empty, completing the proof.
\end{proof}

\subsubsection{Proof of Lemma \ref{lem:fluctuation_averaging_lowrank2}(a)}

We now focus attention on networks with left leaves and no eigen-leaves, and
show Lemma \ref{lem:fluctuation_averaging_lowrank2}(a).
In addition to Definition \ref{def:tensor network}, we define the following.

\begin{defi}\label{def:special vertices}
Let $(\cG,f_\cG)$ be a $(S,L)$-valid tensor network that contains no
eigen-leaves. We say that a left leaf $u \in V_\cG^l$ is \textit{singular} if
its adjacent tensor vertex $v \in V_\cG^t$ is not adjacent to any left leaves
other than $u$, and we call $v$ a \textit{singular tensor vertex}. We write
$V_\cG^{sl},V_\cG^{st}$ for the sets of such vertices, and $t(u)=v$
for the one-to-one correspondence between these two sets.
\end{defi}

The following lemma now expresses products of monomials arising from
Lemma \ref{lem:y_i q_i expansion} via tensor network values.

\begin{lemma}\label{lem:tensor network of sample covariance}
Fix any integer $L\geq 1$. For any $i_1,\ldots,i_{2L}\in[N]$, denote
$S=\{i_1,\ldots,i_{2L}\}$ the set of distinct indices. Let $q_1,\ldots,q_{2L}$
be a sequence of monomials satisfying the conditions of Lemma
\ref{lem:y_i q_i expansion}, where each $q_l$ is a monomial from the expansion
of $\cY_{i_l}^{(i_l)}\cQ_{i_l}^{(i_l)}$, and let
$q=\prod_{l=1}^L q_l \prod_{l=L+1}^{2L} \bar q_l$.
Then there exists a constant $C \equiv C(L)>0$, an exponent $m \in [0,C]$,
and a set $\cN$ of $(S,2L)$-valid tensor networks with $|\cN|\leq C$,
such that
\[|\E_S q|=|\cC^{(S)}|^m \cdot 
\sum_{(\mathcal{G},f_\cG)\in\cN}|\mathrm{val}(\cG,f_\cG)|\]
Moreover, for all $(\mathcal{G}, f_\cG)\in\cN$, there exist disjoint subsets of singular left leaves $U_\cG^1,U_\cG^2\subseteq V_\cG^{sl}$ such that the following holds:
\begin{enumerate}[label=(\alph*)]
    \item We have $|V_\cG^{l}|=|V_\cG^{r}|=|V_\cG^{ro}|=2L$ and
$|V_\cG^{rn}|=|V_\cG^{e}|=0$. (I.e., there are no right new leaves or
eigen-leaves, and exactly $2L$ left leaves and $2L$ right original leaves.)
    \item Let $K \subseteq S$ be the set of indices that appear exactly once
in the list $(i_1,\ldots,i_{2L})$. Then $|K|=|U_\cG^1|+|U_\cG^2|$.
    \item We have $|V_\cG^{m2}|\geq |U_\cG^1|$.
    \item For all $u \in U_\cG^2$, we have $\deg(t(u))\geq 3$
and $m1(t(u)) \geq 1$.
\end{enumerate}
\end{lemma}

\begin{proof}
For each distinct index $i \in S$, collecting all appearances of $\g_i$
and $(\g_i\g_i^*-\Sigma)$ across $q_1,\ldots,q_{2L}$ into a single tensor, and
then evaluating the expectation $\E_S$ over the independent vectors $\{\g_i\}_{i
\in S}$, it is clear from the forms of 
$\cY_i^{(S)},\cZ_{ijk}^{(S)},\cB_{jk}^{(S)},\cP_j^{(S)},\cC^{(S)}$
that $\E_S q$ may be expressed as a product of factors $\cC^{(S)}$,
$\overline{\cC^{(S)}}$ and contractions of tensors of the form
$\{\E[(\g_i\g_i^*-\Sigma)^{\otimes k}\otimes\g_i^{\otimes m}]:
i \in S \text{ and } k,m \in [0,C(L)]\}$ with matrices and vectors of the forms
\[\cL:=\{\u,\Bar{\u}\},\quad
\cR_o:=\left\{\frac{R^{(S)}\u}{\sqrt{N}},\frac{\overline{R^{(S)}\u}}{\sqrt{N}}\right\},\quad\cI_m:=\left\{\frac{R^{(S)}}{N},\frac{\overline{R^{(S)}}}{N}\right\}.
\]
For each distinct index $i \in S$, we may then decompose the corresponding
moment tensor using Lemma \ref{lem:cumulant pairing},
\[
    \E[(\g_i\g_i^*-\Sigma)^{\otimes k}\otimes\g_i^{\otimes m}]=\sum_{\pi\in
\dot{P}_{k,m}}\kappa_\pi(\g_i).
\]
Then expanding by multi-linearity, we claim that $\E_S q$ is a product of
factors $\cC^{(S)},\overline{\cC^{(S)}}$ with 
$\sum_{(\cG,f_\cG) \in \cN} |\mathrm{val}(\cG,f_\cG)|$ over a family $\cN$
of $(S,2L)$-valid tensor networks, where $|\cN| \leq C(L)$.

To see this claim, consider as an illustrative example
\[q_1=\cZ_{ijj}^{(ij)}\cC^{(ij)},
\qquad q_2=\cZ_{jii}^{(ij)}\cC^{(ij)},
\qquad q=q_1\bar{q}_2\]
for two distinct indices $S=\{i,j\}$. Then, defining a permutation
$Q:(\C^n)^8 \to (\C^n)^8$ such that
$Q(T)(i_1,\ldots,i_8)=T(i_1,i_2,i_7,i_8,i_5,i_6,i_3,i_4)$, we have
\begin{align*}
    \E_{\{ij\}}[q]&=|\cC^{(ij)}|^2\E_{\{ij\}}\Big[\u^*(\g_i\g_i^*-\Sigma)(N^{-1}R^{(ij)})\g_j\g_j^*(N^{-1/2}R^{(ij)}\u)
\times \\
&\hspace{2in}\overline{\u^*(\g_j\g_j^*-\Sigma)(N^{-1}R^{(ij)})\g_i\g_i^*(N^{-1/2}R^{(ij)}\u)}\Big]\\
    &=|\cC^{(ij)}|^2 \left\langle\E\left[(\g_i\g_i^*-\Sigma)\otimes\g_i^{\otimes 2}\right]\otimes\E\left[(\g_j\g_j^*-\Sigma)\otimes\g_j^{\otimes 2}\right],\right.\\
    &\hspace{1in}\left.Q\Big(\Bar\u\otimes N^{-1}R^{(ij)} \otimes N^{-1/2}R^{(ij)}\u
\otimes\u\otimes \overline{N^{-1}R^{(ij)}} \otimes \overline{N^{-1/2}R^{(ij)}\u}\Big)\right\rangle\\
    &=\sum_{\pi_1 \in \dot P_{1,2}}\sum_{\pi_2 \in \dot P_{1,2}}|\cC^{(ij)}|^2
\Big\langle\kappa_{\pi_1}(\g_i)\otimes\kappa_{\pi_2}(\g_j),\\
&\hspace{1.5in}Q\Big(\Bar\u\otimes N^{-1}R^{(ij)} \otimes N^{-1/2}R^{(ij)}\u
\otimes\u\otimes \overline{N^{-1}R^{(ij)}} \otimes
\overline{N^{-1/2}R^{(ij)}\u}\Big)\Big\rangle.
\end{align*}
The summand for each pair of partitions $(\pi_1,\pi_2)$ corresponds to
$|\cC^{(ij)}|^2$ times $\mathrm{val}(\cG,f_\cG)$ for a (possibly disconnected)
$(S,2L)$-valid tensor network. For example, the partitions
$\pi_1=\{\{1,3\},\{2,4\}\}$ and $\pi_2=\{\{1,2,3,4\}\}$ correspond to the
network
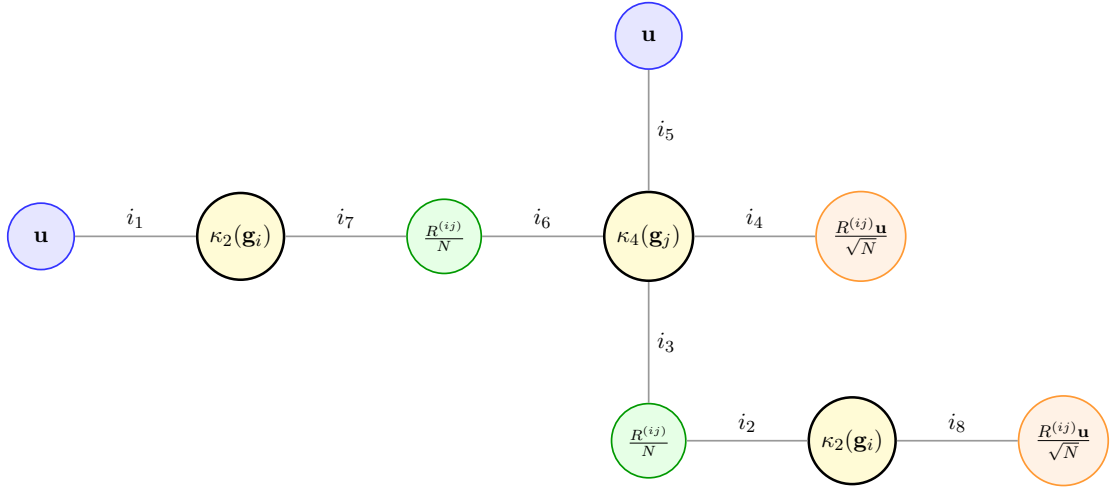
\begin{figure}[H]
    \centering
    \resizebox{0.9\textwidth}{!}{
        \begin{tikzpicture}[
    node distance=2.0cm, 
    tensor_node/.style={circle, draw=black, very thick, fill=yellow!20, minimum size=1.3cm, font=\bfseries},
    leaf_node/.style={circle, draw=blue!80, thick, fill=blue!10, minimum size=1.1cm},
    matrix_node/.style={circle, draw=green!60!black, thick, fill=green!10, minimum size=1.1cm, font=\small},
    conn/.style={thick, draw=gray!80},
    edge_label/.style={font=\footnotesize, fill=white, inner sep=1.5pt}
]

    \node[tensor_node] (center) at (0,0) {$\kappa_4(\g_j)$};
    
    \node[leaf_node] (top_u) [above=of center] {$\u$};
    \draw[conn] (center) -- (top_u) node[midway, right] {$i_5$};

    \node[leaf_node, fill=orange!10, draw=orange!80] (right_leaf) [right=of center] {$\frac{R^{(ij)}\u}{\sqrt{N}}$};
    \draw[conn] (center) -- (right_leaf) node[midway, above] {$i_4$};

    \node[matrix_node] (left_mat) [left=of center] {$\frac{R^{(ij)}}{N}$};
    \node[tensor_node] (left_ten) [left=of left_mat] {$\kappa_2(\g_i)$};
    \node[leaf_node] (left_u) [left=of left_ten] {$\u$};
    
    \draw[conn] (center) -- (left_mat) node[midway, above] {$i_6$};
    \draw[conn] (left_mat) -- (left_ten) node[midway, above] {$i_7$};
    \draw[conn] (left_ten) -- (left_u) node[midway, above] {$i_1$};

    \node[matrix_node] (bot_mat) [below=of center] {$\frac{R^{(ij)}}{N}$};
    \node[tensor_node] (bot_ten) [right=of bot_mat] {$\kappa_2(\g_i)$};
    \node[leaf_node, fill=orange!10, draw=orange!80] (bot_leaf) [right=of bot_ten] {$\frac{R^{(ij)}\u}{\sqrt{N}}$};

    \draw[conn] (center) -- (bot_mat) node[midway, right] {$i_3$};
    \draw[conn] (bot_mat) -- (bot_ten) node[midway, above] {$i_2$};
    \draw[conn] (bot_ten) -- (bot_leaf) node[midway, above] {$i_8$};

\end{tikzpicture}
    }
\end{figure}
\noindent while the
partitions $\pi_1=\{\{1,4\},\{2,3\}\}$ and $\pi_2=\{\{1,4\},\{2,3\}\}$
correspond to the network 
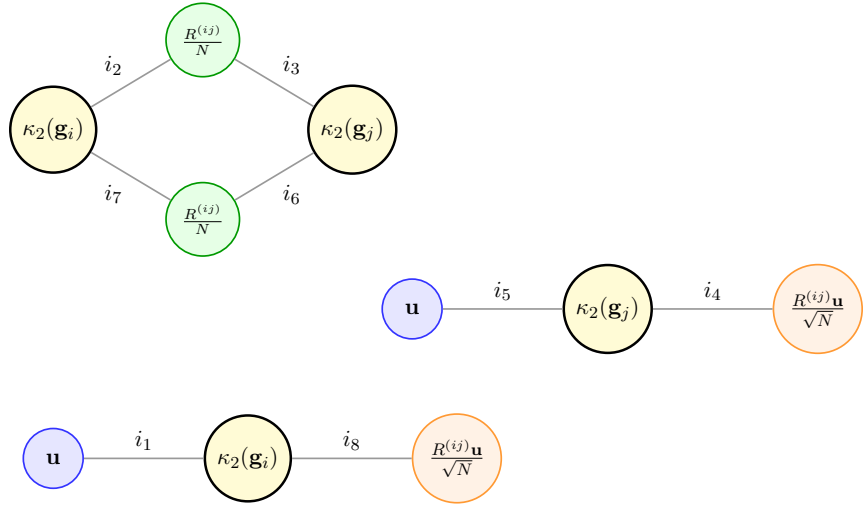
\begin{figure}[H]
    \centering
    \resizebox{0.7\textwidth}{!}{
        \begin{tikzpicture}[
    node distance=2cm,
    tensor_node/.style={circle, draw=black, very thick, fill=yellow!20, minimum size=1.2cm, font=\bfseries},
    leaf_node/.style={circle, draw=blue!80, thick, fill=blue!10, minimum size=1cm},
    matrix_node/.style={circle, draw=green!60!black, thick, fill=green!10, minimum size=1cm, font=\small},
    conn/.style={thick, draw=gray!80}
]

    \begin{scope}[yshift=2cm]
        \node[tensor_node] (d_left) at (0,0) {$\kappa_2(\g_i)$};
        \node[tensor_node] (d_right) at (5,0) {$\kappa_2(\g_j)$};
        
        \node[matrix_node] (d_top) at (2.5, 1.5) {$\frac{R^{(ij)}}{N}$};
        \node[matrix_node] (d_bot) at (2.5, -1.5) {$\frac{R^{(ij)}}{N}$};
        
        \draw[conn] (d_left) -- (d_top) node[midway, above left] {$i_2$};
        \draw[conn] (d_top) -- (d_right) node[midway, above right] {$i_3$};
        \draw[conn] (d_left) -- (d_bot) node[midway, below left] {$i_7$};
        \draw[conn] (d_bot) -- (d_right) node[midway, below right] {$i_6$};
        
    \end{scope}

    
    \begin{scope}[yshift=-3.5cm]
        \node[leaf_node] (l_u) at (0,0) {$\u$};
        \node[tensor_node] (l_ten) [right=of l_u] {$\kappa_2(\g_i)$};
        \node[leaf_node, fill=orange!10, draw=orange!80] (l_ru) [right=of l_ten] {$\frac{R^{(ij)}\u}{\sqrt{N}}$};
        
        \draw[conn] (l_u) -- (l_ten) node[midway, above] {$i_1$};
        \draw[conn] (l_ten) -- (l_ru) node[midway, above] {$i_8$};
        
    \end{scope}

    \begin{scope}[yshift=-1cm, xshift=6cm]
        \node[leaf_node] (r_u) at (0,0) {$\u$};
        \node[tensor_node] (r_ten) [right=of r_u] {$\kappa_2(\g_j)$};
        \node[leaf_node, fill=orange!10, draw=orange!80] (r_ru) [right=of r_ten] {$\frac{R^{(ij)}\u}{\sqrt{N}}$};
        
        \draw[conn] (r_u) -- (r_ten) node[midway, above] {$i_5$};
        \draw[conn] (r_ten) -- (r_ru) node[midway, above] {$i_4$};

    \end{scope}

\end{tikzpicture}
    }
\end{figure}
More generally, any such monomial $q$ may be decomposed in this way, where it is
clear by construction that each resulting $(\cG,f_\cG)$ satisfies conditions (a)
and (b) of Definition \ref{def:tensor network}, that $\cG=(V_\cG,E_\cG)$
is bipartite and $\deg(v)=\deg(f_\cG(v))$ for all $v \in V_\cG$.
It satisfies condition (c) by the following reasoning: If
$u \in V_\cG^m$, then its label $f_\cG(u)$ must be a copy of
$N^{-1}R^{(S)}$ or $\overline{N^{-1}R^{(S)}}$ arising from a term
$\cZ_{ijk}^{(S)}$, $\cB_{jk}^{(S)}$, or $\cP_j^{(S)}$. For $\cZ_{ijk}^{(S)}$,
this vertex $u$ must be type 2 by the constraint $i \neq j$ in $\cZ_{ijk}^{(S)}$, and the same holds for $\cB_{jk}^{(S)}$ by the constraint $j \neq k$.
Thus $u \in V_\cG^m$ can only be type 1 --- i.e.\ both edges connect to a
single neighbor $v \in V_\cG^t$ --- if 
$f_\cG(u)$ arises from a term $\cP_j^{(S)}$. In this case, its neighbor $v \in
V_\cG^t$ cannot have $\deg(v)=2$
because the partitions belonging to $\dot P_{k,m}$
cannot have a cardinality-2 block that
represents a pairing of the two copies of $\g_j$ in $\cP_j^{(S)}$. Thus
$\deg(v) \geq 3$, verifying condition (c). This shows the first claim of
the lemma, that
\[|\E_S q|=|\cC^{(S)}|^m \cdot \sum_{(\cG,f_\cG)\in\cN}
|\mathrm{val}(\cG,f_\cG)|\]
for some $m \in [0,C(L)]$ and family $\cN$ of $(S,2L)$-valid tensor networks
with $|\cN| \leq C(L)$.

We now prove the remaining claims of the lemma: Since each $q_l$ contains
exactly one factor $\cY_{i_l}^{(S)}$ or $\cZ_{i_ljk}^{(S)}$, there are exactly
$2L$ appearances of $\u$ and $2L$ appearances of 
$N^{-1/2}R^{(S)}\u$ in $q$. Thus $|V_\cG^l|=|V_\cG^r|=2L$. There are no
right new leaves or eigen-leaves in the above construction, so
$|V_\cG^{rn}|=|V_\cG^e|=0$, and hence also $|V_\cG^r|=|V_\cG^{ro}|$. This proves
(a).
    
Next, let $K\subseteq S$ be the subset of distinct indices which appear exactly
once in $(i_1,\ldots,i_{2L})$. For each $i_l \in K$, there is exactly 1 factor
of the form $\cY_{i_l}^{(S)}$ or $\cZ_{i_ljk}^{(S)}$ in $q$. It follows that
across all tensor vertices labeled by a cumulant of $\g_{i_l}$,
there is exactly one adjacent left leaf. This left leaf for each $i_l
\in K$ must then be singular.
Let $U_\cG \subseteq V_\cG^{sl}$ be the subset of all such
singular left leaves corresponding to the indices $i_l \in K$,
so $|U_\cG|=|K|$. We will partition $U_\cG$ into disjoint subsets
$U_\cG^1,U_\cG^2$, so that $|K|=|U_\cG^1|+|U_\cG^2|$, showing (b).

For each index $i_l \in K$, let $u(i_l) \in U_\cG$ be the corresponding singular
left leaf identified above, and let $v(i_l)=t(u(i_l)) \in V_\cG^t$ be its
tensor vertex neighbor. We define $U_\cG^1$ to be the set of $u(i_l) \in
U_\cG$ satisfying either:
\begin{enumerate}
    \item $i_l$ appears as a lower index in some $\{q_{l'}:l'\neq l\}$, or
    \item $i_l$ does not appear as a lower index in any $\{q_{l'}:l'\neq l\}$,
and either $\deg(v(i_l))=2$ or $ro(v(i_l))+m2(v(i_l)) \geq 2$.
\end{enumerate}
To bound $|U_\cG^1|$, let $m_Z(l),m_B^*(l),m_B(l)$ be the counts of Lemma
\ref{lem:y_i q_i expansion} for each $q_l$, and let
$m_P^*(l)$ be the number of factors $\cP_{i_l}^{(S)}$ in $q_l$
for which the copy of $N^{-1}R^{(S)}$ in $\cP_{i_l}^{(S)}$ labels a type 2
matrix vertex. (Note that this is different from the proof of Lemma
\ref{lem:fluctuation_averaging_lowrank} where $m_P^*(l)$ counted all factors of
$\cP_{i_l}^{(S)}$ in $q_l$.) Set
\[m_Z=\sum_{l=1}^{2L} m_Z(l),
\qquad m_B^*=\sum_{l=1}^{2L} m_B^*(l),
\qquad m_B=\sum_{l=1}^{2L} m_B(l),
\qquad m_P^*=\sum_{l=1}^{2L} m_P^*(l).\]
Lemma \ref{lem:y_i q_i expansion} ensures that the number of vertices
$u(i_l) \in U_\cG$ satisfying condition (1.) is at most
$m_Z+m_B+\frac{1}{2}m_B^*$.

For condition (2.), if $\deg(v(i_l))=2$, let $u'$ be the vertex other than
$u(i_l)$ which is incident to $v(i_l)$. Then $u'$ is not a left leaf because
$u(i_l)$ is singular. If $u'$ is a right leaf, then its label $f_\cG(u')$ must
be the copy of $N^{-1/2}R^{(S)}\u$ multiplying $\g_{i_l}^*$ in some factor
$\cZ_{ijk}^{(S)}$ of $q$ with $k=i_l$, or multiplying
$(\g_{i_l}\g_{i_l}^*-\Sigma)$ in the factor $\cY_{i_l}^{(S)}$ of $q_l$.
The former is not possible because $i_l$ does not appear as a lower index of any
$\{q_{l'}:l' \neq l\}$, and the latter is not possible because the partitions
belonging to $\dot P_{k,m}$ cannot pair the two copies of $\g_{i_l}$ in
$\cY_{i_l}^{(S)}$. Thus $u'$ is
also not a right leaf, so it must be a type 2 matrix vertex. Since
$i_l$ does not appear as a lower index of any $\{q_{l'}:l' \neq l\}$, and since
$\dot P_{k,m}$ cannot pair the two copies of $\g_{i_l}$ in either
$\cZ_{i_l}^{(S)}$ or $\cP_{i_l}^{(S)}$,
the label $f_\cG(u')$ must be the matrix $N^{-1}R^{(S)}$ arising from a factor
$\cB_{i_lj}^{(S)}$ of $q_l$. Since $m_B^*(l)$ is even by
Lemma \ref{lem:y_i q_i expansion}, this means $\frac{1}{2}m_B^*(l) \geq 1$.
By similar reasoning, if $\deg(v(i_l)) \geq 3$ and
$ro(v(i_l))+m2(v(i_l)) \geq 2$, then the factor $\cY_{i_l}^{(S)}$ or
$\cZ_{i_l}^{(S)}$ of $q$ can contribute at most one right original leaf
or type 2 matrix vertex neighbor for $v(i_l)$. Thus,
since $ro(v(i_l))+m2(v(i_l)) \geq 2$, $v(i_l)$ has some other
neighboring type 2 matrix vertex, which must arise from
a factor $\cP_{i_l}^{(S)}$ or $\cB_{i_lj}^{(S)}$ of $q_l$. Since $m_B^*(l)$ is
even, this means $\frac{1}{2}m_B^*(l)+m_P^*(l) \geq 1$. Thus the
number of vertices $u \in U_\cG$ satisfying condition (2.) is at
most $\frac{1}{2}m_B^*+m_P^*$. Putting this together,
\[|U_\cG^1| \leq \underbrace{m_Z+m_B+\frac{1}{2}m_B^*}_{\text{condition (1.)}}
+\underbrace{\frac{1}{2}m_B^*+m_P^*}_{\text{condition (2.)}}
=m_Z+m_B+m_B^*+m_P^* \leq |V_\cG^{m2}|,\]
the last inequality holding because each factor counted by
$m_Z,m_B,m_B^*,m_P^*$ has a copy of $N^{-1}R^{(S)}$ labeling a different type 2
matrix vertex. This shows (c).

Finally, by construction, all remaining vertices $u(i_l)
\in U_\cG^2=U_\cG \setminus U_\cG^1$ satisfy $\deg(v(i_l)) \geq 3$ and
$ro(v(i_l))+m2(v(i_l)) \leq 1$. Also
$rn(v(i_l))=e(v(i_l))=0$, and $l(v(i_l))=1$ since $u(i_l)$ is singular. Then
the condition $\deg(v(i_l)) \geq 3$ requires $m_1(v(i_l)) \geq 1$.
This shows (d) and completes the proof.
\end{proof}

After removing all tensor vertices $v \in V_\cG^t$ with $\deg(v) \geq 3$, 
the following lemma will be used to bound the value of a tensor network with
only chains and cycles.

\begin{lemma}\label{lem:pairing case}
Suppose the conditions of Lemma \ref{lem:fluctuation_averaging_lowrank2} hold.
Let $(\cG,f_\cG)$ be a $(S,L)$-valid tensor network that contains no
eigen-leaves, and such that $\deg(v)=2$ for all $v \in V_\cG^t$. (This
implies that $\cG$ is a simple graph with no type 1 matrix vertices and only
chains and cycles.) Then for any subset (possibly empty) of singular left
leaves $U_\cG\subseteq V_\cG^{sl}$, uniformly over all $|S|\leq L$ and $z\in\bD$,
\[
    |\mathrm{val}(\cG,f_\cG)|\prec\left(\frac{N^\delta}{\sqrt{N}}\right)^{|U_\cG|}\times\Phi^{|V_\cG^r|+|V_\cG^{m2}|-|U_\cG|}
\]
\end{lemma}
\begin{proof}
By Remark \ref{remk:splitting of network}, we may consider without loss of
generality the case where $\cG$ is a single connected component. 
We will use repeatedly that under the conditions of Lemma
\ref{lem:fluctuation_averaging_lowrank2}, we have also
\[N^{-1}\|R^{(S)}\|_F \leq \max_\alpha N^{-1/2}\|R^{(S)}\e_\alpha\|_2
\prec \Phi\]
Since $\deg(v)=2$ for each $v \in V_\cG^t$, all tensor vertices have label
$\Sigma$. If $\cG$ is a chain, we write $u_1,u_2$ for the two endpoints of the
chain. We then consider all possible cases for $\cG$:

(1) Suppose $\cG$ is a cycle. Then a representative form of $(\cG,f_\cG)$ is
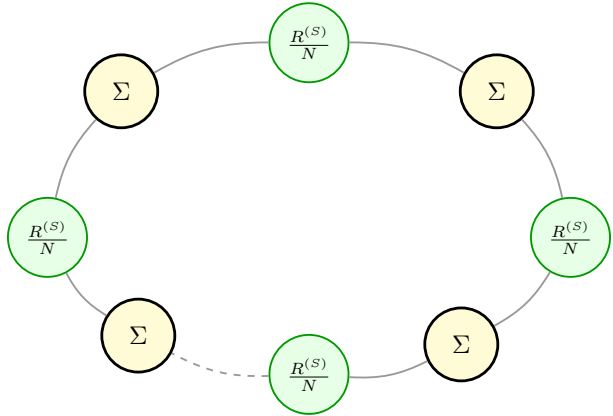
\begin{figure}[H]
    \centering
    \resizebox{0.5\textwidth}{!}{
        \begin{tikzpicture}[
    node distance=2.5cm,
    matrix_node/.style={
        circle, 
        draw=green!60!black, 
        thick, 
        fill=green!10, 
        minimum size=1.1cm,
        font=\small
    },
    sigma_node/.style={
        circle, 
        draw=black, 
        very thick,        
        fill=yellow!20,    
        minimum size=1.1cm,
        font=\bfseries     
    },
    conn/.style={thick, draw=gray!80}
]

    \def\Rx{4cm} 
    \def\Ry{2.5cm} 

    
    \node[matrix_node] (m1) at (90:\Rx\space and \Ry) {$\frac{R^{(S)}}{N}$};
    
    \node[sigma_node] (s1) at (45:\Rx\space and \Ry) {$\Sigma$};
    \node[matrix_node] (m2) at (-10:\Rx\space and \Ry) {$\frac{R^{(S)}}{N}$};
    \node[sigma_node] (s2) at (-55:\Rx\space and \Ry) {$\Sigma$};
    
    \node[sigma_node] (s_last) at (135:\Rx\space and \Ry) {$\Sigma$};
    \node[matrix_node] (m_last) at (190:\Rx\space and \Ry) {$\frac{R^{(S)}}{N}$};

    \node[sigma_node] (s3) at (-130:\Rx\space and \Ry) {$\Sigma$};
    \node[matrix_node] (m_bot) at (-90:\Rx\space and \Ry) {$\frac{R^{(S)}}{N}$};

    \draw[conn] (m1) to[bend left=15] (s1);
    \draw[conn] (s1) to[bend left=15] (m2);
    \draw[conn] (m2) to[bend left=15] (s2);
    \draw[conn] (s2) to[bend left=15] (m_bot);
    
    \draw[conn] (s_last) to[bend left=15] (m1);
    \draw[conn] (m_last) to[bend left=15] (s_last);

    \draw[conn, dashed] (m_bot) to[bend left=15] (s3);
    \draw[conn] (s3) to[bend left=15] (m_last);
    

\end{tikzpicture}
    }
\end{figure}
\noindent The number of vertices with label $N^{-1}R^{(S)}$ in the cycle is $|V_\cG^{m2}|$,
which is at least two since these vertices are type 2. Thus the value is bounded
as
\[\Tr (N^{-1}R^{(S)}\Sigma)^{|V_\cG^{m2}|}
\leq\|\Sigma\|_\op^{|V_\cG^{m2}|}\times(N^{-1}\|R^{(S)}\|_F)^{|V_\cG^{m2}|}\prec\Phi^{|V_\cG^{m2}|}.\]
This implies the lemma, since $|U_\cG|=|V_\cG^r|=0$.
    
(2) Suppose $\cG$ is a chain, with endpoints $u_1,u_2\notin U_\cG$. These
endpoints may be left leaves, right original leaves, or right new leaves,
leading to the following 6 types of representative networks $(\cG,f_\cG)$:
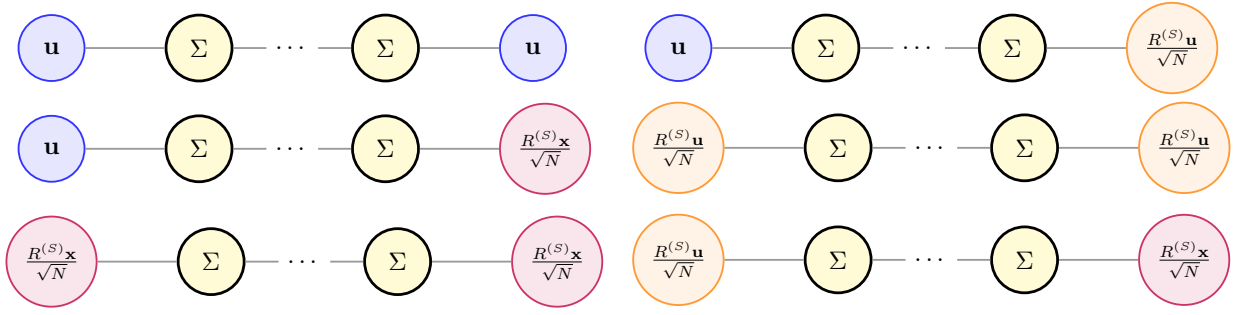
\begin{figure}[H]
    \centering
    \resizebox{1\textwidth}{!}{
        \begin{tikzpicture}[
    node distance=1.2cm,
    leaf_node/.style={circle, draw=blue!80, thick, fill=blue!10, minimum size=1cm},
    sigma_node/.style={circle, draw=black, very thick, fill=yellow!20, minimum size=1cm, font=\bfseries},
    vec_u_node/.style={circle, draw=orange!80, thick, fill=orange!10, minimum size=1.1cm, font=\small},
    vec_x_node/.style={circle, draw=purple!80, thick, fill=purple!10, minimum size=1.1cm, font=\small},
    conn/.style={thick, draw=gray!80}
]


    \node[leaf_node] (l1_1) at (0,0) {$\u$};
    \node[sigma_node] (l1_2) [right=of l1_1] {$\Sigma$};
    \node (l1_dots) [right=0.5cm of l1_2] {$\dots$};
    \node[sigma_node] (l1_3) [right=0.5cm of l1_dots] {$\Sigma$};
    \node[leaf_node] (l1_4) [right=of l1_3] {$\u$};

    \draw[conn] (l1_1) -- (l1_2);
    \draw[conn] (l1_2) -- (l1_dots);
    \draw[conn] (l1_dots) -- (l1_3);
    \draw[conn] (l1_3) -- (l1_4);

    \node[leaf_node] (l2_1) [below=0.5cm of l1_1] {$\u$};
    \node[sigma_node] (l2_2) [right=of l2_1] {$\Sigma$};
    \node (l2_dots) [right=0.5cm of l2_2] {$\dots$};
    \node[sigma_node] (l2_3) [right=0.5cm of l2_dots] {$\Sigma$};
    \node[vec_x_node] (l2_4) [right=of l2_3] {$\frac{R^{(S)}\x}{\sqrt{N}}$};

    \draw[conn] (l2_1) -- (l2_2);
    \draw[conn] (l2_2) -- (l2_dots);
    \draw[conn] (l2_dots) -- (l2_3);
    \draw[conn] (l2_3) -- (l2_4);

    \node[vec_x_node] (l3_1) [below=0.5cm of l2_1] {$\frac{R^{(S)}\x}{\sqrt{N}}$};
    \node[sigma_node] (l3_2) [right=of l3_1] {$\Sigma$};
    \node (l3_dots) [right=0.5cm of l3_2] {$\dots$};
    \node[sigma_node] (l3_3) [right=0.5cm of l3_dots] {$\Sigma$};
    \node[vec_x_node] (l3_4) [right=of l3_3] {$\frac{R^{(S)}\x}{\sqrt{N}}$};

    \draw[conn] (l3_1) -- (l3_2);
    \draw[conn] (l3_2) -- (l3_dots);
    \draw[conn] (l3_dots) -- (l3_3);
    \draw[conn] (l3_3) -- (l3_4);

    
    \node[leaf_node] (r1_1) at (9.5,0) {$\u$};
    \node[sigma_node] (r1_2) [right=of r1_1] {$\Sigma$};
    \node (r1_dots) [right=0.5cm of r1_2] {$\dots$};
    \node[sigma_node] (r1_3) [right=0.5cm of r1_dots] {$\Sigma$};
    \node[vec_u_node] (r1_4) [right=of r1_3] {$\frac{R^{(S)}\u}{\sqrt{N}}$};

    \draw[conn] (r1_1) -- (r1_2);
    \draw[conn] (r1_2) -- (r1_dots);
    \draw[conn] (r1_dots) -- (r1_3);
    \draw[conn] (r1_3) -- (r1_4);

    \node[vec_u_node] (r2_1) [below=0.3cm of r1_1] {$\frac{R^{(S)}\u}{\sqrt{N}}$};
    \node[sigma_node] (r2_2) [right=of r2_1] {$\Sigma$};
    \node (r2_dots) [right=0.5cm of r2_2] {$\dots$};
    \node[sigma_node] (r2_3) [right=0.5cm of r2_dots] {$\Sigma$};
    \node[vec_u_node] (r2_4) [right=of r2_3] {$\frac{R^{(S)}\u}{\sqrt{N}}$};

    \draw[conn] (r2_1) -- (r2_2);
    \draw[conn] (r2_2) -- (r2_dots);
    \draw[conn] (r2_dots) -- (r2_3);
    \draw[conn] (r2_3) -- (r2_4);

    \node[vec_u_node] (r3_1) [below=0.3cm of r2_1] {$\frac{R^{(S)}\u}{\sqrt{N}}$};
    \node[sigma_node] (r3_2) [right=of r3_1] {$\Sigma$};
    \node (r3_dots) [right=0.5cm of r3_2] {$\dots$};
    \node[sigma_node] (r3_3) [right=0.5cm of r3_dots] {$\Sigma$};
    \node[vec_x_node] (r3_4) [right=of r3_3] {$\frac{R^{(S)}\x}{\sqrt{N}}$};

    \draw[conn] (r3_1) -- (r3_2);
    \draw[conn] (r3_2) -- (r3_dots);
    \draw[conn] (r3_dots) -- (r3_3);
    \draw[conn] (r3_3) -- (r3_4);

\end{tikzpicture}
}
\end{figure}
\noindent In each case, we may apply $\|\u\|_2=1$,
$\|\Sigma\|_\op \leq C$, $\|N^{-1}R^{(S)}\|_F \prec \Phi$,
$\|N^{-1/2}R^{(S)}\u\|_2 \prec \Phi$, and
$\|N^{-1/2}R^{(S)}\x\|_2 \prec \Phi$ to get
\[|\mathrm{val}(\cG,f_\cG)|\prec \Phi^{|V_\cG^r|+|V_\cG^{m2}|}.\]
We clarify that in particular, if $|V_\cG^r|+|V_\cG^{m2}|=0$, this is the bound
$|\u^*\Sigma\u|\prec 1$. Here, since $|U_\cG|=0$, the lemma also holds.

(3) Suppose $\cG$ is a chain and exactly one of $u_1,u_2$, say $u_1$, belongs to
$U_\cG$. Then $u_1$ is singular, so its adjacent tensor vertex $t(u_1)$ cannot
be adjacent to any other left leaves, meaning the other neighbor of $t(u_1)$ is
a matrix vertex or a right leaf. If $|V_\cG^r|+|V_\cG^{m2}|=1$, then we have the
cases
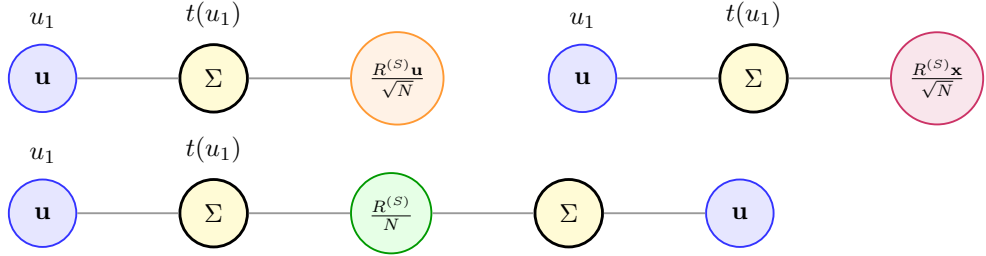
\begin{figure}[H]
    \centering
    \resizebox{0.8\textwidth}{!}{
        \begin{tikzpicture}[
    node distance=1.5cm,
    leaf_node/.style={circle, draw=blue!80, thick, fill=blue!10, minimum size=1cm},
    sigma_node/.style={circle, draw=black, very thick, fill=yellow!20, minimum size=1cm, font=\bfseries},
    matrix_node/.style={circle, draw=green!60!black, thick, fill=green!10, minimum size=1cm, font=\small},
    vec_u_node/.style={circle, draw=orange!80, thick, fill=orange!10, minimum size=1.1cm, font=\small},
    vec_x_node/.style={circle, draw=purple!80, thick, fill=purple!10, minimum size=1.1cm, font=\small},
    conn/.style={thick, draw=gray!80},
    label_text/.style={font=\Large}
]


    \node[leaf_node] (A_u) at (0,0) {$\u$};
    \node[above=0.1cm of A_u] {$u_1$};

    \node[sigma_node] (A_sig) [right=of A_u] {$\Sigma$};
    \node[above=0.1cm of A_sig] {$t(u_1)$};

    \node[vec_u_node] (A_end) [right=of A_sig] {$\frac{R^{(S)}\u}{\sqrt{N}}$};

    \draw[conn] (A_u) -- (A_sig);
    \draw[conn] (A_sig) -- (A_end);

    \begin{scope}[xshift=7cm]
        \node[leaf_node] (B_u) at (1,0) {$\u$};
        \node[above=0.1cm of B_u] {$u_1$};

        \node[sigma_node] (B_sig) [right=of B_u] {$\Sigma$};
        \node[above=0.1cm of B_sig] {$t(u_1)$};

        \node[vec_x_node] (B_end) [right=of B_sig] {$\frac{R^{(S)}\x}{\sqrt{N}}$};

        \draw[conn] (B_u) -- (B_sig);
        \draw[conn] (B_sig) -- (B_end);
    \end{scope}

    \begin{scope}[yshift=-2cm]
        \node[leaf_node] (D_u1) at (0,0) {$\u$};
        \node[above=0.1cm of D_u1] {$u_1$};

        \node[sigma_node] (D_sig1) [right=of D_u1] {$\Sigma$};
        \node[above=0.1cm of D_sig1] {$t(u_1)$};

        \node[matrix_node] (D_mat) [right=of D_sig1] {$\frac{R^{(S)}}{N}$};

        \node[sigma_node] (D_sig2) [right=of D_mat] {$\Sigma$};

        \node[leaf_node] (D_u2) [right=of D_sig2] {$\u$};

        \draw[conn] (D_u1) -- (D_sig1);
        \draw[conn] (D_sig1) -- (D_mat);
        \draw[conn] (D_mat) -- (D_sig2);
        \draw[conn] (D_sig2) -- (D_u2);
    \end{scope}

\end{tikzpicture}
}
\end{figure}
\noindent For these cases, we apply the bounds
\[
    \frac{1}{\sqrt{N}}|\u^* \Sigma R^{(S)}\u|\prec\frac{N^\delta}{\sqrt{N}},\qquad\frac{1}{\sqrt{N}}|\u^*\Sigma R^{(S)}\x|\prec\frac{N^\delta}{\sqrt{N}},
\qquad \frac{1}{N}|\u^*\Sigma R^{(S)}\Sigma\u|
\prec \frac{N^\delta}{N},\]
which imply the lemma since $|U_\cG|=1$ and $|V_\cG^r|+|V_\cG^{m2}|=1$.
If $|V_\cG^r|+|V_\cG^{m2}| \geq 2$, then $t(u_1)$ must be adjacent to a matrix vertex, and we  are left with the following possibilities:
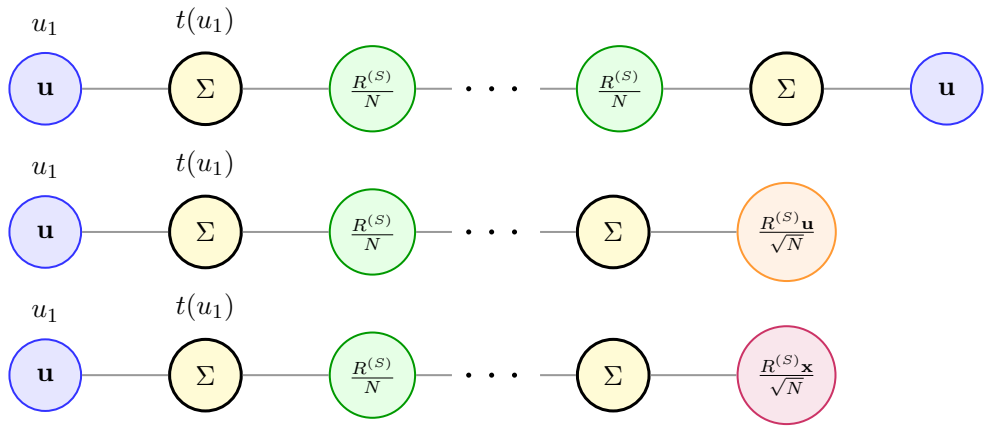
\begin{figure}[H]
    \centering
    \resizebox{0.8\textwidth}{!}{
        \begin{tikzpicture}[
    node distance=1.2cm,
    leaf_node/.style={circle, draw=blue!80, thick, fill=blue!10, minimum size=1cm},
    sigma_node/.style={circle, draw=black, very thick, fill=yellow!20, minimum size=1cm, font=\bfseries},
    matrix_node/.style={circle, draw=green!60!black, thick, fill=green!10, minimum size=1cm, font=\small},
    vec_u_node/.style={circle, draw=orange!80, thick, fill=orange!10, minimum size=1.1cm, font=\small},
    vec_x_node/.style={circle, draw=purple!80, thick, fill=purple!10, minimum size=1.1cm, font=\small},
    conn/.style={thick, draw=gray!80}
]

    \node[leaf_node] (r1_u1) at (0,0) {$\u$};
    \node[above=0.1cm of r1_u1] {$u_1$};
    
    \node[sigma_node] (r1_t1) [right=of r1_u1] {$\Sigma$};
    \node[above=0.1cm of r1_t1] {$t(u_1)$};
    
    \node[matrix_node] (r1_m1) [right=of r1_t1] {$\frac{R^{(S)}}{N}$};
    
    \node (r1_dots) [right=0.5cm of r1_m1] {\huge $\dots$};
    
    \node[matrix_node] (r1_m2) [right=0.5cm of r1_dots] {$\frac{R^{(S)}}{N}$};
    \node[sigma_node] (r1_s2) [right=of r1_m2] {$\Sigma$};
    \node[leaf_node] (r1_u2) [right=of r1_s2] {$\u$};

    \draw[conn] (r1_u1) -- (r1_t1);
    \draw[conn] (r1_t1) -- (r1_m1);
    \draw[conn] (r1_m1) -- (r1_dots);
    \draw[conn] (r1_dots) -- (r1_m2);
    \draw[conn] (r1_m2) -- (r1_s2);
    \draw[conn] (r1_s2) -- (r1_u2);

    \begin{scope}[yshift=-2cm]
        \node[leaf_node] (r2_u1) at (0,0) {$\u$};
        \node[above=0.1cm of r2_u1] {$u_1$};
        
        \node[sigma_node] (r2_t1) [right=of r2_u1] {$\Sigma$};
        \node[above=0.1cm of r2_t1] {$t(u_1)$};
        
        \node[matrix_node] (r2_m1) [right=of r2_t1] {$\frac{R^{(S)}}{N}$};
        
        \node (r2_dots) [right=0.5cm of r2_m1] {\huge $\dots$};
        
        \node[sigma_node] (r2_s2) [right=0.5cm of r2_dots] {$\Sigma$};
        \node[vec_u_node] (r2_vec) [right=of r2_s2] {$\frac{R^{(S)}\u}{\sqrt{N}}$};

        \draw[conn] (r2_u1) -- (r2_t1);
        \draw[conn] (r2_t1) -- (r2_m1);
        \draw[conn] (r2_m1) -- (r2_dots);
        \draw[conn] (r2_dots) -- (r2_s2);
        \draw[conn] (r2_s2) -- (r2_vec);
    \end{scope}

    \begin{scope}[yshift=-4cm]
        \node[leaf_node] (r3_u1) at (0,0) {$\u$};
        \node[above=0.1cm of r3_u1] {$u_1$};
        
        \node[sigma_node] (r3_t1) [right=of r3_u1] {$\Sigma$};
        \node[above=0.1cm of r3_t1] {$t(u_1)$};
        
        \node[matrix_node] (r3_m1) [right=of r3_t1] {$\frac{R^{(S)}}{N}$};
        
        \node (r3_dots) [right=0.5cm of r3_m1] {\huge $\dots$};
        
        \node[sigma_node] (r3_s2) [right=0.5cm of r3_dots] {$\Sigma$};
        \node[vec_x_node] (r3_vec) [right=of r3_s2] {$\frac{R^{(S)}\x}{\sqrt{N}}$};

        \draw[conn] (r3_u1) -- (r3_t1);
        \draw[conn] (r3_t1) -- (r3_m1);
        \draw[conn] (r3_m1) -- (r3_dots);
        \draw[conn] (r3_dots) -- (r3_s2);
        \draw[conn] (r3_s2) -- (r3_vec);
    \end{scope}

\end{tikzpicture}
}
\end{figure}
\noindent In these cases, we may contract the labels $\u$ and $\Sigma$ on $u_1$
and $t(u_1)$ with the first $N^{-1}R^{(S)}$, to get
\[
    |\mathrm{val}(\cG,f_\cG)| \prec \frac{1}{\sqrt{N}}\Phi^{|V_\cG^r|+|V_\cG^{m2}|}.
\]
For example,
\begin{align*}
|\u^*\Sigma (N^{-1}R^{(S)}\Sigma)^{|V_\cG^{m2}|}(N^{-1/2}R^{(S)}\u)|
&\prec \frac{1}{\sqrt{N}}\|N^{-1/2}\u^*\Sigma R^{(S)}\|_2
\times \|N^{-1}R^{(S)}\|_F^{|V_\cG^{m2}|-1} \times \|N^{-1/2}R^{(S)}\u\|_2\\
&\prec \frac{1}{\sqrt{N}}\Phi^{|V_\cG^r|+|V_\cG^{m2}|}.
\end{align*}
This is stronger than the lemma, because $|U_\cG|=1$, $\Phi \leq 1$, and $\delta
\geq 0$.

(4) Finally, suppose both endpoints $u_1,u_2\in U_{\cG}$. Then $|V_\cG^r|=0$.
If $|V_\cG^{m2}|=1$, then a representative case is
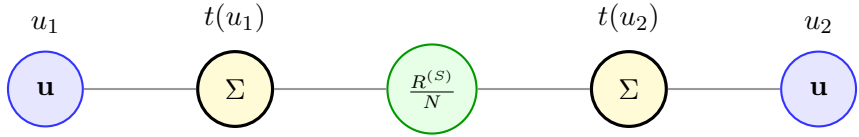
\begin{figure}[H]
    \centering
    \resizebox{0.7\textwidth}{!}{
        \begin{tikzpicture}[
    node distance=1.5cm,
    leaf_node/.style={circle, draw=blue!80, thick, fill=blue!10, minimum size=1cm},
    sigma_node/.style={circle, draw=black, very thick, fill=yellow!20, minimum size=1cm, font=\bfseries},
    matrix_node/.style={circle, draw=green!60!black, thick, fill=green!10, minimum size=1cm, font=\small},
    conn/.style={thick, draw=gray!80}
]

    \node[leaf_node] (u1) at (0,0) {$\u$};
    \node[above=0.1cm of u1] {$u_1$};

    \node[sigma_node] (t1) [right=of u1] {$\Sigma$};
    \node[above=0.1cm of t1] {$t(u_1)$};

    \node[matrix_node] (m1) [right=of t1] {$\frac{R^{(S)}}{N}$};

    \node[sigma_node] (t2) [right=of m1] {$\Sigma$};
    \node[above=0.1cm of t2] {$t(u_2)$};

    \node[leaf_node] (u2) [right=of t2] {$\u$};
    \node[above=0.1cm of u2] {$u_2$};

    \draw[conn] (u1) -- (t1);
    \draw[conn] (t1) -- (m1);
    \draw[conn] (m1) -- (t2);
    \draw[conn] (t2) -- (u2);

\end{tikzpicture}
}
\end{figure}
\noindent We apply the bound
\[
    \frac{1}{N}|\u^*\Sigma R^{(S)}\Sigma\u|\prec\frac{N^\delta}{N}
\prec \left(\frac{N^\delta}{\sqrt{N}}\right)^2,
\]
which implies the lemma since $|U_\cG|=2$ and $\Phi \leq 1$.
If $|V_\cG^{m2}| \geq 2$, then since $u_1,u_2$ are singular left leaves, their
adjacent tensor vertices $t(u_1),t(u_2)$ must be adjacent to matrix vertices.
A representative case is
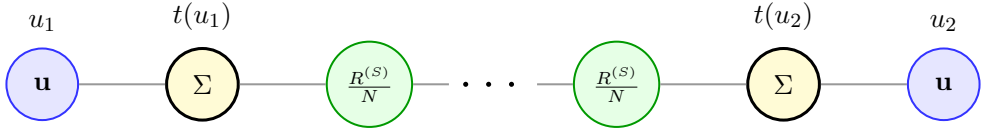
\begin{figure}[H]
    \centering
    \resizebox{0.8\textwidth}{!}{
        \begin{tikzpicture}[
    node distance=1.2cm,
    leaf_node/.style={circle, draw=blue!80, thick, fill=blue!10, minimum size=1cm},
    sigma_node/.style={circle, draw=black, very thick, fill=yellow!20, minimum size=1cm, font=\bfseries},
    matrix_node/.style={circle, draw=green!60!black, thick, fill=green!10, minimum size=1cm, font=\small},
    conn/.style={thick, draw=gray!80}
]

    \node[leaf_node] (u1) at (0,0) {$\u$};
    \node[above=0.1cm of u1] {$u_1$};

    \node[sigma_node] (t1) [right=of u1] {$\Sigma$};
    \node[above=0.1cm of t1] {$t(u_1)$};

    \node[matrix_node] (m1) [right=of t1] {$\frac{R^{(S)}}{N}$};

    \node (dots) [right=0.5cm of m1] {\huge $\dots$};

    \node[matrix_node] (m2) [right=0.5cm of dots] {$\frac{R^{(S)}}{N}$};

    \node[sigma_node] (t2) [right=of m2] {$\Sigma$};
    \node[above=0.1cm of t2] {$t(u_2)$};

    \node[leaf_node] (u2) [right=of t2] {$\u$};
    \node[above=0.1cm of u2] {$u_2$};

    \draw[conn] (u1) -- (t1);
    \draw[conn] (t1) -- (m1);
    \draw[conn] (m1) -- (dots);
    \draw[conn] (dots) -- (m2);
    \draw[conn] (m2) -- (t2);
    \draw[conn] (t2) -- (u2);

\end{tikzpicture}
}
\end{figure}
\noindent for which we may bound
\[
    |\u^*\Sigma (N^{-1}R^{(S)}\Sigma)^{|V_\cG^{m2}|}\u|\prec\frac{1}{N}
\|N^{-1/2}\u^*\Sigma R^{(S)}\|_2
\times \|N^{-1}R^{(S)}\|_F^{|V_\cG^{m2}|-2}
\times \|N^{-1/2}R^{(S)}\Sigma\u\|_2
\prec \frac{1}{N}\Phi^{|V_\cG^{m2}|}.
\]
This again implies the lemma since $|U_\cG|=2$, $\Phi \leq 1$, and $\delta \geq
0$. This completes the proof.
\end{proof}

\begin{proof}[Proof of Lemma \ref{lem:fluctuation_averaging_lowrank2}(a)]
We recall from \eqref{eq:YQbound} that, for any fixed constants $L \geq 1$
and $D>0$, we have
\begin{equation}\label{eq:YQbound2}
    \E\left|\frac{1}{\sqrt{N}}\sum_{i=1}^N\cY_i^{(i)}\cQ_i^{(i)}\right|^{2L}\prec
\max_{\substack{i_1,\ldots,i_{2L} \in [N]\\
q_1\in\cM_{i_1,S},\ldots,q_{2L}\in\cM_{i_{2L},S}}}\left\{N^{|S|-L}\E|\E_S
q|\right\}+N^{-\tau'(D+1)+L}.
\end{equation}
where $S=\{i_1,\ldots,i_{2L}\}$ is the set of distinct indices in
$i_1,\ldots,i_{2L}$, the monomials $q_l \in \cM_{i_l,S}$ are those arising in
the expansion of $\cY_{i_l}^{(i_l)}\cQ_{i_l}^{(i_l)}$ via
Lemma \ref{lem:y_i q_i expansion}, and $q=\prod_{l=1}^L q_l \prod_{l=L+1}^{2L}
\bar q_l$. By Lemma \ref{lem:tensor network of sample covariance} and the bound
$|\cC^{(S)}| \prec 1$, there exists a family $\cN$ of
$(S,2L)$-valid tensor networks for which
\begin{align*}
    |\E_S q|&\prec \sum_{(\mathcal{G},f_\cG)\in\cN}|\mathrm{val}(\mathcal{G},f_\cG)|
    \prec\max_{(\mathcal{G},f_\cG)\in\cN}|\mathrm{val}(\mathcal{G},f_\cG)|,
\end{align*}
where the second inequality applies $|\cN|\leq C$. It remains to bound
$\mathrm{val}(\mathcal{G},f_\cG)$ for each $(\cG,f_\cG) \in \cN$.

The proof for each $(\cG, f_\cG) \in \cN$ proceeds by reducing $(\cG,f_\cG)$ via
a sequence of four steps. In
each of the first three steps, we sequentially apply either
Corollary~\ref{lem:remove tensor weak} or Lemma~\ref{lem:remove tensor strong}
to remove a single tensor vertex $v \in V_\cG^t$ with $\deg(v) \geq 3$.
After removing all such vertices, the graph is reduced to chains and cycles, and
we apply Lemma~\ref{lem:pairing case} in the fourth step to complete the proof.
Letting $U_\cG^1,U_\cG^2 \subseteq V_\cG^{sl}$ be the disjoint
subsets of singular left leaves in Lemma \ref{lem:tensor network of sample
covariance} (which are fixed throughout the below process),
the four steps are as follows:
\begin{enumerate}
    \item While there exists a singular left leaf $u \in U_\cG^1\cup U_\cG^2$
incident to a tensor vertex $v=t(u) \in V_\cG^t$ such that $\deg(v) \geq 3$ and
either $m1(v)\geq 1$ or $r(v)+m1(v)=0$: Remove $v$, and apply Lemma \ref{lem:remove tensor strong}(a) to bound
    \[
        |\mathrm{val}(\cG,f_\cG)|\prec \frac{N^\delta}{\sqrt{N}}\Phi^{r(v)}
\max_{f_{G'}\in\cF_{\cG'}} |\mathrm{val}(\cG',f_{\cG'})|.
    \]
    \item While there exists a singular left leaf $u \in U_\cG^1$ incident to a
tensor vertex $v \in V_\cG^t$ such that $\deg(v) \geq 3$ and $r(v) \geq
1$: Remove $v$, and apply Lemma \ref{lem:remove tensor strong}(b) to bound
    \[
        |\mathrm{val}(\cG,f_\cG)| \prec \frac{N^\delta}{\sqrt{N}}\Phi^{r(v)-1}
\max_{f_{G'}\in\cF_{\cG'}} |\mathrm{val}(\cG',f_{\cG'})|.
    \]
    \item While there exists a tensor vertex $v\in V_\cG^t$ such that $\deg(v) \geq 3$: Remove $v$, and apply Lemma \ref{lem:remove tensor weak} to bound
    \[
        |\mathrm{val}(\cG,f_\cG)| \prec \Phi^{r(v)} \max_{f_{G'}\in\cF_{\cG'}}
|\mathrm{val}(\cG',f_{\cG'})|.
    \]
    \item At this point, the remaining network $(\cG,f_\cG)$ is such that all
$v \in V_\cG^t$ satisfy $\deg(v)=2$. Let
$|V_\cG^{ro}(4)|$, $|V_\cG^{rn}(4)|$, $|V_\cG^{m2}(4)|$, and $|U_\cG^1(4)|$
denote the numbers of right original leaves, right new leaves,
(type 2) matrix vertices, and singular left leaves of $U_\cG^1$ in $(\cG,f_\cG)$.
We apply Lemma \ref{lem:pairing case} to bound
    \[
        |\mathrm{val}(\cG,f_\cG)|\prec\left(\frac{N^\delta}{\sqrt{N}}\right)^{|U_\cG^1(4)|}\times\Phi^{|V_\cG^{ro}(4)|+|V_\cG^{rn}(4)|+|V_\cG^{m2}(4)|-|U_\cG^1(4)|}.
    \]
\end{enumerate}

We clarify that in the 4 inequalities above, $(\cG,f_\cG)$ on the left side
denotes the modified network after having applied all previous removals,
$r(v)$ on the right side denotes the number of right leaves
adjacent to $v$ in this modified network $(\cG,f_\cG)$, and
$(\cG',\cF_{\cG'})$ denotes the family of networks obtained by further
removing $v$ from $(\cG,f_\cG)$ according to
Definition \ref{def:removev}. Within Steps 1--3,
if there are multiple vertices that may be removed,
we may choose one to remove arbitrarily, but we do
this sequentially as each removal may increase $r(v')$ for the
remaining vertices $v' \in V_\cG^t$. For example, suppose two vertices $u_1,u_2
\in U_\cG^1$ with adjacent vertices $v_1=t(u_1)$ and $v_2=t(u_2)$ are
such that $\deg(v_1) \geq 3$ and $\deg(v_2) \geq 3$,
$r(v_1)+m1(v_1)=0$ and $r(v_2)+m1(v_2)=0$, and
$v_1,v_2$ are adjacent to a common type 2 matrix vertex $u \in V_\cG^m$. 
Pictorially, we have
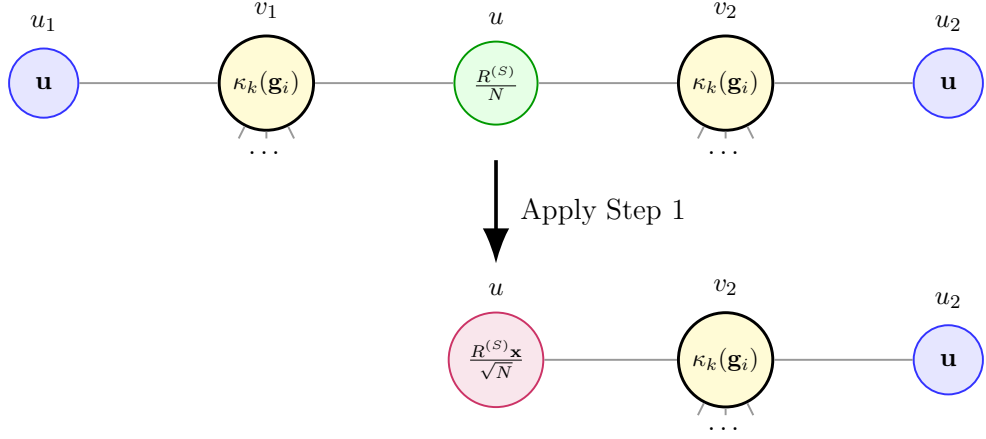
\begin{figure}[H]
    \centering
    \resizebox{0.8\textwidth}{!}{
        \begin{tikzpicture}[
    node distance=2.0cm,
    leaf_node/.style={circle, draw=blue!80, thick, fill=blue!10, minimum size=1cm},
    tensor_node/.style={circle, draw=black, very thick, fill=yellow!20, minimum size=1.2cm, font=\small\bfseries},
    matrix_node/.style={circle, draw=green!60!black, thick, fill=green!10, minimum size=1.2cm, font=\small},
    vec_x_node/.style={circle, draw=purple!80, thick, fill=purple!10, minimum size=1.2cm, font=\small},
    conn/.style={thick, draw=gray!80},
    leg/.style={thick, draw=gray!80}
]

    
    \node[leaf_node] (u1) at (0,0) {$\u$};
    \node[above=0.1cm of u1] {$u_1$};

    \node[tensor_node] (t1) [right=of u1] {$\kappa_k(\g_i)$};
    \node[above=0.1cm of t1] {$v_1$};
    \draw[leg] (t1) -- +(-0.4, -0.8);
    \draw[leg] (t1) -- +(0, -0.8);
    \draw[leg] (t1) -- +(0.4, -0.8);
    \node at ($(t1)+(0,-1)$) {$\dots$};

    \node[matrix_node] (mat) [right=of t1] {$\frac{R^{(S)}}{N}$};
    \node[above=0.1cm of mat] {$u$}; 

    \node[tensor_node] (t2) [right=of mat] {$\kappa_k(\g_i)$};
    \node[above=0.1cm of t2] {$v_2$};
    \draw[leg] (t2) -- +(-0.4, -0.8);
    \draw[leg] (t2) -- +(0, -0.8);
    \draw[leg] (t2) -- +(0.4, -0.8);
    \node at ($(t2)+(0,-1)$) {$\dots$};

    \node[leaf_node] (u2) [right=of t2] {$\u$};
    \node[above=0.1cm of u2] {$u_2$};

    \draw[conn] (u1) -- (t1);
    \draw[conn] (t1) -- (mat);
    \draw[conn] (mat) -- (t2);
    \draw[conn] (t2) -- (u2);

    \draw[-{Latex[scale=1.5]}, line width=1.5pt] ($(mat.south)+(0,-0.5)$) -- ++(0,-1.5) 
        node[midway, right, xshift=0.2cm, font=\large] {Apply Step 1};

    
    \node[vec_x_node] (new_vec) at ($(mat)+(0,-4)$) {$\frac{R^{(S)}\x}{\sqrt{N}}$};
    \node[above=0.1cm of new_vec] {$u$};

    \node[tensor_node] (t2_b) at ($(t2)+(0,-4)$) {$\kappa_k(\g_i)$};
    \node[above=0.1cm of t2_b] {$v_2$};
    \draw[leg] (t2_b) -- +(-0.4, -0.8);
    \draw[leg] (t2_b) -- +(0, -0.8);
    \draw[leg] (t2_b) -- +(0.4, -0.8);
    \node at ($(t2_b)+(0,-1)$) {$\dots$};

    \node[leaf_node] (u2_b) at ($(u2)+(0,-4)$) {$\u$};
    \node[above=0.1cm of u2_b] {$u_2$};

    \draw[conn] (new_vec) -- (t2_b);
    \draw[conn] (t2_b) -- (u2_b);

\end{tikzpicture}
}
\end{figure}
\noindent Then Step 1 may first remove either $v_1$ or $v_2$.
Upon removing, say, $v_1$, the adjacent matrix vertex $u$
is replaced by a right new leaf according to
Definition \ref{def:removev}. Then $r(v_2) \geq 1$ in the
resulting graph, implying that $v_2$ is no longer a candidate for removal in
Step 1, and instead will be removed in Step 2.

Since the removal of any vertex $v \in V_\cG^t$ does not change $\deg(v')$ or
$m_1(v')$ and can only increase $r(v')$ for each remaining
vertex $v' \in V_\cG^t$, the above procedure is forward-progressing in the sense
that each removal of a tensor vertex in Steps 1--3 cannot result in a tensor
vertex being eligible for removal by a previous step.

For each Step $k \in \{1,2,3\}$, let $|V_\cG^{ro}(k)|$, $|V_\cG^{rn}(k)|$,
$|U_\cG^1(k)|$, and $|U_\cG^2(k)|$ denote the total numbers of right original
leaves, right new leaves, singular left leaves of $U_\cG^1$, and singular
left leaves of $U_\cG^2$ that are removed in Step $k$.
Importantly, all vertices $u \in U_\cG^2$ have $m1(t(u)) \geq 1$ by Lemma \ref{lem:tensor network of sample covariance}(d), so after all applications of Step 1, the network has
no remaining vertices in $U_\cG^2$. Thus,
$|U_\cG^2(2)|=|U_\cG^2(3)|=0$, and the number of total applications of Step 2
is $|U_\cG^1(2)|$. Also, after all applications of Steps 1 and 2, no singular
left leaf $u \in U_\cG^1$ has $\deg(t(u)) \geq 3$, so Step 3 does not
remove any vertices of $U_\cG^1$, i.e.\ $|U_\cG^1(3)|=0$.

Thus, letting $\mathrm{val}(\cG,f_\cG)$ denote the value
of the starting network, the above steps yield the bound
\begin{align*}
|\mathrm{val}(\cG,f_\cG)|&\prec
\underbrace{\Big(\frac{N^\delta}{\sqrt{N}}\Big)^{U_\cG^1(1)+U_\cG^2(1)}
\Phi^{|V_\cG^{ro}(1)|+|V_\cG^{rn}(1)|}}_{\text{Step 1}}
\times
\underbrace{\Big(\frac{N^\delta}{\sqrt{N}}\Big)^{U_\cG^1(2)}
\Phi^{|V_\cG^{ro}(2)|+|V_\cG^{rn}(2)|-|U_\cG^1(2)|}}_{\text{Step 2}}\\
&\hspace{1in}\times
\underbrace{\Phi^{|V_\cG^{ro}(3)|+|V_\cG^{rn}(3)|}}_{\text{Step 3}}
\times
\underbrace{\Big(\frac{N^\delta}{\sqrt{N}}\Big)^{|U_\cG^1(4)|}\Phi^{|V_\cG^{ro}(4)|+|V_\cG^{rn}(4)|+|V_\cG^{m2}(4)|-|U_\cG^1(4)|}}_{\text{Step
4}}.
\end{align*}
Furthermore,
letting $|V_\cG^{ro}|$, $|V_\cG^{rn}|$, $|V_\cG^{m2}|$, $|U_\cG^1|$, and
$|U_\cG^2|$ denote these quantities for the original network $(\cG,f_\cG)$, we
observe the following:
\begin{enumerate}
    \item The singular left leaves removed in Steps 1--2 and those remaining
in Step 4 form a partition of $U_\cG^1 \cup U_\cG^2$, so
    \[|U_\cG^1(1)|+|U_\cG^2(1)|+|U_\cG^1(2)|+|U_\cG^1(4)| =
|U_\cG^1|+|U_\cG^2|.\]
\item Similarly, the right original leaves removed in Steps 1--4 form a
partition of $V_\cG^{ro}$, so
\[\sum_{k=1}^4 |V_\cG^{ro}(k)|=|V_\cG^{ro}|.\]
    \item By the removal process of Definition \ref{def:removev}, every right
new leaf removed in Steps 1--4 must have been a type 2 matrix vertex in the
original network $(\cG,f_\cG)$. Together with the type 2 matrix vertices
remaining in Step 4, this accounts for all type 2 matrix vertices of 
$(\cG,f_\cG)$. Thus
    \[
        \sum_{k=1}^4 |V_\cG^{rn}(k)| + |V_\cG^{m2}(4)| = |V_\cG^{m2}|.
    \]

    \item It is clear that we must have $|U_\cG^1(2)|+|U_\cG^1(4)|\leq|U_\cG^1|$.
\end{enumerate}
Thus, we have the bound
\begin{align*}
    |\mathrm{val}(\cG,f_\cG)|&\prec\left(\frac{N^\delta}{\sqrt{N}}\right)^{|U_\cG^1|+|U_\cG^2|}\times\Phi^{|V_\cG^{ro}|+|V_\cG^{m2}|-|U_\cG^1|}\\
&\overset{(a)}{\leq} \left(\frac{N^\delta}{\sqrt{N}}\right)^{|K|}\times\Phi^{|V_\cG^{ro}|}
\overset{(b)}{\leq}N^{-|K|/2}\times(N^\delta\Phi)^{2L}
\overset{(c)}{\leq}N^{-|S|+L}\times(N^\delta\Phi)^{2L}
\end{align*}
where (a) applies $|U_\cG^1|+|U_\cG^2|=|K|$ and
$|V_\cG^{m2}| \geq |U_\cG^1|$ from
Lemma~\ref{lem:tensor network of sample covariance}(b--c),
(b) applies $|K| \leq 2L$ and $|V_\cG^{ro}|=2L$ by
Lemma~\ref{lem:tensor network of sample covariance}(a), and (c) applies the
bound $(2L-|K|)/2+|K| \geq |S|$.
Applying this bound to \eqref{eq:YQbound2}
and choosing $D>0$ large enough shows
\[\E\left|\sum_{i=1}^N N^{-1/2} \cY_i^{(i)} \cQ_i^{(i)}\right|^{2L} \prec
(N^\delta \Phi)^{2L},\]
and the desired bound follows from Markov's inequality.
\end{proof}

\subsubsection{Proof of Lemma \ref{lem:fluctuation_averaging_lowrank2}(b)}

The following lemma is similar to Lemma \ref{lem:tensor network of sample
covariance}.

\begin{lemma}\label{lem:tensor network of gram matrix}
Fix any integer $L\geq 1$. For any $i_1,\ldots,i_{2L},j_1,\ldots j_{2L}\in[N]$
where $i_l \neq j_l$, denote $S=\{i_1,\ldots,i_{2L}\}\cup\{j_1,\ldots,j_{2L}\}$ the set of distinct indices. Let $q_1,\ldots,q_{2L}$ be a sequence of monomials
satisfying the conditions of Lemma \ref{lem:B_ij expansion}, where each
$q_l$ is a monomial in the expansion of
$\cB_{i_lj_l}^{(i_lj_l)}\cQ_{i_l}^{(i_l)}\cQ_{j_l}^{(i_lj_l)}$. Let
$q=\prod_{l=1}^L \prod_{l=L+1}^{2L} \bar q_l$. Then there exists a constant
$C \equiv C(L)>0$, an exponent $m \in [0,C]$, and a set $\cN$ of $(S,4L)$-valid
tensor networks with $|\cN|\leq C$ such that
\[
    |\E_Sq|=|\cC^{(S)}|^m \cdot
\sum_{(\cG,f_\cG)\in\cN}|\mathrm{val}(\cG,f_\cG)|
\]
Moreover, for all $(\mathcal{G}, f_\cG)\in\cN$, there exist disjoint subsets of
tensor vertices $U_\cG^{1},U_\cG^{2}\subseteq V_\cG^{t}$ such that the following
holds:
\begin{enumerate}[label=(\alph*)]
    \item We have $|V_\cG^l|=|V_\cG^r|=0$. (I.e., there are no leaves, and the
only vertices are tensor vertices and matrix vertices.)
    \item Let $K \subseteq S$ be the subset of indices that appear exactly once
in the list $(i_1,\ldots,i_{2L},j_1,\ldots,j_{2L})$. Then $|K|=|U_\cG^{1}|+|U_\cG^{2}|$.
    \item For all $v\in U_\cG^{1}\cup U_\cG^{2}$, $\deg(v)$ is odd and $\deg(v)
\geq 3$.
    \item We have $|V_\cG^{m2}|\geq 2L+|U_\cG^{1}|$.
    \item For all $v\in U_\cG^{2}$, $m1(v) \geq 1$.
\end{enumerate}
\end{lemma}
\begin{proof}
The proof is similar to Lemma~\ref{lem:tensor network of sample covariance},
so we omit some details. The same arguments as in
Lemma~\ref{lem:tensor network of sample covariance} show that
\[
|\E_S q|=|\cC^{(S)}|^m \cdot \sum_{(\cG, f_\cG) \in \cN} |\mathrm{val}(\cG, f_\cG)|
\]
for some $m \in [0,C(L)]$ and family $\cN$ of $(S,4L)$-valid tensor networks
where $|\cN| \leq C(L)$. These networks have no leaves since the terms
$\cB_{ij}^{(S)}$ and $\cP_i^{(S)}$ do not involve the vector $\u$.

Let $K$ be the set of indices appearing exactly once in
$(i_1,\ldots,i_{2L},j_1,\ldots,j_{2L})$. For each $i_l \in K$ (or $j_l \in K$),
Lemma \ref{lem:B_ij expansion}(b) ensures that $\g_{i_l}$ (resp.\ $\g_{j_l}$)
appears an odd number of times in $q_l$ and an even number of times in
$\prod_{l' \neq l} q_{l'}$.
Thus, in the moment-cumulant expansion, there must exist some cumulant tensor
$\kappa_k(\g_{i_l})$ (resp.\ $\kappa_k(\g_{j_l})$) with $k$ odd. Moreover, since
$\E\g_{i_l}=\E\g_{j_l}=0$, we must have $k \geq 3$. Let $U_\cG \subseteq
V_\cG^t$ be a choice of one odd-degree tensor vertex representing an odd-order
cumulant for each $i_l \in K$ or $j_l \in K$; we denote this choice by $v(i_l)
\in U_\cG$ or $v(j_l) \in U_\cG$. It follows that $|U_\cG|=|K|$.

We now partition $U_\cG$ into $U_\cG^{1}$ and $U_\cG^{2}$ satisfying the
conditions (d) and (e) of the lemma. Let $U_\cG^1$ be the set of $v(i_l)$ or
$v(j_l)$ in $U_\cG$ for which either
\begin{enumerate}
    \item $i_l$ (resp.\ $j_l$) appears as a lower index in some $\{q_{l'}:l'
\neq l\}$, or
\item $i_l$ (resp.\ $j_l$) does not appear as a lower index in any
$\{q_{l'}:l' \neq l\}$, and $v(i_l)$ (resp.\ $v(j_l)$) has at least three type 2
matrix neighbors.
\end{enumerate}
To bound $|U_\cG^1|$, let 
$m_B^{**}(l),m_B^{*1}(l),m_B^{*2}(l),m_B(l)$ denote the values of
$m_B^{**},m_B^{*1},m_B^{*2},m_B$ for $q_l$ as defined in
Lemma \ref{lem:B_ij expansion}, and set
\[m_B^{**}=\sum_{l=1}^{2L} m_B^{**}(l),
\qquad m_B^*=\sum_{l=1}^{2L} m_B^{*1}(l)+m_B^{*2}(l),
\qquad m_B=\sum_{l=1}^{2L} m_B(l).\]
Each $i_l \in K$ (or $j_l \in K$) is distinct from $\{i_{l'},j_{l'}:l' \neq l\}$,
so Lemma \ref{lem:B_ij expansion}(c) ensures that the
number of vertices $v(i_l),v(j_l) \in U_\cG$ satisfying condition (1.)
is at most $m_B+\frac{1}{2}m_B^*$.

For condition (2.), denote further
$m_P^{*1}(l)$ as the number of factors of $\cP_{i_l}^{(S)}$ appearing in $q_l$
for which its copy of $N^{-1}R^{(S)}$ labels a type 2 matrix vertex.
Define similarly $m_P^{*2}(l)$ as the number of such factors of
$\cP_{j_l}^{(S)}$ in $q_l$, and set
\[E(l)=m_B^{**}(l)+\frac{1}{2}(m_B^{*1}(l)+m_B^{*2}(l))
+m_P^{*1}(l)+m_P^{*2}(l)-1.\]
Consider any $l \in \{1,\ldots,2L\}$.
Suppose first that both $i_l,j_l \in K$, and $v(i_l),v(j_l) \in U_\cG$
and satisfy condition (2.).
Then either (i) $m_B^{**}(l) \geq 3$; (ii) $m_B^{**}(l)=2$ in
which case $m_B^{*1}(l),m_B^{*2}(l) \geq 1$ by Lemma
\ref{lem:B_ij expansion}(b); (iii) $m_B^{**}(l)=1$ in which case
$m_B^{*1}(l)+m_P^{*1}(l) \geq 2$ 
and $m_B^{*2}(l)+m_P^{*2}(l) \geq 2$ since $v(i_l),v(j_l)$ both have at least
three type 2 matrix neighbors; or (iv) $m_B^{**}(l)=0$ in which
case $m_B^{*1}(l)+m_P^{*1}(l) \geq 3$ 
and $m_B^{*2}(l)+m_P^{*2}(l) \geq 3$. In all four cases (i--iv),
we verify that $E(l) \geq 2$.
Now suppose that only $v(i_l) \in U_\cG$ satisfies condition (2.) (so
either $j_l \notin K$ or $v(j_l) \in U_\cG$ does not satisfy condition (2.)).
Then either (i) $m_B^{**}(l) \geq 2$;
(ii) $m_B^{**}(l)=1$ and $m_B^{*1}(l)+m_P^{*1}(l) \geq 2$;
or (iii) $m_B^{**}(l)=0$ and $m_B^{*1}(l)+m_P^{*1}(l) \geq 3$.
In case (iii), we have
also $m_B^{*2}(l) \geq 1$ since $m_B^{*2}(l)$ is odd by Lemma
\ref{lem:B_ij expansion}(b). Thus in all three cases
(i--iii), we verify that $E(l) \geq 1$. Finally, we have $E(l) \geq 0$ for every
$l=1,\ldots,2L$ by Lemma \ref{lem:B_ij expansion}(a). Thus, we conclude that the
number of vertices $v(i_l),v(j_l) \in U_\cG$ satisfying condition (2.) is at
most $\sum_{l=1}^{2L} E(l)$, and so
\[|U_\cG^1| \leq \underbrace{m_B+\frac{1}{2}m_B^*}_{\text{condition (1.)}}
+\underbrace{\sum_{l=1}^{2L} E(l)}_{\text{condition (2.)}}
\leq m_B+m_B^*+m_B^{**}+m_P^*-2L,\]
where we have set also $m_P^*=\sum_{l=1}^{2L} m_P^{*1}(l)+m_P^{*2}(l)$. Here
$m_B+m_B^*+m_B^{**}+m_P^* \leq |V_\cG^{m2}|$, the total number of type 2 matrix
vertices, showing statement (d) of the lemma.

Finally, each remaining vertex $v \in U_\cG^2=U_\cG \setminus U_\cG^1$ satisfies
$\deg(v) \geq 3$ and $m2(v) \leq 2$, implying that $m1(v) \geq 1$. This shows
(e) and completes the proof.
\end{proof}

\begin{proof}[Proof of Lemma \ref{lem:fluctuation_averaging_lowrank2}(b)]
Recall from \eqref{eq:vvBQQval} that, for any fixed constants $L \geq 1$ and
$D>0$,
\begin{equation}\label{eq:vvBQQval2}
\E\left|\sum_{i\neq
j}\bar{\v}(i)\v(j)\cB_{ij}^{(ij)}\cQ_i^{(i)}\cQ_j^{(ij)}\right|^{2L}
\prec \max_{\substack{i_1,\ldots,i_{2L},j_1,\ldots,j_{2L} \in [N]\\
q_1 \in \cM_{i_1j_1,S},\ldots,q_{2L} \in \cM_{i_{2L}j_{2L},S}}}
\Big\{N^{|K|/2}\E|\E_S q|\Big\}+\Oprec(N^{-\tau'(D+1)+2L}).
\end{equation}
By Lemma \ref{lem:tensor network of gram matrix}, there
exists a family $\cN$ of $(S,4L)$-valid tensor networks $(\cG,f_\cG)$
for which
\[|\E_S q| \prec \max_{(\cG,f_\cG) \in \cN} |\mathrm{val}(\cG,f_\cG)|.\]
Similar to the proof of Lemma~\ref{lem:fluctuation_averaging_lowrank2}(a),
we reduce $(\cG, f_\cG)$ in a sequence of
four steps. These steps will successively remove tensor vertices from the
network, introducing right new leaves and also eigen-leaves to
the network.
Letting $U^1_\cG,U_\cG^2 \subseteq V_\cG^t$ be the subsets of tensor
vertices in Lemma~\ref{lem:tensor network of gram matrix}, these steps are:
\begin{enumerate}
    \item While there exists a tensor vertex $v \in U_\cG^2$: Remove $v$ and
apply Lemma~\ref{lem:remove tensor strong}(a) to bound
    \[
        |\mathrm{val}(\cG,f_\cG)|\prec\frac{N^\delta}{\sqrt{N}}\Phi^{r(v)}\max_{f_{\cG'}
\in \cF_{\cG'}}|\mathrm{val}(\cG',f_{\cG'})|.
    \]
    \item While there exists a tensor vertex $v \in U_\cG^1$:
    \begin{enumerate}[label=(\alph*)]
        \item If there exists such a vertex with either $m1(v) \geq 1$,
$r(v) \geq 1$, or both $e(v)=1$ and $r(v)+m1(v)=0$, remove $v$ and
apply either Lemma~\ref{lem:remove tensor strong}(a) or (b) to bound
        \[
            |\mathrm{val}(\cG,f_\cG)|\prec\frac{1}{\sqrt{N}}\Phi^{r(v)-1}
\max_{f_{\cG'} \in \cF_{\cG'}}|\mathrm{val}(\cG',f_{\cG'})|.
        \]
(Lemma \ref{lem:remove tensor strong}(a) would give a stronger bound with
$r(v)$ in place of $r(v)-1$, and we weaken this to $r(v)-1$ above.)
        \item If all vertices $v \in U_\cG^1$ satisfy
$e(v)+r(v)+m1(v)=0$, then since $\deg(v) \geq 3$, each such $v$ is adjacent to
at least three type 2 matrix vertices. Pick one of them, say $u \in V_\cG^{m2}$.
Decomposing the label $N^{-1}R^{(S)}$ of $u$ as
$N^{-1}\sum_\alpha |\lambda_\alpha-z|^{-1}\v_\alpha\v_\alpha^*$,
let $(\cG'',\cF_{\cG''})$ denote the family of networks obtained
by replacing $u$ with two eigen-leaves having a common
label in $\{\v_\alpha:\alpha \in [N]\}$. Pictorially, this
corresponds to the following operation:
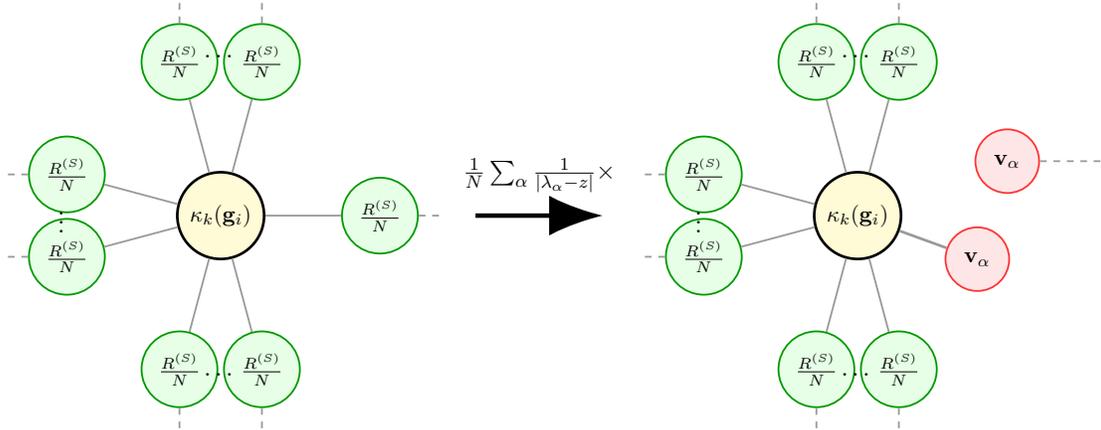
\begin{figure}[H]
    \centering
    \resizebox{0.9\textwidth}{!}{
        \begin{tikzpicture}[
    tensor_node/.style={circle, draw=black, very thick, fill=yellow!20, minimum size=1.3cm, font=\small\bfseries},
    matrix_node/.style={circle, draw=green!60!black, thick, fill=green!10, minimum size=1cm, font=\small},
    eigen_node/.style={circle, draw=red!80, thick, fill=red!10, minimum size=1cm, font=\small},
    conn/.style={thick, draw=gray!80},
    dashed_line/.style={thick, dashed, draw=gray!80}
]

    \begin{scope}
        \node[tensor_node] (c1) at (0,0) {$\kappa_k(\g_i)$};

        \node[matrix_node] (t1_top1) at (75:2.5cm) {$\frac{R^{(S)}}{N}$};
        \node[matrix_node] (t1_top2) at (105:2.5cm) {$\frac{R^{(S)}}{N}$};
        \draw[conn] (c1) -- (t1_top1);
        \draw[conn] (c1) -- (t1_top2);
        \node at (90:2.5cm) {$\dots$};
        \draw[dashed_line] (t1_top1) -- +(0, 1);
        \draw[dashed_line] (t1_top2) -- +(0, 1);

        \node[matrix_node] (t1_right) at (0:2.5cm) {$\frac{R^{(S)}}{N}$};
        \draw[conn] (c1) -- (t1_right);
        \draw[dashed_line] (t1_right) -- +(1, 0);

        \node[matrix_node] (t1_bot1) at (-75:2.5cm) {$\frac{R^{(S)}}{N}$};
        \node[matrix_node] (t1_bot2) at (-105:2.5cm) {$\frac{R^{(S)}}{N}$};
        \draw[conn] (c1) -- (t1_bot1);
        \draw[conn] (c1) -- (t1_bot2);
        \node at (-90:2.5cm) {$\dots$};
        \draw[dashed_line] (t1_bot1) -- +(0, -1);
        \draw[dashed_line] (t1_bot2) -- +(0, -1);

        \node[matrix_node] (t1_left1) at (165:2.5cm) {$\frac{R^{(S)}}{N}$};
        \node[matrix_node] (t1_left2) at (195:2.5cm) {$\frac{R^{(S)}}{N}$};
        \draw[conn] (c1) -- (t1_left1);
        \draw[conn] (c1) -- (t1_left2);
        \node at (180:2.5cm) {$\vdots$};
        \draw[dashed_line] (t1_left1) -- +(-1, 0);
        \draw[dashed_line] (t1_left2) -- +(-1, 0);
    \end{scope}

    \draw[-{Latex[scale=2.0]}, line width=2pt] (4, 0) -- (6, 0) 
        node[midway, above=0.2cm, align=center] {$\frac{1}{N} \sum_{\alpha} \frac{1}{|\lambda_\alpha - z|} \times$};

    \begin{scope}[xshift=10cm]
        \node[tensor_node] (c2) at (0,0) {$\kappa_k(\g_i)$};

        \node[matrix_node] (t2_top1) at (75:2.5cm) {$\frac{R^{(S)}}{N}$};
        \node[matrix_node] (t2_top2) at (105:2.5cm) {$\frac{R^{(S)}}{N}$};
        \draw[conn] (c2) -- (t2_top1);
        \draw[conn] (c2) -- (t2_top2);
        \node at (90:2.5cm) {$\dots$};
        \draw[dashed_line] (t2_top1) -- +(0, 1);
        \draw[dashed_line] (t2_top2) -- +(0, 1);

        \node[matrix_node] (t2_bot1) at (-75:2.5cm) {$\frac{R^{(S)}}{N}$};
        \node[matrix_node] (t2_bot2) at (-105:2.5cm) {$\frac{R^{(S)}}{N}$};
        \draw[conn] (c2) -- (t2_bot1);
        \draw[conn] (c2) -- (t2_bot2);
        \node at (-90:2.5cm) {$\dots$};
        \draw[dashed_line] (t2_bot1) -- +(0, -1);
        \draw[dashed_line] (t2_bot2) -- +(0, -1);

        \node[matrix_node] (t2_left1) at (165:2.5cm) {$\frac{R^{(S)}}{N}$};
        \node[matrix_node] (t2_left2) at (195:2.5cm) {$\frac{R^{(S)}}{N}$};
        \draw[conn] (c2) -- (t2_left1);
        \draw[conn] (c2) -- (t2_left2);
        \node at (180:2.5cm) {$\vdots$};
        \draw[dashed_line] (t2_left1) -- +(-1, 0);
        \draw[dashed_line] (t2_left2) -- +(-1, 0);

        
        \node[eigen_node] (eig1) at (-20:2.0cm) {$\v_\alpha$};
        \draw[conn, very thick] (c2) -- (eig1); 

        \node[eigen_node] (eig2) at (20:2.5cm) {$\v_\alpha$};
        \draw[dashed_line] (eig2) -- +(1.5, 0); 
        
    \end{scope}

\end{tikzpicture}
}
\end{figure}
By the given bound $N^{-1}\|R^{(S)}\|_1 \prec 1$, we have
\[|\mathrm{val}(\cG,f_\cG)| \prec \max_{f_{\cG''} \in \cF_{\cG''}}
|\mathrm{val}(\cG'',f_{\cG''})|\]
The vertex $v$ in $(\cG'',f_{\cG''})$ now satisfies $e(v)=1$ and $r(v)+m1(v)=0$,
so we may apply Lemma~\ref{lem:remove tensor strong}(a) to remove $v$ and bound
the resulting value. Letting $(\cG',\cF_{\cG'})$ denote the family of networks
obtained by removing $v$ from some $(\cG'',f_{\cG''})$ in this way, this yields
\[|\mathrm{val}(\cG,f_\cG)|\prec\frac{N^\delta}{\sqrt{N}}\Phi^{r(v)}
\max_{f_{\cG'} \in \cF_{\cG'}}|\mathrm{val}(\cG',f_{\cG'})|.\]
    \end{enumerate}
    \item While there exists a tensor vertex $v \in V_\cG^t$ such that $\deg(v) \geq 3$: Remove $v$, and apply Corollary~\ref{lem:remove tensor weak} to bound
    \[
        |\mathrm{val}(\cG,f_\cG)|\prec\Phi^{r(v)}\max_{f_{\cG'} \in \cF_{\cG'}}
|\mathrm{val}(\cG',f_{\cG'})|.
    \]
    \item At this point there are no type 1 matrix vertices,
and we have $\deg(v)=2$ for all $v \in V_\cG^t$ so
$\cG$ contains only chains and cycles. The endpoints of a chain can be either right (new) leaves or eigen-leaves. Then, applying the bounds
\begin{align*}
\Tr (N^{-1}R^{(S)}\Sigma)^k &\prec
\|N^{-1}R^{(S)}\|_F^k \prec \Phi^k,\\
(N^{-1/2}R^{(S)}\x)^*(N^{-1}R^{(S)})^k (N^{-1/2}R^{(S)}\y)
&\prec \|N^{-1}R^{(S)}\|_F^k
\|N^{-1/2}R^{(S)}\x\|_2
\|N^{-1/2}R^{(S)}\y\|_2
\prec \Phi^{k+2},\\
\v_\alpha^*(N^{-1}R^{(S)})^k (N^{-1/2}R^{(S)}\x)
&\prec \|N^{-1}R^{(S)}\|_F^k
\|N^{-1/2}R^{(S)}\x\|_2
\prec \Phi^{k+1},\\
\v_\alpha^*(N^{-1}R^{(S)})^k \v_\beta
&\prec \|N^{-1}R^{(S)}\|_F^k \prec \Phi^k
\end{align*}
uniformly over $\x,\y \in \cU$ and eigenvectors $\v_\alpha,\v_\beta$ of $R^{(S)}$, we have
\[
|\mathrm{val}(\cG,f_\cG)|\prec\Phi^{|V_{\cG}^{m2}(4)|+|V_{\cG}^r(4)|}.
\]
where $|V_{\cG}^{m2}(4)|,|V_{\cG}^r(4)|$ denote the numbers of type 2 matrix
vertices and right (new) leaves in the network at the start of Step 4.
\end{enumerate}

We clarify that as in Lemma~\ref{lem:fluctuation_averaging_lowrank2}(a),
Steps 1--3 are applied sequentially, where each application of
Step 2 applies Step 2(b) only when there are no vertices $v \in U_\cG^1$
satisfying the conditions of Step 2(a). At the conclusion of Step 1, all
vertices of $U_\cG^2$ are removed from the network.
Each time Step 2(b) is applied, it creates an eigen-leaf in the
resulting network $(\cG',f_{\cG'})$. If this eigen-leaf is adjacent to a
different vertex $v' \in U_\cG^1$, then Step 2(b) cannot be applied again until
$v'$ is removed by a further application of Step 2(a). Thus, Step 2 preserves
the property that every $v \in U_\cG^1$ always has a number of adjacent
eigen-leaves $e(v)$ that is 0 or 1. Then at the conclusion of Step 2, all
vertices of $U_\cG^1$ are removed from the network.

For each Step $k \in \{1,2,3\}$, let $|V_\cG^r(k)|$ denote the total
number of right (new) leaves removed in Step $k$. Denote also 
$|U_\cG^1(2a)|$ and $|U_\cG^1(2b)|$ as the total numbers of vertices in
$U_\cG^1$ that are removed in Step~2(a) and 2(b), respectively. Then the above
procedure gives the bound
\begin{align*}
|\mathrm{val}(\cG,f_\cG)|&\prec
\underbrace{\Big(\frac{N^\delta}{\sqrt{N}}\Big)^{|U_\cG^2|}
\Phi^{|V_\cG^r(1)|}}_{\text{Step 1}}
\times \underbrace{\Big(\frac{N^\delta}{\sqrt{N}}\Big)^{|U_\cG^1|}
\Phi^{|V_\cG^r(2)|-|U_\cG^1(2a)|}}_{\text{Step 2}}
\times \underbrace{\Phi^{|V_\cG^r(3)|}}_{\text{Step 3}}
\times \underbrace{\Phi^{|V_{\cG}^{m2}(4)|+|V_{\cG}^r(4)|}}_{\text{Step 4}}
\end{align*}
Each type 2 matrix vertex of the original network $(\cG,f_\cG)$ is either
turned into a right new leaf in Steps~1--4, turned into an eigen-leaf in
Step~2(b), or remains a type 2 matrix vertex in Step 4. Moreover, each application of Step~2(b) turns exactly one matrix vertex into an eigen-leaf.
Thus
\[\sum_{k=1}^4 |V_\cG^r(k)| + |V_\cG^{m2}(4)| = |V_\cG^{m2}|-|U_\cG^1(2b)|.\]
Combining with $|U_\cG^1(2a)|+|U_\cG^1(2b)|=|U_\cG^1|$ and
$|U_{\cG}^1| + |U_\cG^2| = |K|$ (by Lemma \ref{lem:tensor network of gram
matrix}) gives
\begin{align*}
    |\mathrm{val}(\cG,f_\cG)| &\prec \Big(\frac{N^\delta}{\sqrt{N}}\Big)^{|K|}
\Phi^{|V_\cG^{m2}|-|U_\cG^1|}
\leq N^{-|K|/2}(N^{2\delta}\Phi)^{2L},
\end{align*}
the second inequality applying $|V_\cG^{m2}| \geq 2L + |U_\cG^1|$ from
Lemma~\ref{lem:tensor network of gram matrix} and $|K| \leq 4L$.
Applying this to
\eqref{eq:vvBQQval2} and choosing $D>0$ large enough shows that 
\eqref{eq:vvBQQval2} is $\Oprec((N^{2\delta}\Phi)^{2L})$, so the result follows
again from Markov's inequality.
\end{proof}

\subsubsection{Proof of Lemma \ref{lem:fluctuation_averaging_lowrank2}(c)}

The following lemma is similar to Lemmas
\ref{lem:tensor network of sample covariance} and
\ref{lem:tensor network of gram matrix}.

\begin{lemma}\label{lem:tensor network of off-diagonal block}
Fix any integer $L\geq 1$. For any $i_1,\ldots,i_{2L}\in[N]$, denote
$S=\{i_1,\ldots,i_{2L}\}$ the set of distinct indices. Let $q_1,\ldots,q_{2L}$
be a sequence of monomials satisfying the conditions of
Lemma \ref{lem:x_i q_i expansion}, where each $q_l$ is a monomial from the
expansion of $\cX_{i_l}^{(i_l)}\cQ_{i_l}^{(i_l)}$, and let
$q=\prod_{l=1}^L q_l\prod_{l=L+1}^{2L} \bar q$.
Then there exists a constant $C \equiv C(L)>0$, an exponent $m \in [0,C]$, and a set
$\cN$ of $(S,2L)$-valid tensor networks with $|\cN|\leq C$, such that
\[
    |\E_Sq|=|\cC^{(S)}|^m \cdot
\sum_{(\mathcal{G},f_\cG)\in\cN}|\mathrm{val}(\cG,f_\cG)|
\]
Moreover, for all $(\mathcal{G}, f_\cG)\in\cN$, there exist disjoint subsets of tensor vertices $U_\cG^1,U_\cG^2\subseteq V_\cG^{t}$ such that the following holds:
\begin{enumerate}[label=(\alph*)]
    \item We have $|V_\cG^{ro}|=2L$ and
$|V_\cG^{l}|=|V_\cG^{rn}|=|V_\cG^{e}|=0$. (I.e., there are no left leaves, right
new leaves, or eigen-leaves, and only right original leaves.)
    \item For all $v\in U_\cG^{1}\cup U_\cG^{2}$, $\deg(v)$ is odd
and $\deg(v) \geq 3$.
    \item Let $K \subseteq S$ be the subset of indices appearing exactly once in
$(i_1,\ldots,i_{2L})$. Then $|K|=|U_\cG^{1}|+|U_\cG^{2}|$.
    \item We have $|V_\cG^{ro}|+|V_\cG^{m2}| \geq 2L+|U_\cG^{1}|$.
    \item For all $v\in U_\cG^{2}$, $m1(v) \geq 1$.
\end{enumerate}
\end{lemma}
\begin{proof}
As in Lemmas \ref{lem:tensor network of sample covariance} and Lemma
\ref{lem:tensor network of gram matrix}, we have
\[
|\E_Sq|=|\cC^{(S)}|^m \cdot
\sum_{(\mathcal{G},f_\cG)\in\cN} |\mathrm{val}(\cG,f_\cG)|
\]
for some $m \in [0,C(L)]$ and family $\cN$ of $(S,2L)$-valid tensor networks
where $|\cN| \leq C(L)$. It is clear that $|V_\cG^{l}|=|V_\cG^{rn}|=|V_\cG^{e}|=0$ in
this representation. The claim
$|V_\cG^{ro}|=2L$ follows from the fact that there is exactly 1 factor of the
form $\cX_{i_l}^{(S)}$ or $\{\cX_j^{(S)}\}_{j \in S^{(i_l)}}$ in each $q_l$,
establishing (a).

As in Lemma \ref{lem:tensor network of gram matrix}, for each $i_l \in K$,
$\g_{i_l}$ appears an odd number of times in $q_l$ and an even number of times
in $\prod_{l' \neq l} q_{l'}$, so we may identify a subset $U_\cG=\{v(i_l):i_l
\in K\} \subseteq V_\cG^t$ of tensor vertices representing an odd-order cumulant
of $\g_{i_l}$ for each $i_l \in K$;
thus $|U_\cG|=|K|$, and $\deg(v) \geq 3$ and $\deg(v)$ is
odd for each $v \in U_\cG$.
We partition $U_\cG$ into $U_\cG^1 \cup U_\cG^2$ by letting $U_\cG^1$ be the set
of vertices $v(i_l)$ for which either
\begin{enumerate}
\item $i_l$ appears as a lower index in some $\{q_{l'}:l' \neq l\}$, or
\item $i_l$ does not appear as a lower index in any $\{q_{l'}:l' \neq l\}$, and
$ro(v(i_l))+m2(v(i_l)) \geq 3$.
\end{enumerate}
Let $m_X^*(l),m_X(l),m_B^*(l),m_B(l)$ denote the counts of Lemma
\ref{lem:x_i q_i expansion} for each $q_l$, and let $m_P^*(l)$ denote the
number of factors $\cP_{i_l}^{(S)}$ of $q_l$ which correspond to type 2 matrix
vertices. Set
\[m_X^*=\sum_{l=1}^{2L} m_X^*(l),
\quad m_X=\sum_{l=1}^{2L} m_X(l),
\quad m_B^*=\sum_{l=1}^{2L} m_B^*(l),
\quad m_B=\sum_{l=1}^{2L} m_B(l),
\quad m_P^*=\sum_{l=1}^{2L} m_P(l)\]
and set also
\[E(l)=m_X^*(l)+\frac{1}{2}(m_X(l)+m_B^*(l))+m_P^*(l)-1.\]
By Lemma \ref{lem:x_i q_i expansion}, the number of vertices $v(i_l) \in U_\cG$ 
satisfying (1.) is at most $\frac{1}{2}m_X+m_B+\frac{1}{2}m_B^*$. For each
$v(i_l) \in U_\cG$ satisfying (2.), we have either $m_X^*(l)=1$ in which case
$m_B^*(l)+m_P^*(l) \geq 2$, or $m_X^*(l)=0$ in which case $m_B^*(l)+m_P^*(l)
\geq 3$ and also $m_X(l)=1$ by Lemma \ref{lem:x_i q_i expansion}(a). Thus in
both cases, $E(l) \geq 1$. We have also $E(l) \geq 0$ for each $l=1,\ldots,2L$
by Lemma \ref{lem:x_i q_i expansion}, so the number of vertices $v(i_l) \in
U_\cG$ satisfying (2.) is at most $\sum_{l=1}^{2L} E(l)$, implying
\[|U_\cG^1| \leq \frac{1}{2}m_X+m_B+\frac{1}{2}m_B^*
+\sum_{l=1}^{2L} E(l)
\leq m_X^*+m_X+m_B^*+m_B+m_P^*-2L \leq |V_\cG^{ro}|+|V_\cG^{m2}|-2L.\]
This shows (d). For each remaining $v \in U_\cG^2$, we have $\deg(v) \geq 3$
and $ro(v)+m2(v) \leq 2$, implying $m1(v) \geq 1$, which shows (e).
\end{proof}

\begin{proof}[Proof of Lemma \ref{lem:fluctuation_averaging_lowrank2}(c)]

Recall from \eqref{eq:vXQval} that, for any fixed constants $L \geq 1$ and
$D>0$,
\begin{equation}\label{eq:vXQval2}
\E\left|\sum_{i=1}^N \bar{\v}(i)\cX_i^{(i)}\cQ_i^{(i)}\right|^{2L}
\prec \mathop{\max_{i_1,\ldots,i_{2L} \in [N]}}_{q_1 \in \cM_{i_1,S},
\ldots,q_{2L} \in \cM_{i_{2L},S}}
\Big\{N^{|K|/2}\E|\E_S q|\Big\}+\Oprec(N^{-\tau'(D+1)+L}).
\end{equation}
By Lemma \ref{lem:tensor network of off-diagonal block}, there exists a family
$\cN$ of $(S,2L)$-valid tensor networks for which
\begin{align*}
|\E_S q|& \prec
\max_{(\mathcal{G},f_\cG)\in\cN}|\mathrm{val}(\mathcal{G},f_\cG)|.
\end{align*}
By following the same four steps and arguments as in the proof of Lemma
\ref{lem:fluctuation_averaging_lowrank2}(b), we arrive at the same bound
\begin{align*}
\mathrm{val}(\cG,f_\cG)&\prec
\underbrace{\Big(\frac{N^\delta}{\sqrt{N}}\Big)^{|U_\cG^2|}
\Phi^{|V_\cG^r(1)|}}_{\text{Step 1}}
\times \underbrace{\Big(\frac{N^\delta}{\sqrt{N}}\Big)^{|U_\cG^1|}
\Phi^{|V_\cG^r(2)|-|U_\cG^1(2a)|}}_{\text{Step 2}}
\times \underbrace{\Phi^{|V_\cG^r(3)|}}_{\text{Step 3}}
\times \underbrace{\Phi^{|V_{\cG}^{m2}(4)|+|V_{\cG}^r(4)|}}_{\text{Step 4}},
\end{align*}
where the only difference is that $|V_\cG^r(k)|$ for each Step $k \in
\{1,2,3,4\}$ counts right original leaves in addition to right new leaves.
By the same arguments as in Lemma \ref{lem:fluctuation_averaging_lowrank2}(b), we have
\[\sum_{k=1}^4
|V_\cG^r(k)|+|V_\cG^{m2}(4)|=|V_\cG^{m2}|+|V_\cG^{ro}|-|U_\cG^1(2b)|,\]
$|U_\cG^1(2a)|+|U_\cG^1(2b)|=|U_\cG^1|$, and $|U_\cG^1|+|U_\cG^2|=|K|$, so
\begin{align*}
|\mathrm{val}(\mathcal{G},f_\cG)|&\prec \Big(\frac{N^\delta}{\sqrt{N}}\Big)^{|K|}
\Phi^{|V_\cG^{ro}|+|V_\cG^{m2}|-|U_\cG^1|} \leq N^{-|K|/2}(N^\delta\Phi)^{2L},
\end{align*}
the last inequality using $|V_\cG^{ro}|+|V_\cG^{m2}| \geq 2L+|U_\cG^1|$ and
$K \leq 2L$.
Applying this above and choosing $D>0$ large enough shows that
\eqref{eq:vXQval2} is $\Oprec((N^\delta\Phi)^{2L})$, so the result follows
again by Markov's inequality.
\end{proof}

\section{Proofs of the main results}\label{sec:mainproofs}

We now show Theorems \ref{thm:entrywise law}, \ref{thm:anisotropic law},
and \ref{thm:outside spec}. Many of the arguments are close to those of
\cite{silverstein1995strong,bai1998no,pillai2014universality,knowles2017anisotropic},
which we reproduce here for the reader's convenience.

\subsection{Resolvent bounds}

For each $z=E+i\eta\in \C^+$, let us define the control parameters
\[\Lambda=\max_{i,j=1}^N|\widetilde{R}_{ij}-\widetilde{m}_0\1\{i=j\}|,\qquad
\Theta =|\widetilde{m} -\widetilde{m}_0|,\qquad
\Psi_\Theta=\sqrt{\frac{\Im{\widetilde{m}_0}+\Theta}{N\eta}}.\]
Fixing a constant $C_0>0$, we define the ($z$-dependent) event
\[\Xi=\{\Lambda \leq N^{-\tau/C_0}\}.\]

\begin{lemma}[Global bounds]\label{lem:trivial_bounds}
Suppose Assumptions \ref{assum:basic} and \ref{assum:concentration} hold. Then
for all $i\in[N]$ and
$z=E+i\eta\in \C^+$ satisfying $|z|\leq C$ for a constant $C>0$, we have
\[
    |\widetilde{m}|\leq \eta^{-1},\quad |\widetilde{R}_{ii}|\leq \eta^{-1},
\quad
    |z\widetilde R_{ii}|=|1+N^{-1}\g_i^* R^{(i)}\g_i|^{-1}
\leq C\eta^{-1},\quad N^{-1/2}\|R^{(i)}\|_F\leq \eta^{-1},
\]
\[
\|(I+\widetilde{m}\Sigma)^{-1}\|_\op\leq C\eta^{-1},\quad \|(I+\widetilde{m}^{(i)}\Sigma)^{-1}\|_\op\leq C\eta^{-1}.
\]
Furthermore, uniformly over $i \in [N]$ and
$z=E+i\eta \in \C^+$ with $|z| \leq C$,
\[|\widetilde{m}|^{-1} \prec\eta^{-1},\quad
|\widetilde{m}-\widetilde{m}^{(i)}|\prec N^{-1}\eta^{-3}.\]
\end{lemma}
\begin{proof}
The first four bounds follow from
$\|\widetilde{R}\|_\op,\|\widetilde{R}^{(i)}\|_\op\leq\eta^{-1}$.
For $\|(I+\widetilde{m}\Sigma)^{-1}\|_\op$, letting
$\lambda_i$ denote the $i$-th eigenvalue of $\widetilde K$, note that
\[\Im{[z\widetilde{m}(z)]}=\frac{1}{N}\sum_{i=1}^N\frac{\eta\lambda_i}{|\lambda_i-z|^2}\geq
0.\]
Then for any $\alpha\in[n]$,
    $|z+z\sigma_\alpha\widetilde{m}(z)|\geq|\eta+\sigma_\alpha\Im{[z\widetilde{m}(z)]}|\geq\eta$,
so $\|(I+\widetilde{m}\Sigma)^{-1}\|_\op \leq |z|\eta^{-1} \leq C\eta^{-1}$,
and similarly
$\|(I+\widetilde{m}^{(i)}\Sigma)^{-1}\|_\op\leq C\eta^{-1}$.
For the bound on $|\widetilde{m}|^{-1}$,
\[
    |\widetilde{m}|\geq\Im{\widetilde{m}}=\frac{1}{N}\sum_{i=1}^N\frac{\eta}{(\lambda_i-E)^2+\eta^2}\geq\frac{\eta}{\max_i(\lambda_i-E)^2+\eta^2}.
\]
Applying $\|\widetilde K\|_\op=\|K\|_\op \prec 1$ from Lemma
\ref{lem:concentration}(b) and $|z|\leq C$ yields $\max_\alpha(\lambda_\alpha-E)^2+\eta^2\prec
1$, so $|\widetilde{m}|^{-1} \prec \eta^{-1}$.
For $|\widetilde{m}-\widetilde{m}^{(i)}|$, apply \eqref{eq:Sherman-Morrison} to obtain
\[
    |\widetilde{m}-\widetilde{m}^{(i)}|=\gamma|m-m^{(i)}|=|z\widetilde{R}_{ii}N^{-2}\g_i^*(R^{(i)})^2\g_i|\prec\eta^{-1}(N^{-1}\|R^{(i)}\|_F)^2\leq N^{-1}\eta^{-3},
\]
where we applied Assumptions \ref{assum:basic} and \ref{assum:concentration}
to show $\g_i^*(R^{(i)})^2\g_i
\prec |\Tr (R^{(i)})^2\Sigma|+\|(R^{(i)})^2\|_F \prec \|R^{(i)}\|_F^2$.
\end{proof}

\begin{lemma}[Local bounds]\label{lem:prec bounds}
Suppose Assumptions \ref{assum:basic} and \ref{assum:concentration} hold, and
$\bD=\{z \in \C^+:E \in U,\,\eta \in [N^{-1+\tau},1]$
is a regular spectral domain.
Then for any fixed $L \geq 0$, uniformly over all $z\in \bD$, $S\subseteq[N]$
with $|S|\leq L$, and $i,j \notin S$ with $i \neq j$,
\[
    |\widetilde{R}^{(S)}_{ii}|\1_\Xi\prec 1,\quad
|\widetilde{R}^{(S)}_{ii}|^{-1} \1_\Xi\prec 1,\quad
|z\widetilde{R}^{(S)}_{ii}|\1_\Xi=|1+N^{-1}\g_i^* R^{(iS)}\g_i|^{-1}
\1_\Xi\prec 1,
\quad |\widetilde{R}_{ij}^{(S)}|\1_\Xi\prec\Psi_\Theta,\]
\[|\widetilde{m}^{(S)}|\1_\Xi\prec 1,
\quad |\widetilde{m}^{(S)}|^{-1}\1_\Xi\prec 1,
\quad \|(I+\widetilde{m}^{(S)}\Sigma)^{-1}\|_\op\1_\Xi \prec 1.
\]
\[N^{-1}\|R^{(S)}\|_F\1_\Xi\prec\Psi_\Theta,
\quad N^{-1} \|R^{(S)}\|_1\1_\Xi \prec 1.\]
Moreover, uniformly also over deterministic unit vectors $\x,\y\in\C^n$,
\[|\widetilde{m}-\widetilde{m}^{(S)}|\1_\Xi\prec\Psi_\Theta^2,\quad
|\widetilde{R}_{ii}^{(S)}-\widetilde{m}^{(S)}|\1_\Xi\prec\Psi_\Theta,\]
\[|\x^*R\y-\x^*R^{(S)}\y|\1_\Xi \prec\frac{\|R\x\|_2}{\sqrt{N}} \cdot
\frac{\|R\y\|_2}{\sqrt{N}},\quad
\|R^{(S)}\x\|_2\1_\Xi \prec \|R\x\|_2.\]
\end{lemma}

\begin{proof}
The bounds
\[|\widetilde{R}_{ii}^{(S)}|\1_\Xi\prec 1,\quad
|\widetilde{R}_{ii}^{(S)}|^{-1}\1_\Xi\prec 1,\quad
|z\widetilde{R}_{ii}^{(S)}|\1_\Xi=|1+N^{-1}\g_i^* R^{(iS)}\g_i|^{-1}\1_\Xi\prec 1\]
follow from $|z|,|\widetilde m_0(z)| \asymp 1$ in
Definition \ref{def:regular} for regularity of $\bD$,
the definition of $\Xi$, and iterative application of the resolvent identity
\eqref{eq:schur complement}.

We next prove $|\widetilde{m}-\widetilde{m}^{(S)}|\1_\Xi\prec\Psi_\Theta^2$ by induction on $|S|$. For $|S|=1$, say $S=\{i\}$, we have
\[
    |\widetilde{m}-\widetilde{m}^{(i)}|=\gamma|m-m^{(i)}|=|N^{-1}\Tr R-N^{-1}\Tr R^{(i)}|.
\]
Applying the Sherman-Morrison identity \eqref{eq:Sherman-Morrison} and
Assumption \ref{assum:concentration} yields
\[
    |\widetilde{m}-\widetilde{m}^{(i)}|\1_\Xi=|z\widetilde{R}_{ii}| |N^{-2}\g_i^*R^{(i)2}\g_i|\1_\Xi\prec N^{-2}\|R^{(i)}\|_F^2\1_\Xi\prec\Psi_\Theta^2+(N\eta)^{-1}|\widetilde{m}-\widetilde{m}^{(i)}|\1_\Xi,
\]
where the last bound uses Ward's identity \eqref{eq:Ward's identity},
\begin{align*}
    N^{-2}\|R^{(i)}\|_F^2&=\frac{\gamma\Im{m^{(i)}}}{N\eta}=\frac{\Im{\widetilde{m}^{(i)}}}{N\eta}+\frac{(1-\gamma)\eta|z|^{-2}}{N\eta}\\
    &\prec\frac{\Im{\widetilde{m}_0}+\Theta}{N\eta}+\frac{|\widetilde{m}-\widetilde{m}^{(i)}|}{N\eta}+\frac{1}{N}\prec\Psi_\Theta^2+\frac{|\widetilde{m}-\widetilde{m}^{(i)}|}{N\eta}.
\end{align*}
Since $(N\eta)^{-1}\leq N^{-\tau}\leq 1/2$ for large $N$, rearranging proves the
case $|S|=1$. Now
assume the claim holds for $|S|=s$. For $i\notin S$ with $|S|=s$, let $S=\{i_1,\dots,i_s\}$ and $S_k=\{i_1,\dots,i_k\}$. Iteratively applying \eqref{eq:Sherman-Morrison} gives
\begin{align*}
    |\widetilde{m}-\widetilde{m}^{(iS)}|\1_\Xi&=\left|z\widetilde{R}^{(S)}_{ii}N^{-2}\g_{i}^*(R^{(iS)})^2\g_{i}+\sum_{k=1}^sz\widetilde{R}^{(S_{k-1})}_{i_ki_k}N^{-2}\g_{i_k}^*(R^{(S_k)})^2\g_{i_k}\right|\1_\Xi\\
    &\prec \Psi_\Theta^2+(N\eta)^{-1}|\widetilde{m}-\widetilde{m}^{(iS)}|\1_\Xi+(N\eta)^{-1}\sum_{k=1}^s|\widetilde{m}-\widetilde{m}^{(S_k)}|\1_\Xi\\
    &\prec\Psi_\Theta^2+(N\eta)^{-1}|\widetilde{m}-\widetilde{m}^{(iS)}|\1_\Xi
\end{align*}
where the last line applies the inductive hypothesis.
Rearranging completes the induction.

Since $|\widetilde m-\widetilde m_0|=\Theta \leq \Lambda$,
the bounds $|\widetilde{m}^{(S)}|\1_\Xi \prec 1$,
$|\widetilde{m}^{(S)}|^{-1}\1_\Xi \prec 1$, and
$\|(I+\widetilde{m}^{(S)}\Sigma)^{-1}\|_\op\1_\Xi \prec 1$ then follow from
the definition of $\Xi$ and the corresponding statements for $\widetilde m_0$
in Definition \ref{def:regular} for the regularity of $\bD$.
The bound $N^{-1}\|R^{(S)}\|_F\1_\Xi\prec\Psi_\Theta$ follows from 
Ward's identity as applied above, and
$|\widetilde{R}_{ij}^{(S)}|\1_\Xi\prec\Psi_\Theta$ follows from
the resolvent identity \eqref{eq:resolvent identity} and Assumption
\ref{assum:concentration}.

For $N^{-1} \|R^{(S)}\|_1$, corresponding to $z=E+i\eta \in \bD$, define
\[
    L=\max \{ l \in \N : \eta 2^{l-1} < 1 \}, \quad \eta_l =
2^l \eta \text{ for } l = 0, \dots, L-1, \quad \eta_L = 1.
\]
Then $L \leq C \log N$ and $\eta_l / \eta_{l-1} \leq 2$. Let
$\{\lambda_\alpha\}_{\alpha=1}^n$ be the eigenvalues of $K^{(S)}$. Define $\eta_{-1}=0$, $\eta_{L+1}=\infty$, and
\[
    U_l=\{\alpha : \eta_{l-1} \leq |\lambda_\alpha - E| < \eta_l\}, \quad l = 0, \dots, L+1.
\]
Then
\[
    N^{-1} \|R^{(S)}\|_1 = \frac{1}{N} \sum_{\alpha=1}^n \frac{1}{|\lambda_\alpha - z|} = \sum_{l=0}^{L+1} \frac{1}{N} \sum_{\alpha \in U_l} \frac{1}{|\lambda_\alpha - z|}.
\]
For $l = 1, \dots, L$, since $\eta_{l-1} \leq |\lambda_\alpha - E| < \eta_l$ for
$\alpha \in U_l$ and $|\lambda_\alpha - z| \geq |\lambda_\alpha-E|$, we have
\begin{align*}
    \frac{1}{n}\sum_{\alpha \in U_l} \frac{1}{|\lambda_\alpha - z|} &\leq
\frac{1}{n}\sum_{\alpha \in U_l} \frac{\eta_l}{(\lambda_\alpha - E)^2} \leq
\frac{2}{n} \sum_{\alpha \in U_l} \frac{\eta_l}{(\lambda_\alpha - E)^2 + \eta_{l-1}^2}
    \leq \frac{2 \eta_l}{\eta_{l-1}} \Im m(E + i \eta_{l-1}).
\end{align*}
Since the map $\eta \mapsto \eta \Im m(E + i \eta)$ is nondecreasing,
$\Im m(E + i \eta_{l-1}) \leq \frac{\eta_l}{\eta_{l-1}} \Im m(E + i \eta_l)$.
Then, applying $\eta_l/\eta_{l-1} \leq 2$,
\[
    \frac{1}{n}\sum_{\alpha \in U_l} \frac{1}{|\lambda_\alpha - z|} \leq 8\Im m(E + i \eta_l).
\]
For $l = 0$, since $0 \leq |\lambda_\alpha - E| < \eta_0 = \eta$ for $\alpha \in
U_0$, we have similarly
\[
    \frac{1}{n}\sum_{\alpha \in U_0} \frac{1}{|\lambda_\alpha - z|} \leq
\frac{2}{n}\sum_{\alpha \in U_0} \frac{\eta}{(\lambda_\alpha - E)^2 + \eta^2}
\leq 2 \Im m(E + i \eta) \leq 4 \Im m(E + i \eta_1).
\]
For $l=L+1$, since $\eta_L=1$, $|\lambda_\alpha-E| \geq 1$ for $\alpha \in
U_{L+1}$, and $\max_\alpha \lambda_\alpha=\|K\|_\op \prec 1$
by Lemma \ref{lem:concentration}(b), we have
\[\frac{1}{n}\sum_{\alpha \in U_{L+1}} \frac{1}{|\lambda_\alpha - z|}
\leq \frac{1}{n}\sum_{\alpha \in U_{L+1}} \frac{|\lambda_\alpha -
E|+\eta}{(\lambda_\alpha - E)^2 + \eta^2} \prec \frac{1}{n}\sum_{\alpha \in
U_{L+1}}\frac{1}{(\lambda_\alpha - E)^2 + 1} \leq \Im m(E+i \eta_L).\]
Thus
\begin{equation}\label{eq:RL1bound}
    N^{-1} \|R^{(S)}\|_1 \prec \sum_{l=1}^L\Im m^{(S)}(E + i \eta_l).
\end{equation}
Since $\gamma m^{(S)} = \widetilde{m}^{(S)} + (1-\gamma) z^{-1}$ and $\gamma
\asymp 1$, this implies
\[N^{-1} \|R^{(S)}\|_1 \1_\Xi \prec \left(\sum_{l=1}^L \Im
\widetilde{m}^{(S)}(E + i \eta_l) + \frac{\eta_l}{|E+i\eta_l|^2}\right)\1_\Xi
\prec 1\]
where the last bound uses $\Im \widetilde{m}^{(S)}(E + i \eta_l)\1_\Xi \prec 1$,
$\eta_l/|E+i\eta_l|^2 \prec 1$ since $|E| \geq \delta$ by Definition
\ref{def:regular} for regularity of $\bD$, and
$L \leq C \log N \prec 1$.

For $|\widetilde{R}_{ii}^{(S)}-\widetilde{m}^{(S)}|$, applying the resolvent
identity \eqref{eq:schur complement}, we have on $\Xi$ that
\begin{align*}
    \frac{1}{\widetilde{R}_{ii}^{(S)}}&=-z-zN^{-1}\g_i^*R^{(iS)}\g_i\\
&=-z-zN^{-1}\Tr\Sigma R^{(iS)}+\Oprec(N^{-1}\|R^{(iS)}\|_F)\\
    &=-z-zN^{-1}\Tr\Sigma
R-z^2\widetilde{R}_{ii}^{(S)}N^{-2}\g_i^*R^{(iS)}\Sigma
R^{(iS)}\g_i+\Oprec(N^{-1}\|R^{(iS)}\|_F)\\
    &=-z-zN^{-1}\Tr\Sigma R-z^2\widetilde{R}_{ii}^{(S)}N^{-2}\Tr\Sigma
R^{(i)}\Sigma
R^{(iS)}+\Oprec(N^{-2}\|R^{(iS)}\|_F^2)+\Oprec(N^{-1}\|R^{(iS)}\|_F)\\
    &=-z-zN^{-1}\Tr\Sigma
R+\Oprec(N^{-2}\|R^{(iS)}\|_F^2)+\Oprec(N^{-1}\|R^{(iS)}\|_F).
\end{align*}
Thus, for $i\neq j$,
\[
    \left|\frac{1}{\widetilde{R}_{ii}^{(S)}}-\frac{1}{\widetilde{R}_{jj}^{(S)}}\right|\1_\Xi\prec\Psi_{\Theta}^2+\Psi_\Theta\prec\Psi_\Theta,
\]
and so
\[
    |\widetilde{R}_{ii}^{(S)}-\widetilde{m}^{(S)}|\1_\Xi\leq\frac{1}{N}\sum_{j=1}^N|\widetilde{R}_{ii}^{(S)}-\widetilde{R}_{jj}^{(S)}|\1_\Xi\leq\frac{1}{N}\sum_{j=1}^N|\widetilde{R}_{ii}^{(S)}\widetilde{R}_{jj}^{(S)}|\left|\frac{1}{\widetilde{R}_{ii}^{(S)}}-\frac{1}{\widetilde{R}_{jj}^{(S)}}\right|\1_\Xi\prec\Psi_\Theta.
\]

Finally, for $|\x^* R\y-\x^*R^{(S)}\y|$, the bound
$|\g_i^*R^{(S)}\x|\prec\|R^{(S)}\x\|_2$ in Lemma \ref{lem:concentration}(a)
and iterative application of the Sherman-Morrison identity
(\ref{eq:Sherman-Morrison}) yield
$|\x^* R\y-\x^*R^{(S)}\y|\1_\Xi \prec \|R\x\|_2\|R\y\|_2/N$.
The last claim for $\|R^{(S)}\x\|_2$ then follows from
\[
    \|R^{(S)}\x\|_2^2\1_\Xi=\frac{\Im{\x^*R^{(S)}\x}}{\eta}\1_\Xi
\prec\frac{\Im{\x^*R\x}}{\eta}+\frac{|\x^*R\x-\x^*R^{(S)}\x|}{\eta}\1_\Xi\prec \|R\x\|_2^2+\frac{\|R\x\|_2^2}{N\eta}\prec \|R\x\|_2^2.
\]
\end{proof}

\subsection{Proof of entrywise law (Theorem \ref{thm:entrywise law})}

We now begin analysis of the Marchenko-Pastur fixed point equation.

\begin{lemma}\label{lem:Approx_fixed_point}
For any $A\in\C^{n\times n}$ and $z\in\C^+$, we have
\begin{equation}\label{eq:deterministic equivalent}
    \Tr RA - \Tr(-zI-z\widetilde{m}\Sigma)^{-1}A=\frac{1}{z}\sum_{i=1}^N\frac{d_i(A)}{1+N^{-1}\g_i^* R^{(i)}\g_i},
\end{equation}
where
\[
    d_i(A) = N^{-1}\g_i^* R^{(i)}A(I+\widetilde{m}\Sigma)^{-1}\g_i-N^{-1}\Tr RA(I+\widetilde{m}\Sigma)^{-1}\Sigma.
\]
In particular, recalling the function $z_0(m)$ from \eqref{eq:z0m}
and setting $A=I$ in \eqref{eq:deterministic equivalent},
\begin{equation}\label{eq:fixed point equation}
    z_0(\widetilde{m})-z=-\frac{1}{\widetilde{m}}\cdot\frac{1}{N}\sum_{i=1}^N\frac{d_i(I)}{1+N^{-1}\g_i^* R ^{(i)}\g_i}.
\end{equation}
\end{lemma}
\begin{proof}
The argument is the same as \cite[Eq.\ (2.4)]{silverstein1995strong}: 
For any $B \in \C^{n \times n}$,
\begin{align*}
\Tr B=\Tr (K-zI)RB
&={-}z\Tr RB+\frac{1}{N}\sum_{i=1}^N \g_i^*RB\g_i
={-}z\Tr RB+\sum_{i=1}^N
\frac{N^{-1}\g_i^*R^{(i)}B\g_i}{1+N^{-1}\g_i^*R^{(i)}\g_i}
\end{align*}
where the last equality applies the Sherman-Morrison identity
\eqref{eq:Sherman-Morrison}. Taking $B=I_n$,
applying $\Tr R(z)=nm(z)=N\widetilde m(z)-(n-N)/z$, and rearranging gives
\[\widetilde m={-}\frac{1}{Nz}\sum_{i=1}^N
\frac{1}{1+N^{-1}\g_i^*R^{(i)}\g_i}.\]
Then taking $B=A(I+\widetilde m \Sigma)^{-1}$ and identifying $d_i(A)$ gives
\begin{align*}
\Tr (I+\widetilde m \Sigma)^{-1}A&=-z\Tr RA(I+\widetilde m\Sigma)^{-1}
+\sum_{i=1}^N \frac{d_i(A)+N^{-1}\Tr RA(I+\widetilde m\Sigma)^{-1}\Sigma}
{1+N^{-1}\g_i^*R^{(i)}\g_i}\\
&=-z\Tr RA(I+\widetilde m\Sigma)^{-1}
-z\widetilde m \Tr RA(I+\widetilde m\Sigma)^{-1}\Sigma
+\sum_{i=1}^N \frac{d_i(A)}{1+N^{-1}\g_i^*R^{(i)}\g_i}\\
&={-}z\Tr RA+\sum_{i=1}^N \frac{d_i(A)}{1+N^{-1}\g_i^*R^{(i)}\g_i}.
\end{align*}
Rearranging shows \eqref{eq:deterministic equivalent},
and specializing to $A=I$ shows \eqref{eq:fixed point equation}.
\end{proof}

The following lemma will allow us to simplify the analysis by replacing factors
of $(I+\widetilde m \Sigma)^{-1}$ in $d_i(A)$
with their deterministic counterpart $(I+\widetilde m_0 \Sigma)^{-1}$.

\begin{lemma}\label{lem:alternative to L1 bound}
Suppose Assumptions \ref{assum:basic} and  \ref{assum:concentration} hold, and
$\bD$ is a regular spectral domain.
Then there exists a constant $l\geq 1$ such that, uniformly over
all $z\in\bD$, deterministic matrices $A\in\C^{n\times n}$, and $i\in[N]$,
\[
    d_i(A)\1_\Xi = \sum_{k=0}^l(\widetilde{m}-\widetilde{m}_0)^k\left[N^{-1}\Tr(\g_i\g_i^*-\Sigma)R^{(i)}A\Sigma^k(I+\widetilde{m}_0\Sigma)^{-k-1}\right]\1_\Xi+\Oprec(\Psi_\Theta^2\|A\|_\op).
\]
Moreover, if $A=\u\v^*$ for vectors $\u,\v\in\C^n$, then
\begin{align*}
    d_i(\u\v^*)\1_\Xi &= \sum_{k=0}^l(\widetilde{m}-\widetilde{m}_0)^k\left[N^{-1}\Tr(\g_i\g_i^*-\Sigma)R^{(i)}\u\v^*\Sigma^k(I+\widetilde{m}_0\Sigma)^{-k-1}\right]\1_\Xi\\
    &\hspace{0.3in}+\Oprec(N^{-2}\|R^{(i)}\u\|_2\|\v^*\Sigma^{k+1}(I+\widetilde{m}_0\Sigma)^{-k-1}R^{(i)}\|_2)+\Oprec(N^{-1}\Psi_\Theta\|\u\|_2\|\v\|_2).
\end{align*}
\end{lemma}
\begin{proof}
    Iteratively applying the identity $A^{-1}-B^{-1}=A^{-1}(B-A)B^{-1}$ gives, for any integer $l\geq 1$,
\[(I+\widetilde{m}\Sigma)^{-1}=\sum_{k=0}^l(\widetilde{m}-\widetilde{m}_0)^k\Sigma^k(I+\widetilde{m}_0\Sigma)^{-k-1}+(\widetilde{m}-\widetilde{m}_0)^{l+1}\Sigma^{l+1}(I+\widetilde{m}_0\Sigma)^{-l-1}(I+\widetilde{m}\Sigma)^{-1},\]
using the commutativity of $\Sigma$, $(I+\widetilde{m}_0\Sigma)^{-1}$, and $(I+\widetilde{m}\Sigma)^{-1}$. Thus, for any $i\in[N]$,
\begin{align*}
    &\frac{1}{N}\g_i^* R^{(i)}A(I+\widetilde{m}\Sigma)^{-1}\g_i-\frac{1}{N}\Tr RA(I+\widetilde{m}\Sigma)^{-1}\Sigma\\
    &= \sum_{k=0}^l(\widetilde{m}-\widetilde{m}_0)^k\left[\frac{1}{N}\g_i^* R^{(i)}A\Sigma^k(I+\widetilde{m}_0\Sigma)^{-k-1}\g_i-\frac{1}{N}\Tr RA\Sigma^k(I+\widetilde{m}_0\Sigma)^{-k-1}\Sigma\right]\\
    &\hspace{1in}+(\widetilde{m}-\widetilde{m}_0)^{l+1}\frac{1}{N}\g_i^* R^{(i)}A\Sigma^{l+1}(I+\widetilde{m}_0\Sigma)^{-l-1}(I+\widetilde{m}\Sigma)^{-1}\g_i\\
    &\hspace{1in}-(\widetilde{m}-\widetilde{m}_0)^{l+1}\frac{1}{N}\Tr RA\Sigma^{l+1}(I+\widetilde{m}_0\Sigma)^{-l-1}(I+\widetilde{m}\Sigma)^{-1}\Sigma.
\end{align*}
Since $\Theta\leq\Lambda\leq N^{-\tau/C_0}$ on $\Xi$, the remainder terms satisfy
\begin{align*}
    &\left|(\widetilde{m}-\widetilde{m}_0)^{l+1}\frac{1}{N}\g_i^* R^{(i)}A\Sigma^{l+1}(I+\widetilde{m}_0\Sigma)^{-l-1}(I+\widetilde{m}\Sigma)^{-1}\g_i\right|\1_\Xi\\
    &\leq|\widetilde{m}-\widetilde{m}_0|^{l+1}\frac{1}{N}\|\g_i\|_2^2 \|R^{(i)}\|_\op\|A\|_\op\|\Sigma\|_\op^{l+1}\|(I+\widetilde{m}_0\Sigma)^{-1}\|_\op^{l+1}\|(I+\widetilde{m}\Sigma)^{-1}\|_\op\1_\Xi\\
    &\prec N^{-(l+1)\tau/C_0}\eta^{-1}\|A\|_\op\leq N^{-(l+1)\tau/C_0+1}\|A\|_\op,
\end{align*}
using $\eta\geq N^{-1+\tau}>N^{-1}$, and similarly
\begin{align*}
    &\left|(\widetilde{m}-\widetilde{m}_0)^{l+1}\frac{1}{N}\Tr RA\Sigma^{l+1}(I+\widetilde{m}_0\Sigma)^{-l-1}(I+\widetilde{m}\Sigma)^{-1}\Sigma\right|\1_\Xi\prec N^{-(l+1)\tau/C_0+1}\|A\|_\op.
\end{align*}
For each $0\leq k \leq l$,
\begin{align*}
    &\frac{1}{N}\g_i^* R^{(i)}A\Sigma^k(I+\widetilde{m}_0\Sigma)^{-k-1}\g_i\1_\Xi-\frac{1}{N}\Tr RA\Sigma^k(I+\widetilde{m}_0\Sigma)^{-k-1}\Sigma\1_\Xi\\
    &=\frac{1}{N}\Tr(\g_i\g_i^*-\Sigma)R^{(i)}A\Sigma^k(I+\widetilde{m}_0\Sigma)^{-k-1}\1_\Xi+\frac{N^{-2}\g_i^*
R^{(i)}A\Sigma^{k+1}(I+\widetilde{m}_0\Sigma)^{-k-1}R^{(i)}\g_i}{1+N^{-1}\g_i^* R^{(i)}\g_i}\1_\Xi,
\end{align*}
and applying Assumption \ref{assum:concentration}, Cauchy-Schwarz, and Lemma \ref{lem:prec bounds} yields
\begin{align*}
    &\frac{N^{-2}|\g_i^* R^{(i)}A\Sigma^{k+1}(I+\widetilde{m}_0\Sigma)^{-k-1}R^{(i)}\g_i|}{|1+N^{-1}\g_i^* R^{(i)}\g_i|}\1_\Xi\\
    &\prec |N^{-2}\Tr\Sigma
R^{(i)}A\Sigma^{k+1}(I+\widetilde{m}_0\Sigma)^{-k-1}R^{(i)}|\1_\Xi+N^{-2}\|R^{(i)}A\Sigma^{k+1}(I+\widetilde{m}_0\Sigma)^{-k-1}R^{(i)}\|_F\1_\Xi\\
    &\prec (N^{-1}\|R^{(i)}\|_F)^2\|A\|_\op\1_\Xi\prec \Psi_\Theta^2\|A\|_\op.
\end{align*}
If $A=\u\v^*$, then we may apply instead
\begin{align*}
    &\frac{N^{-2}|\g_i^* R^{(i)}\u\v^*\Sigma^{k+1}(I+\widetilde{m}_0\Sigma)^{-k-1}R^{(i)}\g_i|}{|1+N^{-1}\g_i^* R^{(i)}\g_i|}\1_\Xi
    \prec
N^{-2}\|R^{(i)}\u\|_2\|\v^*\Sigma^{k+1}(I+\widetilde{m}_0\Sigma)^{-k-1}R^{(i)}\|_2\1_\Xi.
\end{align*}
The result follows by choosing $l$ large enough so that
$N^{-(l+1)\tau/C_0+1}\leq N^{-1}\Psi_\Theta\leq\Psi_\Theta^2$, noting that $l$
is a constant depending only on $\tau$ and $C_0$.
\end{proof}

\begin{lemma}\label{lem:fixed_point_error}
Suppose Assumptions \ref{assum:basic} and \ref{assum:concentration} hold, and
$\bD$ is a regular spectral domain.
Then there exists a constant $l\geq 1$ such that, uniformly over $z\in\bD$,
\begin{align*}
    [z_0(\widetilde{m})-z]\1_\Xi&=\Oprec(\Psi_\Theta^2)+\Oprec\left(\max_{0\leq k\leq l}\left|\frac{1}{N}\sum_{i=1}^N \frac{N^{-1}\Tr(\g_i\g_i^*-\Sigma)R^{(i)}\Sigma^k(I+\widetilde{m}_0\Sigma)^{-k-1}}{1+N^{-1}\g_i^* R^{(i)}\g_i}\right|\1_\Xi\right)\\
    &=\Oprec(\Psi_\Theta).
\end{align*}
\end{lemma}
\begin{proof}
By Lemma \ref{lem:Approx_fixed_point},
\[
    [z_0(\widetilde{m})-z]\1_\Xi=-\frac{1}{\widetilde{m}}\cdot\frac{1}{N}\sum_{i=1}^N\frac{d_i(I)}{1+N^{-1}\g_i^* R^{(i)}\g_i}\1_\Xi.
\]
Applying Lemma \ref{lem:alternative to L1 bound} with $A=I$ gives, for a sufficiently large integer $l\geq 1$,
\[
    d_i(I)\1_\Xi = \sum_{k=0}^l(\widetilde{m}-\widetilde{m}_0)^k\left[N^{-1}\Tr(\g_i\g_i^*-\Sigma)R^{(i)}\Sigma^k(I+\widetilde{m}_0\Sigma)^{-k-1}\right]\1_\Xi+\Oprec(\Psi_\Theta^2).
\]
Since $\Theta=|\widetilde{m}-\widetilde{m}_0|\leq\Lambda\leq N^{-\tau/C_0}\leq 1$ on $\Xi$ and $l$ is fixed, substituting the expansion of $d_i(I)$ into the expression for $z_0(\widetilde{m})-z$ and applying the bounds from Lemma \ref{lem:prec bounds} yield the first equality. 
Moreover, by Assumption \ref{assum:concentration}, uniformly over all $0\leq k\leq l$,
\[
    |N^{-1}\Tr(\g_i\g_i^*-\Sigma)R^{(i)}\Sigma^k(I+\widetilde{m}_0\Sigma)^{-k-1}|\1_\Xi\prec N^{-1}\|R^{(i)}\|_F\1_\Xi\prec\Psi_\Theta.
\]
Then the second equality follows since $\Psi_\Theta^2\prec\Psi_\Theta$.
\end{proof}

\begin{lemma}\label{lem:eta=1_error}
Suppose Assumptions \ref{assum:basic} and \ref{assum:concentration} hold, and
$\bD$ is a regular spectral domain. Then, uniformly over $z=E+i\eta\in\bD$
with $\eta=1$, we have $\Lambda(z)\prec N^{-1/4}$.
\end{lemma}
\begin{proof}
Let $d_i(I)$ be as in Lemma \ref{lem:Approx_fixed_point}. As in \cite{bai1998no},
we decompose $d_i(I)=d_{i,1}+d_{i,2}+d_{i,3}+d_{i,4}$, where
\begin{align*}
    d_{i,1}&=N^{-1}\g_i^* R^{(i)}(I+\widetilde{m} \Sigma)^{-1}\g_i-N^{-1}\g_i^* R^{(i)}(I+\widetilde{m}^{(i)}\Sigma)^{-1}\g_i,\\
    d_{i,2}&=N^{-1}\g_i^* R^{(i)}(I+\widetilde{m}^{(i)}\Sigma)^{-1}\g_i-N^{-1}\Tr\Sigma R^{(i)}(I+\widetilde{m}^{(i)}\Sigma)^{-1},\\
    d_{i,3}&=N^{-1}\Tr\Sigma R^{(i)}(I+\widetilde{m}^{(i)}\Sigma)^{-1}-N^{-1}\Tr\Sigma R (I+\widetilde{m}^{(i)}\Sigma)^{-1},\\
    d_{i,4}&=N^{-1}\Tr\Sigma R (I+\widetilde{m}^{(i)}\Sigma)^{-1}-N^{-1}\Tr\Sigma R (I+\widetilde{m} \Sigma)^{-1}.
\end{align*}
Applying the bounds from Lemma \ref{lem:trivial_bounds} and Assumption \ref{assum:concentration} yields
\begin{align*}
    |d_{i,1}|&=|\widetilde{m}-\widetilde{m}^{(i)}||N^{-1}\g_i^* R^{(i)}(I+\widetilde{m} \Sigma)^{-1}\Sigma(I+\widetilde{m}^{(i)}\Sigma)^{-1}\g_i|\\
    &\prec |\widetilde{m} -\widetilde{m}^{(i)}|(\|R^{(i)}\|_\op+N^{-1}\|R^{(i)}\|_F)\|(I+\widetilde{m} \Sigma)^{-1}\|_\op\|(I+\widetilde{m}^{(i)}\Sigma)^{-1}\|_\op\\
    &\prec N^{-1}\eta^{-6},\\
    |d_{i,2}|&\prec N^{-1}\|R^{(i)}(I+\widetilde{m}^{(i)}\Sigma)^{-1}\|_F\leq N^{-1}\|R^{(i)}\|_F\|(I+\widetilde{m}^{(i)}\Sigma)^{-1}\|_\op\prec N^{-1/2}\eta^{-2},\\
    |d_{i,3}|&= \frac{|N^{-2}\g_i^* R^{(i)}(I+\widetilde{m}^{(i)}\Sigma)^{-1}\Sigma R^{(i)}\g_i|}{|1+N^{-1}\g_i^* R^{(i)}\g_i|}\prec N^{-2}\|R^{(i)}\|_F^2\eta^{-2}\prec N^{-1}\eta^{-4},\\
    |d_{i,4}|&=|\widetilde{m} -\widetilde{m}^{(i)}||N^{-1}\Tr\Sigma R (I+\widetilde{m}^{(i)}\Sigma)^{-1}\Sigma(I+\widetilde{m} \Sigma)^{-1}|\\
    &\leq|\widetilde{m} -\widetilde{m}^{(i)}|\|\Sigma\|_\op^2\|(I+\widetilde{m}^{(i)}\Sigma)^{-1}\|_\op\|(I+\widetilde{m} \Sigma)^{-1}\|_\op \|R \|_\op\prec N^{-1}\eta^{-6}.
\end{align*}
Since $\eta=1$, we have $|d_i(I)|\prec N^{-1/2}$. Thus, by Lemma \ref{lem:Approx_fixed_point},
\begin{align*}
    z_0(\widetilde{m})-z=-\frac{1}{\widetilde{m}}\cdot\frac{1}{N}\sum_{i=1}^N\frac{\Oprec(N^{-1/2})}{1+N^{-1}\g_i^* R^{(i)}\g_i}=\Oprec(N^{-1/2}).
\end{align*}
For any small $\eps>0$ and large $D>0$, there exists $N_0(\eps,D)$ such that for all $N\geq N_0$,
\[
    \P\left(|z_0(\widetilde{m})-z|\leq N^{\eps}N^{-1/2}\right)\geq 1-N^{-D}.
\]
Since $N^{\eps}N^{-1/2}$ satisfies the requirements for $\Delta$ in Definition
\ref{def:regular}, on this event,
\[
    |\widetilde{m} -\widetilde{m}_0|\leq\frac{C
N^{\eps}N^{-1/2}}{\sqrt{\kappa+\eta}+N^{\eps/2} N^{-1/4}}
\leq CN^{\eps/2}N^{-1/4}.
\]
Since $\eps$ is arbitrary, $\Theta\prec N^{-1/4}$. By Lemma \ref{lem:prec bounds} with $\eta=1$, uniformly over distinct $i,j\in[N]$,
\[
    |\widetilde{R}_{ii}-\widetilde{m}_0|\leq |\widetilde{R}_{ii}-\widetilde{m}|+\Theta\prec\Psi_\Theta+\Theta\prec N^{-1/4},\quad |\widetilde{R}_{ij}|\prec\Psi_\Theta\prec N^{-1/2}.
\]
\end{proof}

The following lemma now establishes a weak local law over all $z \in \bD$
using a stochastic continuity argument (see e.g.\ \cite[Lemma
6.12]{pillai2014universality}).

\begin{lemma}\label{lem:weak local law}
Suppose Assumptions \ref{assum:basic} and \ref{assum:concentration} hold, and
$\bD$ is a regular spectral domain. Then uniformly over $z=E+i\eta\in\bD$,
we have $\Lambda(z)\prec (N\eta)^{-1/4} \leq N^{-\tau/4}$.
\end{lemma}
\begin{proof}
For any $z=E+i\eta\in\bD$, let $L(z)$ be as in Definition \ref{def:regular}.
Then $M=|L(z)|\leq N^5$. Define $z_0=E+i$, $z_M=z$, and $z_k:=E+i\eta_k=E+i(1-kN^{-5})$ for $1\leq k\leq M-1$. For a small enough $\eps>0$, introduce the events
\[
    \Omega_k:=\Omega(z_k):=\left\{|\widetilde{m}(z_k)-\widetilde{m}_0(z_k)|\leq\frac{CN^{2\eps}(N\eta_k)^{-1/2}}{\sqrt{\kappa+\eta_k}+N^\eps(N\eta_k)^{-1/4}}\right\},
\]
\[
    \cE_k:=\cE(z_k) := \left\{\Lambda(z_k)\leq N^{2\eps}(N\eta_k)^{-1/4}\right\},
\]
where $\eta_k=\Im z_k$. Choose a large enough $D>0$. We proceed with a
stochastic continuity argument. For $k=0$, by Lemma \ref{lem:eta=1_error} and
Definition \ref{def:regular}, we have
\[
    \P((\Omega_0\cap\cE_0)^c)\leq N^{-D}.
\]
For any $k\geq 1$, since $\Lambda$ is $2N^2$-Lipschitz continuous over $z \in
\bD$, and $|\eta_k - \eta_{k-1}| = N^{-5}$, on the event
$\Omega_{k-1}\cap\cE_{k-1}$ we have $\Lambda(z_k) \leq
N^{2\eps}(N\eta_{k-1})^{-1/4} + 2N^{2-5} \leq N^{2\eps-\tau/4}+2N^{-3}$,
implying that for $C_0>5$, small enough $\eps>0$, and large enough $N$,
the event $\Xi(z_k)=\{\Lambda(z_k) \leq N^{-\tau/C_0}\}$ holds. Thus, by Lemma \ref{lem:fixed_point_error},
\[
    \P\left(\Omega_{k-1}\cap\cE_{k-1}\cap\left\{|z_0(\widetilde{m}(z_k))-z_k|\geq N^{2\eps}(N\eta_k)^{-1/2}\right\}\right)\leq N^{-D},
\]
where we used $\Psi_\Theta(z_k)\1_{\Xi(z_k)}\prec (N\eta_k)^{-1/2}$. Since the
function $\Delta(w)=N^{2\eps}(N\Im w)^{-1/2}$ satisfies the requirements of
Definition \ref{def:regular}, it follows that
\[
    \P\left(\Omega_{k-1}\cap\cE_{k-1}\cap\left\{|\widetilde{m}(z_k)-\widetilde{m}_0(z_k)|\geq
\frac{CN^{2\eps}(N\eta_k)^{-1/2}}{\sqrt{\kappa+\eta_k}+N^\eps(N\eta_k)^{-1/4}}\right\}\right)\leq N^{-D},
\]
which implies $\P(\Omega_{k-1}\cap\cE_{k-1}\cap\Omega_k^c)\leq N^{-D}$. On the other hand, by Lemma \ref{lem:prec bounds}, on the event $\Omega_{k-1}\cap\cE_{k-1}\cap\Omega_k$ we have
\[
    \Lambda(z_k)\1_{\Omega_{k-1}\cap\cE_{k-1}\cap\Omega_k} \prec
(\Psi_\Theta(z_k)+\Theta(z_k))\1_{\Omega_{k-1}\cap\cE_{k-1}\cap\Omega_k} \prec N^\eps(N\eta_k)^{-1/4},
\]
which gives $\P(\Omega_{k-1}\cap\cE_{k-1}\cap\Omega_k\cap\cE_k^c) \leq N^{-D}$ for sufficiently large $N$. Therefore,
\begin{align*}
    &\P\left(\Omega_{k-1}\cap\cE_{k-1}\cap\left(\Omega_k\cap\cE_k\right)^c\right)\\
    &\leq\P(\Omega_{k-1}\cap\cE_{k-1}\cap\Omega_k^c)+\P\left(\Omega_{k-1}\cap\cE_{k-1}\cap\Omega_k\cap\cE_k^c\right)\leq
2N^{-D}.
\end{align*}
In general, for any $0\leq k\leq M$,
\begin{align*}
    \P((\Omega_k\cap\cE_k)^c)&\leq \P((\Omega_0\cap\cE_0)^c)+\sum_{j=1}^k\P\left(\Omega_{j-1}\cap\cE_{j-1}\cap\left(\Omega_j\cap\cE_j\right)^c\right)\\
    &\leq N^{-D}+2kN^{-D}\leq CN^{5}N^{-D}=CN^{5-D}.
\end{align*}
Since $D$ is arbitrary, this proves the claim.
\end{proof}

\begin{remark}
Let us take $C_0>5$ in the definition $\Xi=\{\Lambda \leq N^{-\tau/C_0}\}$.
Then this implies that $\1_\Xi \prec 1$ uniformly
over $z \in \bD$, so all preceding estimates hold with $\1_\Xi$ replaced by 1.
\end{remark}

\begin{proof}[Proof of Theorem \ref{thm:entrywise law}]
Applying Lemma \ref{lem:fixed_point_error} gives
\begin{equation}\label{eq:z0zdiff}
    z_0(\widetilde{m})-z=\Oprec(\Psi_\Theta^2)+\Oprec\left(\max_{0\leq k\leq l}\left|\frac{1}{N}\sum_{i=1}^N \frac{N^{-1}\Tr(\g_i\g_i^*-\Sigma)R^{(i)}\Sigma^k(I+\widetilde{m}_0\Sigma)^{-k-1}}{1+N^{-1}\g_i^*R^{(i)}\g_i}\right|\right).
\end{equation}
Now suppose for some constant $c \in (0,1]$ we have $\Theta \prec (N\eta)^{-c}$. Then
\[\Psi_\Theta\prec\sqrt{\frac{\Im{\widetilde{m}_0}+(N\eta)^{-c}}{N\eta}}:=\Phi.
\]
Note that Definition \ref{def:regular} implies $\eta \lesssim \Im \widetilde
m_0 \lesssim 1$, so $N^{-1/2} \lesssim \Phi \lesssim N^{-\tau/2}$.
By Lemma \ref{lem:prec bounds}, uniformly over all $S\subseteq[N]$ with $|S|\leq
L$ a fixed constant, we have $N^{-1}\|R^{(S)}\|_F \prec \Phi$, and also $\Gamma^{(S)}=\max_i |\widetilde R_{ii}^{(S)}-\widetilde m^{(S)}| \prec \Psi_\Theta+\Theta \prec N^{-c\tau/2}$.
Thus, by the fluctuation averaging result of
Lemma \ref{lem:fluctuation_averaging}, for all $0\leq k\leq l$,
\begin{align*}
    \left|\frac{1}{N}\sum_{i=1}^N
\frac{N^{-1}\Tr(\g_i\g_i^*-\Sigma)R^{(i)}\Sigma^k(I+\widetilde{m}_0\Sigma)^{-k-1}}{1+N^{-1}\g_i^*
R^{(i)}
\g_i}\right|&\prec\|\Sigma\|_\op^k\|(I+\widetilde{m}_0\Sigma)^{-1}\|_\op^{k+1}\Phi^2 \prec \Phi^2.
\end{align*}
This implies $|z_0(\widetilde{m})-z| \prec \Phi^2$. Then by the stability
condition of Definition \ref{def:regular} and the bound $\Im \widetilde m_0(z)
\leq C\sqrt{\kappa+\eta}$,
\begin{align*}
    \Theta=|\widetilde{m}-\widetilde{m}_0|\prec\frac{\frac{\Im{\widetilde{m}_0}+(N\eta)^{-c}}{N\eta}}{\sqrt{\kappa+\eta}
+\sqrt{\frac{\Im{\widetilde{m}_0}+(N\eta)^{-c}}{N\eta}}}&\prec\frac{1}{N\eta}+\frac{(N\eta)^{-c-1}}{(N\eta)^{-c/2-1/2}}
\prec (N\eta)^{-c/2-1/2}.
\end{align*}
Thus, we have shown the implication
$\Theta\prec (N\eta)^{-c} \Rightarrow \Theta\prec (N\eta)^{-c/2-1/2}$.
For any $\eps>0$, starting with $c=1/4$ and iterating this a constant $C_\eps$ number of times yields $\Theta\prec (N\eta)^{-1+\eps}$. Since $\eps>0$ is arbitrary, $\Theta\prec (N\eta)^{-1}$. Substituting into $\Psi_\Theta$ gives
\[\Psi_\Theta\prec\sqrt{\frac{\Im{\widetilde{m}_0}}{N\eta}}+\frac{1}{N\eta}:=\Psi.\]
Then by Lemma \ref{lem:prec bounds},
$|\widetilde{R}_{ii}-\widetilde{m}_0|\prec\Theta+\Psi_\Theta\prec\Psi$ and
$|\widetilde{R}_{ij}|\prec\Psi$, which shows \eqref{eq:entry law in bulk edge}.
For \eqref{eq:average law in bulk edge}, applying the above bounds with $c=1$
gives $|z_0(\widetilde{m})-z|\prec\Psi^2$,
so applying Definition \ref{def:regular} once more gives
\[|\widetilde{m}-\widetilde{m}_0|\prec\frac{\Psi^2}{\sqrt{\kappa+\eta}+\Psi}.\]
Since $\widetilde m-\widetilde m_0=\gamma(m-m_0)$ and $\gamma \asymp 1$,
the same bound holds for $m-m_0$, so this shows
\eqref{eq:average law in bulk edge}. Since all quantities in
\eqref{eq:entry law in bulk edge} and
\eqref{eq:average law in bulk edge} are $N^2$-Lipschitz over $z \in \bD$,
taking a union bound over a covering net shows that with probability
$1-CN^{-D}$, these statements hold simultaneously for every $z \in \bD$.
\end{proof}

Corollary \ref{coro:coro of entrywise law} now follows from known arguments,
which we summarize briefly here. The following lemma is a version of
\cite[Lemmas 10.3, 10.4]{knowles2017anisotropic} (see also the asymptotic exact
separation result of \cite{bai1999exact}),
clarifying that no bulk/edge regularity is needed for this statement, and that
it holds also for values $a \in (0,\infty)$ below the smallest left edge.

\begin{lemma}\label{lem:exactseparation}
Suppose Assumptions \ref{assum:basic} and \ref{assum:concentration} hold.
Let $\lambda_1,\ldots,\lambda_N$ be the eigenvalues of $\widetilde K$, and for
each $a \in (0,\infty)$, let
\begin{align}\label{eq:exactseparation}
\widehat N(a)=\sum_{i=1}^N \1\{\lambda_i \geq a\},
\qquad N(a)=N\int_a^\infty \rho_0(x)dx.
\end{align}
Then for any constants $\delta,D>0$, there exists a constant
$C \equiv C(\delta,D)>0$ such that if $a \in (0,\infty)$ satisfies
$(a-\delta,a+\delta) \cap \supp(\widetilde \mu_0)=\emptyset$, then with
probability at least $1-CN^{-D}$,
\[\widehat N(a)=N(a).\]
\end{lemma}
\begin{proof}
The averaged local law \eqref{eq:average law in bulk edge} applied
for $z=E+i\eta\in \bD^o(\delta,\tau)$ with any fixed constant
$\tau \in (0,1/2)$ implies $|\widetilde m(z)-\widetilde m_0(z)|
\ll (N\eta)^{-1}$. Then, since $\Im \widetilde m_0 \lesssim \eta
\ll (N\eta)^{-1}$ by \eqref{eq:regularbounds} and regularity of 
$\bD^o(\delta,\tau)$ (c.f.\ Lemma \ref{lem:properties of m}),
we have $\sum_{i=1}^N \1\{\lambda_i \in [E-\eta,E+\eta]\}
\leq 2N\eta \Im \widetilde m(z) \ll 1$, so
$\widetilde K$ has no eigenvalue in $[E-\eta,E+\eta]$. Thus, 
(\ref{eq:average law in bulk edge}) implies for some constant
$C \equiv C(\delta,D)>0$, with probability $1-CN^{-D}$,
\begin{equation}\label{eq:dichotomy}
\widetilde K \text{ has no eigenvalues in }
\{x \in \R:|x| \leq \delta^{-1},\dist(x,\supp(\widetilde{\mu}_0) \cup \{0\})
\geq \delta\}.
\end{equation}

To show (\ref{eq:exactseparation}), we begin with the case where $G$ is
Gaussian, i.e.\ $\widetilde K=N^{-1}X^*\Sigma X$ where $X \in \R^{n \times N}$
has i.i.d.\ $N(0,1)$ entries and assume without loss of generality $\Sigma=\diag(\sigma_1,\ldots,\sigma_n)$ with $\sigma_1 \geq \ldots \geq
\sigma_n$. Let $m_1,\ldots,m_{2p}$ and $x_1>x_2 \geq x_3>
\ldots>x_{2p} \geq 0$
be as defined in \eqref{eq:edges}. Note that if $a>x_1$ ($a$ is above the largest edge), then $N(a)=0$,
and \eqref{eq:dichotomy} together with 
$\|\widetilde K\|_\op \leq N^{-1}\|X\|_\op^2\|\Sigma\|_\op$ and the high-probability
bound $\|X\|_\op \leq C\sqrt{N}$ \cite[Corollary 5.35]{vershynin2010introduction}
shows $\widehat N(a)=N(a)=0$ with
probability $1-C(\delta,D)N^{-D}$. 

If $a \in (x_{2k+1},x_{2k})$ for some $k \in \{1,\ldots,p-1\}$ ($a$ is
between two separated bulk components), then \cite[Lemma A.1]{knowles2017anisotropic} shows that
$N(a)=\sum_{\alpha=1}^n\1\{\sigma_\alpha \neq 0,\,{-}\sigma_\alpha^{-1} \in
(m_{2k},0)\}$, i.e.\ $N(a)$ counts the number of non-zero poles of $z_0(m)$ to
the right of $m_{2k}$ with multiplicity. Let $\Sigma=\diag(\Sigma_1,\Sigma_2)$ where
$\Sigma_1=\diag(\sigma_1,\ldots,\sigma_{N(a)})$ contains the eigenvalues satisfying ${-}\sigma_\alpha^{-1} \in (m_{2k},0)$, and
$\Sigma_2=\diag(\sigma_{N(a)+1},\ldots,\sigma_n)$ contains the rest.
Define, for $t \in [0,1]$,
\[\Sigma(t)=\diag\big((1-t)\Sigma_1+t\sigma_1I,(1-t)\Sigma_2\big)
=:\diag\big(\sigma_1(t),\ldots,\sigma_n(t)\big),\]
which interpolates between $\Sigma(0)=\Sigma$ and
$\Sigma(1)=\diag(\underbrace{\sigma_1,\ldots,\sigma_1}_{N(a)},\underbrace{0,\ldots,0}_{N-N(a)})$, and let
\[\widetilde K(t)=N^{-1}X^*\Sigma(t)X \text{ with eigenvalues }
\lambda_1(t),\ldots,\lambda_N(t).\]
Observe that
\[\frac{d}{d\sigma_\alpha}z_0'(m)={-}\frac{1}{N}\frac{d}{d\sigma_\alpha}
\frac{1}{(\sigma_\alpha^{-1}+m)^2}\]
which is positive when ${-}\sigma_\alpha^{-1}>m$ and negative when
${-}\sigma_\alpha^{-1}<m$. Then, writing $z_0(m;t)$ and $z_0'(m;t)$ for
$z_0(m)$ and $z_0'(m)$ defined by $\Sigma(t)$,
we see that $z_0'(m;t)$ is increasing in $t$ for each
$m \in (m_{2k+1},m_{2k})$. This implies that there remain two distinct
critical points of $z_0(m;t)$ between the poles $-\sigma_{N(a)}(t)^{-1}$
and $-\sigma_{N(a)+1}(t)^{-1}$ for all $t \in (0,1)$. Writing
$m_{2k+1}(t)<m_{2k}(t)$ for these two critical points, we have
\[\frac{d}{dt}\Big[z_0(m_{2k}(t);t)-z_0(m_{2k+1}(t);t)\Big]
=\frac{1}{N}\sum_{\alpha=1}^n
\bigg(\frac{1}{(1+\sigma_\alpha(t)m_{2k}(t))^2}
-\frac{1}{(1+\sigma_\alpha(t)m_{2k+1}(t))^2}\bigg)
\frac{d}{dt} \sigma_\alpha(t).\]
Since
$\frac{1}{(1+\sigma_\alpha(t)m_{2k}(t))^2}-\frac{1}{(1+\sigma_\alpha(t)m_{2k+1}(t))^2}>0$
and $\frac{d}{dt}\sigma_\alpha(t)>0$ for $\alpha=1,\ldots,N(a)$, and similarly
both are negative for $\alpha=N(a)+1\ldots,n$, we have
$\frac{d}{dt}[z_0(m_{2k}(t);t)-z_0(m_{2k+1}(t);t)]>0$. Then
\begin{equation}\label{eq:gapremains}
z_0(m_{2k}(t);t)-z_0(m_{2k+1}(t);t)
>z_0(m_{2k}(0);0)-z_0(m_{2k+1}(0);0)>2\delta
\end{equation}
for all $t \in (0,1)$, by the starting assumption that
$(a-\delta,a+\delta) \cap \supp(\widetilde \mu_0)=\emptyset$.
Note that $\Sigma(t)$ satisfies Assumption \ref{assum:basic} for all $t \in
[0,1-c]$ and any constant $c>0$. Then \eqref{eq:dichotomy} holds
for each matrix $\widetilde K(t)$ and $t \in [0,1-c]$. Choosing $c \equiv
c(\delta)>0$ sufficiently small, the operator norm bound $\|X\|_\op \leq C\sqrt{N}$
and a standard covering
argument show that with probability $1-CN^{-D}$,
\eqref{eq:dichotomy} holds simultaneously for all $\{\widetilde K(t):t \in
[0,1]\}$. On this event, \eqref{eq:dichotomy} and
\eqref{eq:gapremains} imply
\begin{equation}\label{eq:hatNainterpolation}
\widehat N(a)=
\sum_{i=1}^N \1\Big\{\lambda_i(0) \geq z_0(m_{2k}(0);0)-\delta\Big\}
=\sum_{i=1}^N \1\Big\{\lambda_i(1) \geq z_0(m_{2k}(1);1)-\delta\Big\}
\end{equation}
where $m_{2k}(1)=\lim_{t \to 1} m_{2k}(t)$.
At $t=1$, denoting $Y \in \R^{N(a) \times N}$ as the
first $N(a)$ rows of $X$, we have
$\widetilde K(1)=\sigma_1 \cdot N^{-1}Y^*Y$,
and $z_0(m_{2k}(1);1)$ is the lower edge of the
single bulk component of the standard Marchenko-Pastur law, given by
$z_0(m_{2k}(1);1)=\sigma_1(1-\sqrt{N(a)/N})^2$. Then
\eqref{eq:gapremains} implies that $N(a) \leq
(1-c(\delta))N$ for some constant $c(\delta)>0$, and by concentration of the
smallest singular value of $Y$ \cite[Corollary 5.35]{vershynin2010introduction}, the right side of
\eqref{eq:hatNainterpolation} is exactly equal to $N(a)$ with probability
$1-CN^{-D}$.

Finally, if $a \in (0,x_{2p})$ ($a$ is below the smallest left edge), then \cite[Lemma A.1]{knowles2017anisotropic} shows $N(a)=\min(N,\bar{n})$ where we denote $\bar{n}=\rank(\Sigma)$. Write $\Sigma=\diag(\Sigma_1,0)$ where $\Sigma_1 \in \R^{\bar n \times \bar n}$, and consider the interpolation
\[\Sigma(t)=\diag\big((1-t)\Sigma_1+t\sigma_1 I,0\big).\]
For the (unique) critical point
$m_{2p}(t) \in (-\infty,-\sigma_{\bar n}(t)^{-1}) \cup (0,\infty]$, the above argument establishes $\frac{d}{dt}z_0(m_{2p}(t);t)>0$ and hence
\begin{equation}\label{eq:gapremains2}
z_0(m_{2p}(t);t)>z_0(m_{2k}(0);0)>\delta.
\end{equation}
Hence, on the event where \eqref{eq:dichotomy} holds for all $\{\widetilde K(t):t \in [0,1]\}$, we have that \eqref{eq:hatNainterpolation} holds with $k=p$. Here $\widetilde K(1)=\sigma_1 \cdot N^{-1}Y^*Y$ where $Y \in \R^{\bar n \times N}$ contains the first $\bar n$ rows of $X$, $z_0(m_{2p}(1);1)=\sigma_1(1-\sqrt{\bar n/N})^2$, and \eqref{eq:gapremains2} implies that $\bar n \notin (1 \pm c(\delta))N$ for some $c(\delta)>0$. Then again, the right side of \eqref{eq:hatNainterpolation} is equal to $N(a)=\min(N,\bar n)$ with probability $1-CN^{-D}$. This shows $\widehat N(a)=N(a)$ in all cases.

To extend this to the general non-Gaussian setting, let
$\widetilde K=N^{-1}G^*G$ where $G$ satisfies Assumptions \ref{assum:basic} and
\ref{assum:concentration}, let $\widetilde G \in \R^{n \times N}$ have columns
$\{\widetilde{\g}_i\}_{i=1}^N$ that are i.i.d.\ $N(0,\Sigma)$ and independent
from $G$, and consider the interpolation for $t \in
\{0,1/N,2/N,\ldots,1\}$
\[G(t)=\sqrt{t}G+\sqrt{1-t}\,\widetilde{G},\qquad
\widetilde{K}(t)=N^{-1}G(t)^*G(t).\]
Then $G(t)$ satisfies
Assumption \ref{assum:concentration} uniformly over all such $t$, so
\eqref{eq:dichotomy} holds with probability
$1-CN^{-D}$ for each such matrix $\widetilde K(t)$. It is easy to check that
for a constant $C>0$ and each $l=1,\ldots,N$,
\[\|\widetilde{K}(l/N)-\widetilde{K}((l-1)/N)\|_\op\leq CN^{-1/2}\max_{0\leq
s\leq N}\|\widetilde{K}(s/N)\|_\op,\]
and we have $\|\widetilde{K}(s/N)\|_\op \leq N^\eps$ for any
$\eps>0$ and each $s=0,1,\ldots,N$ with probability $1-C(\eps,D)N^{-D}$ by
Lemma \ref{lem:concentration}(b). Then on the intersection of these events,
\eqref{eq:exactseparation} for the Gaussian matrix
$\widetilde K(1)$ and \eqref{eq:dichotomy} for each matrix
$\widetilde K(l/n)$ implies \eqref{eq:exactseparation} for the original matrix
$\widetilde K=\widetilde K(0)$, completing the proof.
\end{proof}

\begin{proof}[Proof of Corollary \ref{coro:coro of entrywise law}]
For (a), by \eqref{eq:dichotomy}
and the implication of Lemma \ref{lem:exactseparation}
that $\widehat N(a)=N(a)=0$ for $a>\max(x \in \supp(\widetilde \mu_0))+\delta$,
we know that with probability $1-C(\delta,D)N^{-D}$,
$K$ and $\widetilde K$ have no eigenvalues outside $\{x \in
\R:\dist(x,\supp(\widetilde \mu_0 \cup \{0\})) \leq \delta\}$. If there is a
constant $\delta>0$ for which $\dist(0,\supp(\widetilde \mu_0))>\delta$, then
applying Lemma \ref{lem:exactseparation} with $a=\delta/2$ shows
$\widehat N(a)=N(a)=N$, and thus $\widetilde K$ also has no eigenvalues in
$[0,\delta/2]$. Similarly, if $\dist(0,\supp(\mu_0))>\delta$, then
$N \geq n$ and $\widetilde \mu_0=(n/N)\mu_0+(1-n/N)\delta_0$, so
applying Lemma \ref{lem:exactseparation} with $a=\delta/2$ shows
$\widehat N(a)=N(a)=n$. Then $K$ has full rank with no eigenvalues
in $[0,\delta/2]$. This shows all statements of (a).

For (b), the averaged local law \eqref{eq:average law in bulk edge} applied for
$z=E+i\eta\in \bD_j^e(\delta,\tau)$ with $\kappa=E-x_{2k-1} \geq N^{-2/3+\eps}$
and $\eta=N^{-1/2-\eps/4}\kappa^{1/4} \geq N^{-2/3}$ implies
$|\widetilde m(z)-\widetilde m_0(z)| \ll (N\eta)^{-1}$ (see
e.g.\ \cite[Proof of (3.4), Step 1]{pillai2014universality}). Then again
$\sum_{i=1}^N \1\{\lambda_i \in [E-\eta,E+\eta] \leq 2N\eta\Im \widetilde m(z)
\ll 1$, so $\widetilde K$ has no eigenvalue in $[E-\eta,E+\eta]$.
Applying this over a lattice of such values $z \in \bD_j^e(\delta,\tau)$ shows
(b).

For (c), if $\bD=\{z \in \C^+:E \in [E_1,E_2],\,\eta \in [N^{-1+\tau},1]\}$
is a regular domain, then applying \eqref{eq:average law in bulk edge} for
sufficiently small $\tau \equiv \tau(\eps) \in (0,1)$,
the argument of \cite[Proof of (3.4), Step 2]{pillai2014universality} implies
that for any $a,b \in [E_1,E_2]$,
\[|(\widehat N(b)-\widehat N(a))-(N(b)-N(a))| \leq C(\log N)N^\eps.\]
Thus, together with \eqref{eq:exactseparation}, this shows that if an edge $x_j$
is such that $\bD_j^e(\delta,\tau)$ is regular, then with probability
$1-CN^{-D}$ for some $C \equiv C(\delta,\eps,D)>0$,
$|\widehat N(a)-N(a)| \leq C(\log N)N^\eps$
for each $a \in [x_j-\delta,x_j+\delta]$. Furthermore if $\bD_k^b(\delta,\tau)$
is regular and either $\bD_{2k}^e(\delta,\tau)$ or
$\bD_{2k-1}^e(\delta,\tau)$ is regular, then this holds also
for each $a \in [x_{2k}+\delta,x_{2k-1}-\delta]$. Together with the square-root
decay of $\rho_0(x)$ at regular edges, this implies part (c)
(see e.g.\ \cite[Proof of (3.6)]{pillai2014universality}).

For (d), writing the spectral decomposition $\widetilde K=\sum_{i=1}^N \lambda_i\widetilde \x_i\widetilde \x_i^*$,
for any $j \in [N]$, we have
$\Im \widetilde R_{jj}=\sum_{i=1}^N |\widetilde \x_i[j]|^2 \eta/((\lambda_i-E)^2+\eta^2)$. If $\bD_j^e(\delta,\eps)$ is regular, then for each eigenvalue $\lambda_i \in [x_j-\delta,x_j+\delta]$, applying \eqref{eq:entry law in bulk edge} at $z=\lambda_i+iN^{-1+\eps}$ shows
$|\widetilde \x_i[j]|^2 \leq N^{-1+\eps}\Im \widetilde R_{jj} \prec N^{-1+\eps}$. The same argument holds for a regular bulk domain $\bD_k^b(\delta,\eps)$, showing part (d).
\end{proof}

\subsection{Proof of anisotropic law (Theorems \ref{thm:anisotropic law})}

We now prove Theorem \ref{thm:anisotropic law}
on anisotropic approximations for the linearized
resolvent $\Pi(z)$. Note that for $z\in\C^+$,
\[
    \Pi(z)=\begin{bmatrix}
        -zI_n & \frac{1}{\sqrt{N}}G \\
        \frac{1}{\sqrt{N}}G^* & -I_N
    \end{bmatrix}^{-1}=\begin{bmatrix}
        R(z) & \Pi_o(z)^* \\
        \Pi_o(z) & z\widetilde{R}(z)
    \end{bmatrix}\in\C^{(n+N)\times(n+N)}.
\]
Indexing $[N]$ by Roman indices $i,j,\ldots$ and $[n]$ by Greek indices
$\alpha,\beta,\ldots$, the
following resolvent identities hold by Schur complement:
\begin{equation}\label{eq:linear block}
    (\Pi(z))_{ij}=z\widetilde{R}_{ij}(z),\quad (\Pi(z))_{\alpha\beta}=R_{\alpha\beta}(z),\quad (\Pi(z))_{i\alpha}=(\Pi_o(z))_{i\alpha}=z\widetilde{R}_{ii}(z)\g_i^* R^{(i)}(z)\e_{\alpha}.
\end{equation}
We will first show an anisotropic local law for the upper-left block
$R(z)$, and then use this to show the result for
$z\widetilde{R}(z)$ and $\Pi_o(z)$.

Fix two constants $C_0>0$ and $\delta \in (0, \tau/2C_0)$.
Following the same argument as in \cite{knowles2017anisotropic},
we will bootstrap on the spectral scale in multiplicative increments of
$N^{-\delta}$: For any $\eta \geq N^{-1+\tau}$, define
\[
    L \equiv L(\eta)=\max\{ l \in \N : \eta N^{\delta(l-1)} < 1 \},
\qquad \eta_l = \eta N^{\delta l} \text{ for }
l = 0, \dots, L-1, \qquad \eta_L = 1.\]
Note that $L \leq \delta^{-1} + 1$. The following lemma, applying the
fluctuation averaging result of Lemma
\ref{lem:fluctuation_averaging_lowrank2}(a),
is the main input for the bootstrap argument.

\begin{lemma}\label{lem:self-consistent moment bound}
Suppose Assumptions \ref{assum:basic}, \ref{assum:concentration},
and \ref{assum:cumulant} hold, and $\bD$ is a regular spectral domain.
Suppose there exists a deterministic function $\Phi:\bD \to
[N^{-1/2},N^{-\tau/2}]$ and a constant $\delta > 0$
such that uniformly over all deterministic unit vectors $\x, \y \in \C^n$,
\[\frac{\|R\x\|_2}{\sqrt{N}} \prec \Phi, \qquad |\x^* R \y| \prec N^{2\delta}.\]
Then, uniformly over all deterministic unit vectors $\u \in \C^n$ and
$z \in \bD$,
\[|\u^* R \u - \u^* (-zI - z \widetilde{m}_0 \Sigma)^{-1} \u| \prec
\max(N^{2\delta}\Phi,\Psi)\]
where $\Psi$ is the bound in Theorem \ref{thm:entrywise law}.
\end{lemma}
\begin{proof}
We have
\begin{align*}
    &|\u^*R\u-\u^*(-z-z\widetilde{m}_0\Sigma)^{-1}\u|\\
    &\leq|\underbrace{\u^*R\u-\u^*(-z-z\widetilde{m}\Sigma)^{-1}\u}_{I}|+|\underbrace{\u^*(-z-z\widetilde{m}\Sigma)^{-1}\u-\u^*(-z-z\widetilde{m}_0\Sigma)^{-1}\u}_{II}|
\end{align*}
For $II$, by the identity $A^{-1}-B^{-1}=A^{-1}(B-A)B^{-1}$
\begin{align*}
    |II|&=|z||\widetilde{m}-\widetilde{m}_0||\u^*(-z-z\widetilde{m}\Sigma)^{-1}\Sigma(-z-z\widetilde{m}_0\Sigma)^{-1}\u|\\
    &\prec\Psi\|(-z-z\widetilde{m}\Sigma)^{-1}\|_\op\|(-z-z\widetilde{m}_0\Sigma)^{-1}\|_\op\prec\Psi,
\end{align*}
where we applied Lemma \ref{lem:prec bounds} and Theorem \ref{thm:entrywise
law}. For $I$, by Lemma
\ref{lem:Approx_fixed_point} and Lemma \ref{lem:alternative to L1 bound},
there exists a constant $l\geq 1$ such that
\begin{align*}
    &|\u^*R\u-\u^*(-z-z\widetilde{m}\Sigma)^{-1}\u|\\
&\prec\max_{0\leq k\leq l}
\left|\frac{1}{N}\sum_{i=1}^N\frac{\u^*\Sigma^k(I+\widetilde{m}_0\Sigma)^{-k-1}(\g_i\g_i^*-\Sigma)R^{(i)}\u}{1+N^{-1}\g_i^*
R^{(i)}\g_i}\right|+\max(\Phi^2,\Psi),
\end{align*}
where we have used the given condition and Lemma \ref{lem:prec bounds}
to bound the remainder in Lemma \ref{lem:alternative to L1 bound} as
$N^{-1/2}\|R^{(i)}\u\|_2 \prec N^{-1/2}\|R\u\|_2 \prec \Phi$
and similarly $N^{-1/2}\|\u^*\Sigma^{k+1}(I+\widetilde
m_0\Sigma)^{-k-1}R^{(i)}\|_2 \prec \Phi$.

Note that $\u^*\Sigma^k(I+\widetilde{m}_0\Sigma)^{-k-1}$
is a deterministic vector bounded in $\ell_2$-norm, and the conditions of Lemma \ref{lem:fluctuation_averaging_lowrank2} are satisfied by the given assumptions $N^{-1/2}\|R\x\|_2 \prec \Phi$, $|\x^* R\y| \prec N^{2\delta}$, and the conditions $\Gamma^{(S)} \prec \Psi$ and $N^{-1}\|R^{(S)}\|_1 \prec 1$ in light of Lemma \ref{lem:prec bounds} and Theorem \ref{thm:entrywise law}. Thus
the first term is $\Oprec(N^{2\delta}\Phi)$ by the
polarization identity and fluctuation averaging result of
Lemma \ref{lem:fluctuation_averaging_lowrank2}(a), concluding the proof.
\end{proof}

\begin{lemma}\label{lem:prec bounds on generalized entries}
Under Assumptions \ref{assum:basic} and \ref{assum:concentration}, uniformly over all $z = E + i\eta \in \bD$ and all deterministic unit vectors $\x, \y \in \C^n$,
\[
    |\x^* R(z) \y| \prec N^{2\delta} \sum_{l=1}^{L(\eta)} [\Im \x^* R(E + i\eta_l) \x + \Im \y^* R(E + i\eta_l) \y].
\]
\end{lemma}
\begin{proof}
Writing the spectral decomposition
$K=\sum_{\alpha=1}^n \lambda_\alpha \v_\alpha\v_\alpha^*$,
\[
    |\x^* R(z) \y| \leq \sum_{\alpha} \frac{|\x^* \v_\alpha| |\y^* \v_\alpha|}{|\lambda_\alpha - z|} \leq \sum_{\alpha} \frac{|\x^* \v_\alpha|^2}{|\lambda_\alpha - z|} + \sum_{\alpha} \frac{|\y^* \v_\alpha|^2}{|\lambda_\alpha - z|}.
\]
Define $\eta_{-1}=0$, $\eta_{L+1}=\infty$, and the index sets
$U_l=\{\alpha: \eta_{l-1} \leq |\lambda_\alpha - E| < \eta_l\}$ for
$l=0,\ldots,L+1$. Then the same argument as leading to \eqref{eq:RL1bound},
using $\eta_l/\eta_{l-1} \leq N^\delta$ in place of $\eta_l/\eta_{l-1} \leq 2$,
shows
\[\sum_{\alpha} \frac{|\x^* \v_\alpha|^2}{|\lambda_\alpha - z|} = \sum_{l=0}^{L+1} \sum_{\alpha \in U_l} \frac{|\x^* \v_\alpha|^2}{|\lambda_\alpha - z|}
\prec N^{2\delta}\sum_{l=1}^{L(\eta)} \Im \x^* R(E+i\eta_l)\x,\]
and the analogous bound holds for $\y$.
\end{proof}

\begin{lemma}\label{lem:bootstrap}
Suppose Assumptions \ref{assum:basic}, \ref{assum:concentration}, and
\ref{assum:cumulant} hold, and $\bD$ is a regular spectral domain. 
For $k \in \N$, define the domain
$\bD_k=\{ z \in \bD: \Im z \in [N^{-k\delta},1]\}$.
Let $\Psi$ be the bound in Theorem \ref{thm:entrywise law}.
Then the following holds:
\begin{enumerate}[label=(\alph*)]
    \item Uniformly over all $z \in \bD_0$ and all deterministic unit vectors $\u \in \C^n$,
    \[
        |\u^* R \u - \u^* (-z I - z \widetilde{m}_0 \Sigma)^{-1} \u| \prec N^{C_0 \delta} \Psi,
    \]
    \item For all $0 \leq k \leq \delta^{-1}$, if uniformly over all $z \in
\bD_k$ and all deterministic unit vectors $\u \in \C^n$,
    \[
        |\u^* R \u - \u^* (-z I - z \widetilde{m}_0 \Sigma)^{-1} \u| \prec N^{C_0 \delta} \Psi,
    \]
    then uniformly over all $z \in \bD_k$ and all deterministic unit vectors $\u \in \C^n$,
    \[
        \Im \u^* R \u \prec \Im \widetilde{m}_0 + N^{C_0 \delta} \Psi.
    \]
    \item For all $0 \leq k \leq \delta^{-1}$, if uniformly over all $z \in
\bD_k$ and all deterministic unit vectors $\u \in \C^n$,
    \[
        \Im \u^* R \u \prec \Im \widetilde{m}_0 + N^{C_0 \delta} \Psi,
    \]
    then uniformly over all $z \in \bD_{k+1}$ and all deterministic unit vectors $\u \in \C^n$,
    \[
        |\u^* R \u - \u^* (-z I - z \widetilde{m}_0 \Sigma)^{-1} \u| \prec N^{C_0 \delta} \Psi.
    \]
\end{enumerate}
\end{lemma}
\begin{proof}
For (a), note that for $z \in \bD_0$, we have $\eta=1$, so $\|R \x\|_2 /
\sqrt{N} \leq \|R\|_\op / \sqrt{N} \leq N^{-1/2}$ and
$|\x^* R \y| \leq \|R\|_\op \leq 1$ for any unit vectors $\x, \y \in \C^n$. By Lemma \ref{lem:self-consistent moment bound},
\[
    |\u^* R \u - \u^* (-z I - z \widetilde{m}_0 \Sigma)^{-1} \u| \prec \frac{N^{2 \delta}}{\sqrt{N}} = N^{2 \delta - 1/2} \leq N^{C_0 \delta} \Psi,
\]
where the last bound holds for sufficiently large $C_0>0$ and any $\delta \in
(0,\tau/2C_0)$. This proves (a).

For (b), note that uniformly over $z \in \bD_k$, (a) implies
\[
    \Im \u^* R \u \prec \Im \u^* (-z I - z \widetilde{m}_0 \Sigma)^{-1} \u + N^{C_0 \delta} \Psi.
\]
Denoting $\v_\alpha$ as the eigenvectors of $\Sigma$ with eigenvalues $\sigma_\alpha$,
\[
    \Im \u^* (-z I - z \widetilde{m}_0 \Sigma)^{-1} \u = \sum_{\alpha=1}^n \frac{\eta (1 + \Re \widetilde{m}_0) + E \Im \widetilde{m}_0}{|z|^2 |1 + \sigma_\alpha \widetilde{m}_0|^2} |\v_\alpha^* \u|^2 \leq C \Im \widetilde{m}_0,
\]
since, over a regular domain $\bD$, we have $\Im \widetilde{m}_0 \gtrsim \eta$,
$|z| \asymp 1$, $|\Re \widetilde{m}_0| \leq |\widetilde{m}_0| \lesssim 1$, and
$\min_\alpha |1 + \sigma_\alpha \widetilde{m}_0| \gtrsim 1$. This proves (b).

For (c), let $z = E + i \eta \in \bD_{k+1}$, so $\eta \geq N^{-(k+1) \delta}$.
By construction, $E + i \eta_l \in \bD_k$ for $l = 1, \dots, L(\eta)$, since
$\eta_l = \eta N^{\delta l} \geq N^{-(k+1) \delta} N^{\delta} = N^{-k \delta}$.
The assumption implies uniformly over all such $z=E+i\eta \in \bD_{k+1}$ and
unit vectors $\x \in \C^n$,
\[
    \Im \x^* R(E + i \eta_l) \x \prec \Im \widetilde{m}_0(E + i \eta_l) + N^{C_0
\delta} \Psi(E + i \eta_l) \prec 1,
\]
since $\Im \widetilde{m}_0(E + i \eta_l) \leq C$, $\Psi(E + i \eta_l) \leq
CN^{-\tau/2}$ for any $\eta_l \geq N^{-1+\tau}$, and $\delta \in (0,\tau/2C_0)$.
Then by Lemma \ref{lem:prec bounds on generalized entries}, uniformly over
unit vectors $\x, \y \in \C^n$, $|\x^* R(z) \y| \prec N^{2 \delta}$.
Since the function $\eta \mapsto \eta \Im s(z)$ is nondecreasing and $\eta
\mapsto \Im s(z) / \eta$ is nonincreasing for any Stieltjes transform $s(z)$,
letting $z_1 = E + i N^\delta \eta \in \bD_m$, we have
\[
    \Im \u^* R(z) \u \leq \frac{N^\delta \eta}{\eta} \Im \u^* R(z_1) \u \prec
N^\delta [\Im \widetilde{m}_0(z_1) + N^{C_0 \delta} \Psi(z_1)] \leq N^{2 \delta}
[\Im \widetilde{m}_0(z) + N^{C_0 \delta} \Psi(z)].
\]
Thus,
\[
    \frac{\|R \u\|_2}{\sqrt{N}} = \sqrt{\frac{\Im \u^* R \u}{N \eta}}
\prec N^\delta \sqrt{\frac{\Im \widetilde{m}_0 + N^{C_0 \delta} \Psi}{N
\eta}} \leq N^\delta (1 + N^{C_0 \delta / 2}) \Psi,
\]
the last inequality using $\sqrt{a+b} \leq \sqrt{a}+\sqrt{b}$ and $\sqrt{\Im
\tilde m_0(z)/N\eta} \leq \Psi(z)$.
Noting that $N^\delta (1 + N^{C_0 \delta / 2}) \leq 2N^{(C_0 / 2 + 1) \delta}$
and setting $\Phi = N^{(C_0 / 2 + 1) \delta} \Psi$, Lemma
\ref{lem:self-consistent moment bound} then shows that
\[
    |\u^* R \u - \u^* (-z I - z \widetilde{m}_0 \Sigma)^{-1} \u| \prec N^{2 \delta} \Phi = N^{(C_0 / 2 + 3) \delta} \Psi \leq N^{C_0 \delta} \Psi,
\]
the last inequality holding for $C_0>6$. This completes the proof.
\end{proof}

\begin{proof}[Proof of Theorem \ref{thm:anisotropic law}]
Lemma \ref{lem:bootstrap} implies, uniformly over $z \in \bD$ and all deterministic unit vectors $\u \in \C^n$,
\[|\u^* R \u - \u^* (-z I - z \widetilde{m}_0 \Sigma)^{-1} \u| \prec N^{C_0 \delta} \Psi.
\]
Since the constant $\delta \in (0,\tau/2C_0)$ here is arbitrarily small, this implies
\begin{equation}\label{eq:Rzblockbound}
|\u^* R \u - \u^* (-z I - z \widetilde{m}_0 \Sigma)^{-1} \u| \prec \Psi.
\end{equation}

Taking $\delta>0$ arbitrarily small within the arguments
of Lemma \ref{lem:bootstrap} shows also
\[
    \frac{\|R\x\|_2}{\sqrt{N}}\prec\Psi,\qquad |\x^*R\y|\prec 1
\]
uniformly over unit vectors $\x,\y\in\C^n$ and $z\in\bD$.
Then by the resolvent identity (\ref{eq:schur complement}),
uniformly over unit vectors $\v \in \C^N$,
\begin{align*}
    \v^*\widetilde{R}\v&=\widetilde{m}_0+\sum_{i=1}^N|\v(i)|^2(\widetilde{R}_{ii}-\widetilde{m}_0)+\sum_{i\neq j}\Bar{\v}(i)\v(j)\widetilde{R}_{ij}\\
    &=\widetilde{m}_0+\Oprec(\Psi)+z\sum_{i\neq j}\Bar{\v}(i)\v(j)\widetilde{R}_{ii}\widetilde{R}_{jj}^{(i)}N^{-1}\g_i^*R^{(ij)}\g_j\\
    &=\widetilde{m}_0+\Oprec(\Psi)+z^{-1}\sum_{i\neq
j}\frac{\Bar{\v}(i)\v(j)N^{-1}\g_i^*R^{(ij)}\g_j}{(1+N^{-1}\g_i^*R^{(i)}\g_i)(1+N^{-1}\g_j^*R^{(ij)}\g_j)}.
\end{align*}
Then by the fluctuation averaging result of
Lemma \ref{lem:fluctuation_averaging_lowrank2}(b) applied with $\delta=0$ and $\Phi \asymp \Psi$,
\begin{equation}\label{eq:tildeRzblockbound}
|\v^*\widetilde{R}\v-\widetilde m_0| \prec \Psi.
\end{equation}
Similarly, by the identity \eqref{eq:linear block},
\[
    \v^*\Pi_o\u=\sum_{i=1}^N\bar{\v}(i)N^{-1/2}z\widetilde{R}_{ii}\g_i^*R^{(i)}\u=-\sum_{i=1}^N\frac{\bar{\v}(i)N^{-1/2}\g_i^*R^{(i)}\u}{1+N^{-1}\g_i^*R^{(i)}\g_i},
\]
so by the fluctuation averaging result of
Lemma \ref{lem:fluctuation_averaging_lowrank2}(c) applied with $\delta=0$ and $\Phi \asymp \Psi$,
\begin{equation}\label{eq:Piozblockbound}
|\v^*\Pi_o\u| \prec \Psi.
\end{equation}
The theorem follows from \eqref{eq:Rzblockbound}, \eqref{eq:tildeRzblockbound},
\eqref{eq:Piozblockbound} and the polarization identity.
\end{proof}

\begin{proof}[Proof of Corollary \ref{coro:Singular vector lies in general position}]

Writing
$K=\sum_{\alpha=1}^n \lambda_\alpha \x_\alpha\x_\alpha^*$ and $\widetilde K=\sum_{i=1}^N \lambda_i \widetilde \x_i\widetilde \x_i^*$, for any $\u \in \R^n$ and $\v \in \R^N$ we have
$\Im \u^* R \u=\sum_{\alpha=1}^n \<\u,\x_\alpha\>^2\eta/((\lambda_\alpha-E)^2+\eta^2)$ and
$\Im \v^* \widetilde R \v=\sum_{i=1}^N \<\v,\widetilde \x_i\>^2\eta/((\lambda_i-E)^2+\eta^2)$. Thus the result follows from Theorem \ref{thm:anisotropic law} by the same argument as Corollary \ref{coro:coro of entrywise law}(d).
\end{proof}

\subsection{Proof outside the spectrum (Theorem \ref{thm:outside spec})}

\begin{proof}[Proof of Theorem \ref{thm:outside spec}]
Fixing $\eta \in (0,1)$ and $\delta>0$,
we first show the result over the (regular) spectral domain
$\bD^o \equiv \bD^o(\delta,\eta)$ defined in \eqref{eq:spectraldomains}.

\eqref{eq:dichotomy} shows that
$K$ has no eigenvalues in $\{x \in \R:|x| \leq 2\delta^{-1},\,\dist(x,\supp(\widetilde \mu_0 \cup \{0\})) \geq
\delta/2)\}$, with high probability. In light of Lemma \ref{lem:prec bounds}, Theorem \ref{thm:entrywise law}, and the bound $\Im \widetilde m_0 \lesssim \eta$ for $z \in \bD^o$, we have $|\widetilde m^{(S)}-\widetilde m| \prec \Psi_\Theta^2 \prec N^{-1}+(N\eta)^{-2}$. Then the argument for \eqref{eq:dichotomy} may be applied equally with $\widetilde m^{(S)}$ in place of $\widetilde m$, to show
that for all $S \subset [N]$ with $|S| \leq L$ a fixed constant, also $K^{(S)}$ has no eigenvalues in $\{x \in \R:|x| \leq 2\delta^{-1},\,\dist(x,\supp(\widetilde \mu_0 \cup \{0\})) \geq
\delta/2)\}$ with high probability. Thus
$\|R^{(S)}(z)\|_\op \prec 1$ uniformly over $z \in \bD^o$ and such subsets $S
\subset [N]$.

This implies $N^{-1}\|R^{(S)}(z)\|_F \leq N^{-1/2}\|R^{(S)}\|_\op \prec
N^{-1/2}$, so Lemmas \ref{lem:prec bounds},
\ref{lem:alternative to L1 bound}, and \ref{lem:fixed_point_error} all
hold with $N^{-1/2}$ in place of $\Psi_\Theta$.
Applying Lemma \ref{lem:fluctuation_averaging} with $\Phi=N^{-1/2}$ shows,
in place of \eqref{eq:z0zdiff}, that $|z_0(\widetilde{m})-z|\prec N^{-1}$.
Then regularity of $\bD^o$ and the observation $\kappa \asymp 1$ for
$z \in \bD^o$ implies that
\[|\widetilde{m}-\widetilde{m}_0| \prec
\frac{N^{-1}}{\sqrt{\kappa+\eta}+N^{-1/2}} \prec N^{-1},\]
and hence also $|m-m_0| \prec N^{-1}$.

We have also uniformly over unit vectors $\x,\y \in \C^n$,
$N^{-1/2}\|R^{(S)}\x\|_2 \leq N^{-1/2}\|R^{(S)}\|_\op \prec N^{-1/2}$
and $|\x^* R^{(S)}\y| \leq \|R^{(S)}\|_\op \prec 1$. Then applying Lemma
\ref{lem:fluctuation_averaging_lowrank} with $\Phi=N^{-1/2}$ in place of
Lemma \ref{lem:fluctuation_averaging_lowrank2} in the arguments of
Lemma \ref{lem:self-consistent moment bound},
\eqref{eq:tildeRzblockbound}, and \eqref{eq:Piozblockbound} shows
\[|\u^*R\u-\u^*(-zI-z\widetilde m_0\Sigma)^{-1}\u|
\prec N^{-1/2},
\quad |\v^*\widetilde R\v-\widetilde m_0| \prec N^{-1/2},
\quad |\v^*\Pi_o \u| \prec N^{-1/2}.\]
By the polarization identity, 
this implies all statements of the theorem for $z \in \bD^o$.

To extend these estimates to $\bar\bD^o(\delta)$, first note that by the conjugate symmetry $\overline{\widetilde{m}(\bar{z})}=\widetilde{m}(z)$ and $\overline{\Pi(\bar{z})}=\Pi(z)$ it suffices to extend to the domain $\bar\bD^o(\delta)\cap\overline{\C^+}$. Next, observe that
\eqref{eq:dichotomy} implies also, on an event of high
probability, $R(z)$ and $\widetilde{R}(z)$ are $2/\delta$-Lipschitz in the
operator norm over $\bar\bD^o(\delta)$. Applying this Lipschitz continuity and
the results for $\bD^o(\delta,\eta)$ shows that all statements of
Theorem \ref{thm:outside spec} hold over $\bar\bD^o(\delta)\cap\overline{\C^+}$ with an additional additive error of
$(2/\delta)N^{-1+\tau}$. Since $\tau>0$ is arbitrary, this implies the theorem.
\end{proof}

\section{Analysis of examples}\label{sec:analysis-examples}

\subsection{Separable distributions}

\begin{proof}[Proof of Proposition \ref{prop:linear}]
Suppose first that $d=n$, $X=I$, and $\g=\w$. Then Assumption
\ref{assum:concentration} holds by a standard high moment estimate, see e.g.\
\cite[Lemma 3.1]{bloemendal2014isotropic}.
For Assumption \ref{assum:cumulant}, set $\cU=\{\e_1,\ldots,\e_n\}$, so
$\|T\|_\cU=\|T\|_\infty$.
Since $\w$ has independent entries, $\kappa_k(\w) \in (\R^n)^{\otimes k}$
is diagonal. Then for any $1 \leq m \leq k-1$, $\s_1,\ldots,\s_m\in\R^n$, and
$T\in(\R^n)^{\otimes k-m}$,
\[|\<\kappa_k(\w),\s_1\otimes\cdots\otimes\s_m\otimes
T\>|\leq\sum_{\alpha=1}^n|\kappa_k(\w[\alpha])||T[\alpha,\ldots,\alpha]|\prod_{t=1}^m|\s_t[\alpha]|
\leq C\|T\|_\cU\sum_{\alpha=1}^n \prod_{t=1}^m|\s_t[\alpha]|,\]
where the last inequality holds for some $C \equiv C(k)>0$ by the moment bounds
for $\w[\alpha]$. We may bound $\sum_{\alpha=1}^n |\s_1[\alpha]| \leq
\sqrt{n}\|\s_1\|_2$ for $m=1$, and
$\sum_{\alpha=1}^n \prod_{t=1}^m |\s_t[\alpha]| \leq \prod_{t=1}^m \|\s_t\|_2$
for $m \geq 2$, showing Assumption \ref{assum:cumulant} in the form of
\eqref{eq:linearcumulant}, which is stronger since $k \geq 3$ and $m \leq k-1$.

For general $d \asymp n$ and $X \neq I$,
since $\|X\|_\op \leq \|\Sigma\|_\op^{1/2} \leq C$ for
a constant $C>0$, it is clear that Assumption \ref{assum:concentration} holds
for $\g$ since it holds for $\w$. For any
$1 \leq m \leq k-1$, $\s_1,\ldots,\s_m\in\R^n$, and
$T\in(\R^n)^{\otimes k-m}$, letting
$\s_t'=X^*\s_t \in \R^d$ and $T'=(X^*)^{\otimes (k-m)}\cdot T \in
(\R^d)^{\otimes k-m}$ (denoting the multiplication of $T$ along each axis
by $X^*$), it follows from multilinearity of the cumulants that 
\[\<\kappa_k(\g_i),\s_1\otimes\cdots\otimes\s_m\otimes
T\>=\<\kappa_k(\w_i),\s_1'\otimes\cdots\otimes\s_m'\otimes T'\>.\]
Setting $\cU=\{\e_1,\ldots,\e_n,X\e_1,\ldots,X\e_d\}$, we have
$\|\s_t'\|_2 \leq \|X\|_\op\|\s_t\|_2$ and $\|T'\|_\infty \leq \|T\|_\cU$,
so \eqref{eq:linearcumulant} follows from the above result for $X=I$.
\end{proof}

\subsection{Conditionally mean-zero distributions}

We show Proposition \ref{prop:sign change invariant}.

\begin{lemma}\label{lem:l2 norm of cumulant}
In the setting of Proposition \ref{prop:sign change invariant}, for any fixed $k
\geq 1$, uniformly over $\alpha_1,\ldots,\alpha_k \in [d]$, we have
\begin{equation}\label{eq:signinvariantkappabound}
|\kappa(\w[\alpha_1],\ldots,\w[\alpha_k])| \prec 1
\end{equation}
and
\begin{equation}\label{eq:signinvariantkappasumbound}
\sum_{\beta=1}^d
\kappa_{k+2}(\w[\beta],\w[\beta],\w[\alpha_1],\ldots,\w[\alpha_k])^2 \prec 1.
\end{equation}
\end{lemma}
\begin{proof}
Applying Assumption \ref{assum:concentration} for $\w$
with $A=\e_\alpha\e_\alpha^*$ shows $|\w[\alpha]| \prec 1$. Together with the
moment bound for $\w$ in Assumption \ref{assum:concentration}, this implies
for any fixed integer $k \geq 1$ that $\E|\w[\alpha]|^k \prec 1$.
The first claim $|\kappa(\w[\alpha_1],\ldots,\w[\alpha_k])| \prec 1$ then
follows from the moment-cumulant relations.

For the second claim, let $S=\{\alpha_1,\ldots,\alpha_k\}$. 
Let $P_{k+2}$ denote the set of all partitions of $[k+2]$. For any
$\beta \notin S$, denoting $\alpha_{k+1}=\alpha_{k+2}=\beta$, the
moment-cumulant relations give
\[\kappa_{k+2}(\w[\alpha_1],\ldots,\w[\alpha_k],\w[\beta],\w[\beta])
=\sum_{\pi \in P_{k+2}} (-1)^{|\pi|-1}(|\pi|-1)!
\prod_{B \in \pi} \E\Big[\prod_{j \in B} \w[\alpha_j]\Big]\]
Observe that if any block $B \in \pi$ contains exactly one of the two indices
$\{k+1,k+2\}$, then $\E \prod_{j \in B} \w[\alpha_j]=0$ by the given condition for $\w$. Thus, denoting by $\dot P_{k+2}$ those
partitions which put $k+1,k+2$ in the same block, we have
\begin{align*}
\kappa_{k+2}(\w[\alpha_1],\ldots,\w[\alpha_k],\w[\beta],\w[\beta])
&=\sum_{\pi \in \dot P_{k+2}} (-1)^{|\pi|-1}(|\pi|-1)!
\prod_{B \in \pi} \E\Big[\prod_{j \in B} \w[\alpha_j]\Big]\\
&=\kappa_{k+1}(\w[\alpha_1],\ldots,\w[\alpha_k],\w[\beta]^2)\\
&=\kappa_{k+1}(\w[\alpha_1],\ldots,\w[\alpha_k],\w[\beta]^2-1).
\end{align*}
Now let $P_{k+1}$ be all partitions of $[k+1]$, and for all $\pi \in P_{k+1}$,
let $B^* \in \pi$ denote the block containing the last element $k+1$. Then
we may further expand this as
\begin{align*}
&\kappa_{k+1}(\w[\alpha_1],\ldots,\w[\alpha_k],\w[\beta]^2-1)\\
&=\sum_{\pi\in P_{k+1}}(-1)^{|\pi|-1}(|\pi|-1)!\left(\prod_{B\in\pi\setminus
B^*}\E\left[\prod_{j\in B} \w[\alpha_j]\right]\right)
\E\left[\left(\prod_{j\in B^*\setminus\{k+1\}}\w[\alpha_j]\right)
(\w[\beta]^2-1)\right]\\
&=\E\left[\underbrace{\left(\sum_{\pi\in P_{k+1}}(-1)^{|\pi|-1}(|\pi|-1)!\left(\prod_{B\in\pi\setminus
B^*}\E\left[\prod_{j\in B} \w[\alpha_j]\right]\right)
\prod_{j\in B^*\setminus\{k+1\}}\w[\alpha_j]\right)}_{:=u}
(\w[\beta]^2-1)\right]\\
&=\E[u(\w[\beta]^2-1)],
\end{align*}
where $u \in \R$ is a scalar random variable in the probability space of $\w$,
depending on $\alpha_1,\ldots,\alpha_k$, and satisfying $|u| \prec 1$
uniformly over $\alpha_1,\ldots,\alpha_k \in [d]$.

Then, applying also the first claim \eqref{eq:signinvariantkappabound} for
$\beta \in S$, we have
\begin{align*}
\sum_{\beta=1}^d \kappa_{k+2}(\w[\alpha_1],\ldots,\w[\alpha_k],
\w[\beta],\w[\beta])^2
&=\sum_{\beta \notin
S}\kappa_{k+1}(\w[\alpha_1],\ldots,\w[\alpha_k],\w[\beta]^2-1)^2
+\Oprec(1)\\
&=\sum_{\beta=1}^d \Big(\E[u(\w(\beta)^2-1)]\Big)^2+\Oprec(1).
\end{align*}
The proof is concluded by the observation
\begin{align*}
\sum_{\beta=1}^d \Big(\E[u(\w[\beta]^2-1)]\Big)^2
&=\sup_{\v \in \R^d:\|\v\|_2=1} \bigg(\sum_{\beta=1}^d \v[\beta]
\E[u(\w[\beta]^2-1)]\bigg)^2\\
&\leq \sup_{\v \in \R^d:\|\v\|_2=1} 
\E u^2 \cdot \E\bigg|\sum_{\beta=1}^d \v[\beta](\w[\beta]^2-1)\bigg|^2
\prec 1,
\end{align*}
where the last inequality uses $|u| \prec 1$ and Assumption
\ref{assum:concentration} for $\w$ with $A=\diag(\v)$.
\end{proof}

\begin{proof}[Proof of Proposition \ref{prop:sign change invariant}]
Suppose first that $d=n$, $X=I$, and $\g=\w$. Let
$\cU=\{\e_1,\ldots,\e_d\}$.
Fix any $k \geq 3$, $m \in \{1,\ldots,k-1\}$,
unit vectors $\v_1,\ldots,\v_m\in\R^n$, and $T\in(\R^n)^{\otimes k-m}$. Then
$\|T\|_\cU=\|T\|_\infty$, and
\begin{align*}
   |\<\kappa_k(\w), \v_1\otimes\cdots \otimes \v_m\otimes T\>|
    &\leq \|T\|_\cU \sum_{\alpha_1,\ldots,\alpha_k=1}^n
|\kappa_k(\w[\alpha_1],\ldots,\w[\alpha_k])|\prod_{l=1}^m|\v_l[\alpha_l]|.
\end{align*}
For any index tuple $\alpha=(\alpha_1,\ldots,\alpha_k) \in [n]^k$, write
$\pi(\alpha)$ for the partition of $[k]$ such that $i,j$ belong to the same
block of $\pi(\alpha)$ if and only if $\alpha_i=\alpha_j$.
If any block of $\pi(\alpha)$ has cardinality 1, then
$|\kappa_k(\w[\alpha_1],\ldots,\w[\alpha_k])|=0$ by
the given condition for $\w$.
Thus, letting $P_k^e$ be the set of partitions of $[k]$
where each block has cardinality at least 2,
\[|\<\kappa_k(\w), \v_1\otimes\cdots \otimes \v_m\otimes T\>|
\leq \|T\|_\cU \sum_{\pi\in P_k^e}\sum_{\alpha \in [n]^k:\pi(\alpha)=\pi}
|\kappa_k(\w[\alpha_1],\ldots,\w[\alpha_k])|
\prod_{l=1}^m|\v_l(\alpha_l)|.\]
Now consider any $\pi\in P_k^e$. Note that
\begin{equation}\label{eq:vlpbounds}
\sum_{\alpha=1}^n 1=n, \quad \sum_{\alpha=1}^n |\v_l[\alpha]| \leq
\sqrt{n}\|\v_l\|_2,
\quad \sum_{\alpha=1}^n |\v_{l_1}[\alpha] \ldots \v_{l_p}[\alpha]| \leq 
\|\v_{l_1}\|_2\ldots\|\v_{l_p}\|_2 \text{ for } p \geq 2.
\end{equation}
Let $K_0(\pi)$ and $K_1(\pi)$ be the numbers of blocks $B
\in \pi$ for which 0 and 1, respectively, of the first $m$ indices
$\alpha_1,\ldots,\alpha_m$ belong to $B$. Then, applying
\eqref{eq:signinvariantkappabound} and \eqref{eq:vlpbounds}, we get
\begin{equation}\label{eq:singinvariantkappadecomp}
|\<\kappa_k(\w), \v_1\otimes\cdots \otimes \v_m\otimes T\>|
\prec \|T\|_\cU \prod_{l=1}^m \|\v_l\|_2
\sum_{\pi \in P_k^e} \sqrt{n}^{2K_0(\pi)+K_1(\pi)}.
\end{equation}
Note that if 0 (or 1) such indices
belong to $B$, then $B$ has at least 2 (resp.\ 1) indices from
$m+1,\ldots,k$, so $2K_0+K_1\leq k-m$. This establishes a weaker
version of \eqref{eq:cumulantassumption}, with exponent $k-m$ in place of
$k-m-1$.

To improve this exponent to $k-m-1$, consider any
$\pi \in P_k^e$ such that $2K_0(\pi)+K_1(\pi)=k-m$.
Each block of $\pi$ counted by
$K_0$ must have exactly 2 indices in $m+1,\ldots,k$, and each block of $\pi$
counted by $K_1$ must have exactly 1 index in $m+1,\ldots,k$. If $K_0 \geq 1$,
let $B^* \in \pi$ be a block counted by $K_0$, and suppose for
notational convenience that $B^*=\{k-1,k\}$. Denote
$\alpha_{[k-2]}=(\alpha_1,\ldots,\alpha_{k-2})$. Then we may apply
\eqref{eq:signinvariantkappasumbound} to bound
\begin{align*}
&\sum_{\alpha \in [n]^k:\pi(\alpha)=\pi}
|\kappa_k(\w[\alpha_1],\ldots,\w[\alpha_k])|
\prod_{l=1}^m|\v_l[\alpha_l]|\\
&\leq \sum_{\alpha_{[k-2]}:\pi(\alpha_{[k-2]})=\pi \setminus B^*}
\prod_{l=1}^m|\v_l[\alpha_l]|\sum_{\beta=1}^n |\kappa(\w[\alpha_1],\ldots,
\w[\alpha_{k-2}],\w[\beta],\w[\beta])|\\
&\prec \sqrt{n}\sum_{\alpha_{[k-2]}:\pi(\alpha_{[k-2]})=\pi \setminus B^*}
\prod_{l=1}^m|\v_l[\alpha_l]|
\leq \sqrt{n} \cdot n^{K_0-1} \cdot n^{K_1/2}=\sqrt{n}^{k-m-1}.
\end{align*}
If $K_0=0$ and $K_1=k-m$,
let $B^* \in \pi$ be a block counted by $K_1$, and suppose for
notational convenience that $B^*=\{\alpha_1,\alpha_k\}$. Denote in this case
$\alpha_{[k-2]}=(\alpha_2,\ldots,\alpha_{k-1})$. Then we may apply
\eqref{eq:signinvariantkappasumbound} to bound
\begin{align*}
&\sum_{\alpha \in [n]^k:\pi(\alpha)=\pi}
|\kappa_k(\w[\alpha_1],\ldots,\w[\alpha_k])|
\prod_{l=1}^m|\v_l[\alpha_l]|\\
&\leq \sum_{\alpha_{[k-2]}:\pi(\alpha_{[k-2]})=\pi \setminus B^*}
\prod_{l=2}^m|\v_l[\alpha_l]| \sum_{\beta=1}^n
|\kappa(\w[\beta],\w[\alpha_2],\ldots,\w[\alpha_{k-1}],\w[\beta])||\v_1[\beta]|\\
&\leq \sum_{\alpha_{[k-2]}:\pi(\alpha_{[k-2]})=\pi \setminus B^*}
\prod_{l=2}^m|\v_l[\alpha_l]|\left(\sum_{\beta=1}^n
\kappa(\w[\beta],\w[\alpha_2],\ldots,\w[\alpha_{k-1}],\w[\beta])^2\right)^{1/2}
\|\v_1\|_2\\
&\prec \sum_{\alpha_{[k-2]}:\pi(\alpha_{[k-2]})=\pi \setminus B^*}
\prod_{l=2}^m|\v_l[\alpha_l]|
\leq n^{(K_1-1)/2}=\sqrt{n}^{k-m-1}.
\end{align*}
Applying these cases to \eqref{eq:singinvariantkappadecomp} for the partitions
$\pi \in P_k^e$ where $2K_0(\pi)+K_1(\pi)=k-m$, this shows
that Assumption \ref{assum:cumulant} holds.

The argument to extend to general $d \asymp n$ and $X \neq I$ is the same as in
Proposition \ref{prop:linear}.
\end{proof}

\begin{proof}[Verification of Example \ref{ex:mixture model}]

We verify the claims of Example \ref{ex:mixture model}.

We first check that the conditions of Proposition \ref{prop:sign change invariant} for $\w$
hold under \eqref{eq:mixtureconditionbasic}
and \eqref{eq:mixturecondition}: The conditions $\E\w\w^*=I$ and $\E \prod_{i=1}^k \w[\alpha_i]=0$ if $\alpha_1 \notin \{\alpha_2,\ldots,\alpha_k\}$ are immediate from \eqref{eq:mixtureconditionbasic}. Furthermore, \eqref{eq:mixtureconditionbasic} implies that the conditional law $(\w \mid \lambda)$ satisfies Assumption \ref{assum:concentration} uniformly over $\lambda \in \Lambda$.
To check Assumption \ref{assum:concentration} unconditionally, for any matrix $A \in \R^{n \times n}$, we may decompose
\begin{align}
\w^* A \w-\Tr A
&=(\w^* A \w-\E[\w^* A \w \mid \lambda])
+(\E[\w^* A \w \mid \lambda]-\Tr A)\notag\\
&=(\w^* A \w-\E[\w^* A \w \mid \lambda])
+(\E[\w[\alpha]^2 \mid \lambda]-1) \Tr A.
\label{eq:conditionalconcentration}
\end{align}
The first term is $\Oprec(\|A\|_F)$ by Assumption \ref{assum:concentration} conditional on $\lambda$, while the second is $\Oprec(\|A\|_F)$ by \eqref{eq:mixturecondition} and the bound $|\Tr A| \leq \sqrt{n}\|A\|_F$.

Next, for the specific example of \eqref{eq:mixtureexample}, we note that $\E[\w[\alpha]^2 \mid \lambda]-1=c_n\lambda$. Then taking $A=I$ in \eqref{eq:conditionalconcentration} shows that if $c_n \gg n^{-1/2}$, then Assumption \ref{assum:concentration} does not hold. In this case, we also have
\[\Var[\Tr K]
=N^{-1}\Var[\|\w_i\|_2^2]
=N^{-1}\Big(\E[\Var[\|\w_i\|_2^2 \mid \lambda]]
+\Var[\E[\|\w_i\|_2^2 \mid \lambda]]\Big)
\asymp 1+Nc_n^2 \gg 1.\]
Then Theorem \ref{thm:entrywise law} cannot hold, as the eigenvalue rigidity in Corollary \ref{coro:coro of entrywise law} must instead imply $|\Tr K-\E \Tr K| \prec 1$ for concentration of the linear spectral statistic $\Tr K$.

Finally, let us check that for this model \eqref{eq:mixtureexample}, Assumption \ref{assum:cumulant} holds as long as $c_n \prec n^{-1/4}$: Since $\w$ is sign-invariant, we have $\kappa_k(\w)=0$ for any odd $k$. It is therefore enough to consider the case where $k \geq 4$ is even. Let $\gamma=1+c_n\lambda$,
denote $P_k$ the set of all possible partitions of $[k]$, and let $P_k'$ be the set of all possible pairings of $[k]$. For any $T\in(\R^n)^{\otimes k}$, we have by the law of total cumulance
\begin{align*}
    \<\kappa_k(\w),T\>&=\sum_{\alpha_1,\ldots,\alpha_k}\kappa(\w[\alpha_1],\ldots,\w[\alpha_k])T[\alpha_1,\ldots,\alpha_k]\\
    &=\sum_{\alpha_1,\ldots,\alpha_k}\sum_{\pi\in P_k}\kappa(\kappa(\w[\alpha_b]:b\in B \mid \gamma):B\in\pi)T[\alpha_1,\ldots,\alpha_k]\\
&=\sum_{\alpha_1,\ldots,\alpha_k}\sum_{\pi\in P_k'}\kappa_{k/2}(\gamma)\prod_{\{p,q\}\in\pi}\kappa(\w[\alpha_p],\w[\alpha_q]) T[\alpha_1,\ldots,\alpha_k]\\
    &=\kappa_{k/2}(\gamma)\sum_{\pi\in P_k'}\sum_{\alpha_1,\ldots,\alpha_k}T[\alpha_1,\ldots,\alpha_k]\prod_{\{p,q\}\in\pi}\1_{\alpha_p=\alpha_q}
\end{align*}
where we used the fact that $(\w \mid \gamma)$ has i.i.d.\ $N(0,\gamma)$ entries, so $\kappa(\w[\alpha_b]:b\in B \mid \gamma)=\gamma$ 
if $B=\{p,q\}$ where $\alpha_p=\alpha_q$,
and $\kappa(\w[\alpha_b]:b\in B \mid \gamma)=0$ otherwise.

Then, for any even $k\geq 4$, $1\leq m\leq k-1$, $\s_1,\ldots,\s_m\in\R^n$, and $T\in(\R^n)^{\otimes k-m}$, we have
\begin{align*}
    \<\kappa_k(\w),\s_1\otimes\dots\otimes\s_m\otimes T\>
    &=\kappa_{k/2}(\gamma)\sum_{\pi\in P_k'}\sum_{\alpha_1,\ldots,\alpha_k}\s_1[\alpha_1]\cdots\s_m[\alpha_m]T[\alpha_{m+1},\ldots,\alpha_k]\prod_{\{p,q\}\in\pi}\1_{\alpha_p=\alpha_q}.
\end{align*}
Note that
$|\kappa_{k/2}(\gamma)|=c_n^{k/2} |\kappa_{k/2}(\lambda)| \prec c_n^{k/2} \prec n^{-1/2}$ where the last bound holds for $c_n \prec n^{-1/4}$ and $k \geq 4$.
Thus we can bound, similar to the proof of Proposition \ref{prop:sign change invariant},
\begin{align*}
    &|\<\kappa_k(\w),\s_1\otimes\cdots\otimes\s_m\otimes T\>| \prec n^{-1/2}\|T\|_\infty\max_{\pi\in P_k'}\sum_{\alpha_1,\ldots,\alpha_k}|\s_1[\alpha_1]|\cdots|\s_m[\alpha_m]|\prod_{\{p,q\}\in\pi}\1_{\alpha_p=\alpha_q}\\
    &\prec n^{-1/2}\|T\|_\infty n^{(k-m)/2}\|\s_1\|_2\cdots\|\s_m\|_2 \prec n^{(k-m-1)/2}\|\s_1\|_2\cdots\|\s_m\|_2\|T\|_\infty,
\end{align*}
establishing Assumption \ref{assum:cumulant}.
\end{proof}

\subsection{Random features model}\label{subsec:RF}

We now show Proposition \ref{prop:RF}.
To simplify notation, we will assume that
$X \in \R^{n \times d}$ and $\w \in \R^d$ have equal dimensions $n=d$.
This is without loss of generality, as if $n \neq d$, it
is equivalent to show Proposition \ref{prop:RF} with $X$, $\w$, and/or
$\s_1,\ldots,\s_m,T$ extended by 0-padding to have all dimensions equal to
$\max(n,d)$.

\subsubsection{Tensor networks and bipolar orientations}

\begin{defi}\label{def:tensornetwork}(Tensor network)
We say that a tuple $(G,f)$ is a tensor network if the following holds:
\begin{enumerate}
\item $G=(\cV,\cE)$ is a multigraph with no self-loops.
Each vertex $v \in \cV$ has an ordered list $\partial v \subseteq
\cE$ of its incident edges, and each edge $e=(u,v)$ has an ordering of its 2
incident vertices $u,v \in \cV$. We denote by $\deg(v)=|\partial v|$ the
vertex degree, counting edge multiplicity.
\item $f$ is a labeling of $\cV \cup \cE$ such that $f(v) \in (\R^n)^{\otimes
\deg(v)}$ is a tensor of order $\deg(v)$
for each vertex $v \in \cV$, and $f(e)=\R^{n \times n}$ is a
matrix for each edge $e \in \cE$.
\end{enumerate}
The contracted value of this tensor network is then given by
\[\mathrm{val}(G,f)=\sum_{\substack{i_{v,e} \in [n]:
\,v\in \cV,\,e\in\partial v}}
\,\prod_{v \in \cV}f(v)[i_{v,e}:e\in\partial v]\prod_{e=(u,v)\in \cE}
f(e)[i_{u,e},i_{v,e}]\]
where the summation is over one index $i_{v,e} \in [n]$ for each vertex-edge
pair $(v,e)$ such that $v$ is incident to $e$, and
$[i_{v,e}:e \in \partial v]$ denotes the ordered tuple of such indices for a
given vertex $v \in \cV$.
\end{defi}

Note that if one permutes the order of the edges $\partial v$ incident to $v \in
\cV$, and permutes correspondingly the indices of the tensor label $f(v)$,
then $\val(G,f)$ remains unchanged. Similarly, if one
exchanges the vertex order of $e=(u,v)$ and replaces $f(e)$ by the
transpose $f(e)^*$, then $\val(G,f)$ also remains unchanged. If two tensor
networks $(G,f)$ and $(G',f')$ are equivalent under some such permutations of
their vertex/edge orderings and labels, we will consider $(G,f)$ and $(G',f')$
to be the same tensor network. Otherwise, $(G,f)$ and $(G',f')$ are distinct.
We will omit the specification of the ordering of $\partial v$ or
$e=(u,v)$ if the tensor/matrix label $f(v)$ or $f(e)$ is symmetric.

\begin{defi}(Bipolar orientation)
Let $G=(\cV,\cE)$ be a multigraph with no self-loops. $G$ is
\textit{biconnected} if it consists of a single connected component and
furthermore $G$ is not broken into disconnected components by removing any
single vertex $v \in \cV$ together with its incident edges $\partial v$.

For two vertices $s,t \in \cV$, $G$ admits a \emph{$(s,t)$-bipolar orientation}
if there is a choice of direction for each edge such that $G$ has no
directed cycles, $s$ is the unique source (vertex with no incoming edges),
and $t$ is the unique sink (vertex with no outgoing edges).
\end{defi}

We note that the existence of a $(s,t)$-bipolar orientation depends only on
which pairs of vertices are connected by an edge in $G$, and not on the edge
multiplicities. It is well-known that the following are equivalent (see e.g.\ 
\cite[Section 2]{rosenstiehl1986rectilinear}):
\begin{enumerate}
\item $G$ admits a $(s,t)$-bipolar orientation.
\item $G$ admits a $(s,t)$-numbering, i.e.\ a numbering of its vertices
$v_1,\ldots,v_m$ with $m=|\cV|$ such that $s=v_1$, $t=v_m$, and each
other vertex is adjacent to both a lower-numbered and a higher-numbered vertex.
\item The multigraph $G^+=(\cV,\cE \cup (s,t))$ adding the additional edge
$(s,t)$ is biconnected.
\end{enumerate}

For any tensor $T \in (\R^n)^{\otimes k}$ and any partition of $\{1,\ldots,k\}$
into two disjoint subsets $I,J$, we denote by $\mat_{I,J}(T) \in \R^{n^{|I|}}
\times \R^{n^{|J|}}$ its matricization or flattening
with respect to this index partition, and by
$\vec(T) \in \R^{n^k}$ its vectorization. We define the norms
\[\|T\|_F=\|\vec(T)\|_2, \qquad \|T\|_{I \to J}
=\|\mat_{I,J}(T)\|_\op, \qquad
\|T\|_{\matop}
=\mathop{\max_{I,J:I \sqcup J=[k]}}_{|I| \geq 1,|J| \geq 1} \|T\|_{I \to J}.\]
Thus $\|T\|_{\matop}$ is the maximum of the matrix operator norm
$\|T\|_{I \to J}$ over all flattenings of $T$ corresponding to all possible
partitions of $\{1,\ldots,k\}$ into two non-empty index sets.

The following lemma is similar to the analyses of \cite{mingo2012sharp},
and allows us to bound the value of any tensor network
with bipolar orientation.

\begin{lemma}\label{lem:bounds by bipolar orientation}
Let $(G,f)$ be a tensor network, where $G=(\cV,\cE)$ admits a
$(s,t)$-bipolar orientation. Then
\[
|\mathrm{val}(G,f)| \leq 
\|f(s)\|_F\|f(t)\|_F
\prod_{v \in \cV \setminus \{s,t\}}
\|f(v)\|_{\matop} \prod_{e\in \cE}\|f(e)\|_\op.\]
\end{lemma}
\begin{proof} 
Number the vertices $\cV=\{v_1,\ldots,v_m\}$ according to a $(s,t)$-numbering,
and consider the bipolar orientation such that each edge directs from a
lower-numbered vertex to a higher-numbered vertex. By replacing the label $f(e)$
with $f(e)^*$, we may assume that the vertices for each edge $e=(u,v)$ in
Definition \ref{def:tensornetwork} are also ordered such that $u$ has
lower number than $v$.

For each $k \in \{2,\ldots,m-1\}$, let $I_k,O_k$ be the partition of
$\{1,\ldots,\deg(v_k)\}$ corresponding to the incoming and outgoing
(ordered) edges of $v_k$, respectively. We may then express
\[\mathrm{val}(G,f)=\mathrm{vec}(f(v_1))^\top E_2M_2 \cdots E_{m-1}M_{m-1}
E_m\mathrm{vec}(f(v_m))\]
where
\begin{itemize}
\item Each $M_k$ for $k=2,\ldots,m-1$ is a matrix given by the tensor product
of $\mat_{I_k,O_k}(f(v_k))$ and a number of identity matrices.
\item Each $E_k$ for $k=2,\ldots,m$ is a matrix given by the tensor product of
the matrices $\{f(e):e \text{ is an incoming edge of } v_k\}$
and a number of identity matrices.
\end{itemize}
As an example, for the following network with $V=\{v_1,\ldots,v_6\}$
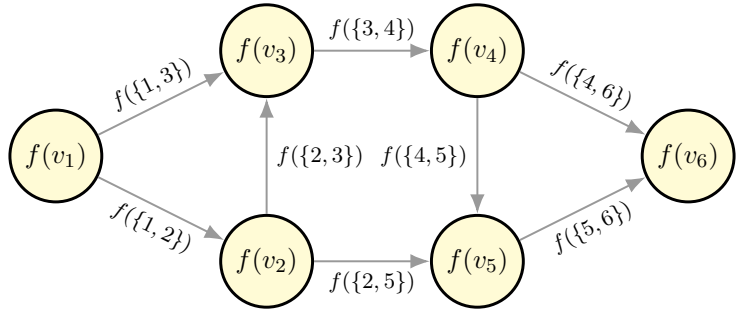
\begin{figure}[H]
    \centering
    \resizebox{0.6\textwidth}{!}{
        \begin{tikzpicture}[
    node distance=2.5cm,
    tensor_node/.style={
        circle, 
        draw=black, 
        very thick, 
        fill=yellow!20, 
        minimum size=1.2cm, 
        font=\bfseries
    },
    conn/.style={
        thick, 
        draw=gray!80, 
        -{Latex[scale=1.2]} 
    },
    edge_label/.style={
        font=\footnotesize, 
        midway, 
        sloped, 
        above, 
        text=black
    }
]

    
    \node[tensor_node] (v1) at (0,0) {$f(v_1)$};
    
    \node[tensor_node] (v3) at (3, 1.5) {$f(v_3)$};
    \node[tensor_node] (v2) at (3, -1.5) {$f(v_2)$};
    
    \node[tensor_node] (v4) at (6, 1.5) {$f(v_4)$};
    \node[tensor_node] (v5) at (6, -1.5) {$f(v_5)$};
    
    \node[tensor_node] (v6) at (9, 0) {$f(v_6)$};

    
    \draw[conn] (v1) -- (v3) node[edge_label] {$f(\{1,3\})$};
    \draw[conn] (v1) -- (v2) node[edge_label, below] {$f(\{1,2\})$}; 
    
    \draw[conn] (v2) -- (v3) node[midway, right, font=\footnotesize] {$f(\{2,3\})$};
    
    \draw[conn] (v4) -- (v5) node[midway, left, font=\footnotesize] {$f(\{4,5\})$};
    
    \draw[conn] (v3) -- (v4) node[edge_label] {$f(\{3,4\})$};
    \draw[conn] (v2) -- (v5) node[edge_label, below] {$f(\{2,5\})$};
    
    \draw[conn] (v4) -- (v6) node[edge_label] {$f(\{4,6\})$};
    \draw[conn] (v5) -- (v6) node[edge_label, below] {$f(\{5,6\})$};

\end{tikzpicture}
}
\end{figure}
\noindent the matrices $M_2,\ldots,M_{m-1}$ and $E_2,\ldots,E_m$ are given by
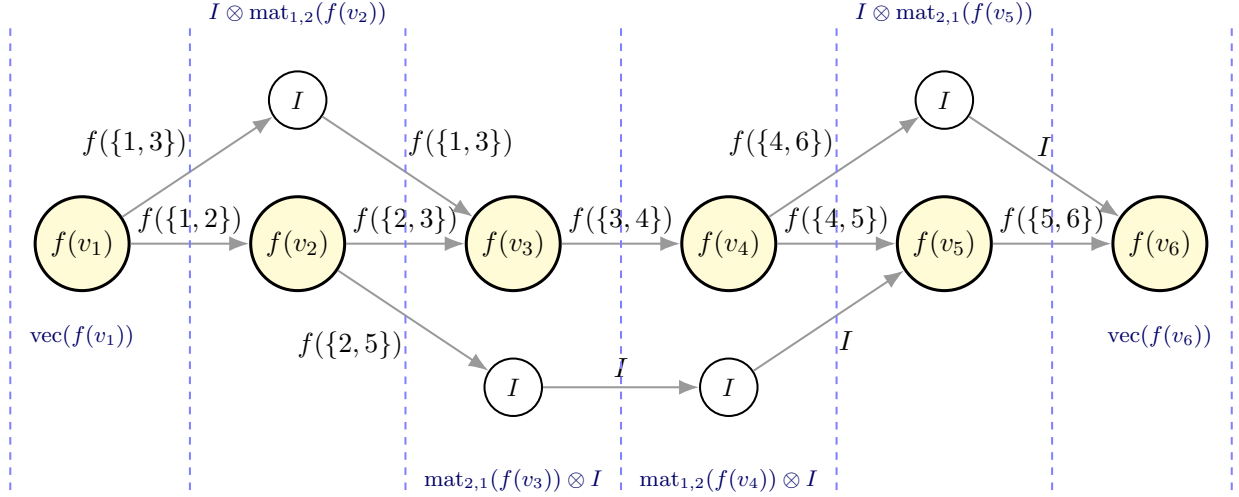
\begin{figure}[H]
    \centering
    \resizebox{1\textwidth}{!}{
        \begin{tikzpicture}[
    node distance=2.0cm,
    tensor_node/.style={
        circle, 
        draw=black, 
        very thick, 
        fill=yellow!20, 
        minimum size=1.1cm, 
        font=\bfseries
    },
    identity_node/.style={
        circle,
        draw=black,
        thick,
        fill=white,
        minimum size=0.8cm,
        font=\small
    },
    conn/.style={
        thick, 
        draw=gray!80, 
        -{Latex[scale=1.2]}
    },
    layer_line/.style={
        dashed,
        draw=blue!50,
        thick
    },
    layer_label/.style={
        font=\small\bfseries,
        text=blue!60!black,
        anchor=south
    },
    math_annot/.style={
        font=\footnotesize,
        text=blue!40!black,
        align=center,
        anchor=north
    }
]

    
    \node[tensor_node] (v1) at (0, 0) {$f(v_1)$};
    \node[tensor_node] (v2) at (3, 0) {$f(v_2)$};
    \node[tensor_node] (v3) at (6, 0) {$f(v_3)$};
    \node[tensor_node] (v4) at (9, 0) {$f(v_4)$};
    \node[tensor_node] (v5) at (12, 0) {$f(v_5)$};
    \node[tensor_node] (v6) at (15, 0) {$f(v_6)$};

    \node[identity_node] (I_13) at (3, 2) {$I$};
    
    \node[identity_node] (I_25a) at (6, -2) {$I$};
    \node[identity_node] (I_25b) at (9, -2) {$I$};

    \node[identity_node] (I_46) at (12, 2) {$I$};

    
    \draw[conn] (v1) -- (v2) node[midway, above] {$f(\{1,2\})$};
    \draw[conn] (v1) -- (I_13) node[midway, above left] {$f(\{1,3\})$};
    
    \draw[conn] (v2) -- (v3) node[midway, above] {$f(\{2,3\})$};
    \draw[conn] (v2) -- (I_25a) node[midway, below left] {$f(\{2,5\})$};
    
    \draw[conn] (I_13) -- (v3) node[midway, above right] {$f(\{1,3\})$};
    
    \draw[conn] (v3) -- (v4) node[midway, above] {$f(\{3,4\})$};
    
    \draw[conn] (I_25a) -- (I_25b) node[midway, above] {$I$};
    
    \draw[conn] (v4) -- (v5) node[midway, above] {$f(\{4,5\})$};
    \draw[conn] (v4) -- (I_46) node[midway, above left] {$f(\{4,6\})$}; 

    \draw[conn] (I_46) -- (v6) node[midway, above] {$I$};
    
    \draw[conn] (I_25b) -- (v5) node[midway, below right] {$I$};
    
    \draw[conn] (v5) -- (v6) node[midway, above] {$f(\{5,6\})$};
    
    

    \draw[layer_line] (-1, 3) -- (-1, -3.5);
    
    \draw[layer_line] (1.5, 3) -- (1.5, -3.5);
    
    \draw[layer_line] (4.5, 3) -- (4.5, -3.5);
    
    \draw[layer_line] (7.5, 3) -- (7.5, -3.5);
    
    \draw[layer_line] (10.5, 3) -- (10.5, -3.5);

    \draw[layer_line] (13.5, 3) -- (13.5, -3.5);
    
    \draw[layer_line] (16, 3) -- (16, -3.5);

    
    \node[math_annot] at (0, -1) {$\vec(f(v_1))$};
    
    \node[math_annot] at (3, 3.5) {$I \otimes \mat_{1,2}(f(v_2))$};
    
    \node[math_annot] at (6, -3) {$\mat_{2,1}(f(v_3)) \otimes I$};
    
    \node[math_annot] at (9, -3) {$\mat_{1,2}(f(v_4)) \otimes I$};
    
    \node[math_annot] at (12, 3.5) {$I \otimes \mat_{2,1}(f(v_5))$};
    
    \node[math_annot] at (15, -1) {$\vec(f(v_6))$};

\end{tikzpicture}
}
\end{figure}
\noindent 
More generally, for each directed edge $e=(v_{k'},v_k)$, we introduce a
factor of the identity matrix $I \in \R^{n \times n}$ in the tensor product
defining each matrix $M_{k'+1},\ldots,M_{k-1}$ and $E_{k'+1},\ldots,E_{k-1}$,
and a factor of $f(e) \in \R^{n \times n}$ into that defining $E_k$, to form a
path of degree-2 vertices from $v_{k'}$ to $v_k$ whose labels have product
$f(e)$.

Then, since $\|M_k\|_\op=\|f(v_k)\|_{I_k \to O_k}$
and $\|E_k\|_\op=\prod_{\text{incoming edges $e$ of } v_k} \|f(e)\|_\op$,
where each edge appears in exactly one such term $\|E_k\|_\op$,
this implies the bound
\begin{align*}
|\mathrm{val}(G,f)|&\leq \|\mathrm{vec}(f(v_1))\|_2\|\mathrm{vec}(f(v_m))\|_2
\prod_{k=2}^{m-1}\|M_k\|_\op \prod_{k=2}^m \|E_k\|_\op\\
&\leq \|f(s)\|_F\|f(t)\|_F
\prod_{v \in V \setminus \{s,t\}}
\|f(v)\|_{\matop} \prod_{e\in E}\|f(e)\|_\op.
\end{align*}
\end{proof}

Many of the networks that arise in our proof will be of the following form, which by the following lemma admits a bipolar orientation:
\begin{lemma}\label{lem:source-sink network}
Suppose $(G,f)$ is a tensor network where
$G=(\{s,t\} \cup \cA \cup \cB,\cE)$ has its vertices partitioned into three
disjoint sets $\{s,t\},\cA,\cB$, such that
\begin{itemize}
\item The subgraph induced by $\cA \cup \cB$ is connected.
\item Each vertex of $\{s,t\}$ has at least 1 incident edge connecting to $\cA$,
and each vertex $v \in \cA$ has at least 1 incident edge connecting to
$\{s,t\}$.
\item Each vertex $u \in \cB$ has at least 2 distinct neighbors in $\cA$.
\end{itemize}
Then $G$ admits a $(s,t)$-bipolar orientation.
\end{lemma}
\begin{proof}
It suffices to check that the augmented multigraph $G^+=(\{s,t\} \cup \cA \cup
\cB,\cE \cup (s,t))$ is biconnected, i.e.\ it remains connected upon removing
any vertex $v$ and its incident edges. We denote this graph by
$G^+ \setminus \{v\}$.

For $a \in \{s,t\}$, $G^+ \setminus \{a\}$ remains connected since the
subgraph induced by $\cA \cup \cB$ is connected, and the remaining vertex
of $\{s,t\}$ different from $a$ is a neighbor of $\cA$.

For $u \in \cB$, note that all vertices of $\cA$ are neighbors of $\{s,t\}$,
which are neighbors of each other via the added edge $(s,t)$.
Thus $\{s,t\} \cup \cA$ forms
a single connected component. Upon removing $u$, each remaining vertex of $\cB$
is a neighbor of $\cA$, so $G^+ \setminus \{u\}$ is connected. 

For $v \in \cA$, by the same reasoning, all
vertices of $\{s,t\} \cup (\cA \setminus \{v\})$ form a single connected
component of $G^+ \setminus \{v\}$. Upon removing $v$ and its incident edges,
each $u \in \cB$ remains a neighbor of $\cA \setminus \{v\}$, since $u$ has at
least 2 distinct neighbors in $\cA$. Thus $G^+ \setminus \{v\}$ is connected.
\end{proof}

\subsubsection{Tensor network representation}

\begin{lemma}\label{lem:graph of random feature}
Let $\x \in \R^n$ be a random vector having mean 0 and
finite moments of all orders, and let $\x^{\odot l}=(\x[i]^l)_{i=1}^n \in \R^n$
be its entrywise $l^\text{th}$ power.
Let $k \geq 1$, $l_1,\ldots,l_k \geq 1$, and
$U\in (\R^n)^{\otimes k}$. Then
\[\<\kappa_k(\x^{\odot l_1},...,\x^{\odot l_k}),U\>
=\sum_{(G,f)\in \cG} \mathrm{val}(G,f)\]
where $\cG \equiv \cG(l_1,\ldots,l_k,U)$ is the set of all distinct tensor
networks $(G,f)$ satisfying the properties:
\begin{enumerate}
\item $G=(\{t\} \cup \cK \cup \cV,\cE)$ where the vertices are partitioned
into disjoint sets $\{t\},\cK,\cV$. Each edge $e \in \cE$ has label $f(e)=I
\in \R^{n \times n}$, and connects one vertex of $\{t\} \cup \cK$ with one
vertex of $\cV$ (hence $G$ is bipartite). Furthermore,
the bipartite subgraph induced by $\cK \cup \cV$ is connected.
\item $\cV=\{v_1,\ldots,v_k\}$ where each $v_i$ has degree $\deg(v_i)=l_i+1$ 
and label $f(v_i)=I \in (\R^n)^{\otimes l_i+1}$. Here $I$ is the
order-$(l_i+1)$ identity tensor,
having all diagonal entries 1 and off-diagonal entries 0.
\item The vertex $t$ has $\deg(t)=k$, label $f(t)=U$, and its ordered edges
connect to $v_1,\ldots,v_k$.
\item Each vertex $u \in \cK$ has $\deg(u) \geq 2$ and label
$f(u)=\kappa_{\deg(u)}(\x) \in (\R^n)^{\otimes \deg(u)}$.
\end{enumerate}
\end{lemma}
\begin{proof}
By linearity, it suffices to show the lemma when $U=\e_{i_1} \otimes \ldots
\otimes \e_{i_k}$ is a standard basis element for every choice of indices
 $i_1,\ldots,i_k \in [n]$.

We induct on $k$. For $k=1$ and $U=\e_i$,
we have by the moment-cumulant relation
\[\<\kappa_1(\x^{\odot l}),U\>
=\kappa_1(\x[i]^l)=\E[\x[i]^l]=\sum_{\pi\in
P_l}\prod_{B\in\pi}\kappa_{|B|}(\x[i]),\]
where $P_l$ is the set of all partitions of $\{1,\ldots,l\}$.
Since $\x$ has mean 0, the summand corresponding to $\pi \in P_l$ is 0 if $\pi$
has a singleton block. Each remaining summand corresponding to
$\pi=\{B_1,\ldots,B_j\}$
may be understood as $\val(G,f)$ for the tensor network described by the lemma,
where $\cK$ has $j$ vertices of degrees $|B_1|,\ldots,|B_j|$ with all edges
connecting to the single vertex $v \in \cV$. The set $\cG(l,U)$
of such tensor networks is in 1-to-1 correspondence with such summands $\pi
\in P_l$. This proves the lemma for $k=1$.

Now suppose the lemma holds for all $k\leq t$, and consider $k=t+1$ and
$U=\e_{i_1} \otimes \ldots \otimes \e_{i_{t+1}}$. By the
moment-cumulant relation for a mixed moment of order $t+1$,
\[
\E\left[\prod_{m=1}^{t+1}\x[i_m]^{l_m}\right]=\sum_{\pi\in P_{t+1}}
\prod_{B\in\pi}\kappa_{|B|}(\x[i_b]^{l_b}:b\in B).
\]
Let $1_{t+1} \in P_{t+1}$ denote the partition having one block containing
all $t+1$ elements. Then rearranging gives 
\begin{align*}
\<\kappa_{t+1}(\x^{\odot l_1},\ldots,\x^{\odot l_{t+1}}),U\>
&=\kappa_{t+1}(\x[i_1]^{l_1},\ldots,\x[i_{t+1}]^{l_{t+1}})\\
&=\E\left[\prod_{m=1}^{t+1}\x[i_m]^{l_m}\right]
-\sum_{\pi\in P_{t+1}\setminus 1_{t+1}}\prod_{B\in\pi}
\kappa_{|B|}(\x[i_b]^{l_b}:b\in B).
\end{align*}

On the other hand, denote by $\cA(l_1,\ldots,l_k,U)$
the set of all distinct tensor networks $(G,f)$ satisfying
all properties of the lemma except the property that $\cK \cup \cV$ is
connected. Let $L=l_1+\ldots+l_{t+1}$, and let $j_1=\ldots=j_{l_1}=i_1$,
$j_{l_1+1}=\ldots=j_{l_1+l_2}=i_2$, ...,
$j_{l_t+1}=\ldots=j_{l_{t+1}}=i_{t+1}$. Applying the moment-cumulant relation
for a mixed moment of order $L$, we have
\[
\E\left[\prod_{m=1}^{t+1}\x[i_m]^{l_m}\right]
=\E\left[\prod_{m=1}^L\x[j_m]\right]=\sum_{\pi\in
P_L}\prod_{B\in\pi}\kappa_{|B|}(\x[j_b]:b\in
B)=\sum_{(G,f)\in\cA(l_1,\ldots,l_{t+1},U)}\val(G,f).
\]
For each $\pi\in P_{t+1}\setminus 1_{t+1}$ and block $B \in \pi$, let
$U_B=\otimes_{b \in B} \e_{i_b}$ be the corresponding factors of $U$.
Then by the induction hypothesis, we have
\[
    \kappa_{|B|}(\x[i_b]^{l_b}:b\in B)=\sum_{(G,f)\in\cG((l_b:b\in B),\,U_b)}
\mathrm{val}(G,f)
\]
Taking the product over blocks $B \in \pi$, we get
\[
\prod_{B\in\pi}\kappa_{|B|}(\x[i_b]^{l_b}:b\in
B)=\sum_{(G,f)\in\cG_\pi(l_1,\ldots,l_{t+1},U)}\mathrm{val}(G,f)
\]
where, if $\pi=\{B_1,\ldots,B_j\}$, then
$\cG_\pi(l_1,\ldots,l_{t+1},U)$ is the set of tensor networks in
$\cA(l_1,\ldots,l_{t+1},U)$ for which $\cK \cup \cV$ splits exactly into
$j$ connected components, with one component corresponding to each block $B \in
\pi$ that contains the vertices $\{v_b \in \cV:b \in B\}$.
Then, since $\cG(l_1,\ldots,l_{t+1},U)
=\cG_{1_{t+1}}(l_1,\ldots,l_{t+1},U)$ is precisely the set of tensor
networks with a single connected component,
\[
    \cA(l_1,\ldots,l_{t+1},U)\setminus\cG(l_1,\ldots,l_{t+1},U)=
\bigcup_{\pi\in P_{t+1}\setminus 1_{t+1}}\cG_\pi(l_1,\ldots,l_{t+1},U).
\]
Thus
\begin{align*}
\<\kappa_{t+1}(\x^{\odot l_1},\ldots,\x^{\odot l_{t+1}}),U\>
&=\sum_{(G,f)\in\cA(l_1,\ldots,l_{t+1},U)}\mathrm{val}(G,f)-\sum_{(G,f)\in\cA(l_1,\ldots,l_{t+1},U)\setminus\cG(l_1,\ldots,l_{t+1},U)}\mathrm{val}(G,f)\\
    &=\sum_{(G,f)\in\cG(l_1,\ldots,l_{t+1},U)}\mathrm{val}(G,f).
\end{align*}
This completes the inductive argument.
\end{proof}

\subsubsection{Proof of Assumption \ref{assum:concentration}}

\begin{lemma}\label{lem:concentration of quadratic form}
Let $\x \in \R^n$ be a random vector having mean 0 and finite moments of all
orders. Then for any $k\in\N$ and any deterministic, symmetric matrix
$A\in \R^{n\times n}$,
\[\E(\x^*A\x-\E\x^*A\x)^{2k}\leq \sum_{_{\text{partitions } \pi \text{ of } [4k]}}\prod_{B \in \pi} \|\kappa_{|B|}(\x)\|_\matop\|A\|_F^{2k}.\]
\end{lemma}
\begin{proof}
Let $\Sigma = \E\x \x^*$. We can rewrite the expectation as
\[
    \E(\x^* A \x - \E \x^* A \x)^{2k} = \E [\Tr((\x \x^* - \Sigma) A)]^{2k} = \< \E [(\x \x^* - \Sigma)^{\otimes 2k}], A^{\otimes 2k} \>.
\]
Applying Lemma \ref{lem:cumulant pairing},
the above value may be decomposed as
\[
    \< \E [(\x \x^* - \Sigma)^{\otimes 2k}], A^{\otimes 2k} \>=\sum_{(G,f)\in\cN}\val(G,f)
\]
where $\cN$ is the set of all distinct tensor networks $(G,f)$ satisfying the properties:
\begin{enumerate}
    \item $G=(\cV,\cE)$ where the vertices $\cV=\cA \cup \cB$
are partitioned into disjoint sets $\cA,\cB$. Each edge $e\in\cE$ has label
$f(e)=I\in\R^{n\times n}$, and connects one vertex of $\cA$ with one vertex of
$\cB$ (hence $G$ is bipartite).
    \item $\cA=\{a_1,\ldots,a_{2k}\}$ where each vertex $a_i$ has degree $2$ and label $f(a_i)=A\in\R^{n\times n}$.
    \item $\cB=\{b_1,\ldots,b_m\}$ for some $1\leq m\leq 2k$, where each vertex
$b_i$ satisfies $2\leq\deg(b_i)\leq 4k$ and has label $f(b_i)=\kappa_{\deg(b_i)}(\x)$.
    \item If any $b_i$ has degree 2, then the two edges incident to $b_i$
connect to distinct neighbors in $\cA$.
\end{enumerate}
Here $\deg(b_i) \geq 2$ because $\E \x=\kappa_1(\x)=0$, so only diagrams without
$\kappa_1(\x)$ contribute. Additionally, Lemma \ref{lem:cumulant pairing}
ensures that no copy of $A$ is contracted with a single cumulant tensor
$\kappa_2(\x)=\Sigma$, implying property 4. Note that these properties ensure
each connected component of $G$ contains at least 2 vertices from $\cA$,
and each vertex of $\cB$ has at least 2 distinct neighbors in $\cA$.

Since $\val(G,f)$ factorizes across connected components of $G$,
we may bound $\val(G',f)$ for each connected component $G'$ of $G$.
Let us modify $G'$ to yield a
network satisfying the conditions of Lemma \ref{lem:source-sink network}:
Let $A=U\diag(\v)U^*$ be the spectral decomposition of $A$. For each $a\in\cA
\cap G'$,
denote its two incident edges $e_1,e_2$. We introduce a new vertex $v$, connect
$v$ to $a$ via a new edge $e_3$ so that $\deg(v)=1$ and $\deg(a)=3$,
and relabel as
$f(v)=\v$, $f(a)=I\in(\R^n)^{\otimes 3}$ (the diagonal order-3 tensor with all
diagonal entries 1),
$f(e_3)=I\in\R^{n\times n}$, and $f(e_1)=f(e_2)=U\in\R^{n\times n}$. Note that
this does not change $\val(G',f)$.
Let $\cV$ be the set of newly added vertices. Then $|\cV|=|\cA \cap G'|\geq 2$. We
partition $\cV$ into two non-empty sets $\cV_1,\cV_2$ of cardinalities $i,j \geq
1$, merge the vertices in
$\cV_1$ into a single vertex $s$ having label $f(s)=\v^{\otimes i}$, and merge
the vertices in $\cV_2$ into a single vertex $t$ having label $f(t)=\v^{\otimes
j}$. This also does not change $\val(G',f)$.

It is readily checked that $G'$ now satisfies the condition of Lemma
\ref{lem:source-sink network}. Thus by Lemma \ref{lem:bounds by bipolar
orientation},
\[|\val(G',f)| \leq \|\v\|_2^i\|\v\|_2^j\|U\|_\op^{2|\cA \cap G'|}
\prod_{b \in \cB \cap G'}\|\kappa_{\deg(b)}(\x)\|_{\matop}
=\|A\|_F^{|\cA \cap G'|}\prod_{b \in \cB \cap
G'}\|\kappa_{\deg(b)}(\x)\|_{\matop},\]
where we have used $i+j=|\cV|=|\cA \cap G'|$ and $\|\v\|_2=\|A\|_F$.
Multiplying across all connected components $G'$ of $G$ and then
summing over all $(G,f) \in \cN$ completes the proof.
\end{proof}

\begin{lemma}\label{lem:bound i-to-j norm}
Fix any $k \geq 1$ and $l_1,\ldots,l_k \geq 1$.
Let $|\cG(l_1,\ldots,l_k)|$ be the cardinality of the set $\cG \equiv
\cG(l_1,\ldots,l_k,U)$ in Lemma \ref{lem:graph of random feature} (where this
cardinality does not depend on $U$). Let $l=l_1+\ldots+l_k$, and set
$C_l(\w)=\sup_{\text{partitions } \pi \text{ of } [l]}
\prod_{B \in \pi} \|\kappa_{|B|}(\w)\|_\infty$. Then
\[\|\kappa_k((X\w)^{\odot l_1},\ldots,(X\w)^{\odot l_k})\|_{\matop}
\leq C_l(\w)\|X\|_\op^l |\cG(l_1,\ldots,l_k)|.\]
\end{lemma}
\begin{proof}
Consider any partition of $[k]$ into two disjoint non-empty sets $I,J$,
and suppose $U \in (\R^n)^{\otimes k}$ factorizes as $S \otimes T$ with respect
to this partition, for some $S\in(\R^n)^{\otimes |I|}$ and
$T\in(\R^n)^{\otimes |J|}$. We have by Lemma \ref{lem:graph of random feature}
\[
    \<\kappa_k((X\w)^{\odot l_1},\ldots,(X\w)^{\odot l_k}),U\>=\sum_{G\in\cG}\val(G,f)
\]
where $\cG\equiv\cG(l_1,\ldots,l_k,U)$. We proceed to modify each tensor network $(G,f)\in\cG$ via a sequence of steps that do not change $\val(G,f)$, to yield a network satisfying the requirements of Lemma \ref{lem:source-sink network}:
\begin{enumerate}
    \item The vertex $t$ in Lemma \ref{lem:graph of random feature} has 
label $f(t)=U=S\otimes T$. Since $f(t)$ factorizes, we may split $t$ into two
new vertices $s,t$ with labels $f(s)=S$ and $f(t)=T$, such that edges
corresponding to $I$ connect to $s$ and edges corresponding to $J$ connect to
$t$.
\item Each vertex $u \in \cK$ has label $f(u)=\kappa_{\deg(u)}(X\w)$, with
$\deg(u)$ incident edges having label $I$. By
multilinearity of mixed cumulants, we may relabel
$f(u)=\kappa_{\deg(u)}(\w)$ and relabel each incident edge $e=(v,u)$ with $X$.
Note that $f(u)$ is then diagonal because $\w$ has independent entries.
    \item For each vertex $u\in\cK$ having a single unique neighbor
$v\in\cV$: Let its incident edges be $e_1,\ldots,e_{\deg(u)}=(v,u)$. Note that
$\deg(v)\geq\deg(u)+2$ since $v$ is connected to at least one of $\{s,t\}$
and to $(\cK\cup\cV)\setminus\{u,v\}$. (If $\cK \cup \cV=\{u,v\}$, this
continues to hold as then $v$ is connected to both $\{s,t\}$.)
We contract $u$ by removing $u$ and
$e_1,\ldots,e_{\deg(u)}$ from $(G,f)$, and replacing the diagonal tensor label $f(v)\equiv\diag(\widetilde{\v})\in(\R^n)^{\otimes\deg(v)}$ with the new label $\diag(\hat{\v})\in(\R^n)^{\otimes \deg(v)-\deg(u)}$ having entries
\[\hat\v[i]=\tilde \v[i] \times \bigg(\sum_{j=1}^n X[i,j]^{\deg(u)}
\kappa_{\deg(u)}(\w)[j]\bigg).\]
\end{enumerate}
At the end of Step 2, we have
\begin{equation}\label{eq:matopcondition}
\|f(s)\|_F\|f(t)\|_F \prod_{v \in \cK \cup \cV} \|f(v)\|_\infty
\prod_{e \in \cE} \|f(e)\|_\op \leq C_l(\w)\|X\|_\op^l\|S\|_F\|T\|_F.
\end{equation}
For each application of Step 3, since $\deg(u) \geq 2$, we have the bound
\begin{align*}
\|\hat\v\|_\infty &\leq \|\tilde\v\|_\infty \|\kappa_{\deg(u)}(\w)\|_\infty
\|X\|_\infty^{\deg(u)-2} \max_{i=1}^n \sum_{j=1}^n X[i,j]^2\\
&\leq \|\tilde\v\|_\infty\|\kappa_{\deg(u)}(\w)\|_\infty\|X\|_\op^{\deg(u)},
\end{align*}
and thus the inequality \eqref{eq:matopcondition} continues to hold after Step
3. It is readily checked that after all three steps of the above procedure, $G$
satisfies the condition of Lemma \ref{lem:source-sink network} with
$\cA\equiv\cV$ and $\cB\equiv\cK$. So applying Lemma \ref{lem:bounds by bipolar orientation} gives
\begin{align*}
&|\<\kappa_k(\x^{\odot l_1},\ldots,\x^{\odot
l_k}),U\>|\leq\sum_{G\in\cG}|\val(G,f)| \leq C_l(\w)\|X\|_\op^l |\cG(l_1,\ldots,l_k)| \cdot \|S\|_F\|T\|_F.
\end{align*}
Since $I,J$ and $S,T$ defining $U=S \otimes T$ are arbitrary,
by definition of $\|\cdot\|_\matop$, this completes the proof.
\end{proof}

\begin{proof}[Proof of Proposition \ref{prop:RF},
Assumption \ref{assum:concentration}]
Denote $\a_l=(a_{il})_{i=1}^n$ (where in the setting of Proposition
\ref{prop:RF}(a), $\a_l=0$ for $l>D$). For any fixed $k \geq 1$ and
$l_1,\ldots,l_k \geq 1$, we may expand by multilinearity of the mixed cumulant
\begin{align*}
\kappa_k(\a_{l_1} \odot (X\w)^{\odot l_1},\ldots,\a_{l_k} \odot
(X\w)^{\odot l_k})
=(\a_{l_1} \otimes \ldots \otimes \a_{l_k}) \odot
\kappa_k((X\w)^{\odot l_1},\ldots,(X\w)^{\odot l_k}).
\end{align*}
Thus, noting that $\g=\sigma(X\w)=\sum_{l=0}^\infty \a_l \odot (X\w)^{\odot l}$
and any mixed cumulant $\kappa_k(\ldots)$ is 0 if an argument is constant
(non-random), we have
\begin{align*}
\|\kappa_k(\g)\|_\matop&=\left\|\sum_{l_1,\ldots,l_k=1}^\infty
(\a_{l_1} \otimes \ldots \otimes \a_{l_k}) \odot
\kappa_k((X\w)^{\odot l_1},\ldots,(X\w)^{\odot l_k})\right\|_\matop\\
&\leq \sum_{l_1,\ldots,l_k=1}^\infty \Big(\prod_{i=1}^k
\|\a_{l_i}\|_\infty\Big)
\|\kappa_k((X\w)^{\odot l_1},\ldots,(X\w)^{\odot l_k})\|_\matop\\
&\leq \sum_{l_1,\ldots,l_k=1}^\infty \Big(\prod_{i=1}^k \|\a_{l_i}\|_\infty\Big)
C_l(\w)\|X\|_\op^l |\cG(l_1,\ldots,l_k)|,
\end{align*}
the last inequality using Lemma \ref{lem:bound i-to-j norm}.

In the setting of Proposition \ref{prop:RF}(a), $C_l(\w)$, $\|X\|_\op^l$,
$|\cG(l_1,\ldots,l_k)|$, and $\prod_{i=1}^k \|\a_{l_i}\|_\infty$
 are all bounded by a constant depend only on $k,D$, so
$\|\kappa_k(\g)\|_\matop \leq C$ for a constant $C \equiv C(k,D)>0$.

In the setting of Proposition \ref{prop:RF}(b) where $\w \sim N(0,I)$, we have
$\kappa_2(\w)=I$ and $\kappa_k(\w)=0$ for all $k \geq 3$. Then
$C_l(\w)=1$ for all even $l \geq 2$, and $C_l(\w)=0$ for odd $l$. In Lemma
\ref{lem:graph of random feature}, also $\val(G,f)=0$ unless $l=l_1+\ldots+l_k$
is even and each vertex $u \in \cK$ has degree exactly 2, in which
case $|\cG(l_1,\ldots,l_k)| \leq l!! \leq C^l l^{l/2}$
where $l!!$ is the number of partitions of $l$ into $l/2$ blocks of size 2.
Applying the given condition
$\|\a_{l_i}\|_\infty<C(l_i!)^{-\beta}<C(l_i/e)^{-\beta l_i}$ and
$\sum_{i=1}^k l_i \log l_i>l \log (l/k)$,
we have that $\prod_{i=1}^k \|\a_{l_i}\|_\infty
<(Cke^\beta)^l l^{-\beta l}$. Then, for a constant $C'>0$ not depending on
$k,l$ and for any $l \geq k$,
\begin{equation}\label{eq:GaussianRFbound}
\mathop{\sum_{l_1,\ldots,l_k \geq 1}}_{l_1+\ldots+l_k=l}
\Big(\prod_{i=1}^k \|\a_{l_i}\|_\infty\Big)
C_l(\w)\|X\|_\op^l |\cG(l_1,\ldots,l_k)|
\leq \mathop{\sum_{l_1,\ldots,l_k \geq 1}}_{l_1+\ldots+l_k=l}
(Ck)^l l^{(\frac{1}{2}-\beta)l}
\leq (Ck^2)^l l^{(\frac{1}{2}-\beta)l}.
\end{equation}
Since $\beta>1/2$, this is summable over $l \geq k$, so also
$\|\kappa_k(\g)\|_\matop \leq C$ for a constant $C \equiv C(k,\beta)>0$ in this
case.

Then by Lemma \ref{lem:concentration of quadratic form}, for any
$A \in \R^{n \times n}$, we have $\E(\x^* A\x-\E\x^* A\x)^{2k}
\leq C\|A\|_F^{2k}$ for each fixed $k \geq 1$ and a constant $C \equiv
C(k)>0$ (where this holds also for asymmetric $A$ upon applying the lemma
with $(A+A^*)/2$). This implies Assumption \ref{assum:concentration} by
Markov's inequality.
\end{proof}

\subsubsection{Proof of Assumption \ref{assum:cumulant}}

For each integer $l \geq 0$, let us now define a set of vectors
$\cU(l) \subset \R^n$ for verifying Assumption \ref{assum:cumulant},
as follows: Define first the normalized matrices and
scalar cumulants, for $k \geq 2$,
\[\bar X=X/\|X\|_\op, \qquad \bar
\kappa_k(\w[j])=\kappa_k(\w[j])/\|\kappa_k(\w)\|_\infty.\]
(If $\|\kappa_k(\w)\|_\infty=0$, we set $\bar\kappa_k(\w[j])=0$.)
Define the diagonal matrices
\[D_k=\diag\bigg(\bigg(\sum_{j=1}^n \bar X[i,j]^k
\bar\kappa_k(\w[j])\bigg)_{i=1}^n\bigg),
\qquad
K_k=\diag\big(\big(\bar\kappa_k(\w[j])\big)_{j=1}^n\big).\]
Let $\cM_0=\{I,\bar X,\bar X^*\} \cup \{D_k\}_{k \geq 2} \cup
\{K_k\}_{k \geq 2}$. We say that $I$ has degrees 0 in both $X$ and $\w$, that
$\bar X,\bar X^*$ have degree 1 in $X$ and 0 in $\w$,
that $D_k$ has degrees $k$ in both
$X$ and $\w$, and that $K_k$ has degree 0 in $X$ and $k$ in $\w$.
For each $j \geq 1$, let $\cM_j=\{AB,A \odot B:A,B \in \cM_{j-1}\}$ (where $AB$
is the matrix product and $A \odot B$ is the entrywise product). Each
set $\cM_j$ is closed under matrix transpose, since this holds for $\cM_0$.
If $A,B$ have degrees $a_x,b_x$ in $X$ and $a_w,b_w$ in $\w$, then
we say that $AB$ and $A \odot B$ have degrees $a_x+b_x$ in $X$ and $a_w+b_w$
in $\w$. Let $\cM=\bigcup_{j \geq 0} \cM_j$, i.e.\ $\cM$ consists of all
matrices obtained from $I,\bar X,D_k,K_k$ through combinations of matrix
products, entrywise products, and matrix transposes, and let
$\cM(l) \subset \cM$ be
those elements with degree at most $l$ in $X$ and at most $l$ in $\w$. We set
\begin{equation}\label{eq:Uldef}
\cU(l)=\{M\e_i:M \in \cM(l),\,M \neq 0,\,i \in [n]\}.
\end{equation}

The following lemma is the main step to verify Assumption \ref{assum:cumulant}.

\begin{lemma}\label{lem:RFmonomials}
Fix any $k \geq 3$ and $l_1,\ldots,l_k \geq 1$, and let $l=l_1+\ldots+l_k$.
Let $|\cG(l_1,\ldots,l_k)|$ and $C_l(\w)$ be as defined in
Lemma \ref{lem:bound i-to-j norm}.
Then for any $m \in \{1,\ldots,k-1\}$, $\s_1,\ldots,\s_m \in \R^n$ and
$T \in (\R^n)^{\otimes k-m}$,
\begin{equation}\label{eq:RFmonomialsbound}
\Big|\<\kappa_k((X\w)^{\odot l_1},\ldots,(X\w)^{\odot l_k}),
\s_1 \otimes \ldots \otimes \s_m \otimes T\>\Big|
\leq C_l(\w)\|X\|_\op^l |\cG(l_1,\ldots,l_k)| \times
\sqrt{n}^{k-m-1}\|T\|_{\cU(l)}\prod_{i=1}^m \|\s_i\|_2.
\end{equation}
\end{lemma}

We proceed to show Lemma \ref{lem:RFmonomials}.
For the later combinatorial arguments and casework, it will be convenient
to first specialize Lemma \ref{lem:graph of random feature} to this setting 
where $U=\s_1 \otimes \ldots \otimes \s_m \otimes T$,
and to simplify the networks by contracting
vertices with 1 or 2 unique neighbors. The structure of the resulting simplified
networks is captured by the following lemma.

\begin{lemma}\label{lem:contract network into degree 3}
In the setting of Lemma \ref{lem:RFmonomials}, there exists a set $\cG$ of
tensor networks $(G,f)$ having cardinality $|\cG(l_1,\ldots,l_k)|$ such that
\[\<\kappa_k((X\w)^{\odot l_1},\ldots,(X\w)^{\odot l_k}),
\s_1 \otimes \ldots \otimes \s_m \otimes T\>
=\sum_{(G,f)\in \cG} \mathrm{val}(G,f),\]
and each $(G,f) \in \cG$ satisfies the following properties:
\begin{enumerate}
\item $G=(\{s,t\} \cup \cK \cup \cV,\cE)$ where the vertices are
partitioned into disjoint sets $\{s,t\},\cK,\cV$. The label
$f(v)$ for each $v \in \cK \cup \cV$ is a diagonal tensor,
the label $f(s)$ is a rank-1 tensor (or a vector if $\deg(s)=1$),
and the label $f(e)$ for each edge incident to $\{s,t\}$ is $I \in \R^{n
\times n}$.
\item The subgraph induced by $\cK \cup \cV$ is connected. Every
edge that is incident to a vertex in $\{s,t\} \cup \cK$ connects to a
vertex in $\cV$. (Vertices of $\cV$ may also connect to each other, so
$G$ is not necessarily bipartite.)
The vertices $\{s,t\}$ have degrees $\deg(s)=m$
and $\deg(t)=k-m$, and each vertex of $\cV$ is incident to
at least one edge connecting to $\{s,t\}$.
\item Each vertex $u \in \cK \cup \cV$ has $\deg(u) \geq 3$, and any two
vertices $u,v \in \cK \cup \cV$ are connected by at most 1 edge (i.e.\ the
subgraph induced by $\cK \cup \cV$ is simple).
\item Let $\cU(l)$ be as defined in \eqref{eq:Uldef}, and let
$C_l(\w)$ be as defined in Lemma \ref{lem:RFmonomials}. Then
\begin{equation}\label{eq:RFnetworkmainbound}
\|f(s)\|_F\|f(t)\|_\infty
\prod_{v \in \cK \cup \cV}
\|f(v)\|_\infty \prod_{e \in \cE} \|f(e)\|_\op
\leq C_l(\w)\|X\|_\op^l \|T\|_{\cU(l)} \prod_{i=1}^m \|\s_i\|_2
\end{equation}
\end{enumerate}
\end{lemma}
\begin{proof}
We begin with the representation of Lemma \ref{lem:graph of random feature}
applied to $\x=X\w$,
\[\<\kappa_k((X\w)^{\odot l_1},\ldots,(X\w)^{\odot l_k}),
\s_1 \otimes \ldots \otimes \s_m \otimes T\>
=\sum_{(G,f) \in \cG} \val(G,f)\]
where $\cG \equiv \cG(l_1,\ldots,l_k,\s_1 \otimes \ldots \otimes \s_m \otimes
T)$.
We proceed to modify each tensor network $(G,f) \in \cG$ via a sequence of
steps that do not change $\val(G,f)$, to yield a network satisfying the stated
properties:
\begin{enumerate}
\item As in Lemma \ref{lem:bound i-to-j norm},
we split the vertex $t$ with label
$f(t)=\s_1 \otimes \ldots \otimes \s_m \otimes T$ into
two vertices $s,t$ with labels $f(s)=\s_1 \otimes \ldots \otimes \s_m$
and $f(t)=T$, such that the first $m$ edges are
incident to $s$ and the last $k-m$ edges are incident to $t$.
\item As in Lemma \ref{lem:bound i-to-j norm}, we relabel each $u \in \cK$ by
(the diagonal label) $f(u)=\kappa_{\deg(u)}(\w)$ and relabel each incident edge 
$e=(v,u)$ with $X$. This resulting network satisfies
Conditions 1--2 of Lemma \ref{lem:contract network into degree 3}.
\item We then repeatedly apply the following steps. It is readily checked
that each step preserves $\val(G,f)$ and Conditions 1--2:
\begin{enumerate}[label=(\alph*)]
\item As in Lemma \ref{lem:bound i-to-j norm}, if
a vertex $u \in \cK$ has a single unique neighbor
$v \in \cV$ with incident edges $e_1,\ldots,e_{\deg(u)}=(v,u)$,
we contract $u$ by removing $u$ and
$e_1,\ldots,e_{\deg(u)}$ from $(G,f)$, and replacing the label
label $f(v) \equiv \diag(\tilde \v) \in (\R^n)^{\otimes \deg(v)}$
with the new label $\diag(\hat\v) \in (\R^n)^{\otimes \deg(v)-\deg(u)}$
having entries
\[\hat\v[i]=\tilde \v[i] \times \bigg(\underbrace{\sum_{j=1}^n X[i,j]^{\deg(u)}
\kappa_{\deg(u)}(\w)[j]}_{=\|X\|_\op^{\deg(u)} D_{\deg(u)}[i,i]}\bigg).\]
\item Suppose two vertices $u,v \in \cK \cup \cV$ each have at least 2
distinct neighbors, and $u,v$ are also connected by
$a \geq 2$ edges $e_1,\ldots,e_a=(u,v)$. We replace these edges by a
single edge $e=(u,v)$ having label $f(e)=f(e_1) \odot \ldots \odot f(e_a)$ (the
entrywise product), and reduce the orders of $f(u),f(v)$ by $a-1$ while keeping
their diagonal entries the same.
\item Suppose a vertex $u \in \cK$ has $\deg(u)=2$ and two distinct
neighbors $v,v' \in \cV$. Let $e=(v,u)$ and $e'=(u,v')$.
We contract $u$ by removing $u$ and $e,e'$, and replacing these with the
single new edge $e''=(v,v')$ labeled by $f(e'')=f(e)f(u)f(e')$.
\item Suppose a vertex $v \in \cV$ has $\deg(v)=2$ and is adjacent to one
vertex in $\{s,t\}$ --- say, $s$ via its last incident edge --- and
one vertex $u \in \cK \cup \cV$. Let $e=(s,v)$ and $e'=(v,u)$,
and note that $f(u)$ is diagonal and $f(e)=I$. We contract $v$ by removing
$v$ and $e,e'$, replacing these by a single new edge $e''=(s,u)$ with
label $f(e'')=I$, 
reassigning $u$ to the vertex set $\cV$ if it belongs to $\cK$ (so that each
edge incident to $s$ still connects to $\cV$), and replacing
$f(s) \equiv \widetilde S \in (\R^n)^{\otimes \deg(s)}$ by the new label
$\widehat S \in (\R^n)^{\otimes \deg(s)}$ given by the contraction
\[\widehat S[i_1,\ldots,i_{\deg(s)}]
=\sum_{j=1}^n \widetilde S[i_1,\ldots,i_{\deg(s)-1},j]f(v)[j,j]
f(e')[j,i_{\deg(s)}]\]
Note that if
$\widetilde S=\tilde \s_1 \otimes \ldots \otimes \tilde \s_{\deg(s)}$,
then $\widehat S=\tilde \s_1 \otimes \ldots \otimes \tilde \s_{\deg(s)-1}
\otimes (f(e')^*f(v)\tilde \s_{\deg(s)})$, so
$\widehat S$ remains rank-1, as required in Condition 1.
\end{enumerate}
We apply the above steps (a--d) iteratively in any order, until none of these
steps are further applicable. (The number of such applications is finite, since
each step removes at least one vertex or edge.)
\end{enumerate}

We now check Conditions 3--4 for this resulting network: Note first that since
Conditions 1--2 hold, it is either the case that every $v \in \cV$ has at
least 1 neighbor in $\{s,t\}$ and at least 1 neighbor in $(\cK \cup \cV)
\setminus \{v\}$, or $\cK \cup \cV=\{v\}$ consists of the single vertex $v$
which has $m$ edges connecting to $s$ and $k-m$ edges connecting to $t$.
In both cases, $v$ has at least 2 distinct neighbors.
We have $\deg(v) \geq 3$ in the former case because Step 3(d) is no longer
applicable, and $\deg(v) \geq 3$ in the latter case because $k \geq 3$.
For each $u \in \cK$, we also have that $u$ has at least 2 distinct neighbors
and $\deg(u) \geq 3$, because Steps 3(a,c) are no
longer applicable. Then any two vertices $u,v \in \cK \cup \cV$ are connected
by at most 1 edge, as Step 3(b) is no longer applicable. This shows Condition 3.

To check Condition 4, note that the network $(G,f)$ after Step 2 has $f(s)=\s_1
\otimes \ldots \otimes \s_m$, $f(t)=T$, and
\begin{equation}\label{eq:Condition4initial}
\|f(s)\|_F\|f(t)\|_\infty \prod_{v \in \cK \cup \cV}
\|f(v)\|_\infty \prod_{e \in \cE} \|f(e)\|_\op
=\Big(\prod_{i=1}^m \|\s_i\|_2\Big)\|T\|_\infty
\Big(\prod_{u \in \cK} \|\kappa_{\deg(u)}(\w)\|_\infty\Big)
\|X\|_\op^l.
\end{equation}
Thus Condition 4 holds for this network. For each application of Step 3(a),
since $\deg(u) \geq 2$, we have
\begin{align*}
\|\hat\v\|_\infty &\leq \|\tilde\v\|_\infty \|\kappa_{\deg(u)}(\w)\|_\infty
\|X\|_\infty^{\deg(u)-2} \max_{i=1}^n \sum_{j=1}^n X[i,j]^2\\
&\leq \|\tilde\v\|_\infty\|\kappa_{\deg(u)}(\w)\|_\infty\|X\|_\op^{\deg(u)}.
\end{align*}
For each application of Step 3(b), we have
\[\|f(e)\|_\op 
=\|f(e_1) \odot \ldots \odot f(e_a)\|_\op
\leq \|f(e_1)\|_\op \ldots \|f(e_a)\|_\op.\]
For each application of Step 3(c), we have
\[\|f(e'')\|_\op=\|f(e)f(u)f(e')\|_\op
\leq \|f(e)\|_\op\|f(u)\|_\infty\|f(e')\|_\op.\]
For each application of Step 3(d) to $s$, we have
\[\|\widehat S\|_F
=\|\tilde\s_1\|_2\ldots\|\tilde \s_{\deg(s)-1}\|_2
\|f(e')^*f(u)\tilde \s_{\deg(s)}\|_2
\leq \|f(e')\|_\op\|f(u)\|_\infty \|\widetilde S\|_F.\]
Thus, each of these operations do not increase the left side of
\eqref{eq:Condition4initial}.

Each time Step 3(d) is applied to $t$, a matrix of
the form $f(v)f(e')$ is contracted with one axis of the tensor $f(t) \in
(\R^n)^{\otimes k-m}$.
Here, $f(v)f(e')$ is obtained from successively applying matrix products,
entrywise products, and transposes to the matrices $\|X\|_\op^k D_k$,
$\diag((\kappa_k(\w[j]))_{j=1}^n)=\|\kappa_k(\w)\|_\infty K_k$, and
$X=\|X\|_\op\bar X$ arising in the preceding applications of Steps 3(a--c).
Hence, the final label $f(t)$ is given by the contraction of $T$ along each
axis $j=1,\ldots,k-m$ with a matrix of the form
$\|X\|_\op^{a_j} \prod_{k \in B_j}
\|\kappa_k(\w)\|_\infty \cdot M_j$ for some $a_j \geq 0$, some list $B_j$ (possibly empty) of integers $\geq 2$,
and some matrix $M_j \in \cM$ having degree $a_j$ in $X$ and
$b_j:=\sum_{k\in B_j} k$ in $\w$, where $\cM$ and this notion of degree are as
defined preceding \eqref{eq:Uldef}. In the network after Step 2, we have $l$
edges with label $X$, and also $\sum_{u \in \cK} \deg(u)=l$. Thus, for the
label $f(t)$ after the conclusion of the above steps, we must have
$a_1+\ldots+a_{k-m} \leq l$ and $b_1+\ldots+b_{k-m} \leq l$, and in particular
$M_j \in \cM(l)$ for each $j=1,\ldots,k-m$. Thus
\[\|f(t)\|_\infty \leq \|T\|_{\cU(l)} \prod_{j=1}^{k-m}
\left(\|X\|_\op^{a_j} \prod_{k \in B_j} \|\kappa_k(\w)\|_\infty\right).\]
Whenever Step 3(d) is applied to contract a matrix of the form
$f(v)f(e')=\|X\|_\op^a \prod_{k \in B} \|\kappa_k(\w)\|_\infty M$
into $f(t)$, this factor of
$\|X\|_\op^a \prod_{k \in B} \|\kappa_k(\w)\|_\infty$ is removed
from the bound \eqref{eq:Condition4initial} for the remaining factors
on the left side of \eqref{eq:Condition4initial}. Thus, for the final network
$(G,f)$, we have the bound
\[\|f(s)\|_F\|f(t)\|_\infty \prod_{v \in \cK \cup \cV}
\|f(v)\|_\infty \prod_{e \in \cE} \|f(e)\|_\op
\leq \Big(\prod_{i=1}^m \|\s_i\|_2\Big)\|T\|_{\cU(l)}
\Big(\prod_{u \in \cU} \|\kappa_{\deg(u)}(\w)\|_\infty\Big)
\|X\|_\op^l\]
which shows Condition 4 since
$\prod_{u \in \cK} \|\kappa_{\deg(u)}(\w)\|_\infty \leq C_l(\w)$.
\end{proof}

Note that each network $(G,f)$ of
Lemma \ref{lem:contract network into degree 3} satifies
the conditions of Lemma \ref{lem:source-sink network} with $\cA=\cV$ and
$\cB=\cK$. Then, by Lemma \ref{lem:bounds by bipolar orientation},
we have
\[|\val(G,f)| \leq \|f(s)\|_F\|f(t)\|_F
\prod_{u \in \cK \cup \cV} \|f(u)\|_\matop \prod_{e \in \cE} \|f(e)\|_\op.\]
Applying also $\|f(t)\|_F \leq \sqrt{n}^{k-m}\|f(t)\|_\infty$ for any $f(t) \in
(\R^n)^{\otimes k-m}$, and $\|f(u)\|_\matop=\|f(u)\|_\infty$ because $f(u)$ is
diagonal, Lemma \ref{lem:contract network into degree 3} then implies
the bound
\begin{align*}
&\Big|\<\kappa_k((X\w)^{\odot l_1},\ldots,(X\w)^{\odot l_k}),
\s_1 \otimes \ldots \otimes \s_m \otimes T\>\Big|
\leq \sum_{(G,f) \in \cG} |\val(G,f)|\\
&\hspace{2in}\leq \sum_{(G,f) \in \cG} \sqrt{n}^{k-m}\|f(s)\|_F\|f(t)\|_\infty
\prod_{u \in \cK \cup \cV} \|f(u)\|_\infty
\prod_{e \in \cE} \|f(e)\|_\op\\
&\hspace{2in}\leq C_l(\w)\|X\|_\op^l|\cG(l_1,\ldots,l_k)| \times \sqrt{n}^{k-m}
\|T\|_{\cU(l)}\prod_{i=1}^m \|\s_i\|_2.
\end{align*}
This shows a weaker form of \eqref{eq:RFmonomialsbound} with factor
$\sqrt{n}^{k-m}$ instead of $\sqrt{n}^{k-m-1}$.

We conclude the proof of Lemma \ref{lem:RFmonomials} by
improving this factor to $\sqrt{n}^{k-m-1}$ using a more involved combinatorial argument and some casework.

\begin{lemma}\label{lem:good vertices}
Let $(G,f) \in \cG$ be any tensor network satisfying the properties of
Lemma \ref{lem:contract network into degree 3}. Suppose that $k-m \geq 2$,
and that the vertex $t$ connects to $k-m$ distinct neighbors $\cV_T \subseteq
\cV$. Let $G'$ be the subgraph of $G$ induced by $\{s\} \cup \cK \cup \cV$.

Then $G'$ contains a tree $H$ satisfying following property:
Call a leaf vertex $v$ of $H$ ``good'' if $v \in \cV_T$ and
$v$ is connected via a path of edges in $G' \setminus H$
to either a different vertex $u \in \cV_T$ or to $s$.
Then $H$ has at least 2 good leaf vertices.
\end{lemma}
\begin{proof}
Since the subgraph induced by $\cK \cup \cV$ is connected, there
exists at least one tree in this subgraph containing $\cV_T$. Among all such
trees, let $H$ be one having the smallest number of vertices.
Then each leaf of $H$ belongs to $\cV_T$, and all edges incident to $s$ belong
to $G' \setminus H$.

Let $\cV_S \subseteq \cV$ be the vertices neighboring $s$. Note that
$\cV=\cV_S \cup \cV_T$, and $\cV_S \cap \cV_T$ may be non-empty as a vertex
can neighbor both $s$ and $t$.
Let us call an edge of $G' \setminus H$ a ``non-tree edge''.
Since $|\cV_T|=\deg(t)=k-m$, the vertex  $t$ has only a single
edge in $G$ connecting to each vertex of $\cV_T$. Then, since every
$v \in \cV_T$ has degree at least 3 in $G$,
every leaf $v$ of $H$ must have at least one incident non-tree edge $(v,w)$.
If $w=s$, $w \in \cV_T$, or $w \in \cV_S$ (and hence connects to $s$ via a
non-tree edge), then by definition $v$ is
good. If $w \in \cK$ and $w\notin H$, then letting $u \in \cV_S \cup \cV_T$
be any neighbor of $w$ different from $v$,
$(w,u)$ must be a non-tree edge, so $v$ is also good. Thus we have
established the following claim:
\begin{center}
If a leaf vertex $v$ of $H$ is not good,
then there exists a non-tree edge $(v,w)$ where $w \in H \cap \cK$.
\end{center}
Note that if $H$ consists of two leaf vertices
connected by a single edge, or if $H$ is a star consisting of a single vertex $u
\in \cK \cup \cV$ connecting by an edge to all other vertices of $H$
(which are leaves and hence belong to $\cV_T$),
then no such non-tree edge $(v,w)$ can exist. Then all leaf vertices
of $H$ are good, so in particular $H$ has at least 2 good leaves, and the lemma
holds for $H$.

For any other tree $H$, there is a path of at least 3 distinct edges
in $H$, say $u-r-r'-v$. Let us call the two components of $H$ connected to $r$
and to $r'$ upon removing the edge $r-r'$ the
``sub-trees rooted at $r$ and $r'$'' respectively. Let $\cL$ be the set of
leaves in the sub-tree rooted at $r$ that have maximal distance from $r$,
and define similarly $\cL'$ for the sub-tree rooted at $r'$. If $\cL$ and $\cL'$
each contain a good leaf, then again $H$ has at least 2 good leaves, so the
lemma holds for $H$.

Otherwise, at least one of $\cL$ and $\cL'$ --- say, $\cL$ --- does not contain
any good leaf. We then consider the following modification of $H$:
Take any $v \in \cL$, let $p(v)$ be its parent in the sub-tree rooted at $r$,
and let $\cL_v$ be the set of all children of $p(v)$ (including $v$).
Each $v' \in \cL_v$ is a leaf vertex that is not good, so by the above claim,
for each $v' \in \cL_v$, there exists a non-tree edge $(v',w(v'))$ where
$w(v') \in H \cap \cK$. We consider the graph
$\tilde H$ that removes the edges $\{(v',p(v)):v' \in \cL_v\}$ from $H$ and adds
the new edges $\{(v',w(v')):v' \in \cL_v\}$. Note that $\tilde H$ is a tree
having the same vertices as $H$, that each $v' \in \cL_v$ and $p(v)$ itself are
now leaves of $\tilde H$, and that each leaf of $\tilde H$ other than $p(v)$
was a leaf of $H$ and hence belongs to $\cV_T$.
We must have $p(v) \in \cV_T$, because otherwise
we may remove $p(v)$ from $\tilde H$ to obtain a smaller tree containing
$\cV_T$, contradicting that $H$ was such a tree having the smallest number of
vertices. Thus all leaves of $\tilde H$ belong to $\cV_T$. Furthermore,
each $v' \in \cL_v$ connects to $p(v)$ by an edge not belonging to
$\tilde H$, so $p(v)$ and all vertices of $\{v' \in \cL_v\}$ are good leaves of
$\tilde H$. Thus $\tilde H$ has at least 2 good leaves, so the lemma holds for
$\tilde H$.
\end{proof}

\begin{proof}[Proof of Lemma \ref{lem:RFmonomials}]
Consider any tensor network $(G,f) \in \cG$ given in
Lemma \ref{lem:contract network into degree 3}. 
Recall that the vertex $t$ has $\deg(t)=k-m$ and label $f(t)
\in (\R^n)^{\otimes k-m}$. We will establish the bound
\begin{equation}\label{eq:valGfmainbound}
|\val(G,f)| \leq \sqrt{n}^{k-m-1}\|f(s)\|_F\|f(t)\|_\infty
\prod_{v \in \cK \cup \cV} \|f(v)\|_\infty
\prod_{e \in \cE}\|f(e)\|_\op.
\end{equation}

{\bf Case 1:} 
Suppose that $t$ has strictly fewer than $k-m$ distinct neighbors,
i.e.\ that some two edges incident to $t$ --- say, its last two edges ---
connect to
the same vertex of $\cV$. Consider the modified network $(\tilde G,\tilde f)$
that replaces the last two edges incident to $t$ by a single edge with
label $\tilde f(e)=I$, and replaces $f(t)$ by $\tilde f(t) \in
(\R^n)^{\otimes k-m-1}$ having entries
\[\tilde f(t)[i_1,\ldots,i_{k-m-1}]
=f(t)[i_1,\ldots,i_{k-m-2},i_{k-m-1},i_{k-m-1}].\]
Then $\val(G,f)=\val(\tilde G,\tilde f)$. Lemma \ref{lem:source-sink network}
(with $\cA=\cV$ and $\cB=\cK$)
implies that $\tilde G$ still admits a $(s,t)$-bipolar orientation. Then,
denoting by $\tilde \cE$ the edges of $\tilde G$, we have
by Lemma \ref{lem:bounds by bipolar orientation} that
\begin{align*}
|\val(G,f)|=|\val(\tilde G,\tilde f)| &\leq \|\tilde f(s)\|_F\|\tilde f(t)\|_F
\prod_{v \in \cK \cup \cV} \|\tilde f(v)\|_\infty
\prod_{e \in \tilde \cE}\|\tilde f(e)\|_\op\\
&\leq \sqrt{n}^{k-m-1}\|f(s)\|_F\|f(t)\|_\infty
\prod_{v \in \cK \cup \cV} \|f(v)\|_\infty
\prod_{e \in \cE}\|f(e)\|_\op,
\end{align*}
the second inequality applying
$\|\tilde f(t)\|_F \leq \sqrt{n}^{k-m-1}\|\tilde f(t)\|_\infty
=\sqrt{n}^{k-m-1}\|f(t)\|_\infty$.\\

{\bf Case 2:}
Suppose now that $t$ has exactly $k-m$ distinct neighbors, and $k-m=1$.
Let $v_t \in \cV$ be the vertex adjacent to $t$. We modify the network $(G,f)$
into a new network $(\tilde G,\tilde f)$ via two steps:
\begin{itemize}
\item We contract $f(t) \in \R^n$
by removing $f(t)$ and its incident edge, and replacing
the diagonal label $f(v_t) \equiv \diag(\v) \in (\R^n)^{\otimes \deg(v)}$ by
$\tilde f(v_t)=\diag(\v \odot f(t)) \in (\R^n)^{\otimes \deg(v)-1}$.
\item Note that $m \geq 2$, since $k \geq 3$. Let $e_1,\ldots,e_m$ be the
ordered edges incident to $s$ in $G$, and suppose its (rank-1) label factorizes
as $f(s)=\tilde \s_1 \otimes \ldots \otimes \tilde \s_m$. We 
split $s$ into two vertices $s',s''$ of $\tilde G$,
where $s'$ is incident to $e_1$ with label $\tilde f(s')=\tilde \s_1$, and
$s''$ is incident to $e_2,\ldots,e_m$ with label $\tilde f(s'')=\tilde \s_2 \otimes \ldots \otimes \tilde \s_m$.
\end{itemize}
Then $\val(G,f)=\val(\tilde G,\tilde f)$. We claim that $(\tilde G,\tilde
f)$ admits a $(s',s'')$-bipolar orientation. Indeed, if $v_t$ has an edge
connecting to either $s'$ or $s''$ in $\tilde G$, then Lemma
\ref{lem:source-sink network} (with $\cA=\cV$ and $\cB=\cK$) applies to show
$(\tilde G,\tilde f)$ has a $(s',s'')$-bipolar orientation.
Otherwise, each edge of $\tilde G$
incident to $v_t$ must connect to a distinct neighbor in $\cK \cup \cV$,
and there are at least 2 such edges, with each such neighbor
having degree $\geq 3$. Then
the graph $\tilde G \setminus \{v_t\}$ removing $v_t$ and its incident edges
admits a $(s',s'')$-bipolar orientation by Lemma \ref{lem:source-sink network}
applied with $\cA=\cV \setminus \{v_t\}$ and $\cB=\cK$.
Since $v_t$ has at least 2 distinct neighbors in 
$\tilde G \setminus \{v_t\}$, this implies that $\tilde G$ also admits
a $(s',s'')$-bipolar orientation, verifying our claim.
Then, denoting by $\tilde \cE$ the edges of $(\tilde G,\tilde f)$, we have
by Lemma \ref{lem:bounds by bipolar orientation} that
\begin{align*}
|\val(G,f)|=|\val(\tilde G,\tilde f)|
&\leq \|\tilde f(s')\|_F\|\tilde f(s'')\|_F
\prod_{v \in \cK \cup \cV} \|\tilde f(v)\|_\infty
\prod_{e \in \tilde \cE} \|\tilde f(e)\|_\infty\\
&\leq \|f(s)\|_F\|f(t)\|_\infty
\prod_{v \in \cK \cup \cV} \|f(v)\|_\infty
\prod_{e \in \cE} \|f(e)\|_\infty,
\end{align*}
the second inequality applying
$\|\tilde f(s')\|_F\|\tilde f(s'')\|_F=\|f(s)\|_F$ and
$\|\tilde f(v_t)\|_\infty
\leq \|f(t)\|_\infty \|f(v_t)\|_\infty$.\\

{\bf Case 3:} Suppose $t$ has exactly $k-m$ distinct neighbors,
and $k-m \geq 2$. Let $\partial s$ and $\partial t$ denote the edges incident
to $s$ and $t$, respectively, where $|\partial s|=m$ and $|\partial t|=k-m$.
Let $\cE_{\cK \cup \cV}$ denote all remaining edges, which belong to the
connected subgraph induced by $\cK \cup \cV$.
For each edge $e \in \partial s \cup \partial t$, denote by $v_e \in \cV$
the vertex other than $\{s,t\}$ that is incident to $e$. As $f(s)$ is rank-1,
it admits a factorization $f(s)=\otimes_{e \in \partial s} \tilde \s_e$
for some $m$ vectors $(\tilde \s_e:e \in \partial s)$ in $\R^n$. Then, since
$f(e)=I$ for each $e \in \partial s \cup \partial t$ and $f(v)$ is diagonal
for each vertex $v \in \cK \cup \cV$, we may express $\val(G,f)$ as
\begin{align*}
\val(G,f)=\sum_{i_v \in [n]:v \in \cK \cup \cV}
f(t)[i_{v_e}:e \in \partial t]
\prod_{e \in \partial s} \tilde \s_e[i_{v_e}]
\prod_{v \in \cK \cup \cV} f(v)[i_v,\ldots,i_v]
\prod_{e=(u,v) \in \cE_{\cK \cup \cV}} f(e)[i_u,i_v]
\end{align*}
where the summation is over one distinct index $i_v \in [n]$ for each vertex of
$\cK \cup \cV$.

Let $\cV_T=\{v_e:e \in \partial t\} \subseteq \cV$ 
and $\cV_S=\{v_e:e \in \partial s\} \subseteq \cV$ denote the sets of vertices
adjacent to $t$ and $s$, where $\cV=\cV_T \cup \cV_S$ and $|\cV_T|=k-m \geq 2$. 
(We recall that $\cV_S \cap \cV_T$ may be non-empty.)
We will next partition the edges of $\cE_{\cK \cup \cV} \cup \partial s$ ---
including those incident to $s$ but not those incident to $t$ --- into two
non-empty
disjoint sets $\cE_1 \cup \cE_2$, so as to bound the above expression using
Cauchy-Schwarz. Given any such partition $\cE_1 \cup \cE_2$, let us
split the vertices of $\cK \cup \cV$ that are
incident to $\cE_1$ into two types: Let $\cU_1^* \subseteq \cK \cup \cV$ be
those vertices incident to some edge of $\cE_1$, but not belonging to $\cV_T$
or incident to any edge of $\cE_2$.
Let $\cU_1 \subseteq \cK \cup \cV$ be those remaining vertices which are
incident to an edge of $\cE_1$ and either belong to $\cV_T$ or are incident
to an edge of $\cE_2$. We define a tensor $U_1 \in (\R^n)^{\otimes |\cU_1|}$ of
order $|\cU_1|$, which sums over vertices of $\cU_1^*$ and is 
indexed by $(i_v:v \in \cU_1)$ corresponding to vertices of $\cU_1$. This tensor
$U_1$ has entries
\[U_1[i_v:v \in \cU_1]=\sum_{i_v \in [n]:v \in \cU_1^*}
\prod_{e \in \partial s \cap \cE_1}\tilde \s_e[i_{v_e}]
\prod_{v \in \cU_1^*}f(v)[i_v,\ldots,i_v]
\prod_{e \in (u,v) \in \cE_1 \cap \cE_{\cK \cup \cV}} f(e)[i_u,i_v].\]
We define analogously the sets $\cU_2^*,\cU_2 \subseteq \cK \cup \cV$ and tensor
$U_2$ for $\cE_2$. Note that $\cK \cup \cV$ is then the disjoint union of the
three sets $\cU_1^*$, $\cU_2^*$, and $\cU_1 \cup \cU_2$, where $\cV_T \subseteq
\cU_1 \cup \cU_2$ because each vertex of $\cV_T$ must be incident to at least
one edge of the connected subgraph $\cK \cup \cV$. Then
\[\val(G,f)=\sum_{i_v \in [n]:v \in \cU_1 \cup \cU_2}
f(t)[i_{v_e:e \in \partial t}]\,
\bigg(\prod_{v \in \cU_1 \cup \cU_2} f(v)[i_v,\ldots,i_v]\bigg)
U_1[i_v:v \in \cU_1]\,U_2[i_v:v \in \cU_2].\]
Applying Cauchy-Schwartz over the indices of $\cU_1 \cup \cU_2$,
\begin{align*}
|\val(G,f)|
&\leq \|f(t)\|_\infty
\bigg(\prod_{v \in \cU_1 \cup \cU_2} \|f(v)\|_\infty\bigg)\\
&\hspace{1in}\bigg(\sum_{i_v \in [n]:v \in \cU_1 \cup \cU_2}
U_1[i_v:v \in \cU_1]^2\bigg)^{1/2}
\bigg(\sum_{i_v \in [n]:v \in \cU_1 \cup \cU_2}
U_2[i_v:v \in \cU_2]^2\bigg)^{1/2}.
\end{align*}
Note that $(\cU_1 \cup \cU_2) \setminus \cU_1$ consists of the vertices of
$\cV_T$ that are not incident to any edge of $\cE_1$. Let $k_1$ be the
number of such vertices, and define similarly $k_2$ as the number of vertices of
$\cV_T$ that are not incident to any edge of $\cE_2$. Then the above yields the
bound
\begin{equation}\label{eq:Case3valGfbound}
|\val(G,f)| \leq \sqrt{n}^{k_1+k_2}\|f(t)\|_\infty
\bigg(\prod_{v \in \cU_1 \cup \cU_2} \|f(v)\|_\infty\bigg)
\|U_1\|_F\|U_2\|_F.
\end{equation}

To further bound $\|U_1\|_F$,
consider the connected components of the subgraph induced by $\cE_1$ (i.e.\ the
edges $\cE_1$, their incident vertices $\cU_1,\cU_1^*$, and the vertex $s$ if
$\cE_1 \cap \partial s$ is non-empty).
If $\cE_1 \cap \partial s$ is non-empty, and the component of $s$
breaks into $j$ pieces upon removing $s$, let us further split this
component of $s$ into $j$ separate components, each containing a copy of $s$.
Let $\cC$ be the set of all components. For each component $C \in \cC$, let
$\cU_C^* \subseteq \cU_1^*$, $\cU_C \subseteq \cU_1$,
$\cE_C \subseteq \cE_1 \cap \cE_{\cK \cup \cV}$,
and $\partial s_C \subseteq \partial s$ be the vertices and edges belonging to
$C$, where we set $\partial s_C=\emptyset$ if $C$ does not contain (a copy of)
$s$. Note that $\cU_C$ must be non-empty, because 
Lemma \ref{lem:contract network into degree 3} ensures that
the subgraph $(\cK \cup \cV,\cE_{\cK \cup \cV})$ is connected. 
Then $U_1$ factorizes as
\[U_1=\bigotimes_{C \in \cC} U_C\]
where each $U_C \in (\R^n)^{\otimes |\cU_C|}$ is the tensor indexed by vertices
of $\cU_C$, with entries
\[U_C[i_v:v \in \cU_C]=\sum_{i_v \in [n]:v \in \cU_C^*}
\prod_{e \in \partial s_C} \tilde \s_e[i_{v_e}]
\prod_{v \in \cU_C^*} f(v)[i_v,\ldots,i_v]
\prod_{e=(u,v) \in \cE_C} f(e)[i_u,i_v].\]
Thus,
\begin{equation}\label{eq:U1bound}
\|U_1\|_F=\prod_{C \in \cC} \|U_C\|_F
=\prod_{C \in \cC}\bigg(\sup_{T' \in (\R^n)^{\otimes |\cU_C|}:
\|T'\|_F=1} \<U_C,T'\>\bigg).
\end{equation}

For each component $C \in \cC$ where $C$ contains (a copy of) $s$,
let $\tilde G$ be the graph that adds a sink vertex $t'$ to $C$,
connected by an edge to each vertex of $\cU_C$.
We claim that $\tilde G$ admits a $(s,t')$-bipolar
orientation: Let $\cA$ be the vertices of $\tilde G \setminus \{s,t'\}$
that are connected to $\{s,t'\}$ by an edge, and let $\cB$ be all remaining
vertices of $\tilde G \setminus \{s,t'\}$. The subgraph induced by
$\cA \cup \cB$ in $\tilde G$ is connected,
by the construction of $C$ as a single component in $\cC$. 
By definition, any vertex $v \in \tilde G \cap \cU_C$ is adjacent to $t'$, hence
$v \in \cA$. Recalling the definition of $\cU_1$, note that any vertex
$v \in \tilde G \cap \cV_T$ must be in $\cU_C$, hence also $v \in \cA$.
Any vertex $v \in \tilde G \cap \cV_S$ must either neighbor $s$
in $\tilde G$ (if the edge $(s,v)$ belongs to $\cE_1$ and hence to $\tilde G$)
or neighbor $t'$ in $\tilde G$ (if $(s,v)$ belongs to $\cE_2$ and hence
$v \in \cU_C$); thus also $v \in \cA$. Then $\tilde G \cap (\cU_C \cup \cV) \subseteq\cA$, implying that $\cB \subseteq
\cU_C^* \cap \cK$. Any such vertex $u \in \cB$ has at least 3
incident edges, all of these edges belong to $\tilde G$ (since $u \in \cU_C^*$), and each such edge
must connect to a distinct neighbor in $\cV=\cV_S \cup \cV_T$ (since $u \in \cK$) which
must belong to $\cA$. Thus, $\tilde G$ satisfies the conditions of Lemma
\ref{lem:source-sink network}, implying that $\tilde G$ has a $(s,t')$-bipolar
orientation as claimed. We observe that
$\<U_C,T'\>=\val(\tilde G,\tilde f)$ for a labeling $\tilde f$ where
$\tilde f(s)=\otimes_{e \in \partial s_C} \tilde \s_e$ and
$\tilde f(t')=T'$. Then, applying Lemma \ref{lem:bounds by bipolar orientation}
and $\|\tilde f(s)\|_F=\prod_{e \in \partial s_C} \|\tilde \s_e\|_2$, we get
\begin{equation}\label{eq:CCbound1}
|\<U_C,T'\>|
=|\val(\tilde G,\tilde f)|
\leq \|T'\|_F\prod_{e \in \partial s_C}\|\tilde \s_e\|_2
\prod_{v \in \cU_C^*} \|f(v)\|_\infty \prod_{e \in \cE_C} \|f(e)\|_\op.
\end{equation}

For each $C \in \cC$ where $C$ does not contain $s$,
let $\tilde G$ be the graph that adds two vertices $s',t'$ to $C$,
where $s'$ connects by an edge to any single vertex of $\cU_C$,
and $t'$ connects by an edge to every vertex of $\cU_C$. The same arguments as
above show that $\tilde G$ satisfies the conditions of
Lemma \ref{lem:source-sink network} (with $\cA$ being the vertices adjacent to
$\{s',t'\}$ and $\cB$ being the rest), and hence admits a $(s',t')$-bipolar
orientation. We observe that
$\<U_C,T'\>=\val(\tilde G,\tilde f)$ for a labeling $\tilde f$ where
$\tilde f(s)=(1,1,\ldots,1) \in \R^n$ and
$\tilde f(t')=T'$. Applying
Lemma \ref{lem:bounds by bipolar orientation} and $\|\tilde f(s)\|_F=\sqrt{n}$,
we get
\begin{equation}\label{eq:CCbound2}
|\<U_C,T'\>|=|\val(\tilde G,\tilde f)|
\leq \sqrt{n}\|T'\|_F
\prod_{v \in \cU_C^*} \|f(v)\|_\infty \prod_{e \in \cE_C} \|f(e)\|_\op.
\end{equation}
Let $m_1,m_2$ be the numbers of connected components in the subgraphs induced by
$\cE_1,\cE_2$ that do not contain $s$.
Then, applying the bound \eqref{eq:CCbound1}
or \eqref{eq:CCbound2} to each component $C \in \cC$ of \eqref{eq:U1bound}, we
obtain
\[\|U_1\|_F \leq \sqrt{n}^{m_1} \prod_{e \in \partial s \cap \cE_1}
\|\tilde \s_e\|_2 \prod_{v \in \cU_1^*}\|f(v)\|_\infty
\prod_{e \in \cE_{\cK \cup \cV} \cap \cE_1} \|f(e)\|_\op,\]
and similarly for $\|U_2\|_F$ and $m_2$. Applying these bounds back to
\eqref{eq:Case3valGfbound} and recalling that $\cU_1^*$, $\cU_2^*$, and $\cU_1
\cup \cU_2$ form a partition of $\cK \cup \cV$,
\[|\val(G,f)| \leq \sqrt{n}^{k_1+k_2+m_1+m_2}
\|f(s)\|_F \|f(t)\|_\infty \prod_{v \in \cK \cup \cV}\|f(v)\|_\infty
\prod_{e \in \cE_{\cK \cup \cV}} \|f(e)\|_\op.\]

Finally, to show \eqref{eq:valGfmainbound}, let us argue that we may choose
the edge partition $\cE_1 \cup \cE_2$ of $\cE_{\cK \cup \cV} \cup \partial s$
to ensure that $k_1+k_2+m_1+m_2 \leq k-m-1$. Let $G'$ be the subgraph induced
by $\{s\} \cup \cK \cup \cV$, i.e.\ containing the vertices
$\{s\} \cup \cK \cup \cV$ and edges $\cE_{\cK \cup \cV} \cup \partial s$.
Since $t$ has exactly $k-m$
distinct neighbors in $\cV_T$ and $k-m \geq 2$, Lemma \ref{lem:good vertices}
implies that there exists a tree $H$ in $G'$
with at least two good leaves in $\cV_T$. Let $G''$ be the graph $G'$ removing
all edges of $H$; here $G''$ may be disconnected and/or have isolated vertices.
By definition, each good leaf $v \in \cV_T$
is connected by a path in $G''$ either to another vertex $u \in \cV_T$
or to $s$. We consider two cases:

(a) Suppose there is a good leaf $v \in \cV_T$ whose connected component in
$G''$ does not contain $s$. Let $\cE_1$ be all edges of this connected
component, and let $\cE_2$ be all remaining edges of $G'$.
Since $v$ is connected by a path in $G''$ to some other $u \in \cV_T$,
the (connected) subgraph induced by $\cE_1$ has at least two vertices of
$\cV_T$, so $k_1 \leq k-m-2$ and $m_1=1$. The subgraph induced by $\cE_2$ is
also connected, because all connected components of $G''$ must be
connected via the removed tree $H$. Furthermore, this subgraph contains both $s$
and all vertices of $\cV_T$, since $\cE_2$ contains all edges of $H$.
Thus $k_2=0$ and $m_2=0$, so $k_1+k_2+m_1+m_2 \leq k-m-1$.

(b) Suppose all good leaves $v \in \cV_T$ belong to the connected component of
$s$ in $G''$. Again let $\cE_1$ be all edges of this connected component,
and let $\cE_2$ be all remaining edges of $G'$.
The subgraph induced by $\cE_1$ has at least two good leaves in $\cV_T$ and also
contains $s$, so $k_1 \leq k-m-2$ and $m_1=0$. The subgraph induced by $\cE_2$
again consists of a single connected component which contains all vertices of
$\cV_T$, so $k_2=0$ and $m_2=1$. Thus, we also have
$k_1+k_2+m_1+m_2 \leq k-m-1$.\\

This shows that \eqref{eq:valGfmainbound} holds in all cases. 
The desired bound \eqref{eq:RFmonomialsbound} then follows from
\eqref{eq:valGfmainbound}, \eqref{eq:RFnetworkmainbound}, and 
$|\cG|=|\cG(l_1,\ldots,l_k)|$ in Lemma
\ref{lem:contract network into degree 3}, completing the proof of
Lemma \ref{lem:RFmonomials}.
\end{proof}

\begin{proof}[Proof of Proposition \ref{prop:RF}, Assumption \ref{assum:cumulant}]
Denote $\a_l=(a_{il})_{i=1}^n$.
We may expand by multilinearity of the mixed cumulant
\begin{align*}
&\<\kappa_k(\g),\,\s_1 \otimes \ldots \otimes \s_m \otimes T\>\\
&=\sum_{l_1,\ldots,l_k=1}^\infty
\Big\<\kappa_k(\a_{l_1} \odot (X\w)^{\odot l_1},\ldots,\a_{l_k} \odot
(X\w)^{\odot l_k}),\,\s_1 \otimes \ldots \otimes \s_m \otimes T\Big\>\\
&=\sum_{l_1,\ldots,l_k=1}^\infty
\Big\<(\a_{l_1} \otimes \ldots \otimes \a_{l_k}) \odot
\kappa_k((X\w)^{\odot l_1},\ldots,(X\w)^{\odot l_k}),
\s_1 \otimes \ldots \otimes \s_m \otimes T\Big\>\\
&=\sum_{l_1,\ldots,l_k=1}^\infty
\Big\<\kappa_k((X\w)^{\odot l_1},\ldots,(X\w)^{\odot
l_k}),(\a_{l_1} \odot \s_1) \otimes \ldots \otimes (\a_{l_m} \odot \s_m)
\otimes ((\a_{l_{m+1}} \otimes \ldots \otimes \a_{l_k}) \odot T)\Big\>.
\end{align*}
Then Lemma \ref{lem:RFmonomials} implies
\begin{align*}
&\<\kappa_k(\g),\,\s_1 \otimes \ldots \otimes \s_m \otimes T\>\\
&\leq \sum_{l_1,\ldots,l_k=1}^\infty
C_l(\w)\|X\|_\op^l|\cG(l_1,\ldots,l_k)|
\sqrt{n}^{k-m-1}
\|(\a_{l_{m+1}} \otimes \ldots \otimes \a_{l_k}) \odot T\|_{\cU(l)}
\prod_{i=1}^m \|\a_{l_i} \odot \s_i\|_2.
\end{align*}
In the definitions preceding \eqref{eq:Uldef}, we have
$\|\bar X\|_\op \leq 1$, $\|K_k\|_\op \leq 1$, and
$\|D_k\|_\op \leq \|K_k\|_\op \cdot \max_{i=1}^n \sum_{j=1}^n |\bar X[i,j]|^k
\leq 1$ for all $k \geq 2$. Then each $M \in \cM_0$ has $\|M\|_\op \leq 1$,
implying that also each $M \in \cM=\bigcup_{j \geq 0} \cM_j$ has $\|M\|_\op \leq
1$. Thus, $\sup_{\x \in \cU(l)} \|\x\|_2 \leq 1$. Furthermore $I \in \cM(l)$,
so $\e_1,\ldots,\e_n \in \cU(l)$. Let us define
$\e=(1,\ldots,1) \in \R^n$, and set
\[\cU'(l)=\{(\a/\|\a\|_\infty) \odot \x:\a \in \{\e,\a_1,\ldots,\a_l\},\,\x \in
\cU(l)\}.\]
Then also $\e_1,\ldots,\e_n \in \cU'(l)$,
$\sup_{\x \in \cU'(l)} \|\x\|_2 \leq 1$, and we have
$\|(\a_{l_{m+1}} \otimes \ldots \otimes \a_{l_k}) \odot
T\|_{\cU(l)} \leq  \|T\|_{\cU'(l)}\prod_{i=k-m}^k \|\a_{l_i}\|_\infty$.
Applying this and
$\|\a_{l_i} \odot \s_i\|_2 \leq \|\a_{l_i}\|_\infty\|\s_i\|_2$,
\begin{align*}
&\<\kappa_k(\g),\,\s_1 \otimes \ldots \otimes \s_m \otimes T\>\\
&\leq \sum_{l_1,\ldots,l_k=1}^\infty
C_l(\w)\|X\|_\op^l|\cG(l_1,\ldots,l_k)|
\bigg(\prod_{i=1}^k \|\a_{l_i}\|_\infty\bigg)
\sqrt{n}^{k-m-1} \|T\|_{\cU'(l)} \prod_{i=1}^m \|\s_i\|_2.
\end{align*}

Under Proposition \ref{prop:RF}(a),
noting that $C_l(\w)$, $\|X\|_\op^l$, $|\cG(l_1,\ldots,l_k)|$, and
$\prod_{i=1}^k \|\a_{l_i}\|_\infty$ are all bounded by a constant, this shows
for a constant $C \equiv C(k,D)>0$ that
\begin{align*}
|\<\kappa_k(\g),\s_1 \otimes \ldots \otimes \s_m \otimes T\>|
&\leq C\sqrt{n}^{k-m-1}\|T\|_{\cU'(kD)} \prod_{i=1}^m \|\s_i\|_2.
\end{align*}
Since also $|\cU'(kD)| \leq C(k,D)n$ for a constant $C(k,D)>0$,
this shows Assumption \ref{assum:cumulant}.

Under Proposition \ref{prop:RF}(b), recalling the bound
\eqref{eq:GaussianRFbound}, we have
\begin{align*}
|\<\kappa_k(\g),\,\s_1 \otimes \ldots \otimes \s_m \otimes T\>|
\leq \sum_{l=k}^\infty
\underbrace{(Ck^2)^l l^{(\frac{1}{2}-\beta)l} \sqrt{n}^{k-m-1}\|T\|_{\cU'(l)}
\prod_{i=1}^m \|\s_i\|_2}_{:=F(l)}.
\end{align*}
Note that in the definitions preceding \eqref{eq:Uldef}, 
the number of distinct non-zero matrices in $\cM(l)$ is at most the number of
strings in the symbols $I,\bar X,\bar X^*,D_2,K_2,\odot,\times,(,)$ of
length $Cl$, for some absolute constant $C>0$. Then for some constants
$C',C''>0$, we have $|\cM(l)| \leq {C'}^l$ and $|\cU'(l)| \leq {C''}^l$.
For a sufficiently large constant $C_0>0$, let us set
\[L=\lfloor C_0k\log n+k^{C_0} \rfloor, \qquad \cU \equiv \cU'(L).\]
Then $|\cU| \leq n^C$ for a constant $C \equiv C(C_0,k)>0$.
For all $l>L$, recalling that $\sup_{\x \in \cU'(l)} \|\x\|_2 \leq 1$,
we may bound $\|T\|_{\cU'(l)} \leq \|T\|_F \leq \sqrt{n}^{k-m}\|T\|_\infty
\leq \sqrt{n}^k\|T\|_\cU$. This gives
\begin{align*}
\sum_{l>L} F(l)
\leq \sqrt{n}^k \cdot \sqrt{n}^{k-m-1}\|T\|_{\cU}
\prod_{i=1}^m \|\s_i\|_2 \times
\sum_{l>L} (Ck^2)^l l^{(\frac{1}{2}-\beta)l}.
\end{align*}
Under the given condition $\beta>1/2$, choosing a large enough constant
$C_0>0$ that defines $L$ ensures
$\sum_{l>L} (Ck^2)^l l^{(\frac{1}{2}-\beta)l}<\sqrt{n}^{-k}$, and thus
\[\sum_{l>L} F(l)<\sqrt{n}^{k-m-1}\|T\|_{\cU}
\prod_{i=1}^m \|\s_i\|_2.\]
For $l \leq L$, since $\cU'(l) \subseteq \cU$, we may bound
$\|T\|_{\cU'(l)} \leq \|T\|_\cU$. Then for a constant $C(k,\beta)>0$,
\[\sum_{l=k}^L F(l)
\leq \sqrt{n}^{k-m-1}\|T\|_{\cU} \prod_{i=1}^m \|\s_i\|_2
\sum_{l=1}^\infty (Ck^2)^l l^{(\frac{1}{2}-\beta)l}
\leq C(k,\beta)\sqrt{n}^{k-m-1}\|T\|_{\cU} \prod_{i=1}^m \|\s_i\|_2.\]
This implies
$|\<\kappa_k(\g),\s_1 \otimes \ldots \otimes \s_m \otimes T\>|
\leq (C(k,\beta)+1)\sqrt{n}^{k-m-1}\|T\|_{\cU} \prod_{i=1}^m \|\s_i\|_2$,
which again shows Assumption \ref{assum:cumulant}.
\end{proof}

\subsection{Random features tilt}\label{subsec:random features tilt}

To prove Proposition \ref{prop:spiked cumulant}, we use techniques developed in constructive field
theory. In particular, we follow a similar approach as outlined in \cite[Section
A]{gurau2014tensor}. The main technical tool we rely on is the universal
Brydges-Kennedy-Abdesselam-Rivasseau forest formula \cite[Theorem
III.1]{10.1007/3-540-59190-7_20}, which we state here for the reader's
convenience:

\begin{lemma}[BKAR forest formula]\label{lem:BKAR forest formula}
For some $l\geq 1$, let $f:[0,1]^{l(l-1)/2} \rightarrow\R$
be a smooth function. Then
\begin{align*}
    f(\underbrace{1,\ldots,1}_{l(l-1)/2})&=\sum_{G
\in\cF_l}\int_{[0,1]^{l(l-1)/2}}\left.\frac{\partial^{|\cE(G)|}}{\prod_{(i,j)\in
\cE(G)}\partial
y_{ij}}f(y)\right|_{y=y(G,u)}\prod_{i<j}du_{ij}
\end{align*}
where $\cF_l$ is the set of all forests (i.e.\ simple graphs $G$ with no loops
and cycles) on $l$ labeled vertices, $\cE(G)$ is the edge set of $G$,
$u=(u_{ij})_{i<j}$, and $y(G,u)=(y_{pq}(G,u))_{p<q} \in \R^{l(l-1)/2}$
has the coordinates
\[y_{pq}(G,u)=\min_{(i,j)\in 
\cE(G) \cap \{\text{unique path from } p \text{ to } q\}} u_{ij}\]
if $p,q$ belong to the same connected component of $\cF_l$,
or $y_{pq}(G,u)=0$ otherwise.

Moreover, the symmetric matrix $Y(G,u)$ with entries
\[
    Y(G,u)[p,q]=\begin{cases}
        1 & \text{if $p=q$}\\
        y_{pq}(G,u) & \text{if $p<q$}\\
        y_{qp}(G,u) & \text{if $p>q$}\\
    \end{cases}
\]
is positive-semidefinite for all $0\leq u_{ij}\le 1$.
\end{lemma}

We also need the following simple identity.
\begin{lemma}\label{lem:Gaussian differentiation}
Fix any $l\geq 1$, and let $Y=(y_{ij})_{1\leq i,j\leq l}\in\R^{l\times l}$
be any positive-semidefinite matrix. For a smooth function $\Phi:\R^{nl}
\rightarrow \R$ defined by $(\w_1,\ldots,\w_l) \in \R^{nl}
\mapsto \Phi(\w_1,\ldots,\w_l)$, denote
\[\<\nabla_{\w_i},\nabla_{\w_j}\>\Phi
=\sum_{\alpha=1}^n \frac{\partial^2}{\partial \w_i[\alpha]\partial \w_j[\alpha]}
\Phi.\]
If $(\w_1,\ldots,\w_l) \sim N(0,Y \otimes I_n)$, then for any $i<j$,
\[\frac{\partial}{\partial
y_{ij}}\E[\Phi(\w_1,\ldots,\w_l)]=\E[\<\nabla_{\w_i},\nabla_{\w_j}\>\Phi(\w_1,\ldots,\w_l)].\]
\end{lemma}
\begin{proof}
Since the covariance $Y \otimes I_n$ can be singular, we shall prove this via
the Fourier transform. Let $\w=(\w_1,\ldots,\w_l) \sim N(0,Y \otimes I_n)$.
Its characteristic function, for any $\t=(\t_1,\ldots,\t_l)\in\R^{nl}$,
is given by
\[\mu_\w(\t)=\E\exp(i\<\t,\w\>)=\exp\left({-}\frac{1}{2}\t^*(Y\otimes
I_n)\t\right)=\exp\left({-}\frac{1}{2}\sum_{i=1}^l
y_{ii}\|\t_i\|_2^2-\sum_{i<j} y_{ij}\t_i^*\t_j\right).\]
Thus,
\[\frac{\partial}{\partial y_{ij}}\mu_\w(\t)=-(\t_i^*\t_j)\mu_\w(\t).\]
Denote by
\[\cF[\Phi](\t)=\int_{-\infty}^\infty \Phi(\w)e^{{-}i\<\t,\w\>}d\w,
\quad \cF[\<\nabla_{\w_i},\nabla_{\w_j}\>\Phi](t)
=\int_{-\infty}^\infty \<\nabla_{\w_i},\nabla_{\w_j}\>
\Phi(\w)e^{{-}i\<\t,\w\>}d\w\]
the Fourier transforms of $\Phi$ and $\<\nabla_{\w_i},\nabla_{\w_j}\>\Phi$.
Then by standard properties of the Fourier transform, for any $i<j$,
\[
    \cF[\<\nabla_{\w_i},\nabla_{\w_j}\>\Phi](\t)=\sum_{\alpha=1}^n\cF\left[\frac{\partial^2}{\partial
\w_i[\alpha]\partial \w_j[\alpha]}\Phi\right](\t)=\sum_{\alpha=1}^n
(-it_i[\alpha])(-it_j[\alpha])\cF[\Phi](\t)=-\t_i^*\t_j \cF[\Phi](\t).
\]
So by the Fourier inversion formula,
\begin{align*}
\frac{\partial}{\partial y_{ij}}\E[\Phi(\w)]
&=\frac{\partial}{\partial y_{ij}}
\E\left[\frac{1}{(2\pi)^{nl}}\int \cF[\Phi](\t)e^{i\<\t,\w\>} d\t\right]\\
&=\frac{1}{(2\pi)^{nl}}\int\cF[\Phi](\t)\frac{\partial}{\partial y_{ij}}\mu_\w(\t)\;d\t
    =\frac{1}{(2\pi)^{nl}}\int(-\t_i^*\t_j)\cF[\Phi](\t)\mu_\w(\t)\;d\t\\
    &=\frac{1}{(2\pi)^{nl}}\int\cF[\<\nabla_{\w_i},\nabla_{\w_j}\>\Phi](\t)\mu_\w(\t)\;d\t=\E[\<\nabla_{\w_i},\nabla_{\w_j}\>\Phi(\w)].
\end{align*}
\end{proof}

For any $l\geq 1$, any tree $G$ on $l$ vertices, and any $u
\in[0,1]^{l(l-1)/2}$, let $Y(G,u)\in\R^{l\times l}$ be defined as in Lemma
\ref{lem:BKAR forest formula}. We write $\E_{Y(G,u)}$ for the expectation with
respect to $(\w_1,\ldots,\w_l) \sim N(0,Y(G,u) \otimes I_n)$. We will assume without loss of generality in Proposition \ref{prop:spiked cumulant} that $\sigma_i$ is centered such that
\begin{equation}\label{eq:meanzerosigma}
\E_{\w \sim N(0,I_n)} \sigma_i(\x_i^*\w)=0 \text{ for all } i=1,\ldots,d
\end{equation}
Using the above two lemmas, we now derive the following series expansion of the cumulant generating function. 

\begin{lemma}\label{lem:formal expansion of CGF}
In the setting of Proposition \ref{prop:spiked cumulant}, for any $\w,\t \in
\R^n$, denote
\[\Phi(\w,\t,\lambda)=\<\t,\w\>-\lambda\sum_{i=1}^d\sigma_i(\x_i^*\w).\]
Let $\cT_l$ be the set of all trees on $l$ labeled vertices. For each $G \in
\cT_l$ and $\t \in \R^n$, let
\[F(G,\lambda,\t)=\int_{[0,1]^{l(l-1)/2}}\E_{Y(G,u)}\left[\prod_{(i,j)\in
\cE(G)}\<\nabla_{\w_i},\nabla_{\w_j}\>\prod_{i=1}^l\Phi(\w_i,\t,\lambda)\right]\prod_{i<j}du_{ij}.\]
Then there exists a constant $c_*>0$ such that if $\max(|\lambda|,\|\t\|_2)<c_*$ and
$\sum_{l=1}^\infty \sum_{G \in \cT_l} \frac{1}{l!} |F(G,\lambda,\t)|<\infty$,
then
\begin{align*}
\log\E_{\w \sim N(0,I_n)} \exp \Phi(\w,\t,\lambda)
=\sum_{l=1}^\infty\sum_{G\in\cT_l} \frac{1}{l!}\,F(G,\lambda,\t).
\end{align*}
\end{lemma}
\begin{proof}
To begin, we expand $\E \exp \Phi(\w,\t,\lambda)$ as
\begin{equation}\label{eq:expand exponential}
    \E\exp \Phi(\w,\t,\lambda)=\sum_{l=0}^\infty\frac{1}{l!}\E
\Phi(\w,\t,\lambda)^l.
\end{equation}
This exchange of summation and expectation
is justified by the Fubini-Tonelli theorem: Since $\w\sim N(0,I_n)$, we have by \cite[eq.\ (2.4)]{PS_2017__21__467_0} that 
under the conditions of Proposition \ref{prop:spiked cumulant} and the assumption \eqref{eq:meanzerosigma}, for
some constants $C,C'>0$,
\begin{align*}
    (\E|\Phi(\w,\t,\lambda)|^l)^{1/l}&=(\E|\t^*\w-\lambda\1_d^*\sigma(X\w)|^l)^{1/l}\leq(\E|\t^*\w|^l)^{1/l}+|\lambda|(\E|\1_d^*\sigma(X\w)|^l)^{1/l}\\
    &\leq C\sqrt{l}(\|\t\|_2+|\lambda|(\E\|X^*\sigma'(X\w)\|_2^l)^{1/l})\\
    &\leq C'\sqrt{l}(\|\t\|_2+|\lambda|)
\end{align*}
For $\max(\|\t\|_2,|\lambda|)<c_*$ sufficiently small,
this implies
$\sum_{l=0}^\infty\frac{1}{l!}\E|\Phi(\w,\t,\lambda)|^l<\infty$,
which justifies \eqref{eq:expand exponential}.

We may then write
\[
    \E \Phi(\w,\t,\lambda)^l=\E\prod_{i=1}^l\Phi(\w_i,\t,\lambda)
\]
where $\w_1=\ldots=\w_l=\w$, i.e.\ $(\w_1,\ldots,\w_l) \sim
N(0,1_l1_l^*\otimes I_n)$. Let $S_l \subset [0,1]^{l \times l}$ be the space of
positive-semidefinite matrices such that $S_l[i,i]=1$ for all $i \in [l]$
and $S_l[i,j] \in [0,1]$ for all $i \neq j$, and define the function
$f:S_l \rightarrow\R$ by
\[f(Y)=\E_Y \prod_{i=1}^l\Phi(\w_i,\t,\lambda).\]
Identifying $S_l$ with a subset of $[0,1]^{l(l-1)/2}$ given by its
upper-triangular entries, we have by Lemmas \ref{lem:BKAR forest formula}
and \ref{lem:Gaussian differentiation}
\begin{align*}
    &\E \Phi(\w,\t,\lambda)^l=f(1_l1_l^*)\\
    &=\sum_{G
\in\cF_l}\int_{[0,1]^{l(l-1)/2}}\left.\frac{\partial^{|\cE(G)|}}{\prod_{(i,j)\in
\cE(G)}\partial
y_{ij}}f(Y)\right|_{Y=Y(G,u)}\prod_{i<j}du_{ij}\\
    &=\sum_{G\in\cF_l}\int_{[0,1]^{l(l-1)/2}}\E_{Y(G,u)}\left[\prod_{(i,j)\in
\cE(G)}\<\nabla_{\w_i},\nabla_{\w_j}\>\prod_{i=1}^l\Phi(\w_i,\t,\lambda)\right]\prod_{i<j}du_{ij}.
\end{align*}
Note that this is well-defined since the matrix $Y(G,u)$ is
positive-semidefinite. If $i,j \in [l]$ belong to two disconnected components of
$G$, then by construction we have $Y(G,u)[i,j]=0$, so $\w_i,\w_j$ are
independent. Thus, $\E_{Y(G,u)}$ factorizes over the the connected components of
$G$. By a standard combinatorial identity for exponential
generating functions (see e.g.\ \cite[eq.(109)]{gurau2014tensor} and
\cite[pg.3]{Rivasseau2008constructive}), if a function is a sum over forests of
contributions which factorize over its connected trees, then its logarithm is
the sum over trees of these individual factors. That is, as an identity of
formal series,
\begin{align}
    &\log \E\exp \Phi(\w,\t,\lambda)
=\log \sum_{l=0}^\infty \frac{1}{l!}\E \Phi(\w,\t,\lambda)^l\nonumber\\
    &=\sum_{l=1}^\infty\frac{1}{l!}\sum_{G \in\cT_l}
\underbrace{\int_{[0,1]^{l(l-1)/2}}\E_{Y(G,u)}\left[\prod_{(i,j)\in
\cE(G)}\<\nabla_{\w_i},\nabla_{\w_j}\>\prod_{i=1}^l\Phi(\w_i,\t,\lambda)\right]\prod_{i<j}du_{ij}}_{=F(G,\lambda,\t)}
\end{align}
where $\cT_l$ is the set of all trees on $l$ labeled vertices. This holds as an
identity of analytic functions as long as the right side is absolutely
convergent, i.e.\
$\sum_{l=1}^\infty\sum_{G\in\cT_l}\frac{1}{l!}\,|F(G,\lambda,\t)|<\infty$.
\end{proof}

We now show that the series expansion established in Lemma \ref{lem:formal
expansion of CGF} is indeed absolutely convergent, and represent its terms using
the tensor network definitions/notions developed in Section \ref{subsec:RF}.
Here, all dimensions of tensors and matrices in the network are $d$ rather than
$n$.

\begin{lemma}\label{lem:CGF converges absolutely}
In the setting of Proposition \ref{prop:spiked cumulant}, 
let $\sigma^{(k)}(X\w)=(\sigma_i^{(k)}(\x_i^*\w))_{i=1}^d$ for all $k \geq 1$.
Then there exists a constant $c_*>0$ such that if $\max(|\lambda|,\|\t\|_2)<c_*$, then
\begin{equation}\label{eq:CGFseries}
    \log\E\exp
\Phi(\w,\t,\lambda)=\sum_{l=1}^\infty\frac{1}{l!}\sum_{G\in\cT_l}\sum_{f\in\cF(G)}\int_{[0,1]^{l(l-1)/2}}\E_{Y(G,u)}[\val(G,f)]\prod_{i<j}du_{ij}.
\end{equation}
Here $\cF(G)$ is a set of labelings of cardinality $|\cF(G)| \leq 2^l$.
Denoting the vertices of $G$ by $v_1,\ldots,v_l$, for each $f \in \cF(G)$,
$(G,f)$ is a tensor network satisfying:
\begin{enumerate}
    \item If $v_i\in G$ is such that $\deg(v_i)=1$ (it is a leaf), then $f(v_i)
\in\{X\t,-\lambda\sigma^{(1)}(X\w_i)\} \subset \R^d$.
    \item If $v_i\in G$ is such that $\deg(v_i)\ge 2$, then
$f(v_i)=\diag({-}\lambda\sigma^{(\deg(v_i))}(X\w_i))\in(\R^d)^{\otimes \deg(v_i)}$.
    \item For all edges $e=(v_i,v_j)\in \cE(G)$, if $f(v_i)=X\t$ or
$f(v_j)=X\t$, then $f(e)=I_d$, and otherwise $f(e)=XX^* \in \R^{d \times d}$.
\end{enumerate}
Moreover, for any $k\geq 1$ and $\alpha_1,\ldots,\alpha_k\in[n]$,
the derivative
$\frac{\partial}{\partial \t[\alpha_1]}\cdots\frac{\partial}{\partial
\t[\alpha_k]}\log\E\exp \Phi(\w,\t,\lambda)$ is given by differentiating the
series \eqref{eq:CGFseries} term-by-term.
\end{lemma}
\begin{proof}
In the series expansion developed in Lemma \ref{lem:formal expansion of CGF},
we first show that for any $l\geq 1$ and $G \in\cT_l$,
\begin{equation}\label{eq:network representation of CGF}
    \prod_{(i,j)\in
\cE(G)}\<\nabla_{\w_i},\nabla_{\w_j}\>\prod_{i=1}^l\Phi(\w_i,\t,\lambda)=\sum_{f\in\cF(G)}\val(G,f)
\end{equation}
for a family of tensor networks $\{(G,f)\}_{f \in \cF(G)}$ satisfying the given
conditions.
We will then use this tensor network representation to check the condition of
absolute convergence in Lemma \ref{lem:formal expansion of CGF}.

Expanding $\Phi(\w_i,\t,\lambda)=\t^*\w_i-\lambda\1_d^*\sigma(X\w)$, we have
\begin{equation}\label{eq:differential operator}
    \prod_{(i,j)\in
\cE(G)}\<\nabla_{\w_i},\nabla_{\w_j}\>\prod_{i=1}^l\Phi(\w_i,\t,\lambda)=\sum_{S\subseteq[l]}\prod_{(i,j)\in
\cE(G)}\<\nabla_{\w_i},\nabla_{\w_j}\>\prod_{i\in S}(\t^*\w_i)\prod_{j\in[l]\setminus S}(-\lambda\1_d^*\sigma(X\w_j))
\end{equation}
If $S \subseteq [l]$ contains a vertex $i$ that is not a leaf of the tree $G$,
then the term corresponding to $S$ in \eqref{eq:differential operator} is 0
because the differential $\nabla_{\w_i}$ is applied at least twice. For all
remaining subsets $S$, we may repeatedly apply the identities,
for any $k,m \geq 0$ and $\u,\v \in \R^d$,
\begin{align*}
&\<\nabla_{\w_i},\nabla_{\w_j}\>
\left(\t^*\w_i \cdot \u^*[{-}\lambda\sigma^{(k)}(X\w_j)]\right)
=(X\t)^*\diag({-}\lambda\sigma^{(k+1)}(X\w_j))\u,\\
&\<\nabla_{\w_i},\nabla_{\w_j}\>
\left(\u^*[{-}\lambda\sigma^{(k)}(X\w_i)] \cdot
\v^*[{-}\lambda\sigma^{(m)}(X\w_j)]\right)\\
&\hspace{1in}=\u^* \diag({-}\lambda\sigma^{(k+1)}(X\w_i))
XX^*\diag({-}\lambda\sigma^{(m+1)}(X\w_j))\v,
\end{align*}
to evaluate $\<\nabla_{\w_i},\nabla_{\w_j}\>$ for each edge of $G$.
Then writing $\diag({-}\lambda\sigma^{(1)}(X\w_j))\1_d
={-}\lambda \sigma^{(1)}(X\w_j) \in \R^d$ for the leaf vertices of $G$ belonging
to $[l] \setminus S$, we see that the summand
for $S$ in \eqref{eq:differential operator} takes the form $\val(G,f)$
for some labeling $f$ satisfying the given properties. Letting $\cF(G)$ denote the set of
labelings corresponding to all such non-zero summands, we then have
$|\cF(G)| \leq 2^l$. This shows (\ref{eq:network representation of CGF}).

We now check absolute convergence by bounding $\val(G,f)$:
We may identify one leaf vertex of $G$ as a source $s$ and merge
all remaining leaf vertices $u_1,\ldots,u_p \in G$ into a sink $t$ with label
$f(t)=f(u_1) \otimes \ldots \otimes f(u_p)$. This does not change $\val(G,f)$,
and the resulting network clearly admits a $(s,t)$-bipolar orientation since $G$
is a tree. Then, further bounding the operator norm of a diagonal matrix by the 
Euclidean norm of its vector of diagonal entries,
\[
    |\val(G,f)|\leq\prod_{i=1}^l\|f(v_i)\|_2\times
\max\{\|X^*X\|_\op,1\}^{|\cE(G)|}
\]
Here, by the conditions of Proposition \ref{prop:spiked cumulant},
$\|f(v_i)\|_2 \leq C\max(\|\t\|_2,|\lambda|C^{\deg(v_i)})$ for a constant $C>0$.
Since $G$ is a tree on $l$ vertices, it has $l-1$ edges, and $\sum_{i=1}^l
\deg(v_i)=2|\cE(G)|=2(l-1)$. Then this implies
$|\val(G,f)| \leq {C'}^l\max(\|\t\|_2,\lambda)^l$ (with probability 1 over
$(\w_1,\ldots,\w_l)$), so also
\[|F(G,\lambda,\t)|
\leq \int_{[0,1]^{l(l-1)/2}}
\E_{Y(G,u)} \left[\sum_{f \in \cF(G)} |\val(G,f)|\right]
\prod_{i<j} du_{ij} \leq (C'')^l \max(\|\t\|_2,\lambda)^l\]
for a constant $C''>0$. Using Cayley's formula $|\cT_l|=l^{l-2}$
and $l! \geq (l/e)^l$, we obtain for a constant $C>0$ that
\[\sum_{l=1}^\infty \sum_{G \in \cT_l} \frac{1}{l!} |F(G,\lambda,\t)|
\leq \sum_{l=1}^\infty \frac{1}{l^2}\,C^l\max(\|\t\|_2,\lambda)^l.\]
This is finite for $\max(\|\t\|_2,|\lambda|)<c_*$ sufficiently small,
so Lemma \ref{lem:formal expansion of CGF} applies to show \eqref{eq:CGFseries}.

All preceding arguments apply equally with $\z \in \C^n$ in place of
$\t \in \R^n$, so for each fixed $\lambda$ with $|\lambda|<c_*$, the Weierstrass M-test implies $\sum_{l=1}^\infty \sum_{G \in \cT_l} \frac{1}{l!} F(G,\lambda,\z)$ is
uniformly convergent to an analytic limit $F(\z)$ over the disk
$\{\z \in \C^n:\|\z\|_2<c_*\}$. This implies the uniform convergence of the
partial derivatives in $\z$, hence also
\[\frac{\partial}{\partial \t[\alpha_1]}\cdots\frac{\partial}{\partial
\t[\alpha_k]}\log\E\exp \Phi(\w,\t,\lambda)=\sum_{l=1}^\infty \sum_{G \in \cT_l}
\frac{1}{l!}\frac{\partial}{\partial \t[\alpha_1]}\cdots\frac{\partial}{\partial
\t[\alpha_k]}F(G,\lambda,\t).\]
\end{proof}

We now show that the quantity to be bounded in Assumption \ref{assum:cumulant}
admits a series expansion in terms of expected tensor network values.

\begin{lemma}\label{lem:series expansion of cumulant}
In the setting of Proposition \ref{prop:spiked cumulant},
there exist a constant $\lambda_*>0$ such that for all $|\lambda|<\lambda_*$ the
following holds: For any $k\geq 3$, $1\leq m\leq k-1$, any
$\s_1,\ldots,\s_m\in\R^n$, and any $T \in(\R^n)^{\otimes k-m}$, we have the convergent series expansion
\[
    \<\kappa_k(\g),\s_1\otimes\cdots\otimes\s_m\otimes
T\>=\sum_{l=k+1}^\infty\frac{1}{l!}\sum_{G\in\cT_l}\sum_{(H,f)\in\cG(G)}
\int_{[0,1]^{l(l-1)/2}}\E_{Y(G,u)}[\val(H,f)]\prod_{i<j}du_{ij}
\]
where $\cG(G)$ is a family of tensor networks satisfying $|\cG(G)| \leq 2^lk!$,
and for all $(H,f)\in\cG(G)$, the following holds:
\begin{enumerate}
    \item $H=(\{s_1,\ldots,s_m,t\}\cup\cA,\cE)$ is a connected, simple graph,
and the subgraph induced by the vertices $\cA$ is a tree on $l-k$ vertices.
    \item For all $1\leq i\leq m$, we have $\deg(s_i)=1$ and $f(s_i)=X\s_i$.
Similarly, $\deg(t)=k-m$ and $f(t)=X^{\otimes k-m}\cdot T$ (the contraction of
$T$ with $X$ along each axis). Moreover, the neighbors of $s_1,\ldots,s_m,t$ are
all in $\cA$, and the edges connecting them to $\cA$ all have label $f(e)=I_d$.
    \item Each vertex $v\in\cA$ has label
$f(v)=\diag({-}\lambda\sigma^{(\deg(v))}(X\w_{i_v}))\in(\R^d)^{\otimes \deg(v)}$
for some $i_v \in [l]$, and each edge $e$ between two vertices of $\cA$ has
label $f(e)=XX^*$.
\end{enumerate}
\end{lemma}
\begin{proof}
Note that $\kappa_k(\g)=\kappa_k(\w)$ since $\g=\w-\E\w$ and $k \geq 3$.
We begin with an entrywise understanding of $\kappa_k(\w)$. For any
$\alpha_1,\ldots,\alpha_k\in[n]$, Lemma \ref{lem:CGF converges absolutely}
implies we can differentiate term-by-term to get
\begin{align}
    &\kappa(\w[\alpha_1],...,\w[\alpha_k])=\left.\frac{\partial}{\partial \t[\alpha_1]}\cdots\frac{\partial}{\partial\t[\alpha_k]}\log\E_{\w
\sim N(0,I_n)}\exp \Phi(\w,\t,\lambda)\right|_{\t=0}\notag\\
    &=\sum_{l=1}^\infty\frac{1}{l!}\sum_{G\in\cT_l}\sum_{f\in\cF(G)}\int_{[0,1]^{l(l-1)/2}}\E_{Y(G,\u)}\left[\left.\frac{\partial}{\partial\t[\alpha_1]}\cdots\frac{\partial}{\partial\t[\alpha_k]}\val(G,f)\right|_{\t=0}\right]\prod_{i<j}du_{ij}\label{eq:cumulantseries}
\end{align}
The above summand corresponding to $G \in \cT_l$ and $f \in \cF(G)$ is
0 if the number of
leaves of $G$ having label $X\t$ is less than $k$. Since the minimum number of
vertices to form a tree with $k$ leaves is $k+1$, the above sum must start with
$l=k+1$. Moreover, since we evaluate the sum at $\t=0$, the non-zero summands
must correspond to $(G,f)$ where exactly $k$ leaves of $G$ are labeled as $X\t$.
For any such $(G,f)$, observe that
$\frac{\partial}{\partial\t[\alpha_1]}\cdots\frac{\partial}{\partial\t[\alpha_k]}
\val(G,f)$ is the sum over $k!$ tensor network values that replace the $k$
labels of $X\t$ in $(G,f)$ by the labels
$X\e_{\alpha_1},\ldots,X\e_{\alpha_k}$ in some order. Then
\[\sum_{f \in \cF(G)} \sum_{\alpha_1,\ldots,\alpha_k=1}^n
\frac{\partial}{\partial\t[\alpha_1]}\cdots\frac{\partial}{\partial\t[\alpha_k]}
\val(G,f)\s_1[\alpha_1]\cdots\s_m[\alpha_m]T[\alpha_{m+1},\ldots,\alpha_k]
=\sum_{(H,f) \in \cG(G)} \val(H,f)\]
where $\cG(G)$ is the set of at most $2^l k!$
tensor networks that take some labeling $f \in \cF(G)$, replace the labels of
some choice of $m$ vertices having label $X\t$ by
$X\s_1,\ldots,X\s_m$, and merge the remaining $k-m$ vertices
having label $X\t$ into a single vertex with (incident edges in some
order and) label $X^{\otimes k-m} \cdot T$. This network $(H,f)$ satisfies all
properties of the lemma, by the properties established
for $G$ in Lemma \ref{lem:CGF converges absolutely}. Since
\[\<\kappa_k(\w),\s_1\otimes\cdots\otimes\s_m\otimes
T\>=\sum_{\alpha_1,\ldots,\alpha_k=1}^n\kappa(\w[\alpha_1],\ldots,\w[\alpha_k])\s_1[\alpha_1]\cdots\s_m[\alpha_k]T[\alpha_{m+1},\ldots,\alpha_k],
\]
the lemma follows from applying this representation to
\eqref{eq:cumulantseries}.
\end{proof}

We are now ready to prove Proposition \ref{prop:spiked cumulant}.

\begin{proof}[Proof of Proposition \ref{prop:spiked cumulant}]
To check Assumption \ref{assum:basic} for $\Sigma$ and Assumption
\ref{assum:concentration} for $\g$,
note that the Hessian of the potential satisfies
\[
    \nabla^2U(\w)=\frac{1}{2}I_n+\lambda\sum_{i=1}^d\sigma_i''(\x_i^*\w)\x_i\x_i^*=\frac{1}{2}I_n+\lambda X^*\diag(\sigma''(X\w))X.
\]
Then for $|\lambda|<\lambda_*$ sufficiently small,
$(1/4)I_n\preceq \nabla^2 U(\w) \preceq (3/4)I_n$. Therefore the distribution of
$\g-\E \g$ is strongly log-concave, so Assumption \ref{assum:concentration}
follows from the Bakry-Emery criterion and a standard Lipschitz extension
argument. Furthermore by the Brascamp-Lieb inequality,
\[\Sigma=\E\g\g^*\preceq 4I_n.\]
To prove a lower bound on $\Sigma$, consider the random vector
$\z=(\w-\E\w,\nabla U(\w)) \in \R^{2n}$. We have
\[
    \E\z\z^*=\begin{bmatrix}
        \Sigma & \E(\w-\E\w)\nabla U(\w)^*\\
        \E\nabla U(\w)(\w-\E\w)^* & \E\nabla U(\w)\nabla U(\w)^*
    \end{bmatrix}\succeq 0.
\]
We deduce, from integration by parts
\[
    \E(\w-\E\w)\nabla U(\w)^*=\E\nabla U(\w)(\w-\E\w)^*=I_n,\quad \E\nabla U(\w)\nabla U(\w)^*=\E\nabla^2U(\w).
\]
Since both $\E\z\z^*$ and $\E\nabla U(\w)\nabla U(\w)^*=\E\nabla^2U(\w)$ are
positive-semidefinite, Schur's complement implies
$\Sigma_\w-I_n[\E\nabla^2 U(\w)]^{-1}I_n \succeq 0$, hence
\[\Sigma \succeq (4/3)I_n.\]
This implies the conditions for $\Sigma$ in Assumption \ref{assum:basic}.

Finally we check Assumption \ref{assum:cumulant}. It suffices to bound
$\val(H,f)$ in Lemma \ref{lem:series expansion of cumulant}, and we will do so deterministically over $(\w_1,\ldots,\w_l)$. For any $l\geq
k+1$, any $G \in\cT_l$, and any $(H,f)\in\cG(G)$, let $v_1,\ldots,v_{k-m}\in\cA$
be the vertices adjacent to $t$ in $H$.
We modify $(H,f)$ as follows: For each such
vertex $v_i$ adjacent to $t$, add a degree-1 vertex $v_i'$ that connects to
$v_i$ by an edge with label $I_d$, label $v_i'$ by the label 
$f(v_i)=\diag({-}\lambda\sigma^{(\deg(v_i))}(X\w_{j_i}))$ of $v_i$, and
re-label $f(v_i)$ by the identity tensor. This does not change $\val(H,f)$.
Now, letting $K\in(\R^d)^{\otimes p}$ be the tensor representing the contraction
of internal edges of
$H\setminus\{t,v_1',...,v_{k-m}'\}$, we apply Cauchy-Schwartz to bound
\begin{align*}
    |\val(H,f)|&=\left|\sum_{\alpha_1,\ldots,\alpha_{k-m}=1}^d
K[\alpha_1,\ldots,\alpha_{k-m}]f(v_1')[\alpha_1]\cdots
f(v_{k-m}')[\alpha_{k-m}]f(t)[\alpha_1,\ldots,\alpha_{k-m}]\right|\\
    &\leq \|K\|_F\|f(v_1')\|_2\cdots\|f(v_{k-m}')\|_2\|f(t)\|_\infty
\end{align*}
Here, since $f(t)=X^{\otimes k-m} \cdot T$, we have
\[\|f(t)\|_\infty \leq \|T\|_\cU\]
by definition of $\cU=\{\e_1,\ldots,\e_n\} \cup \{\x_1,\ldots,\x_d\}$.
To bound $\|K\|_F$, we have by duality
\[
    \|K\|_F=\sup_{\|U\|_F=1}\<U,K\>=\val(H',f)
\]
where $(H',f)$ is formed by removing $t,v_1',...,v_{k-m}'$ from $H$, adding a
new vertex $t'$ with label $f(t')=U$, and adding $k-m$ edges between $t'$ and
$v_1,\ldots,v_{k-m}$ with labels $I_d$. Note that these vertices
$v_1,\ldots,v_{k-m}$ are leaves of the tree $H' \setminus \{t'\}$.
Let $u_1,\ldots,u_p$ be the remaining
leaves of $H' \setminus \{t'\}$, which include the leaves $s_1,\ldots,s_m$ of $H$ together with any leaves in $\cA$. We set $s=u_1$ as a source, and
further merge the leaves $u_2,\ldots,u_p$ with $t'$ into a sink $t''$
having label
$f(t'')=U \otimes f(u_2)\otimes\cdots\otimes f(u_p)$. This does not change
$\val(H',f)$, and the resulting network admits a $(s,t'')$-bipolar orientation
because $H'\setminus \{t'\}$ is a tree. Then applying
Lemma \ref{lem:bounds by bipolar orientation} and combining with the above,
\begin{align*}
    |\val(H,f)|&\leq
\prod_{i=1}^m \|X\s_i\|_2\prod_{v\in\cA}\|{-}\lambda\sigma^{(\deg(v))}(X\w_{i_v})\|_2
\times \|XX^*\|_\op^{|\cE(\cA)|}\|T\|_\cU\\
    &\leq C^{m+2(l-k-1)}|\lambda|^{l-k}\left(\prod_{i=1}^m\|\s_i\|_2\right)\|T\|_\cU
\end{align*}
where we applied that $\cA$ is a tree on $l-k$ vertices so
$\sum_{v\in\cA}\deg(v)=2|\cE(\cA)|=2(l-k-1)$. Substituting this back into the
series expansion of Lemma \ref{lem:series expansion of cumulant}, and applying
$|\cG(G)| \leq 2^lk!$, a
similar counting argument as in the proof of Lemma \ref{lem:CGF converges
absolutely} gives
\begin{align*}
    |\<\kappa_k(\w),\s_1\otimes\cdots\otimes\s_m\otimes
T\>|&\leq\sum_{l=k+1}^\infty\frac{1}{l!}\sum_{G\in\cT_l}\sum_{(H,f)\in\cG(G)}
C^{m+2(l-k-1)}|\lambda|^{l-k}\left(\prod_{i=1}^m\|\s_i\|_2\right)\|T\|_\cU\\
    &\leq
k!C^m\left(\prod_{i=1}^m\|\s_i\|_2\right)\|T\|_\cU\sum_{l=k+1}^\infty\frac{1}{l^2}{C'}^l|\lambda|^{l-k}.
\end{align*}
For $|\lambda|<\lambda_*$ sufficiently small, this is at most
$C_k(\prod_{i=1}^m\|\s_i\|_2)\|T\|_\cU$ for a constant $C_k>0$,
completing the proof.
\end{proof}

\subsection*{Acknowledgments} RM would like to thank Garrett G. Wen for many insightful discussions on the random features model, and Jing Guo for a helpful discussion on elements of graph theory. ZF was supported in part by NSF DMS-2142476 and a Sloan Research Fellowship. EP was supported by an NSERC Discovery Grant RGPIN-2025-04643, an FRQNT-NSERC NOVA Grant, a CIFAR Catalyst Grant, and a gift from Google Canada.

\printbibliography

\end{document}